\documentclass[11pt]{amsart}
\usepackage{amsthm,amsfonts,amssymb,amsmath,oldgerm}
\numberwithin{equation}{section}
\usepackage{fullpage}
\usepackage{mathtools}
\usepackage{color}
\usepackage{amsfonts}
\usepackage{calrsfs}
\usepackage{graphicx}
\DeclareMathAlphabet{\pazocal}{OMS}{zplm}{m}{n}

\renewcommand\d{\partial}

\def\eps{\varepsilon }

\renewcommand\d{\partial}

\def\eps{\varepsilon}

\newcommand\br{\begin{remark}}
	\newcommand\er{\end{remark}}
\newcommand\bp{\begin{pmatrix}}
	\newcommand\ep{\end{pmatrix}}
\newcommand{\be}{\begin{equation}}
	\newcommand{\ee}{\end{equation}}
\newcommand\ba{\begin{equation}\begin{aligned}}
		\newcommand\ea{\end{aligned}\end{equation}}

\newcommand{\bap}{\begin{app}}
	\newcommand{\eap}{\end{app}}
\newcommand{\begs}{\begin{exams}}
	\newcommand{\eegs}{\end{exams}}
\newcommand{\beg}{\begin{example}}
	\newcommand{\eeg}{\end{exaplem}}
\newcommand{\bpr}{\begin{proposition}}
	\newcommand{\epr}{\end{proposition}}
\newcommand{\bt}{\begin{theorem}}
	\newcommand{\et}{\end{theorem}}
\newcommand{\bc}{\begin{corollary}}
	\newcommand{\ec}{\end{corollary}}
\newcommand{\bl}{\begin{lemma}}
	\newcommand{\el}{\end{lemma}}
\newcommand{\bd}{\begin{definition}}
	\newcommand{\ed}{\end{definition}}
\newcommand{\brs}{\begin{remarks}}
	\newcommand{\ers}{\end{remarks}}

\newcommand{\A }{\mathcal{A}}

\newcommand{\RR}{{\mathbb R}}

\newcommand{\NN}{{\mathbb N}}
\newcommand{\ZZ}{{\mathbb Z}}

\newcommand{\CC}{{\mathbb C}}

\newcommand{\oM}{{\overline{M}}}
\newcommand{\onu}{{\overline{\nu}}}
\newcommand{\omm}{{\overline{\omega}}}
\newcommand{\tom}{{\widetilde{\omega}}}
\newcommand{\teps}{{\widetilde{\varepsilon}}}
\newcommand{\tih}{{\widetilde{h}}}
\newcommand{\oor}{\mathrm{ord}}
\newcommand{\tcS}{{\widetilde{\mathcal{S}}}}

\newcommand{\rml}{{\mathrm{l}}}
\newcommand{\ta}{{\widetilde{a}}}
\newcommand{\ttr}{{\widetilde{r}}}
\newcommand{\tp}{{\widetilde{p}}}
\newcommand{\tc}{{\widetilde{c}}}
\newcommand{\ts}{{\widetilde{s}}}
\newcommand{\tx}{{\widetilde{x}}}
\newcommand{\tA}{{\widetilde{A}}}

\newcommand{\tF}{{\widetilde{F}}}
\newcommand{\tf}{{\widetilde{f}}}
\newcommand{\tM}{{\widetilde{M}}}
\newcommand{\tN}{{\widetilde{N}}}

\newcommand{\tQ}{{\widetilde{Q}}}
\newcommand{\tGamma}{{\widetilde{\Gamma}}}

\newcommand{\nK}{{K}}
\newcommand{\mfkp}{{\mathfrak{p}}}
\newcommand{\mfkq}{{\mathfrak{q}}}
\newcommand{\tmfkp}{{\widetilde{\mathfrak{p}}}}
\newcommand{\tmfkq}{{\widetilde{\mathfrak{q}}}}

\newcommand{\tcE}{{\widetilde{\mathcal{E}}}}
\newcommand{\hcE}{{\widehat{\mathcal{E}}}}

\newcommand{\rmi}{{\mathrm{i}}}
\newcommand{\rmd}{{\mathrm{d}}}
\newcommand{\rms}{{\mathrm{s}}}
\newcommand{\rmu}{{\mathrm{u}}}
\newcommand{\rmc}{{\mathrm{c}}}
\newcommand{\os}{{\overline{s}}}

\newcommand{\na}{{\alpha}}
\newcommand{\oa}{{\overline{a}}}
\newcommand{\PV}{{\mathit{p.v.}}}
\newcommand{\vk}{{\varkappa}}
\newcommand{\Ooor}{{\Omega_*\setminus\big(\{\mu_k:k=1,\cdots,m\}\cup\{\nu_\ell:\ell=1,\dots,p\}\big)}}

\newcommand{\re}{{\mathrm{Re}}}
\newcommand{\im}{{\mathrm{Im}}}
\newcommand{\Ker}{{\mathrm{Ker}}}

\newcommand{\Range}{{\rm Range }}
\newcommand{\Graph}{{\rm Graph }}

\newtheorem{theorem}{Theorem}[section]
\newtheorem{proposition}[theorem]{Proposition}
\newtheorem{corollary}[theorem]{Corollary}
\newtheorem{lemma}[theorem]{Lemma}

\theoremstyle{remark}
\newtheorem{remark}[theorem]{Remark}
\theoremstyle{definition}
\newtheorem{definition}[theorem]{Definition}

\newtheorem{example}[theorem]{Example}

\newcommand\cA{{\mathcal A}}
\newcommand\cB{{\mathcal B}}

\newcommand\cU{{\mathcal U}}

\newcommand\cW{{\mathcal W}}
\newcommand\cC{{\mathcal C}}

\newcommand\cR{{\mathcal R}}
\newcommand\cG{{\mathcal G}}
\newcommand\cK{{\mathcal K}}

\newcommand\cL{{\mathcal L}}

\newcommand\cE{{\mathcal E}}
\newcommand\cF{{\mathcal F}}
\newcommand\cN{{\mathcal N}}

\newcommand\cM{{\mathcal M}}
\newcommand\cT{{\mathcal T}}
\newcommand\cS{{\mathcal S}}

\newcommand\bH{{\mathbb H}}

\newcommand\bV{{\mathbb V}}
\newcommand\bW{{\mathbb W}}

\newcommand\bL{{\mathbb L}}

\newcommand\bX{{\mathbb X}}
\newcommand\bY{{\mathbb Y}}

\newcommand{\M}{\,\mbox{\bf M}}

\newcommand{\tE}{{\widetilde{E}}}

\newcommand{\tR}{{\widetilde{R}}}

\newcommand{\tP}{{\widetilde{P}}}
\newcommand{\tG}{{\widetilde{G}}}
\newcommand{\oP}{{\widehat{P}}}
\newcommand{\oN}{{\widehat{N}}}
\newcommand{\oE}{{\overline{E}}}
\newcommand{\tcK}{{\widetilde{\mathcal{K}}}}
\newcommand{\tcJ}{{\widetilde{\mathcal{J}}}}

\newcommand{\tbX}{{\widetilde{\mathbb{X}}}}

\newcommand{\Res}{{\mathrm{Res}}}
\newcommand{\mi}{{\mathrm{mid}}}
\newcommand{\up}{{\mathrm{up}}}
\newcommand{\down}{{\mathrm{down}}}
\newcommand{\out}{{\mathrm{out}}}
\newcommand{\loc}{\text{\rm{loc}}}

\newcommand{\dom}{\text{\rm{dom}}}

\newcommand{\beq}{\begin{equation}}
	\newcommand{\eeq}{\end{equation}}
\allowdisplaybreaks

\title{Solutions of abstract wave equations, eigenvalues and resonances}

\author{ Yuri Latushkin}
\address{University of Missouri, Columbia, MO 65211}
\email{latushkiny@missouri.edu}
\thanks{\hspace*{-0.17in}\textbf{AMS MSC 2020 Mathematics Subject Classification:} 37L15, 81U24, 47D06, 34G10Y.L. was supported by the NSF grants DMS-2106157.\\A. P. research  was partially supported under the Simons Foundation Grant nr. 524928.}
\author{Alin Pogan}
\address{Miami University, Oxford, OH 45056}
\email{pogana@miamioh.edu}
\begin{document}

\begin{abstract}
We prove general representation formulas for strongly continuous cosine and sine operator families in terms of scattering resonances of their generators. This generalizes known results related to decay, growth and oscillatory behavior of solutions of abstract wave equations to a wide class of non-self-adjoint operators in Banach spaces. Inspired by the classical results on scattering resonances for Schr\"odinger operators with compactly supported potentials, we develop quite general abstract scheme of resonances that involves extensions of the resolvent operators from Banach to Frechet spaces. We split the solutions of the wave equations in two parts: The first part is related to finite rank operators induced by the eigenvalues and resonances while the second part involves a partial inversion of the Laplace transform whose exponential behavior is effectively controlled. Illustrations and applications cover a wide class 
of generators including the Schr\"odinger operators with non-symmetric complex matrix potentials, linearizations of nonlinear wave equations, Aharonov-Bohm and block-box Hamiltonians, etc.
\end{abstract}
	
\maketitle

	\vspace{0.3cm}
	\begin{minipage}[h]{0.48\textwidth}
		\begin{center}
			University of Missouri \\
			Department of Mathematics\\
			810 East Rollins Street\\ Columbia, MO 65211, USA
		\end{center}
	\end{minipage}
	\begin{minipage}[h]{0.48\textwidth}
		\begin{center}
			Miami University\\
			Department of Mathematics\\
			100 Bishop Circle\\
			Oxford, OH 45056, USA
		\end{center}
	\end{minipage}
	
	\vspace{0.3cm}

	%\tableofcontents
	%{\scriptsize{\tableofcontents}}

\section{Introduction }\label{s1}

In this paper we use eigenvalues and scattering resonances to 
build a general representation formula for the mild solutions 
of the abstract second order Cauchy problem,
\begin{equation}\label{abstract-Cauchy}
\begin{cases}
u''(t)=Au(t), &  t\in\RR,\\
\,\, u(0)=x_0, \,
u'(0)=x_1, & x_0,x_1\in\bX, \end{cases}
\end{equation}
where $A$ is a linear operator acting in a Banach space $\bX$. We develop a general abstract scheme, involving Frechet spaces, that allows us to study meromorphic extensions of the resolvent of $A$ and to represent  the solution of the second order  Cauchy problem as a sum of two parts: The first part is given by effectively computed finite rank operators induced by the eigenvalues and resonances of $A$ while the second part involves partial inverse Laplace transforms and has effectively controlled exponential behavior. The representation formula generalizes known results to a wide class of non-self-adjoint operators $A$.

To describe our strategy, we first notice that a fundamental problem in the study of second order abstract Cauchy problems in Banach spaces is indeed to find representations of solutions that help to describe their oscillations and long time growth or decay. In the case of the classical wave equation 
$u_{tt}=u_{xx}$ on a compact interval one can use  Fourier series to find a representation of solutions. If a real-valued potential is added in the equation, the problem can still be solved using the same method by passing to the eigenvalues of the  associated Schr\"odinger operator. More generally, in the case of non-compact odd-dimensional domains, expansions of scattered waves are obtained by studying scattering resonances and their properties. 
This topic has a decades long history and can be traced to Lax and Phillips \cite{LaxPhillips} or Vainberg \cite{V}. 
Results of this type are ubiquitous in the literature, and have been proven  for various classes of the second order wave equations  \eqref{abstract-Cauchy}  with $A$ being a differential operator, see \cite{BZ,CD1,CDY1,CDY2,CZ,Sjostrand,Z1,Z2,Z3} and references therein. For a very interesting and comprehensive treatment of these problems we refer to the recent book by Dyatlov and Zworski \cite{DZ} which was our inspiration in writing the current paper.

In this paper we assume that $A:\dom(A)\subseteq \bX\to\bX$ is the generator of a \textit{cosine family}, denoted $\{C(t)\}_{t\in\RR}$, on the Banach space $\bX$, that is,  $C:\RR\to\cB(\bX)$ is a strongly continuous operator-valued function so that  $(\omega^2,\infty)\subset\rho(A)$ for some $\omega>0$, $\lambda R_A(\lambda)x=\displaystyle\int_0^\infty e^{-\lambda t}C(t)x\rmd t$ for any $x\in\bX$ and
\begin{equation}\label{cosine-def}
	C(0)=I_\bX,\; C(t+s)+C(t-s)=2C(t)C(s)\;\mbox{for any}\; t,s\in\RR;	
\end{equation}	
here and in what follows we denote  the resolvent of $A$ at $\lambda^2$ by
\begin{equation}\label{def-RA}
	R_A(\lambda):=(\lambda^2-A)^{-1},\, \lambda\in\CC_\omega^+:=\{\lambda\in\CC: \re(\lambda)>\omega\}.
\end{equation}
The mild solution of \eqref{abstract-Cauchy} 
is then given by the formula $u(t)=C(t)x_0+S(t)x_1$, $t\in\RR$, where $S:\RR\to\mathcal{B}(\bX)$
is the associated \textit{sine family}  defined by
\begin{equation}\label{sine-def}
	S(t)x=\int_0^tC(s)x\rmd s \text{ for any $x\in\bX$}.	
\end{equation}	
There is a vast literature on well-posedness of the second order Cauchy problems and related topics. We refer to \cite[Section 3.14]{ABHN} for a concise introduction and to \cite{CiKe,deLP,Fattorini,Ki,Kr,TW} for more detailed exposition. The main facts that we need are also reminded in the next section.

To find the
expansions of the solutions $u$ of the wave equation in terms of the eigenvalues and resonances of the generator $A$, we therefore need to find expansions for the cosine family $\{C(t)\}_{t\in\RR}$ and its associated sine family $\{S(t)\}_{t\in\RR}$. To achieve this we need more information on the function
$R_A(\cdot)$ originally defined in \eqref{def-RA}. Since $A$ is a generator of a cosine family, 
it is well-known that $R_A(\cdot)$ can be defined on a set that contains the open right-half plane $\CC_\omega^+$ for some $\omega\geq 0$. %Moreover, it is straightforward to control the decay for large $\re\lambda$. 
Also, in many instances one can show that the operator-valued function $R_A(\cdot)$ could be defined on a subset of $\CC$ which is larger than $\CC_\omega^+$. In particular, $R_A(\cdot)$ can be extended \textit{meromorphically} to a set of the form $\{\lambda\in\CC:\re\lambda>g_0(\im\lambda)\}$, where $g_0:\RR\to\RR$ is a smooth function of our choosing such that $\sup g_0<\omega$. In this larger set $R_A(\cdot)$ has only finitely many \textit{meromorphic poles} the set of which we denote by $\cM$. Typically, the function $g_0$ is chosen such that the curve $\{\lambda\in\CC: \re\lambda=g_0(\im\lambda)\}$ is contained in $\{\lambda\in\CC:\lambda^2\in\rho(A)\}$, and is located to the right of (and as close as possible to) the set $\{\lambda\in\CC:\lambda^2\in\sigma_{\mathrm{ess}}(A)\}$.  
In most examples (think $\sigma_{\mathrm{ess}}(\partial_{xx})=(-\infty,0]$ for the Laplacian in $L^2(\RR)$), we expect the set $\{\lambda\in\CC:\lambda^2\in\sigma_{\mathrm{ess}}(A)\}$ to contain at least parts of the imaginary axis. This affects the choice of $g_0$ and so we expect that $g_0(s)$ takes positive values at least for  large $|s|$.

One of our important tasks is to identify the leading order in $t$ terms in the representation of the cosine and sine functions. A typical (and the simplest) formula proved in this paper is of the form
\begin{equation}\label{representation-aim1}
	C(t)x=\sum_{\mu\in\cM}e^{\mu t}P^\rmc_\mu(t)x +\chi_{[0,\infty)}(g_0(0))\displaystyle\sum\limits_{j=0}^{n-1}\frac{t^{2j}}{(2j)!}A^jx+\frac{1}{2\pi\rmi}\displaystyle\int_{\Lambda_0^\eps}\frac{e^{\lambda t}}{\lambda^{2n-1}}R_A(\lambda)A^nx\rmd\lambda,
\end{equation}
where $t\in\RR$, $x\in\dom(A^n)$, $\eps>0$ is small enough and $\chi_{[0,\infty)}$ is the characteristic function of $[0,\infty)$. Here $\Lambda_0^\eps$ is the path in $\CC$ along the curve  $\{\lambda\in\CC:\re\lambda=g_0(\im\lambda)+\eps\}$ and $P^\rmc_\mu:\RR\to\cB(\bX)$
is a polynomial whose coefficients are finite rank operators.
The first sum in \eqref{representation-aim1} is the fastest growing in $t$ term, while the integral term in this representation has the slower growth rate. The positive integer $n$ is fixed such that $|\lambda|^{-(2n-1)}\|R_A(\lambda)\|=o(|\lambda|)$ as $|\im\lambda|\to\infty$ for $\lambda$ in the set $\{\lambda\in\CC: g_0(\im\lambda)<\re\lambda\leq\omega\}$ and so in formula \eqref{representation-aim1} we have to balance the ``smoothness'' $n$ of the initial condition and the behavior of $\|R_A(\lambda)\|$ for large values of $|\im\lambda|$.

To prove \eqref{representation-aim1} we proceed as follows. First, if $\mu\in\cM$ then $\mu^2$ belongs to the discrete spectrum of the generator $A$ which allows us to show the existence of and give a precise formula  for the polynomial $P^\rmc_\mu(t)$ that depends on the spectral projection of finite rank associated to the isolated eigenvalue $\mu^2$ of $A$, on the algebraic multiplicity of the eigenvalue,  and the structure of the respective  Jordan blocks. Next, we find an integral representation of $C(t)$ by integrating along a vertical line located to the right of the line $\re\lambda=\omega$. Furthermore, we show that we can switch the contour of integration to a path parallel to the curve $\re\lambda=g_0(\im\lambda)$; this leads to the appearance of the first sum in \eqref{representation-aim1}. Finally, we estimate the integral term in \eqref{representation-aim1} as follows,
\begin{equation}\label{est-cosine1-intro}
	\Big\|\displaystyle\int_{\Lambda_0^\eps}\frac{e^{\lambda t}}{\lambda^{2n-1}}R_A(\lambda)A^nx\rmd\lambda\Big\|\lesssim e^{(\sup g_0+\eps)t}(\|A^{n-1}x\|+\|A^nx\|)\;\mbox{for any}\;x\in\dom(A^n),\eps\thicksim0.
\end{equation}
This estimate shows that the integral term in \eqref{representation-aim1} grows as $t\to\infty$ at most as $\mathcal{O}(e^{(\sup g_0+\eps)t})$ for any $\eps>0$, that is,  has a slower exponential growth rate than the first sum in \eqref{representation-aim1}.

To introduce a representation of the sine function similar to formula \eqref{representation-aim1} we recall the definition of the phase space associated to a cosine family, sometime also called the Kisyński space (see, e.g., \cite{Ki}). We introduce notation
\begin{equation}\label{Kisynski}
\bW=\Big\{x\in\bX:S(\cdot)x\in\mathcal{C}\big([0,1],\dom(A)\big)\Big\},\;\bW_{n-1}=\{x\in\dom(A^{n-1}):A^{n-1}x\in\bW\}
\end{equation}
for $n\in\NN$. A discussion on the norms on the spaces and their properties can be found in Section~\ref{sec2}. Our second aim is to prove the following formula,
\begin{equation}\label{representation-aim2}
S(t)w=\sum_{\mu\in\cM}e^{\mu t}P^\rms_\mu(t)w +\chi_{[0,\infty)}(g_0(0))\displaystyle\sum\limits_{j=0}^{n-1}\frac{t^{(2j+1)}}{(2j+1)!}A^jw+\frac{1}{2\pi\rmi}A\displaystyle\int_{\Lambda_0^\eps}\frac{e^{\lambda t}}{\lambda^{2n}}R_A(\lambda)A^{n-1}w\rmd\lambda,
\end{equation}
for $t\in\RR$, $w\in\bW_{n-1}$ and $\eps>0$ small enough. Here $P^\rms_\mu:\RR\to\cB(\bX)$
is a polynomial whose coefficients are finite rank operators. The formula for $P^\rms_\mu(t)$ is similar to that of $P^\rmc_\mu(t)$  in representation \eqref{representation-aim1} and uses the spectral projection of finite rank associated to the isolated eigenvalue $\mu^2$ of $A$, its algebraic multiplicity and the Jordan blocks.
The integral term can be estimated as follows,
\begin{equation}\label{est-sine1-intro}
\Big\|A\displaystyle\int_{\Lambda_0^\eps}\frac{e^{\lambda t}}{\lambda^{2n}}R_A(\lambda)A^{n-1}w\rmd\lambda\Big\|\lesssim e^{(\sup g_0+\eps)t}\|A^{n-1}w\|_\bW\;\mbox{for any}\;x\in\bW_{n-1}
\end{equation}
and any $\eps>0$ small enough. Therefore, similarly to the representation of the cosine function, in \eqref{representation-aim2} the first sum is the fastest growing term, while the integral term is again of order $\mathcal{O}(e^{(\sup g_0+\eps)t})$ for any $\eps>0$.
Representation \eqref{representation-aim2} works for $w\in\bW_{n-1}$ which is a larger subspace than $\dom(A^n)$ from where we required to take the initial condition in the case of the cosine function. This fact confirms the intuition that the sine functions is smoother than the cosine function, as shown by \eqref{sine-def}.  Also, the positive integer $n$ in formula \eqref{representation-aim2} is the same as the one used in formula \eqref{representation-aim1} and is related to the order of magnitude of $\|R_A(\lambda)\|$ whenever $g_0(\im\lambda)<\re\lambda\leq \omega$ and $|\im\lambda|\to\infty$. To prove \eqref{representation-aim2} we rely on the following property of the phase-space $\bW$: if $w\in\bW$ then $\|AR_A(\lambda)w\|\thicksim{\|R_A(\lambda)w\|}/{|\lambda|}$, whenever $|\im\lambda|$ is large enough, see Lemma~\ref{l3.7} below.

Our next goal is to ``extract" more finite rank terms from the integrals in \eqref{representation-aim1} and \eqref{representation-aim2}. To this end we refine the representations \eqref{representation-aim1} and \eqref{representation-aim2} by finding the leading order finite rank terms in the integrals therein. To achieve this we need to involve the \textit{scattering resonances} of $A$. To explain this, we first note that the function $R_A(\cdot)$ cannot be extended meromorphically as a $\cB(\bX)$-valued function to a set intersecting $\{\lambda\in\CC:\lambda^2\in\sigma_{\mathrm{ess}}(A)\}$. To overcome this obstacle, many methods have been developed, in particular, in the special case when $A$ is a differential operator in $L^2(\RR^m)$. In this case the function $R_A(\cdot)$ is an integral operator whose Green's integral kernel is often meromorphic on a large subset of $\CC$, thus allowing us to extend the operator-valued function $R_A(\cdot)_{|L^2_{\mathrm{comp}}(\RR^m)}$ to a meromorphic function defined on a subset of $\CC$ containing $\{\lambda\in\CC:\lambda^2\in\sigma_{\mathrm{ess}}(A)\}$. The extended function, however,  is
taking values in the set of continuous linear operators $\cC\cL\big(L^2_{\mathrm{comp}}(\RR^m),L^2_{\mathrm{loc}}(\RR^m)\big)$, that is, even in the case of Schr\" odinger operators acting in $L^2(\RR)$, the values of the meromorphic extension are not bounded linear operators from say, $L^2_{\mathrm{comp}}(\RR)$ to $L^2(\RR)$ or to some other normed space. 

A proper abstract generalization of this idea is one of the main points of the current paper. In the general case  when $A$ is the generator of an abstract cosine family in a Banach space $\bX$, we need to find $\bX_\infty$, a ``localizing'' subspace of $\bX$, and a \textit{Frechet} space $\bY$, such that $\bX\hookrightarrow\bY$ and the function $R_A(\cdot)_{|\bX_\infty}$ has a meromorphic extension, denoted $R_A^\infty(\cdot)$, to a subset of $\CC$ of the form $\{\lambda\in\CC:\re\lambda>g_*(\im\lambda)\}$; here, the extended function takes values $R_A^\infty(\lambda)$ that are operators from $\cC\cL(\bX_\infty,\bY)$. The function $g_*$ is chosen such that $g_*<g_0$, the meromorphic extension $R_A^\infty(\cdot)$ has finitely many meromorphic poles of finite algebraic multiplicity in the set $\{\lambda\in\CC:g_*(\im\lambda)<\re\lambda<g_0(\im\lambda)\}$ and the growth of $R_A^\infty(\cdot)$ for large $|\lambda|$ can be controlled in the natural Frechet seminorms of the space $\cC\cL(\bX_\infty,\bY)$. We have less restrictions in the choice of $g_*$ than that of $g_0$ since we may now include parts of the set $\{\lambda\in\CC:\lambda^2\in\sigma_{\mathrm{ess}}(A)\}$ in the set $\{\lambda\in\CC:\re\lambda>g_*(\im\lambda)\}$.

As a result, in this paper we propose a general framework that assumes that 
the Banach space $\bX$ can be embedded continuously into a Frechet space $\bY$ and there exists a sequence $\{Q_n\}_{n\in\NN}$ of linear continuous operators from $\bY$ to $\bX$ such that the topology on $\bY$ is equivalent to the topology induced by the family of seminorms $\{\mfkq_n\}_{n\in\NN}$ defined by $\mfkq_n:\bY\to[0,\infty)$, $\mfkq_n(y):=\|Q_ny\|_\bX$, $n\in\NN$. Metaphorically, the operators $Q_n$ mimic the operators of multiplication by smooth scalar cut-off  functions on $L^p(\RR^m)$ or other spaces of functions;  they must have several compatibility properties with respect to the generator $A$, see Hypothesis (Q) and Example~\ref{e4.1} in Section~\ref{sec4}. Also, we show that one can always choose the localizing subspace $\bX_\infty$ by letting $\bX_\infty=\bigcup_{n=1}^{\infty}\Ker(Q_n-I)$. The general framework of this paper covers important classes of non-self-adjoint differential operators acting in $L^p(\RR^m)$ or its subspaces so that  $\bX_\infty$ and $\bY$ mimic $L^p_{\mathrm{comp}}(\RR^m)$ and $L^p_{\mathrm{loc}}(\RR^m)$, respectively. 

Next, to construct the desired representation of the cosine and sine functions, we need to show that the generator $A$ has an extension, denoted by $\tA$, to the Frechet space $\bY$, which is a closed operator in the topology of the Frechet space. It is important to show that $\dom(\tA)$,  the domain of $\tA$, contains the range of the operator $R_A^\infty(\lambda)$ whenever $\re\lambda>g_*(\im\lambda)$. In addition, it is crucial to show that $Q_n(\dom(\tA))\subseteq\dom(A)$ for any $n\in\NN$, that is, that $Q_n$ ``localizes" the domain of $\tA$ by mapping it into the domain of $A$, similarly to how $Q_n$ localizes the Frechet space $\bY$ into $\bX$. To achieve this, we show that $A$ is closable in the Frechet space topology on $\bY$ by using a classical result by Browder \cite{B}. We refer to Section~\ref{sec4} for details.

Once the construction outlined above is set up, we are ready to refine the expansions for the cosine and sine functions. We denote the set of finitely many \textit{meromorphic poles} of the extension $R_A^\infty(\cdot)$ by $\cN$. Our final goal is to improve formula \eqref{representation-aim1} and to show that
\begin{equation}\label{representation-aim3}\begin{split}
		C(t)x=\sum_{\mu\in\cM}e^{\mu t}P^\rmc_\mu(t)x
		&+\sum_{\nu\in\cN}e^{\nu t}\tP^\rmc_\nu(t)x\\&+\chi_{[0,\infty)}(g_*(0))
		\displaystyle\sum\limits_{j=0}^{n-1}\frac{t^{2j}}{(2j)!}A^jx+\frac{1}{2\pi\rmi}\displaystyle\int_{\Lambda_*^\eps}\frac{e^{\lambda t}}{\lambda^{2n-1}}R_A^\infty(\lambda)A^nx\rmd\lambda\end{split}
\end{equation}
for all $t\in\RR$, $x\in\dom(A^n)\cap\bX_\infty$ and $\eps>0$ small enough.
Here $\Lambda_*^\eps$ is the path along the curve  $\re\lambda=g_*(\im\lambda)+\eps$. The polynomial $P^\rmc_\mu$ is the same as in formula \eqref{representation-aim1}, $\tP^\rmc_\nu:\RR\to\cC\cL(\bX_\infty,\bY)$
is yet another polynomial whose coefficients are finite rank operators. Unlike the coefficients of $P_\mu^\rmc$, the coefficients of $\tP_\nu^\rmc$ cannot be expressed  via some spectral projections. This difference is due to the fact that, in general, if $\nu\in\cN$ is a finitely meromorphic resonance then $\nu^2$ is not necessarily a discrete eigenvalue of $A$, see, e.g., \cite{Hislop}. However, for any $\nu\in\cN$, we can still find finitely many generalized eigenvectors of the extension $\tA$ of $A$ which allows us to conclude that the coefficients of $\tP_\nu^\rmc$ are operators of finite rank.  When $0$ is a discrete eigenvalue of $A$ then $P_0^\rmc(t)$ has only even in $t$ terms, while if $0$ is a resonance then $\tP_0^\rmc(t)$ might have even and odd in $t$ terms (this does happen when $A$ is a one dimensional Schr\" odinger operator). The precise formula for $\tP_\nu^\rmc$ is given in Section~\ref{sec5}.

The first two sums  in formula \eqref{representation-aim3}
make up the finite rank leading order in $t$ part of the decomposition of $C(t)$, while the integral part is of lower order.
More precisely, under natural assumptions, we prove the estimate
\begin{equation}\label{est-cosine2-intro}
	\Big\|Q_i\displaystyle\int_{\Lambda_*^\eps}\frac{e^{\lambda t}}{\lambda^{2n-1}}R_A^\infty(\lambda)A^nx\rmd\lambda\Big\|\lesssim e^{(\sup g_*+\eps)t}\|A^{n-1}x\|
\end{equation}
that holds for any $x\in\dom(A^n)\cap\Ker(Q_j-I)$, $i,j\in\NN$, $\eps>0$ small enough and $t$ large enough. 
Generically, we expect the function $g_*$ to be chosen such that $\{\lambda\in\CC:\re\lambda>g_*(\im\lambda)\}$ contains some parts of the set $\{\lambda\in\CC:\lambda^2\in\sigma_{\mathrm{ess}}(A)\}$. Therefore,
in many examples the function $g_*$ can be chosen such that $\sup g_*<0$, thus making the integral term in \eqref{representation-aim3} exponentially decaying as $|t|\to\infty$. This last estimate is in an appropriate seminorm, since the meromorphic extension $R_A^\infty(\cdot)$ is a function that takes values in the Frechet space of linear and continuous operators from $\bX_\infty$ to $\bY$. In addition, if $\sup g_*<0$, then the third term in \eqref{representation-aim3} vanishes. The positive integer $n$ is the same as in \eqref{representation-aim1}. Depending on the shape of the graph of the scalar function $g_*$, formula \eqref{representation-aim3} might hold only provided we multiply all terms from the left by the operator $Q_i$, $i\in\NN$, and for $x\in\dom(A^n)\cap\Ker(Q_j-I_\bY)$, $j\in\NN$, and only when $|t|$ is bigger than a positive constant $\beta_{i,j}$. This is a common occurrence in the theory of scattering resonances for differential operators, see, e.g., \cite{CD1, CDY1,CDY2, DZ, LaxPhillips,Sjostrand,TZ,V,Z1,Z2,Z3}. In Section~\ref{sec6} we give an estimate for the integral term in \eqref{representation-aim1} in the natural Frechet space seminorm for two distinct cases: First, when  $g_*(s)$ approaches $-\infty$ as $|s|\to\infty$ at least logarithmically and, second, in the case when $g_*,g_*'\in L^\infty(\RR)$.

Finally, we can refine representation \eqref{representation-aim2} using the resonances of the operator $A$ and the meromorphic extension $R_A^\infty(\cdot)$ as follows,
\begin{equation}\label{representation-aim4}\begin{split}
		S(t)w&=\sum_{\mu\in\cM}e^{\mu t}P^\rms_\mu(t)w
		+\sum_{\nu\in\cN}e^{\nu t}\tP^\rms_\nu(t)w+\chi_{[0,\infty)}(g_*(0))\displaystyle\sum\limits_{j=0}^{n-1}\frac{t^{(2j+1)}}{(2j+1)!}A^jw\\&\qquad\qquad\qquad\qquad+\tA\displaystyle\int_{\Lambda_*^\eps}\frac{e^{\lambda t}}{\lambda^{2n}}R_A^\infty(\lambda)A^{n-1}w\rmd\lambda,
\end{split}\end{equation}
for $t\in\RR$ and $w\in\bW_{n-1}\cap\bX_\infty$ and $\eps\thicksim0$.
The polynomial $P^\rms_\mu(t)$ is the same as in formula \eqref{representation-aim2}, $\tP^\rms_\nu:\RR\to\cC\cL(\bX_\infty,\bY)$
is a polynomial whose coefficients are finite rank operators. An important difference from representation \eqref{representation-aim2} is that in representation \eqref{representation-aim4} we have the operator $\tA$ instead of $A$ in front of the integral. This is necessary since the meromorphic extension $R_A^\infty(\cdot)$ takes values in the Frechet space of linear and continuous operators from $\bX_\infty$ to $\bY$.  There are several similarities between formulas \eqref{representation-aim2} and \eqref{representation-aim3}. The formula for the coefficients $\tP_\nu^\rms$ is similar to the one for $\tP_\nu^\rmc$, including the case when $\nu=0$ is a resonance. The smoothness of the initial condition $w$ in \eqref{representation-aim4} is the same as in \eqref{representation-aim2}, with the added condition that $w$ needs to belong to the localizing subspace $\bX_\infty$. The estimate similar to \eqref{est-cosine2-intro} is as follows, 
\begin{equation}\label{est-sine2-intro}
	\Big\|Q_i\tA\displaystyle\int_{\Lambda_*^\eps}\frac{e^{\lambda t}}{\lambda^{2n}}R_A^\infty(\lambda)A^{n-1}w\rmd\lambda\Big\|\lesssim e^{(\sup g_*+\eps)t}\|A^{n-1}w\|;\end{equation}
it holds for any $x\in\bW_{n-1}\cap\Ker(Q_j-I)$, $i,j\in\NN$,
$\eps>0$ small enough and $t$ large enough.  We see from \eqref{est-sine2-intro} that the first two sums in \eqref{representation-aim4} are of leading order in $t$ while the integral term is of lower order. Similarly to representation \eqref{representation-aim3}, when $\sup g_*<0$ the last term of \eqref{representation-aim4} is exponentially decaying as $|t|\to\infty$ and the third term vanishes. In addition, depending on the behavior of $g_*(s)$ at $s\to\pm\infty$, formula \eqref{representation-aim4} holds only  provided we multiply all terms from the left by the operator $Q_i$, $i\in\NN$, and for $w\in\bW_{n-1}\cap\Ker(Q_j-I_\bY)$, $j\in\NN$, and only when $|t|$ is larger than a positive constant $\beta_{i,j}$. The Frechet space estimates for the last term of \eqref{representation-aim4} are more involved due to the fact that we multiply the integral term by the closed but not bounded operator $\tA$.

One of the big advantages in using formulas proved in this paper is that they allow one to identify in the representations of the cosine and sine functions as many  finite rank pieces as possible and estimate the exponential growth and decay of the remaining terms. 
Indeed, it is possible that  $\re\mu_k<\sup g_0$ for some pole $\mu_k\in\cM$. Therefore, $e^{\mu_k t}P^{\rmc}_{\mu_k}(t)$, respectively, $e^{\mu_k t}P^{\rms}_{\mu_k}(t)$ for this particular $\mu_k$ might grow slower that the integral term in representations \eqref{representation-aim1}, respectively, \eqref{representation-aim2}. However, the overall sum $\sum_{\mu\in\cM}e^{\mu t}P^{\rmc/\rms}_\mu(t)$ is growing faster than the integral term. Since $P^{\rmc/\rms}_{\mu_k}(t)$ is a polynomial with operator coefficients of finite rank  it is advantageous to isolate the term $e^{\mu_k t}P^{\rmc/\rms}_{\mu_k}(t)$ from the integral part of the cosine function representation. Similarly,  if $\re\nu_j<\sup g_*$ for some resonance $\nu_j\in\cN$ then $e^{\nu_j t}\tP^{\rmc/\rms}_{\nu_j}(t)$ might be a lower order in $t$ term when compared to the integral terms in \eqref{representation-aim3}/\eqref{representation-aim4}. But, again, the overall sum $\sum_{\nu\in\cN}e^{\nu t}\tP^{\rmc/\rms}_\nu(t)$ is of higher order in $t$ when compared to the integral terms in the representations. 

Formulas \eqref{representation-aim1}, \eqref{representation-aim2}, \eqref{representation-aim3} and \eqref{representation-aim4} show that we can evaluate the leading order terms of the cosine and sine functions, and therefore of the solution of \eqref{abstract-Cauchy}, by studying the spectral properties of the generator $A$. Indeed, by finding the eigenvalues, the resonances, their multiplicities, identifying generalized sets of eigenfunctions, we can compute the leading order terms in the formulas for cosine and sine functions and then use \eqref{est-cosine1-intro}, \eqref{est-sine1-intro}, \eqref{est-cosine2-intro} and \eqref{est-sine2-intro} to estimate the error terms. This is very useful especially in the case when there are no simple formulas readily available for the cosine and sine functions. Interestingly, the representations for the sine function are slightly more involved than the corresponding representations for the cosine functions. This fact is somewhat counter intuitive since the sine functions is smoother due to \eqref{sine-def}. The main explanation of this phenomenon is the fact that \eqref{representation-aim2} and \eqref{representation-aim4} involve the unbounded operators $A$ and $\tA$, respectively. In addition, in both representations we need to work on the phase-space $\bW$ and its generalization $W_{n-1}$ defined in \eqref{Kisynski}. 

When $A$ acts in a Hilbert space and is self-adjoint, many authors have obtained representations of the solutions of \eqref{abstract-Cauchy} using the spectral decomposition of $A$, in particular, Stone's formula. This method was also generalized to the case of black box Hamiltonians, see, e.g., \cite[Chapter 4]{DZ} or \cite{CD1}. Instead, in the current paper in the more general case of the generators of cosine families in a Banach space, in Subsection~\ref{subsec2.1} we use complex inversion results for the Laplace transform to obtain a representation of the cosine and sine function as complex integrals along the vertical line $\re\lambda=\omega$.  Once this representation is established, we use the residues theorem to shift the contour of integration first to the curve $\Lambda_0^\eps$ and then to the curve $\Lambda_*^\eps$. This leads to representations \eqref{representation-aim1}, \eqref{representation-aim2}, \eqref{representation-aim3} and \eqref{representation-aim4}. The advantage of this method is that we impose only minimal conditions on the generator $A$. 

We chose to distinguish between the poles $\mu$'s of the resolvent  $R_A(\cdot)$ located in the set $\Omega_0$ defined as the maximal subset of $\CC$ where the function  $R_A(\cdot)$ is meromorphic and takes values in $\cB(\bX)$,  and the poles  $\nu$'s of its extension $R_A^\infty(\cdot)$ located outside of $\Omega_0$. We use the term \textit{resonances} only for the
latter poles. We make this distinction because if $\mu$ is a pole of $R_A(\cdot)$ then $\mu^2$ is a discrete eigenvalue of the generator $A$, a fact that has consequences for the coefficients of the operator-valued polynomials $P^{\rmc/\rms}_\mu(t)$ in \eqref{representation-aim1}, \eqref{representation-aim2}, \eqref{representation-aim3} and \eqref{representation-aim4}. 

In the current paper we work with generators $A$ of cosine families.  As it is well known, see, e.g., \cite[Theorem 3.14.17]{ABHN}, each such $A$ generates an analytic semigroups. By this reason we write, say,  $A=\partial_{xx}+V$ rather then $A=-\partial_{xx}+V$ as it is done in, e.g., \cite{DZ}. As a consequence, the spectral parameter $z$ with $\im z>0$ in the eigenvalue equation $Au=z^2u$ in, say, \cite{DZ}
and the spectral parameter $\lambda$ with $\re\lambda>0$ in the eigenvalue equation $Au=\lambda^2u$ in the current paper are related via $\lambda=-\rmi z$. Because of  this change of variables, one should keep in mind to rotate counterclockwise by $\pi/2$ the graphs of functions of the spectral parameter found in the literature on scattering resonances to match them with the graphs used in the current paper.  

\noindent\textbf{Plan of the paper.} In Section~\ref{sec2} we review several elementary properties of cosine families, the associated sine families and their generators. In particular, we introduce the phase space associated to a cosine family, and prove integral formulas for the cosine and sine functions. In Section~\ref{sec3} we prove \eqref{representation-aim1} and \eqref{representation-aim2}. In Section~\ref{sec4} we prove the functional analytic results that allow us to define the ``localizing'' space $\bX_\infty$, which, in our abstract setting, plays the role of $L^p_{\mathrm{loc}}(\RR^m)$ for the case of differential operators. Also, we discuss the extension of the generator $A$ to a closed linear operator $\tA$ acting in the Frechet space $\bY$, as explained above. In Section~\ref{sec5} we study the properties of the finitely meromorphic extension $R_A^\infty(\cdot)_{|\bX_\infty}$ and prove \eqref{representation-aim3} and \eqref{representation-aim4}. In Section~\ref{sec6} we prove the estimates  \eqref{est-cosine2-intro} and \eqref{est-sine2-intro}. 
The main results of this paper are summarized in Theorem~\ref{t4.24}, Theorem~\ref{t4.25}, Theorem~\ref{t4.31} and Theorem~\ref{t4.32}. In Section~\ref{sec7} we discuss several classes of examples that fit our framework. In Appendix~\ref{Appendix A} we present some Laplace transform results needed to prove the cosine and sine family representations from Section~\ref{sec2}. In Appendix~\ref{Appendix B} we review results concerning holomorphic functions taking values in Frechet spaces. 

\noindent\textbf{A glossary of notation.}
$\bX$ is a Banach space with norm $\|\cdot\|$, $L^p(\RR^m,\bX)$, $p\geq 1$, denotes the Lebesgue space of functions on $\RR^m$ with values in $\bX$ with the Lebesgue measure $\rmd x$, 
$W^{s,p}(\RR^m,\bX)$, $s> 0$, $p\geq 1$, is the Sobolev space of $\bX$-valued functions, $\mathcal{C}^m(\Omega,\bX)$ is the space of functions on $\Omega$ having continuous derivatives of order up to $m\in\NN\cup\{0\}$. The space of locally $p$-integrable functions with values in a Frechet space $\bY$ is denoted by $L^p_{\mathrm{loc}}(\RR^m,\bY)$. We denote by $L^p_{\mathrm{comp}}(\RR^m,\bY)$ the set of $p$-integrable, $\bY$-valued functions having compact support. 
The identity operator on the Banach space $\bX$ is denoted by $I_\bX$. The set of bounded linear operators in the Banach space $\bX$ is denoted by $\cB(\bX)$. The space of continuous linear operator between two topological vector spaces $\bV_1$ and $\bV_2$ is denoted by $\cC\cL(\bV_1,\bV_2)$. For an operator $B$ on a Frechet space $\bY$, we use  $\dom(B)$, $\Ker B$, $\Range B$, $\mathrm{Graph}(B)$, $\sigma(B)$, $\rho(B)$, and $B_{|\bW}$ to denote the domain, kernel, range, graph, spectrum, resolvent set, and the restriction of $B$ to a subspace $\bW$ of $\bY$. The adjoint operator is denoted by $B^*$. We denote by $\sigma_{\mathrm{disc}}(B)$ the set of isolated eigenvalues of finite algebraic multiplicity of the linear operator $B$, and by $\sigma_{\mathrm{ess}}(B)$ its complement in the spectrum of $B$. We use notation $R(\mu,B)=(\mu I_\bX-B)^{-1}$ for the resolvent of $B$ at $\mu\in\rho(B)$. For any Frechet space $\bY$ we use notation $\langle y,y^*\rangle:=y^*(y)$ for any $(y,y^*)\in\bY\times\bY^*$.  The weak$^*$ topology on $\bY^*$ is denoted $\sigma(\bY^*,\bY)$. The direct sum of two subspaces $\bW_1$ and $\bW_2$ is denoted by $\bW_1\oplus \bW_2$. The operator of multiplication by a function $g$ is denoted by $M_g$. We use $\omega_0(T)$ or $\omega_0(A)$ to denote the growth bound of a semigroup $\{T(t)\}_{t\geq 0}$ with generator $A$. The spectral bound of the generator $A$ is defined by $\rms(A)=\sup\mathrm{Re}\,\sigma(A)$. The open disc in $\CC$ centered at $a$ of radius $\eps>0$ is denoted by $D(a,\eps)$. We denote by $\CC_\omega^+$ the set of all $\lambda\in\CC$ such that $\mathrm{Re}\lambda>\omega$, by $\Sigma_{a,\theta}$ the set of all $\lambda\in\CC\setminus\{0\}$ such that $|\arg(\lambda-a)|<\theta$, by $\chi_\Omega$ the characteristic function of the set $\Omega\subset\CC$,
by $\sqrt{z}$ the brunch of the square root obtained by cutting $\CC$ along $(-\infty,0]$ so that $\re\sqrt{z}>0$ for $z\notin(-\infty,0]$ and $\sqrt{z}=\sqrt{z+\rmi 0}$ for $z\in(-\infty,0]$.

\noindent\textbf{Acknowledgments.} We are thankful to T. Christiansen, F. Gesztesy and I. M. Karabash for many useful discussions and for pointing out many interesting results in the resonance literature. Y.\ L.\  would like to thank the
Courant Institute of Mathematical Sciences at NYU and especially Prof.\ Lai-Sang Young for their hospitality. This work began as a part of IMPRESS-U activities at the Institute of Mathematics of Polish Academy of Sciences; the hospitality of IMPAN is gratefully acknowledged.

\section{Preliminaries}\label{sec2}
\subsection{A growth estimate}\label{new2.1}
In this section we assume that $\bX$ is a Banach space and $A:\rm{dom}(A)\subseteq\bX\to\bX$ is the generator of a cosine operator family $\{C(t)\}_{t\in\RR}$. We denote by$\{S(t)\}_{t\in\RR}$ the associated sine family defined by \eqref{sine-def}. We recall that the cosine family has exponential growth, that is, there exist $M\geq 1$ and $\omega\geq 0$ such that
\begin{equation}\label{exp-growth}
\|C(t)\|\leq Me^{\omega|t|}\;\mbox{for any}\;t\in\RR.	
\end{equation}
Similar to the case of semigroups of linear operators there is an extensive literature on the properties of cosine operator families, the associated sine families, various connections of these families to the generator $A$ via the Laplace transform, etc.  We refer to the references \cite{ABHN,Bob,Fattorini,Goldstein,TW} for details and many important results.

One classical method  of treating a second order Cauchy problem is to reduce it to a first order Cauchy problem, and ideally prove that it is well-posed. Introducing the new variable $v=u'$, we can convert \eqref{abstract-Cauchy} into the first order Cauchy problem,
 \begin{equation}\label{abstract-Cauchy-first-order}
	\left\{\begin{array}{ll}
		\left(\begin{matrix}u'(t)\\
			v'(t)
		\end{matrix}\right)=\cA \left(\begin{matrix}u(t)\\
		v(t)
		\end{matrix}\right),\,t\in\RR,\\
			\left(\begin{matrix}u\\
			v
		\end{matrix}\right)(0)=\left(\begin{matrix}x_0\\
		x_1
		\end{matrix}\right).\end{array}\right.
\end{equation}
Here the linear operator $\cA:\dom(A)\times\bX\to\bX\times\bX$ is defined by the block-matrix formula
\begin{equation}\label{def-cA}
\cA=\begin{bmatrix} 0 & I_\bX\\
A& 0\end{bmatrix}.	
\end{equation}	
It is natural to ask if the linear operator $\cA$ generates a $C_0$-semigroup on $\bX\times\bX$ or on a subspace of $\bX\times\bX$ invariant under $\cA$. In the next remark we collect the most important results on this topic, c.f., e.g. \cite{ABHN}.
\begin{remark}\label{r2.1}
\begin{enumerate}
\item[(i)] The linear operator $A$ generates a cosine family on $\bX$ if and only if the linear operator $\cA$ generates a \textit{once integrated $C_0$-semigroup} on $\bX\times\bX$, see e.g, \cite{KH};
\item[(ii)] The linear operator $\cA$ generates a $C_0$-semigroup on $\bX\times\bX$ if and only if the linear operator $A$ is \textit{bounded};
\item[(iii)] The exists a unique Banach space $\bW$, called the \textit{phase-space}, such that $\dom(A)\hookrightarrow \bW\hookrightarrow\bX$ and $\cA_{|\bW\times\bX}$ generates a $C_0$-semigroup on $\bW\times\bX$ denoted $\{\cT(t)\}_{t\geq 0}$.  The phase-space $\bW$ is defined by
\begin{equation}\label{def-bW}
\bW=\{x\in\bX:S(\cdot)x\in\mathcal{C}([0,1],\dom(A))\},\;\;\|x\|_\bW=\|x\|+\sup_{0\leq s\leq1}\|AS(s)x\|.
\end{equation}
Moreover, the semigroup $\cT$ can be extended to a $C_0$-group given by
\begin{equation}\label{def-cT}
	\cT(t)=\begin{bmatrix} C(t)_{|\bW} & S(t)\\
		AS(t)_{|\bW}& C(t)\end{bmatrix},\; t\in\RR.		
\end{equation}	
We refer to \cite{Ki} for the details of this celebrated result.
\end{enumerate}	
\end{remark}	
In the sequel it will be important to show that the $C_0$-group $\cT$ defined above is of order at most $\mathcal{O}(e^{\omega|t|})$, whenever $|t|\to\infty$. To prove this  result we recall some important properties of trajectories of the cosine and sine functions.
\begin{remark}\label{r2.2}
If $x\in\bX$ and $w\in\bW$ then the following properties hold true:
\begin{enumerate}
	\item[(i)]  $C(\cdot)w\in\mathcal{C}(\RR,\bW)\cap\mathcal{C}^1(\RR,\bX)$ and $\bigl(C(\cdot)w\bigr)'=AS(\cdot)w$;
	
	\item[(ii)] $S(\cdot)x\in\mathcal{C}(\RR,\bW)$;
	
	\item[(iii)] $S(\cdot)w\in\mathcal{C}(\RR,\dom(A))\cap\mathcal{C}^2(\RR,\bX)$ and $\bigl(S(\cdot)w\bigr)'=C(\cdot)w$,  $\bigl(S(\cdot)w\bigr)''=AS(\cdot)w$.
\end{enumerate}
\end{remark}
\begin{lemma}\label{l2.3}
If $M\geq 1$ and $\omega>0$ are the constants from \eqref{exp-growth}, then
\begin{equation}\label{exp-growth-cT}
\|\cT(t)\|_{\mathcal{B}(\bW\times\bX)}\leq M_1e^{\omega|t|}\;\mbox{for any}\;t\in\RR,\;\mbox{where}\;M_1:=M\max\Bigl\{2+\frac{M}{\sinh(\omega/2)},1+\omega^{-1}+e^{\omega}\Bigr\}.	
\end{equation}		
\end{lemma}	
\begin{proof} Since $\omega>0$ from \eqref{sine-def} and \eqref{exp-growth} one can readily check that
\begin{equation}\label{l2.3.1}	
\|S(t)\|\leq M\omega^{-1}(e^{\omega|t|}-1)\;\mbox{for any}\; t\in\RR.	
\end{equation}
Fix $x\in\bX$. We recall from \cite{ABHN} that $S(s)S(t)x\in\dom(A)$ and $AS(s)S(t)x=\frac{1}{2}\bigl(C(t+s)x-C(t-s)x\bigr)$ for any $t\in\RR$, $s\in[0,1]$. Using again \eqref{exp-growth} we obtain
\begin{equation}\label{l2.3.2}	
\|AS(s)S(t)x\|=\frac{1}{2}\|C(t+s)x-C(t-s)x\|\leq\frac{M}{2}\bigl(e^{\omega|t-s|}+e^{\omega|t+s|}\bigr)\leq Me^\omega e^{\omega|t|}\|x\|
\end{equation}
for any $t\in\RR$, $s\in[0,1]$. From \eqref{l2.3.1} and \eqref{l2.3.2} we infer that
\begin{equation}\label{l2.3.3}	
\|S(t)x\|_{\bW}=\|S(t)x\|+\sup_{0\leq s\leq 1}{\|AS(s)S(t)x\|}\leq M\Bigl(\omega^{-1}+e^\omega\Bigr)e^{\omega|t|}\|x\|\;\mbox{for any}\; t\in\RR.
\end{equation}
Next, we fix $w\in\bW$. From the definition of the phase-space, we have $S(s)w\in\dom(A)$ for any $s\in [0,1]$. Moreover, since $C(t)S(s)=S(s)C(t)$, it follows that $S(s)C(t)w=C(t)S(s)w\in\dom(A)$ and  $AS(s)C(t)w=AC(t)S(s)w=C(t)AS(s)w$ for any $t\in\RR$, $s\in[0,1]$. Hence,
\begin{align}\label{l2.3.4}
\|C(t)w\|_{\bW}&=\|C(t)w\|+\sup_{0\leq s\leq 1}{\|AS(s)C(t)w\|}=\|C(t)w\|+\sup_{0\leq s\leq 1}{\|C(t)AS(s)w\|}\nonumber\\
&\leq Me^{\omega|t|}\|w\|+Me^{\omega|t|}\sup_{0\leq s\leq 1}\|AS(s)w\|=Me^{\omega|t|}\|w\|_{\bW}\;\mbox{for any}\; t\in\RR.		
\end{align}	
We recall that $S(k)-S(k-1)=2C(k-\frac{1}{2})S(\frac{1}{2})$ for any $k\in\NN$. Since $S(\frac{1}{2})w\in\dom(A)$, we obtain that $C(k-\frac{1}{2})S(\frac{1}{2})w\in\dom(A)$ and $AC(k-\frac{1}{2})S(\frac{1}{2})w=C(k-\frac{1}{2})AS(\frac{1}{2})w$, hence
\begin{equation}\label{l2.3.5}
AS(k)w=AS(k-1)w+2C(k-\frac{1}{2})AS(\frac{1}{2})w\;\mbox{for any}\; k\in\NN.	
\end{equation}	
From \eqref{exp-growth}, \eqref{def-bW} and \eqref{l2.3.5} it follows that
\begin{equation}\label{l2.3.6}
\|AS(k)w\|\leq\|AS(k-1)w\|+2Me^{\omega(k-\frac{1}{2})}(\|w\|_\bW-\|w\|)\;\mbox{for any}\; k\in\NN.	
\end{equation}	
Since $S(0)=0$, from\eqref{l2.3.6} we conclude that
\begin{align}\label{l2.3.7}
\|AS(k)w\|&=\sum_{j=1}^{k}\Bigl(\|AS(j)w\|-\|AS(j-1)w\|\Bigr)\leq 2M\bigl(\|w\|_\bW-\|w\|\bigr)\sum_{j=1}^{k}e^{\omega(j-\frac{1}{2})}\nonumber\\
&=2Me^{-\omega/2}\frac{e^{\omega k}-1}{e^\omega-1}\bigl(\|w\|_\bW-\|w\|\bigr)\leq\frac{2M}{e^{3\omega/2}-e^{\omega/2}}e^{\omega k}\bigl(\|w\|_\bW-\|w\|\bigr)\;\mbox{for any}\; k\in\NN.
\end{align}	
Using again the definition of the phase-space $\bW$ in \eqref{def-bW} we have $S(k)w,S(s)w\in\dom(A)$, thus $C(k)S(s)w,C(s)S(k)w\in\dom(A)$ and
$AC(k)S(s)w=C(k)AS(s)w$, $AC(s)S(k)w=C(s)AS(k)w$ for any $k\in\NN$ and $s\in [0,1]$. Fix $t\geq1$ and let $m=[t]\in\NN$ be the integer part of $t$ and $\tau=t-m$.  From \eqref{l2.3.7} we conclude that
\begin{align}\label{l2.3.8}
\|AS(t)w\|&=\|AS(m+\tau)w\|=\Bigl\| A\Bigl(C(m)S(\tau)w+C(\tau)S(k)w \Bigr)\Bigr\|=\|C(m)AS(\tau)w+C(\tau)AS(k)w \|\nonumber\\
&\leq\|C(m)AS(\tau)w\|+\|C(\tau)AS(k)w\|\leq Me^{\omega m}\|AS(\tau)w\|+Me^{\omega\tau}\|AS(m)w\|\nonumber\\
&\leq\Bigl(M+\frac{M^2}{\sinh{(\omega/2)}}\Bigr)e^{\omega m}\bigl(\|w\|_\bW-\|w\|\bigr)\leq M\Bigl(1+\frac{M}{\sinh{(\omega/2)}}\Bigr)e^{\omega t}\bigl(\|w\|_\bW-\|w\|\bigr).
\end{align}	
Using the definition of the norm in $\bW$ given in \eqref{def-bW} one can readily see that \eqref{l2.3.8} trivially holds true for any $t\in[0,1]$. Since the sine function is odd  it follows that \eqref{l2.3.8} holds for any $t\in\RR$, and so
\begin{equation}\label{l2.3.9}
\|AS(t)w\|\leq M\Bigl(1+\frac{M}{\sinh{(\omega/2)}}\Bigr)e^{\omega |t|}\bigl(\|w\|_\bW-\|w\|\bigr)\;\mbox{for any}\;t\in\RR.
\end{equation}
Recalling the definition of $M_1$ in \eqref{exp-growth-cT} and summarizing, from \eqref{exp-growth}, \eqref{def-cT},  \eqref{l2.3.1}, \eqref{l2.3.3}, \eqref{l2.3.4} and \eqref{l2.3.9} we conclude that
\begin{align}\label{l2.3.10}
\Bigl\|\cT(t) \left(\begin{matrix}w\\x\end{matrix}\right) &\Bigr\|_{\bW\times\bX}=\Bigl\|\left(\begin{matrix}C(t)w+S(t)x\\AS(t)w+C(t)x\end{matrix}\right) \Bigr\|_{\bW\times\bX}=\|C(t)w+S(t)x\|_{\bW}+\|AS(t)w+C(t)x\|\nonumber\\
&=Me^{\omega |t|}\|w\|_\bW+M\Bigl(\omega^{-1}+e^\omega\Bigr)e^{\omega|t|}\|x\|+M\Bigl(1+\frac{M}{\sinh{(\omega/2)}}\Bigr)e^{\omega |t|}\|w\|_\bW+Me^{\omega |t|}\|x\|\nonumber\\
&=M\Bigl(2+\frac{M}{\sinh{(\omega/2)}}\Bigr)e^{\omega |t|}\|w\|_\bW+M(1+\omega^{-1}+e^\omega)e^{\omega|t|}\|x\|\leq M_1e^{\omega|t|}\Bigl\| \left(\begin{matrix}w\\x\end{matrix}\right) \Bigr\|_{\bW\times\bX}
\end{align}	
for any $t\in\RR$, $x\in\bX$, $w\in\bW$, proving the lemma.
\end{proof}
Next, we recall on some well-known spectral properties of the cosine family generator $A$ and its resolvent operator function, see e.g., \cite{ABHN,Bob,Fattorini,Goldstein}.
\begin{remark}\label{r2.4}
\begin{enumerate}
	\item[(i)] If $M\geq 1$ and  $\omega\geq0$ are the constants from \eqref{exp-growth}, then $\{\lambda^2:\lambda\in\CC_\omega^+\}\subset\rho(A)$ and
\begin{equation}\label{res-omega-est}
\|\lambda R(\lambda^2,A)\|\leq\frac{M}{\re\lambda-\omega}\;\mbox{for any}\;\lambda\in\CC_\omega^+;	
\end{equation}		
\item[(ii)] The linear operator $A$ generates on $\bX$ an analytic semigroup of angle $\frac{\pi}{2}$ denoted $\{T(z)\}_{z\in\CC_0^+}$ that can be expressed via the Weierstrass formula
\begin{equation}\label{Weierstrss}
T(z)x=\frac{1}{\sqrt{\pi z}}\int_0^\infty e^{-\frac{t^2}{4z}}C(t)x\rmd t,\; z\in\CC_0^+, x\in\bX.
\end{equation}
\item[(iii)] The analytic semigroup $\{T(z)\}_{z\in\CC_0^+}$ satisfies the following estimates,
\begin{align}\label{sem-est}
\|T(z)\|\leq &M \sqrt{\frac{|z|}{\re z}} \mathrm{exp}\Bigl({\omega^2\frac{|z|^2}{\re z}}\Bigr)\;\mbox{for any}\;z\in\CC_0^+,\nonumber\\
\|T(z)\|\leq &\oM(\varphi)\mathrm{exp}\Bigl({\omm(\varphi)\re z}\Bigr)\;\mbox{for any}\;z\in\Sigma_{0,\varphi},\,\varphi\in\bigl(0,\frac{\pi}{2}\bigr),
\end{align}
where the functions $\oM,\omm:\bigl(0,\frac{\pi}{2}\bigr)\to[0,\infty)$ are defined by
\begin{equation}\label{def-oM-omm}
\oM(\varphi)=\frac{M}{\sqrt{\cos{(\varphi)}}}  \mathrm{exp}\Bigl({\frac{\omega}{\cos{(\varphi)}}}\Bigr),\quad \omm(\varphi)=\frac{\ln(\oM(\varphi))}{\cos{(\varphi)}}.	
\end{equation}	
\item[(iv)] The linear operator $A$ is sectorial, more precisely, $\Sigma_{\omm(\varphi),\frac{\pi}{2}+\varphi}\subset\rho(A)$ and
\begin{equation}\label{sectorial-est}
\tM(\varphi):=\sup\bigl\{\big\| \big(\mu-\tom(\varphi)\big)R(\mu,A)\big\|:\mu\in\Sigma_{\omm(\varphi),\frac{\pi}{2}+\varphi}\bigr\}<\infty\;\mbox{for any}\;\varphi\in\bigl(0,\frac{\pi}{2}\bigr), 	 
\end{equation}
where $\tom:\bigl(0,\frac{\pi}{2}\bigr)\to[0,\infty)$ is defined by $\tom(\varphi)=\omm(\frac{\pi}{4}+\frac{\varphi}{2})$.
\end{enumerate}
\end{remark}	
We note that some of the spectral properties of generators of cosine families are different from those of generators of analytic semigroups. For example, in the case of analytic semigroups generators one has that the least exponential bound of the semigroup is equal to the spectral bound of the generator, that is $\omega_0(A)=\rms(A)$. However, it is possible to construct an example of a cosine family generator such that $\rms(A)<0$, and yet the norm of the cosine family has polynomial growth, see \cite[Example 0.6]{deLP}.

We note that \eqref{res-omega-est} is equivalent to
\begin{equation}\label{RA-est}
	\|R_A(\lambda)\|\leq\frac{M}{|\lambda|(\re\lambda-\omega)}\;\mbox{for any}\;\lambda\in\CC_\omega^+.
\end{equation}			
We are interested in meromorphic extensions of the operator-valued function $R_A$ to other than $\CC_\omega^+$ regions of the complex plane. The best for our aims extension is not necessarily defined on the largest possible domain,  but on a domain where the norm of the operator-valued function $R_A$ can be estimated by a product of a function of $|\lambda|$ and a function of $\re\lambda$, that mimics estimate \eqref{RA-est} as close as possible. One can readily check that, with $\tom$ defined above,
\begin{equation}\label{lambda-square-domain}
\lambda^2\in\Sigma_{\tom(\varphi),\frac{\pi}{2}+\varphi}\;\mbox{if and only if}\;|\re\lambda|>\sqrt{ \sec^2{(\varphi)}(\im\lambda)^2+\tom(\varphi)}-\tan{(\varphi)}|\im\lambda|  	
\end{equation}
for any $\varphi\in\bigl(0,\frac{\pi}{2}\bigr)$. This implies that the operator-valued function $R_A$ can be extended analytically to the open set
\begin{equation}\label{def-Omega-phi}
\Omega(\varphi)=\{\lambda\in\CC:|\re\lambda|>\sqrt{ \sec^2{(\varphi)}(\im\lambda)^2+\tom(\varphi)}-\tan{(\varphi)}|\im\lambda|\},\;\varphi\in\bigl(0,\frac{\pi}{2}\bigr).
\end{equation}
Moreover, from Remark~\ref{r2.4}(iv) and \eqref{lambda-square-domain} we infer that
\begin{equation}\label{RA-Omega-phi-region-est}
\|R_A(\lambda)\|\leq\frac{\tM(\varphi)}{|\lambda^2-\tom(\varphi)|}\;\mbox{for any}\;\lambda\in\Omega(\varphi)\;\mbox{for each}\; \varphi\in\bigl(0,\frac{\pi}{2}\bigr).
\end{equation}
From \eqref{def-oM-omm} one can readily check that $\lim\limits_{\varphi\nearrow\frac{\pi}{2}}\omm(\varphi)=\lim\limits_{\varphi\nearrow\frac{\pi}{2}}\tom(\varphi)=\infty$. This implies that the estimate \eqref{RA-Omega-phi-region-est} is useful only in the case when the angle $\varphi$ stays away from some neighborhood of $\pi/2$. In addition, it is straightforward to check that
$\lim\limits_{s\to\infty}\Bigl(\sqrt{ \sec^2{(\varphi)}s^2+\tom(\varphi)}-\tan{(\varphi)}|s|\Bigr)=\infty$. It follows that the set $\Omega(\varphi)$ defined in \eqref{def-Omega-phi} does not contain an entire neighborhood of $\CC_\omega^+$. To prove our results, below we will need to have a meromorphic extension of the operator-valued function $R_A$ to the set located to the right of the graph $\re\lambda=g_0(\im\lambda)$, of a smooth, real-valued function $g_0$, satisfying the condition $\sup g_0<\omega$.
In general, we will need to assume the existence of such an extension. However, in many instances one expects that the set $\{\lambda\in\CC:\lambda^2\in\sigma_{\mathrm{ess}}(A)\}$ to be located in a neighborhood of the imaginary axis. Hence, we expect the set $\{\lambda\in\CC:\lambda^2\in\rho_{\mathrm{ess}}(A)\}$ to contain the  open set $\{\lambda\in\CC:\re\lambda>g_0(\im\lambda)\}$, on which $R_A$ is meromorphic and has only finitely many poles of finite multiplicity, related to the discrete eigenvalues of $A$.

\subsection{Integral Representations of Trajectories of Cosine and Sine functions }\label{subsec2.1} Next, we look for representations of trajectories of the cosine and sine functions, using complex integration involving the operator-valued function $R_A$ along vertical lines contained in $\CC_\omega^+$. The intuition behind such representations are the various forms of complex inversion results for the Laplace Transform and the fact that $\cL\{C(\cdot)x\}=\lambda R(\lambda^2,A)x$ and  $\cL\{S(\cdot)x\}= R(\lambda^2,A)x$ for any $\lambda\in\CC_\omega^+$, $x\in\bX$. Also, we note that in the case of trajectories of $C_0$-semigroups, that is, solutions of well-posed first order Cauchy problems, such results are well-known, see, e.g., \cite{ABHN,CL,DK,EN,lunardi,P}. These types of results have been discussed in \cite{Fattorini} for the case of cosine families. In the next two lemmas we give all the details for completeness. First, we discuss the basic result to better grasp the nature of the representation.
\begin{lemma}\label{l2.5}
Assume that $\{C(t)\}_{t\in\RR}$ is a strongly continuous cosine family satisfying \eqref{exp-growth}. Then,
\begin{enumerate}
\item[(i)] $\displaystyle\int_{\re\lambda=\gamma}\frac{e^{\lambda t}}{\lambda}R_A(\lambda)Ax\rmd\lambda$ is absolutely convergent for any $t\geq 0$, $x\in\dom(A)$, $\gamma>\omega$;
\item[(ii)] $C(t)x=x+\frac{1}{2\pi\rmi}\displaystyle\int_{\re\lambda=\gamma}\frac{e^{\lambda t}}{\lambda}R_A(\lambda)Ax\rmd\lambda$ for any $t\geq 0$, $x\in\dom(A)$, $\gamma>\omega$.	
\end{enumerate}
\end{lemma}
\begin{proof}
See the proof of Lemma~\ref{l2.10} below for a more general setting.
\end{proof}
It is well-known that in the case of first order Cauchy problems, one can use the complex Laplace transform inversion formula in general only for classical solutions. In the case of mild, or general trajectories of $C_0$-semigroups, inversion formulas can be proved in the case $\mathrm{UMD}$-spaces. In the case of cosine families in $\mathrm{UMD}$-spaces a similar result is proved in \cite{CiKe}. In the case of general Banach spaces if $x\in\bX$ then the term $\frac{e^{\lambda t}}{\lambda}AR_A(\lambda)x=\lambda e^{\lambda t}R_A(\lambda)x-\frac{e^{\lambda t}}{\lambda}x$ might not be integrable, not even in the principal value sense.

To formulate the representation for the trajectories of the sine function we need to present some additional estimates of the operator-valued function $R_A$.
\begin{remark}\label{r2.6}
Taking Laplace Transform in \eqref{def-cT} we infer that $\CC_\omega^+\subset\rho(\cA)$ and
\begin{equation}\label{r2.6.1}
R(\lambda,\cA_{|\bW\times\bX})=\begin{bmatrix} \lambda R_A(\lambda)_{|\bW} &  R_A(\lambda)\\
		AR_A(\lambda)_{|\bW}& \lambda R_A(\lambda)\end{bmatrix}\;\mbox{for any}\;\lambda\in\CC_\omega^+.		
\end{equation}	
Moreover, from Lemma~\ref{l2.3} it follows that
\begin{equation}\label{r2.6.2}
\|R(\lambda,\cA_{|\bW\times\bX})\|_{\mathcal{B}(\bW\times\bX)}\leq\frac{M_1}{\re\lambda-\omega}\;\mbox{for any}\;\lambda\in\CC_\omega^+.
\end{equation}
From \eqref{r2.6.1} and \eqref{r2.6.2} we conclude that
\begin{equation}\label{r2.6.3}
\|\lambda R_A(\lambda)_{|\bW}\|_{\mathcal{B}(\bW)}+\|	AR_A(\lambda)_{|\bW}\|_{\mathcal{B}(\bW,\bX)}\leq\frac{M_1}{\re\lambda-\omega}\;\mbox{for any}\;\lambda\in\CC_\omega^+.
\end{equation}	
\end{remark}
\begin{lemma}\label{l2.7}
Assume that $\{C(t)\}_{t\in\RR}$ is a strongly continuous cosine family, $\{S(t)\}_{t\in\RR}$ is the associate sine family defined by \eqref{sine-def}, $\bW$ is the phase-space defined in \eqref{def-bW} and assume that \eqref{exp-growth} holds. Then,
\begin{enumerate}
\item[(i)] The integral $\displaystyle\int_{\re\lambda=\gamma}\frac{e^{\lambda t}}{\lambda^2}AR_A(\lambda)w\rmd\lambda$ is absolutely convergent for any $t\geq 0$, $w\in\bW$, $\gamma>\omega$;
\item[(ii)] The integral $\displaystyle\int_{\re\lambda=\gamma}\frac{e^{\lambda t}}{\lambda^2}R_A(\lambda)x\rmd\lambda$ is absolutely convergent for any $t\geq 0$, $x\in\bX$, $\gamma>\omega$;
\item[(iii)] $\displaystyle\int_{\re\lambda=\gamma}\frac{e^{\lambda t}}{\lambda^2}R_A(\lambda)w\rmd\lambda\in\dom(A)$ and
\begin{equation}\label{A-switch}	
A\int_{\re\lambda=\gamma}\frac{e^{\lambda t}}{\lambda^2}R_A(\lambda)w\rmd\lambda=\displaystyle\int_{\re\lambda=\gamma}\frac{e^{\lambda t}}{\lambda^2}AR_A(\lambda)w\rmd\lambda\;\mbox{for any}\;w\in\bW, \gamma>\omega;	
\end{equation}	
\item[(iv)] $S(t)w=tw+\frac{1}{2\pi\rmi}\displaystyle\int_{\re\lambda=\gamma}\frac{e^{\lambda t}}{\lambda^2}AR_A(\lambda)w\rmd\lambda=tw+\frac{1}{2\pi\rmi}A\displaystyle\int_{\re\lambda=\gamma}\frac{e^{\lambda t}}{\lambda^2}R_A(\lambda)w\rmd\lambda$ for any $t\geq 0$, $w\in\bW$, $\gamma>\omega$.	
\end{enumerate}	
\end{lemma}	
\begin{proof}
	See the proof of Lemma~\ref{l2.10} below for a more general setting.
\end{proof}
Next, we aim to generalize the integral representations of trajectories $C(\cdot)x$ and $S(\cdot)w$, given in Lemma~\ref{l2.5} and Lemma~\ref{l2.7}, respectively, to the case when $x$ belongs to a subspace of $\dom(A)$ and $w$ belongs to a subspace of the phase-space $\bW$. In particular we are interested in representations that involve an integral with a higher order decay rate as $|\lambda|\to\infty$. We start with the following lemma.
\begin{lemma}\label{l2.8}
Let $\{C(t)\}_{t\in\RR}$ be a strongly continuous cosine family generated by $A$. If $n\in\NN$, then
\begin{equation}\label{higher-order-cos-int}
C(t)x=\sum_{j=0}^{n-1}\frac{t^{2j}}{(2j)!}A^jx+\frac{1}{(2n-1)!}\int_0^t(t-s)^{2n-1}C(s)A^nx\rmd s\;\mbox{for any}\;t\in\RR,x\in\dom(A^n).
\end{equation}
\end{lemma}	
\begin{proof} Since the cosine family $\{C(t)\}_{t\in\RR}$ is an even function of $t\in\RR$, using a simple change of variables one can readily check that all functions in \eqref{higher-order-cos-int} are even. Hence, it is enough to prove the identity for $t\geq 0$. We use induction to prove the lemma.
Fix $x\in\dom(A)$. Since $\displaystyle\int_0^t(t-s)C(s)x\rmd s\in\dom(A)$ and the linear operator $A$ is closed, it follows that
\begin{align}\label{l2.5.2}
C(t)x&=x+A\int_0^t(t-s)C(s)x\rmd s=x+\int_0^t(t-s)AC(s)x\rmd s\nonumber\\
&=x+\int_0^t(t-s)C(s)Ax\rmd s=x+\int_0^t(t-s)f_x(s)\rmd s\;\mbox{for any}\; t\geq 0, x\in\dom(A).	
\end{align}	
We note that for $n=1$ formula \eqref{higher-order-cos-int} follows from  \eqref{l2.5.2}. Next, we fix $n\in\NN$, $x\in\dom(A^{n+1})$ and assume that \eqref{higher-order-cos-int} holds true. From \eqref{l2.5.2} we have	
\begin{equation}\label{l2.8.1}
C(s)A^nx=A^nx+\int_0^s(s-\tau)C(\tau)A^{n+1}x\rmd\tau\;\mbox{for any}\;s\in\RR. 	
\end{equation}		
Moreover, an elementary computation shows that
\begin{equation}\label{l2.8.2}	
\int_\tau^t(t-s)^{2n-1}(s-\tau)\rmd s=\int_0^{t-\tau}\xi^{2n-1}(t-\tau-\xi)\rmd\xi=\frac{(t-\tau)^{2n+1}}{2n(2n+1)}\;\mbox{for any}\;t,\tau\in\RR.	
\end{equation}	
From \eqref{higher-order-cos-int}, \eqref{l2.8.1} and \eqref{l2.8.2} we obtain
\begin{align}\label{l2.8.3}
C(t)x=&\sum_{j=0}^{n-1}\frac{t^{2j}}{(2j)!}A^jx+\frac{1}{(2n-1)!}\int_0^t(t-s)^{2n-1}\Bigl(A^nx+\int_0^s(s-\tau)C(\tau)A^{n+1}x\rmd\tau\Bigr)\rmd s\nonumber\\
&=\sum_{j=0}^{n-1}\frac{t^{2j}}{(2j)!}A^jx+\frac{1}{(2n)!}A^nx
+\frac{1}{(2n-1)!}\int_0^t\int_0^s(t-s)^{2n-1}(s-\tau)C(\tau)A^{n+1}x\rmd\tau\rmd s\nonumber\\
&=\sum_{j=0}^{n}\frac{t^{2j}}{(2j)!}A^jx
+\frac{1}{(2n-1)!}\int_0^t\Bigl(\int_\tau^t(t-s)^{2n-1}(s-\tau)\rmd s\Bigr)C(\tau)A^{n+1}x\rmd\tau\nonumber\\
&=\sum_{j=0}^{n}\frac{t^{2j}}{(2j)!}A^jx
+\frac{1}{(2n+1)!}\int_0^t(t-\tau)^{2n+1}C(\tau)A^{n+1}x\rmd\tau.
\end{align}	
Hence formula \eqref{higher-order-cos-int} holds when $n$ is replaced by $n+1$, proving the lemma.
\end{proof}
In order to formulate a representation of trajectories of $\{S(t)\}_{t\in\RR}$ similar to Lemma~\ref{l2.8} we recall \eqref{def-bW} and introduce the space
\begin{equation}\label{def-bW-n}
\bW_n=\{x\in\dom(A^n):A^nx\in\bW\},\; n\in\NN\cup\{0\}.	
\end{equation}	
Since $A$ is a closed linear operator with non-empty resolvent set, $A^n$ is a closed linear operator. Hence,  $\bW_n$ is a Banach space when endowed with the norm
\begin{equation}\label{norm-bW-n}
\|x\|_{\bW_n}=\|x\|_{\dom(A^n)}+\|A^nx\|_{\bW},\;\; x\in\bW_n, n\in\NN\cup\{0\}.	
\end{equation}	
From Remark~\ref{r2.1}(iii), \eqref{def-bW-n} and \eqref{norm-bW-n} we infer that
\begin{equation}\label{domain-A-W-n}
\bW_n\hookrightarrow\dom(A^n)\hookrightarrow\bW_{n-1}\hookrightarrow\dom(A^{n-1})\;\mbox{for any}\;n\in\NN.
\end{equation}
\begin{lemma}\label{l2.9}
Let $\{C(t)\}_{t\in\RR}$ be a strongly continuous cosine family generated by $A$, and  $\{S(t)\}_{t\in\RR}$ be the associate sine family defined by \eqref{sine-def}. If $n\in\NN$, then
\begin{equation}\label{higher-order-sine-int}
S(t)w=\sum_{j=0}^{n-1}\frac{t^{2j+1}}{(2j+1)!}A^jw+\frac{1}{(2n-1)!}\int_0^t(t-s)^{2n-1}AS(s)A^{n-1}w\rmd s\;\mbox{for any}\;t\in\RR,w\in\bW_{n-1}.
\end{equation}	
\end{lemma}	
\begin{proof} As in the proof of Lemma~\ref{l2.8} we use induction.
Integrating by parts, from Remark~\ref{r2.2} and \eqref{sine-def} we obtain
\begin{equation}\label{l2.7.4}
\int_0^t(t-s)S(s)w\rmd s=\int_0^t(t-s)\bigl(C(s)w\bigr)'\rmd s=-tw+\int_0^tC(s)w\rmd s=S(t)w-tw
\end{equation}	
for any $t\geq0$, $w\in\bW$. Since $\bW_0=\bW$, for $n=1$ formula \eqref{higher-order-sine-int} follows from \eqref{l2.7.4}. Fix $n\in\NN$, $w\in\bW_n$ and assume \eqref{higher-order-sine-int}. From \eqref{domain-A-W-n} we have $w\in\bW_{n-1}$ and $A^nw\in\bW$. From Remark~\ref{r2.2}(iii) it follows that
\begin{equation}\label{l2.9.1}
S(\cdot)A^nw\in\mathcal{C}^2(\RR,\bX)\;\mbox{and}\;\bigl(S(\cdot)A^nw\bigr)'=C(\cdot)A^nw, \bigl(S(\cdot)A^nw\bigr)''=AS(\cdot)A^nw.	
\end{equation}		
Integrating by parts twice, from \eqref{higher-order-sine-int} and \eqref{l2.9.1} we obtain
\begin{align}\label{l2.9.2}
\int_0^t&(t-s)^{2n+1}AS(s)A^nw\rmd s=\int_0^t(t-s)^{2n+1}\bigl(S(s)A^nw\bigr)''\rmd s\nonumber\\
&=-t^{2n +1}A^nw+(2n+1)\int_0^t(t-s)^{2n}\bigl(S(s)A^nw\bigr)'\rmd s\nonumber\\
&=-t^{2n +1}A^nw+2n(2n+1)\int_0^t(t-s)^{2n-1}S(s)A^nw\rmd s\nonumber\\
&=-t^{2n +1}A^nw+2n(2n+1)\int_0^t(t-s)^{2n-1}AS(s)A^{n-1}w\rmd s\nonumber\\
&=-t^{2n+1}A^nw+(2n+1)!\Bigg(S(t)w-\sum_{j=0}^{n-1}\frac{t^{2j+1}}{(2j+1)!}A^jw\Bigg)\nonumber\\
&=(2n+1)!\Bigg(S(t)w-\sum_{j=0}^{n}\frac{t^{2j+1}}{(2j+1)!}A^jw\Bigg),
\end{align}
and so \eqref{higher-order-sine-int} holds with $n$ is replaced by $n+1$, proving the lemma.
\end{proof}
\begin{lemma}\label{l2.10}
Assume that $\{C(t)\}_{t\in\RR}$ is a strongly continuous cosine family satisfying \eqref{exp-growth} and $n\in\NN$. Then,
\begin{enumerate}
\item[(i)] $\displaystyle\int_{\re\lambda=\gamma}\frac{e^{\lambda t}}{\lambda^{2n-1}}R_A(\lambda)A^nx\rmd\lambda$ is absolutely convergent for any $t\geq 0$, $x\in\dom(A^n)$, $\gamma>\omega$;
\item[(ii)] $C(t)x=\displaystyle\sum\limits_{j=0}^{n-1}\frac{t^{2j}}{(2j)!}A^jx+\frac{1}{2\pi\rmi}\displaystyle\int_{\re\lambda=\gamma}\frac{e^{\lambda t}}{\lambda^{2n-1}}R_A(\lambda)A^nx\rmd\lambda$ for any $t\geq 0$, $x\in\dom(A^n)$, $\gamma>\omega$.	
\end{enumerate}
\end{lemma}
\begin{proof} Fix $x\in\dom(A^n)$ and $\gamma>\omega$. Let $f_{n,x}:[0,\infty)\to\bX$ be defined by $f_{n,x}(t)=C(t)A^nx$. Then, $f_{n,x}\in\mathcal{C}([0,\infty),\bX)$ and  $\|f_{n,x}(t)\|\leq M\|A^nx\|e^{\omega t}$ for any $t\geq 0$, by \eqref{exp-growth}. Taking Laplace Transform we have
\begin{equation}\label{l2.10.1}
(\cL f_{n,x})(\lambda)=\lambda R_A(\lambda)A^nx\;\mbox{for any}\;\lambda\in\CC_\omega^+.
\end{equation}
From \eqref{l2.10.1} and Lemma~\ref{lAI7}(i) with $k=2n-1$  it follows that
$\displaystyle\int_{\re\lambda=\gamma}\frac{e^{\lambda t}}{\lambda^{2n-1}}R_A(\lambda)A^nx\rmd\lambda=\displaystyle\int_{\re\lambda=\gamma}\frac{e^{\lambda t}}{\lambda^{2n}}(\cL f_{n,x})(\lambda)\rmd\lambda$ is absolutely convergent for any $t\geq 0$, proving (i). Assertion (ii) follows from Lemma~\ref{l2.8} and Lemma~\ref{lAI7}(ii) with $k=2n-1$.		   	
\end{proof}
\begin{lemma}\label{l2.11}
Assume that $\{C(t)\}_{t\in\RR}$ is a strongly continuous cosine family satisfying \eqref{exp-growth} and let $\{S(t)\}_{t\in\RR}$ the associate sine family defined by \eqref{sine-def},  $n\in\NN$. Then,
\begin{enumerate}
\item[(i)] $\displaystyle\int_{\re\lambda=\gamma}\frac{e^{\lambda t}}{\lambda^{2n}}AR_A(\lambda)A^{n-1}w\rmd\lambda$ is absolutely convergent for any $t\geq 0$, $w\in\bW_{n-1}$, $\gamma>\omega$;
\item[(ii)] $\displaystyle\int_{\re\lambda=\gamma}\frac{e^{\lambda t}}{\lambda^{2n}}R_A(\lambda)x\rmd\lambda$ is absolutely convergent for any $t\geq 0$, $x\in\bX$, $\gamma>\omega$;
\item[(iii)] $\displaystyle\int_{\re\lambda=\gamma}\frac{e^{\lambda t}}{\lambda^{2n}}R_A(\lambda)A^{n-1}w\rmd\lambda\in\dom(A)$ and
\begin{equation}\label{A-switch-bis}	
A\int_{\re\lambda=\gamma}\frac{e^{\lambda t}}{\lambda^{2n}}R_A(\lambda)A^{n-1}w\rmd\lambda=\displaystyle\int_{\re\lambda=\gamma}\frac{e^{\lambda t}}{\lambda^{2n}}AR_A(\lambda)A^{n-1}w\rmd\lambda\;\mbox{for any}\;w\in\bW_{n-1}, \gamma>\omega;	
\end{equation}	
\item[(iv)] $S(t)w=\sum\limits_{j=0}^{n-1}\frac{t^{2j+1}}{(2j+1)!}A^jw+\frac{1}{2\pi\rmi}\displaystyle\int_{\re\lambda=\gamma}\frac{e^{\lambda t}}{\lambda^{2n}}AR_A(\lambda)A^{n-1}w\rmd\lambda$
		
		$\quad\;\;\;=\sum\limits_{j=0}^{n-1}\frac{t^{2j+1}}{(2j+1)!}A^jw+\frac{1}{2\pi\rmi}A\displaystyle\int_{\re\lambda=\gamma}\frac{e^{\lambda t}}{\lambda^{2n}}R_A(\lambda)A^{n-1}w\rmd\lambda$ for any $t\geq 0$, $w\in\bW_{n-1}$, $\gamma>\omega$.	
\end{enumerate}	
\end{lemma}	
\begin{proof} Fix $w\in\bW_{n-1}$, $\gamma>\omega$.  If $\lambda=\gamma+\rmi s$, $s\in\RR$, then \eqref{r2.6.3} yields
\begin{equation}\label{l2.11.1}
\Bigl\|\frac{e^{\lambda t}}{\lambda^{2n}}AR_A(\lambda)A^{n-1}w\Bigr\|=\frac{e^{(\re\lambda)t}}{|\lambda|^{2n}}\|	AR(\lambda^2,A)A^{n-1}w\|\leq\frac{M_1e^{\gamma t}}{(\gamma-\omega)(\gamma^2+s^2)^n}\|A^{n-1}w\|_\bW,
\end{equation}
proving assertion (i). Next, we fix $x\in\bX$. If $\lambda=\gamma+\rmi s$, $s\in\RR$, then from \eqref{RA-est} we obtain
\begin{equation}\label{l2.11.2}
\Bigl\|\frac{e^{\lambda t}}{\lambda^{2n}}R_A(\lambda)x\Bigr\|=\frac{e^{(\re\lambda)t}}{|\lambda|^{2n}}\|R(\lambda^2,A)x\|\leq\frac{Me^{\gamma t}}{(\gamma-\omega)(\gamma^2+s^2)^{\frac{2n+1}{2}}}\|x\|_\bX,
\end{equation}
proving assertion (ii). Since $A$ is closed, \eqref{A-switch-bis} follows by \cite[Proposition 1.1.7]{ABHN}. $A^{n-1}w\in\bW$ implies that the function $\tf_{n,w}:[0,\infty)\to\bX$ defined by $\tf_{n,w}(t)=AS(t)A^{n-1}w$ is continuous.
From Lemma~\ref{l2.3} and \eqref{def-cT} we have
\begin{equation}\label{l2.11.3}
\|\tf_{n,w}(t)\|=\|AS(t)_{|\bW}A^{n-1}w\|\leq M_1e^{\omega t}\|A^{n-1}w\|_\bW \leq M_1e^{\omega t}\|w\|_{\bW_{n-1}} \;\mbox{for any}\; t\geq 0.
\end{equation}
Moreover, from \eqref{def-cT} and \eqref{r2.6.1} we infer that
\begin{equation}\label{l2.11.4}
	(\cL \tf_w)(\lambda)=AR_A(\lambda)A^{n-1}w\;\mbox{for any}\;\lambda\in\CC_\omega^+.
\end{equation}
Assertion (iv) follows from Lemma~\ref{l2.9} and Lemma~\ref{lAI7}(ii) with $k=2n$.		   		
\end{proof}

\section{Leading Order Terms of Cosine and Sine Functions}\label{sec3}
In this section we aim to identify the leading order terms in the representations given in Lemma~\ref{l2.5}(ii) and Lemma~\ref{l2.7}(iv), for the  Cosine and Sine Functions, respectively. Estimating directly the integral terms of these representations we can readily obtain a growth rate of order $\mathcal{O}(e^{\gamma t})$ as $t\to\infty$, for any $\gamma>\omega$, where $\omega$ is the constant from \eqref{exp-growth}. Our goal is to find some sufficient conditions that would allow us to replace integration along the vertical line $\re\lambda=\gamma$ by integration along a path contained to the left of the vertical line $\re\lambda=\omega$.  First, we recall the definition of an important type of singular points of a meromorphic function. For more information on finitely meromorphic point and its applications we refer to  \cite{GHR,GGK,GL-book,GSig,How1}.
\begin{definition}\label{d3.1}
Let $\Xi\subseteq\CC$ be an open, connected set and suppose that $E:\Xi\to\mathcal{B}(\bX)$ is an operator-valued analytic function on $\Xi$ except for isolated singularities.  We say that $E$ is a \textit{finitely meromorphic} at $\xi_0\in\Xi$ and that $\xi_0\in\Xi$ is a finitely meromorphic point of $E$, if $E$ is analytic on $D(\xi_0,r_0)\setminus\{\xi_0\}\subset\Xi$ for some $r_0>0$ and there exists $\tE:D(\xi_0,r_0)\to\mathcal{B}(\bX)$  an analytic function, a number $m_E(\xi_0)\in\NN$, and \textit{finite rank operators} $G_j\in\mathcal{B}(\bX)$, $j=1,\cdots,m_E(\xi_0)$, such that
\begin{equation}\label{finite-meromorphic}
E(\xi)=\sum_{j=1}^{m_E(\xi_0)}\frac{1}{(\xi-\xi_0)^j}G_j+\tE(\xi)\;\mbox{for any}\;\xi\in D(\xi_0,r_0)\setminus\{\xi_0\}.
\end{equation}
\end{definition}
Next, we formulate the main assumptions on the resolvent function $R_A(\cdot)$ for the generator of the cosine family satisfying \eqref{exp-growth} for some $\omega>0$.

\noindent\textbf{Hypothesis (H).}  We assume that there exists a function $g_0:\RR\to\RR$, \textit{piecewise} of class $\mathcal{C}^1$, with $\sup_{s\in\RR}g_0(s)<\omega$ such that
\begin{enumerate}
\item[(i)] The operator valued-function $R_A(\cdot)$ has a meromorphic extension from $\CC_\omega^+$ to the set
\begin{equation}\label{def-Omega-0}
\Omega_0=\{\lambda\in\CC:\re\lambda>g_0(\im\lambda)\},
\end{equation}
also denoted $R_A(\cdot)$, slightly abusing the notation;
\item[(ii)] The operator-valued function $R_A(\cdot)$ has finitely many singularities in $\Omega_0\setminus\CC_\omega^+$, denoted $\mu_k$, $k=1,\dots,m$. All singularities are finitely meromorphic points of $R_A(\cdot)$;
\item[(iii)] There exist two piecewise continuous functions $h_1,h_2:[0,\infty)\to[0,\infty)$ such that
\begin{equation}\label{RA-h1-h2-est}
\|R_A(\lambda)\|\leq h_1(|\lambda|)h_2(\re\lambda)\;\mbox{for any}\;\lambda\in\Omega_0\setminus\CC_\omega^+\;\mbox{such that}\;|\im\lambda|>\max_{1\leq k\leq m}|\im\mu_k|;
\end{equation}	
\item[(iv)] There exists $n_0\in\NN$ such that
\begin{equation}\label{limit-n-zero}
\lim\limits_{s\to\infty}\frac{h_1(s)}{s^{2n_0-1}}=0.
\end{equation}
Moreover, the function $h_2$ is bounded on any compact interval.
\end{enumerate}
In the sequel we denote the set of singularities by
\begin{equation}\label{def-calM}
\cM=\{\mu_k:k=1,\dots,m\}.	
\end{equation}	
In general one expects the set $\Omega_0$ to be a subset of the set $\{\lambda\in\CC:\lambda^2\in\rho_{\mathrm{ess}}(A)\}$ for which the estimate \eqref{RA-h1-h2-est} holds true at the right of the curve $\re\lambda=g_0(\im\lambda)$.
\begin{figure}[h]
	\begin{center}
		\includegraphics[width=0.6\textwidth]{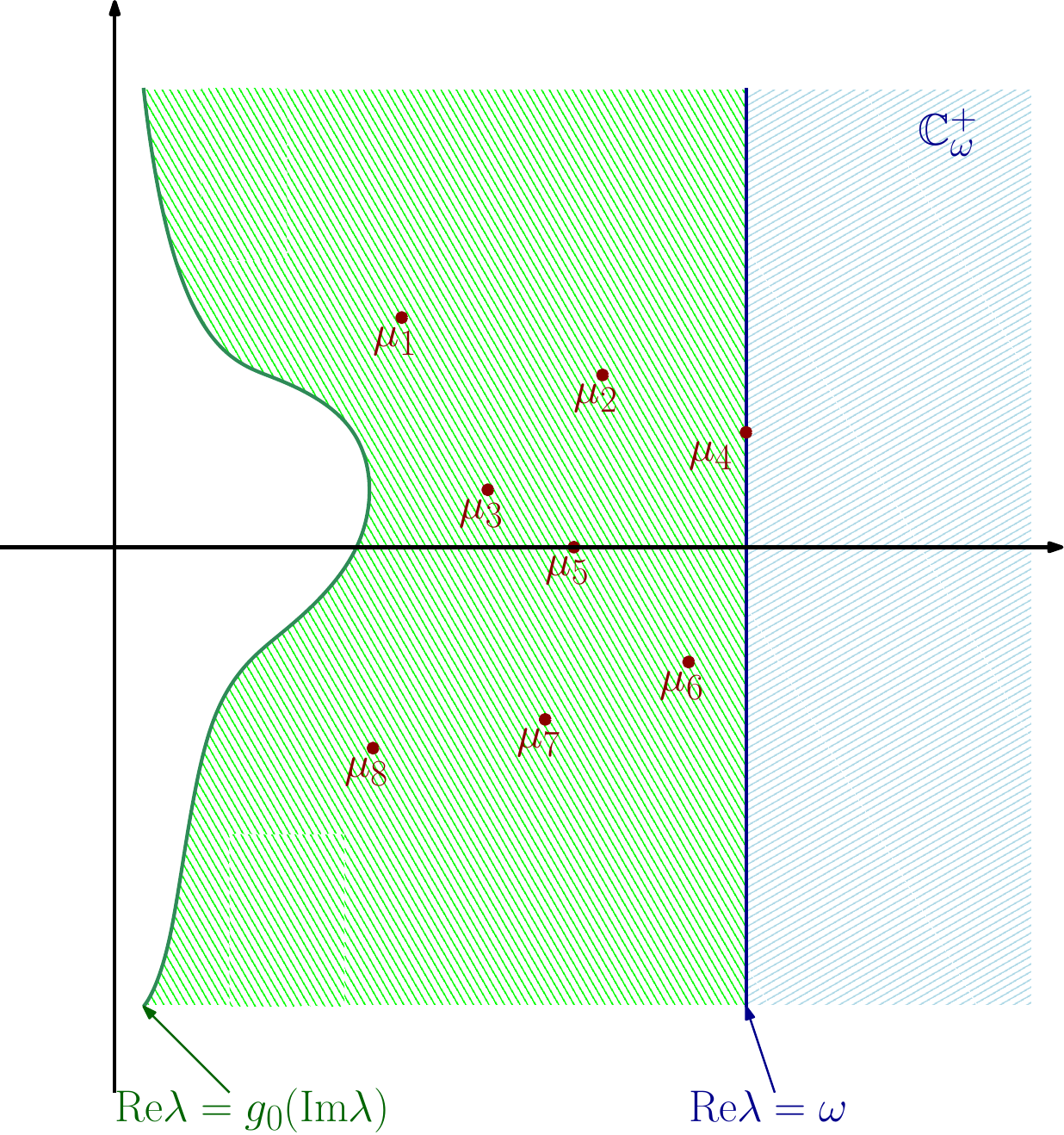}
		
		Figure 1. A generic depiction of the sets introduced in Hypothesis (H). The set $\Omega_0$ consists of the union of the regions depicted in green and blue. The finitely meromorphic points are depicted in dark red and the graph $\re\lambda=g_0(\im\lambda)$ in dark green.
	\end{center}
\end{figure}
Our next task is to study the consequences of Hypothesis (H). First, we describe the properties of the regular points of the operator-valued function $R_A(\cdot)$. Then, we will show how the finitely meromorphic points $\mu_k$, $k=1,\dots m$, are connected to the discrete spectrum of the linear operator $A$.
\begin{lemma}\label{l3.2}
Assume Hypothesis (H). Then, the following assertions hold true:
\begin{enumerate}
\item[(i)] $\big\{\lambda^2:\lambda\in\Omega_0\setminus\cM\big\}\subset\rho(A)$;
\item[(ii)] $R(\lambda^2,A)=R_A(\lambda)$ for any $\lambda\in\Omega_0\setminus\cM$;
\item[(iii)] For any $n\in\NN$, $x\in\dom(A^n)$ we have
\begin{equation}\label{RA-reduction}
R_A(\lambda)A^nx=\lambda^{2n}R_A(\lambda)x-\displaystyle\sum_{j=0}^{n-1}\lambda^{2n-2-2j}A^jx\;\mbox{for any}\; \lambda\in\Omega_0\setminus\cM.	
\end{equation}		
\end{enumerate}
\end{lemma}	
\begin{proof}
Remark~\ref{r2.4}(i) yields the inclusion $\{\lambda^2:\lambda\in\CC_\omega^+\}\subset\rho(A)$ and
\begin{equation}\label{l3.2.1}	
R(\lambda^2,A)Ax=\lambda^2R(\lambda^2,A)x-x\;\mbox{for any}\;\lambda\in\CC_\omega^+,\,x\in\dom(A).
\end{equation}
Since $R_A(\cdot)$ is an analytic extension to $\Omega_0\setminus\cM$ of the function $\lambda\to R(\lambda^2,A):\CC_\omega^+\to\mathcal{B}(\bX)$ by Hypothesis (H)(i) and both sides of \eqref{l3.2.1} are analytic, we infer that
\begin{equation}\label{l3.2.2}	
	R_A(\lambda)Ax=\lambda^2R_A(\lambda)x-x\;\mbox{for any}\;\lambda\in\Omega_0\setminus\cM, x\in\dom(A).
\end{equation}
Next, for each $x\in\bX$, we define the function $F_x:\Omega_0\setminus\cM\to\bX\times\bX$ by $F_x(\lambda)=\bigl(R_A(\lambda)x,\lambda^2R_A(\lambda)x-x\bigr)$. By Hypothesis (H)(i), $F_x$ is analytic. Moreover, from Remark~\ref{r2.4}(i) it follows that $\mathrm{Range}(R_A(\lambda))=\mathrm{Range}(R(\lambda^2,A))=\dom(A)$ and $AR_A(\lambda)x=\lambda^2R_A(\lambda)x-x$ for any $\lambda\in\CC_\omega^+$. Hence,
\begin{equation}\label{l3.2.3}
F_x(\lambda)=\bigl(R_A(\lambda)x,\lambda^2R_A(\lambda)x-x\bigr)\in\mathrm{Graph}(A)\;\mbox{for any}\;\lambda\in\CC_\omega^+.
\end{equation}	
 Since $\mathrm{Graph}(A)$ is a closed subspace of $\bX\times\bX$ as $A$ is closed,  from Lemma~\ref{lB3} and \eqref{l3.2.3} we conclude that $F_x(\lambda)\in\mathrm{Graph}(A)$ for any $\lambda\in\Omega_0\setminus\cM$. Reformulating, we have
 \begin{equation}\label{l3.2.4}
R_A(\lambda)x\in\dom(A),\,AR_A(\lambda)x=\lambda^2R_A(\lambda)x-x\;\mbox{for any}\;\lambda\in\Omega_0\setminus\cM, x\in\bX.
\end{equation}
Assertions (i) and (ii) follow from \eqref{l3.2.2} and \eqref{l3.2.4}, while assertion (iii) follows  from \eqref{l3.2.2} by induction. 	
\end{proof}	
\begin{lemma}\label{l3.3}
Assume Hypothesis (H). Then, the following assertions hold true:
\begin{enumerate}
\item[(i)] $\mu_k^2\in\sigma_{\mathrm{disc}}(A)$ for any $k=1,\dots,m$;
\item[(ii)] There exists $r_*>0$, small enough, such that the spectral projection relative to the spectral subset $\{\mu_k^2\}$ of $\sigma(A)$ is given by
\begin{equation}\label{spectral-projection}
P_k=\frac{1}{\pi\rmi}\int_{\partial D(\mu_k,r_*)}\lambda R_A(\lambda)\rmd\lambda\;\mbox{for any}\;k=1,\dots,m;	
\end{equation}	
\item[(iii)]  For any $k\in\{1,\dots,m\}$ there exists $N_k\in\NN$ and an analytic function $E_k:D(\mu_k,2r_*)\to\mathcal{B}(\bX)$ such that $\Range(P_k)\subset\dom(A^{N_k})$ and
\begin{equation}\label{Laurent-muk}
R_A(\lambda)=\sum_{j=1}^{N_k}\frac{1}{(\lambda^2-\mu_k^2)^j}(A-\mu_k^2 I_\bX)^{j-1}P_k+E_k(\lambda^2)\;\mbox{for any}\;\lambda\in D(\mu_k,2r_*);
\end{equation}
\item[(iv)] For any $k\in\{1,\dots,m\}$ we have the following residues formula
\begin{equation}\label{res-cos-sin}
\Res\Bigl(\lambda^\kappa e^{\lambda t}R_A(\lambda),\lambda=\mu_k\Bigr)=\sum_{j=1}^{N_k}\Bigg(\frac{1}{2\pi\rmi}\int_{\partial D(\mu_k,r_*)}\frac{\lambda^\kappa e^{\lambda t}}{(\lambda^2-\mu_k^2)^j}\rmd\lambda\Bigg)(A-\mu_k^2 I_\bX)^{j-1}P_k,\;\kappa=0,1.
\end{equation}	
\end{enumerate}
\end{lemma}	
\begin{proof} First, we fix $s_0>0$ independent of $k\in\{1,\dots,m\}$  such that the elements of the set $\{\mu_k^2: k=1,\dots,m\}$ are separated by disks of radius $2s_0$, that is
\begin{equation}\label{l3.3.1}	
D(\xi,2s_0)\cap\{\mu_k^2: k=1,\dots,m\}=\{\xi\}\;\mbox{for any}\;\xi\in\{\mu_k^2: k=1,\dots,m\}.	
\end{equation}	
From \eqref{l3.3.1} it follows that there exists $r_0>0$, small enough, such that
\begin{equation}\label{l3.3.2}
D(\mu_k,2r_0)\subset\Omega_0,\;D(\mu_k,2r_0)\cap D(\mu_\ell,2r_0)=\emptyset,\;\bigl\{\lambda^2:\lambda\in D(\mu_k,2r_0)\bigr\}\subset D(\mu_k^2, s_0)
\end{equation}
for any $k,\ell\in\{1,\dots,m\}$, $k\ne\ell$. Since the entire transformation $\lambda\to\lambda^2:\CC\to\CC$ is non-constant, from the Open Mapping Theorem we infer that the set $\{\lambda^2:\lambda\in D(\mu_k,r_0)\}$ is an open set for any $k=1,\dots,m$. We infer that there exists $s_*\in(0,\frac{s_0}{2})$, independent of $k\in\{1,\dots,m\}$, such that
\begin{equation}\label{l3.3.3}
D(\mu_k^2,2s_*)\subset\bigl\{\lambda^2:\lambda\in D(\mu_k,2r_0)\bigr\}\subset D(\mu_k^2, s_0)\;\mbox{for any}\; k=1,\dots,m.	
\end{equation}	
Next, we fix $\xi\in D(\mu_k^2,2s_*)\setminus\{\mu_k^2\}$. From \eqref{l3.3.3} we have $\xi=\lambda^2$ for some $\lambda\in D(\mu_k, 2r_0)\subset\Omega_0$. Clearly, $\lambda\ne\mu_k$. Moreover, from \eqref{l3.3.2} we infer that $\lambda\ne\mu_\ell$ for any $\ell\in\{1,\dots,m\}$, hence $\lambda\in\Omega_0\setminus\cM$. From Lemma~\ref{l3.2} we conclude that $\xi=\lambda^2\in\rho(A)$. Summarizing, we have
\begin{equation}\label{l3.3.4}
D(\mu_k^2,2s_*)\setminus\{\mu_k^2\}\subset\rho(A)\;\mbox{for any}\;k=1,\dots,m.	
\end{equation}
From \eqref{l3.3.2}, we obtain that there exists $r_*\in(0,r_0)$, small enough, such that
\begin{equation}\label{l3.3.5}
D(\mu_k,2r_*)\subset\Omega_0,\;D(\mu_k,2r_*)\cap D(\mu_\ell,2r_*)=\emptyset,\;\bigl\{\lambda^2:\lambda\in D(\mu_k,2r_*)\bigr\}\subset D(\mu_k^2, s_*)	
\end{equation}	
for any $k,\ell\in\{1,\dots,m\}$, $k\ne\ell$. From \eqref{l3.3.4} it follows that $R(\cdot,A)$ is analytic on $D(\mu_k^2,2s_*)\setminus\{\mu_k^2\}$ for any $k=1,\dots,m$, hence it has a Laurent series representation, see, e.g. \cite{GL-book}. That is, there exists $E_k:D(\mu_k^2,2s_*)\to\mathcal{B}(\bX)$ an analytic operator-valued function and $G_{k,j}\in\mathcal{B}(\bX)$, $k=1,\dots,m$, $j\in\NN$ such that
\begin{equation}\label{l3.3.6}
R(\xi,A)=E_k(\xi)+\sum_{j=1}^{\infty}\frac{1}{(\xi-\mu_k^2)^j}G_{k,j}\;\mbox{for any}\;\xi\in D(\mu_k^2,2s_*)\setminus\{\mu_k^2\}, k=1,\dots,m.
\end{equation}
The coefficients of the principal part of the Laurent series are given by
\begin{equation}\label{l3.3.7}
G_{k,j}=\frac{1}{2\pi\rmi}\int_\Lambda(\xi-\mu_k^2)^{j-1}R(\xi,A)\rmd\xi\;\mbox{for any}\;k=1,\dots,m,\,j\in\NN,
\end{equation}
where $\Lambda$ is closed, simple, counter-clockwise oriented path, surrounding $\mu_k^2$, whose range is contained in $D(\mu_k^2,2s_*)\setminus\{\mu_k^2\}$. Next, we define the path
\begin{equation}\label{l3.3.8}
\Lambda_{k,r}: \xi=(\mu_k+re^{\rmi t})^2,\, t\in [0,2\pi],\, k=1,\dots,m,\, r\in(0,2r_*). 	
\end{equation}	
The radius $r_*>0$ can be chosen small enough such that the path $\Lambda_{k,r}$ is simple.
From \eqref{l3.3.5} we infer that we can use the path $\Lambda_{k,r}$ in formula \eqref{l3.3.7} for any $r\in (0,2r_*)$ and any $k=1,\dots,m$. It follows that
\begin{align}\label{l3.3.9}
G_{k,j}&=\frac{1}{\pi\rmi}\int_{0}^{2\pi}r\rmi e^{\rmi t}(\mu_k+re^{\rmi t})\Bigl((\mu_k+re^{\rmi t})^2-\mu_k^2\Bigr)^{j-1}R\Bigl((\mu_k+re^{\rmi t})^2,A\Bigr)\rmd t\nonumber\\
&=\frac{1}{\pi\rmi}\int_{\partial D(\mu_k,r)}\lambda\bigl(\lambda^2-\mu_k^2\bigr)^{j-1}R_A(\lambda)\rmd\lambda\;\;\mbox{for any}\;k=1,\dots,m, j\in\NN, r\in(0,2r_*).
\end{align}	
Next, we will show that the principal part of the Laurent series in \eqref{l3.3.6} has finitely many terms and all coefficients are of finite rank. Indeed, since $\mu_k$ is a finitely meromorphic point by Hypothesis (H) (ii), from Definition~\ref{d3.1} we know that for each $k\in\{1,\dots,m\}$ there exists $N_k\in\NN$, an analytic operator-valued function $F_k:D(\mu_k,2r_*)\to\mathcal{B}(\bX)$  and  finite rank operators $R_{k,j}\in\mathcal{B}(\bX)$, $j=1,\dots,N_k$, such that
\begin{equation}\label{l3.3.10}
R_A(\lambda)=F_k(\lambda)+\sum_{\ell=1}^{N_k}\frac{1}{(\lambda-\mu_k)^\ell}R_{k,\ell}\;\mbox{for any}\;\lambda\in D(\mu_k,2r_*)\setminus\{\mu_k\}, k=1,\dots,m.
\end{equation}
From \eqref{l3.3.9} and \eqref{l3.3.10} we obtain that
\begin{equation}\label{l3.3.11}
G_{k,j}=\frac{1}{\pi\rmi}\sum_{\ell=1}^{N_k}\Bigg(\int_{\partial D(\mu_k,r_*)}\frac{\lambda\bigl(\lambda^2-\mu_k^2\bigr)^{j-1}}{(\lambda-\mu_k)^\ell}\rmd\lambda\Bigg)R_{k,\ell}\;\mbox{for any}\;k=1,\dots,m, j\in\NN.
\end{equation}
Since $R_{k,j}$ are finite rank operators, from \eqref{l3.3.11} we conclude that
\begin{equation}\label{l3.3.12}
G_{k,j}=0\;\mbox{for any}\; j\geq 1+N_k,\;\;G_{k,j}\;\mbox{is of finite rank for any}\; j=1,\dots,N_k,\; k=1,\dots,m.	
\end{equation}
From \eqref{l3.3.6} and \eqref{l3.3.12} we infer that $\mu_k^2$ is a finitely meromorphic point of $R(\cdot,A)$. From \eqref{l3.3.9}, the resolvent equation and \cite[Proposition 1.1.7]{ABHN} it follows that $\mathrm{Range}(G_{k,j})\subseteq\dom(A)$ and
\begin{align}\label{l3.3.13}
(A-\mu_k^2I_\bX)G_{k,j}x&=\frac{1}{\pi\rmi}\int_{\partial D(\mu_k,r_*)}\lambda\bigl(\lambda^2-\mu_k^2\bigr)^{j-1}(A-\mu_k^2I_\bX)R_A(\lambda)x\rmd\lambda\nonumber\\
&=\frac{1}{\pi\rmi}\int_{\partial D(\mu_k,r_*)}\lambda\bigl(\lambda^2-\mu_k^2\bigr)^{j-1}\Bigl((\lambda^2-\mu_k^2)R_A(\lambda)x-x\Bigr)\rmd\lambda\nonumber\\
&=\frac{1}{\pi\rmi}\int_{\partial D(\mu_k,r_*)}\lambda\bigl(\lambda^2-\mu_k^2\bigr)^jR_A(\lambda)x\rmd\lambda=G_{k,j+1}x
\end{align}	
for any $j=1,\dots,N_k$, $k=1,\dots,m$, $x\in\bX$.  One can use another contour integration argument to show that $P_k:=G_{k,1}$, $k=1,\dots,m$ is a projection that commutes with the linear operator $A$, see e. g. \cite {GGK,Kato}. Since $R_A(\cdot)$ has a singularity at $\mu_k$, we have that $P_k\ne0$. Moreover, from \eqref{l3.3.12} we have $P_k$ is of finite rank for any $k=1,\dots,m$. From \eqref{l3.3.13} it follows that
\begin{equation}\label{l3.3.14}
\mathrm{Range}(P_k)\subset\dom(A^{N_k})\;\mbox{and}\;G_{k,j}=(A-\mu_k^2I_\bX)^{j-1}P_k\;\mbox{for any}\;k=1,\dots,m, j=1,\dots,N_k+1.
\end{equation}
From \eqref{l3.3.12} and \eqref{l3.3.14} we obtain that $(A-\mu_k^2I_\bX)^{N_k}P_k=0$ for any $k=1,\dots,m$. Since the projection $P_k$ commutes with $A$ we conclude that $(A-\mu_k^2I_\bX)_{|\mathrm{Range}(P_k)}$ is a nilpotent linear operator on the non-trivial, finite dimensional space $\mathrm{Range}(P_k)$ for any $k=1,\dots,m$, which implies assertion (i). Assertion (ii) follows from Lemma~\ref{l3.2}(i), (i), \eqref{l3.3.9} for $j=1$ and $r=r_*$. Assertion (iii) follows from \eqref{l3.3.5}, \eqref{l3.3.6}, \eqref{l3.3.12} and \eqref{l3.3.14}. Assertion (iv) follows immediately from representation \eqref{Laurent-muk}.
\end{proof}
\begin{remark}\label{r3.4}
An elementary computation shows for any $\kappa=0,1$, $k=1,\dots,m$, $j\in\NN$ there exists a polynomial $p_{\kappa,k,j}$ of degree at most $j-1$ such that
\begin{equation}\label{polynomial-representation}
\frac{1}{2\pi\rmi}\int_{\partial D(\mu_k,r_*)}\frac{\lambda^\kappa e^{\lambda t}}{(\lambda^2-\mu_k^2)^j}\rmd\lambda=p_{\kappa,k,j}(t)e^{\mu_k t}\;\mbox{for any}\; t\in\RR.	
\end{equation}
This fact allows us to infer that the norm of the residue from \eqref{res-cos-sin} is of order $\mathcal{O}(t^{N_k-1}e^{\mu_k t})$. 	 	
\end{remark}	
Next, we study the properties of residues of the operator-valued functions involved in the representations of cosine family and its associated sine family from Section~\ref{sec2}.
\begin{lemma}\label{l3.5}
Assume Hypothesis (H) and let $n\in\NN$ and $x\in\dom(A^n)$. Then,
\begin{enumerate}
\item[(i)]$\Res\Bigl(\frac{e^{\lambda t}}{\lambda^{2n-1}}R_A(\lambda)A^nx,\lambda=\mu_k\Bigr)=\Res\Bigl(\lambda e^{\lambda t}R_A(\lambda)x,\lambda=\mu_k\Bigr)$, if $\mu_k\ne0$;
\item[(ii)] $\Res\Bigl(\frac{e^{\lambda t}}{\lambda^{2n-1}}R_A(\lambda)A^nx,\lambda=0\Bigr)=\Res\Bigl(\lambda e^{\lambda t}R_A(\lambda)x,\lambda=0\Bigr)-\displaystyle\sum_{j=0}^{n-1}\frac{t^{2j}}{(2j)!}A^jx$, if $0\in\cM$;
\item[(iii)] $\Res\Bigl(\frac{e^{\lambda t}}{\lambda^{2n-1}}R_A(\lambda)A^nx,\lambda=0\Bigr)=-\displaystyle\sum_{j=0}^{n-1}\frac{t^{2j}}{(2j)!}A^jx$, if $0\in\Omega_0\setminus\cM$.
\end{enumerate}
\end{lemma}
\begin{proof} From Lemma~\ref{l3.2}(iii) we obtain that
\begin{equation}\label{l3.5.1}	
\frac{e^{\lambda t}}{\lambda^{2n-1}}R_A(\lambda)A^nx=\lambda R_A(\lambda)x-\displaystyle\sum_{j=0}^{n-1}\frac{e^{\lambda t}}{\lambda^{2j+1}}A^jx\;\mbox{for any}\; \lambda\in\Omega\setminus\bigl(\cM\cup\{0\}\bigr),
\end{equation}
while an elementary computation shows that
\begin{equation}\label{l3.5.2}
\Res\Bigl(\frac{e^{\lambda t}}{\lambda^\ell},\lambda=\mu\Bigr)=\left\{\begin{array}{ll} 0,& \mu\ne0,\\
\frac{t^{\ell-1}}{(\ell-1)!},&\mu=0, \end{array}\right.\;\mbox{for any}\;t\in\RR,\ell\in\NN,
\end{equation}	
thus proving the lemma.	
\end{proof}
\begin{lemma}\label{l3.6}
Assume Hypothesis (H) and let $n\in\NN$, $x\in\bX$ and $y\in\dom(A^{n-1})$. Then,
\begin{enumerate}
\item[(i)] $\Res\Bigl(\frac{e^{\lambda t}}{\lambda^{2n}}R_A(\lambda)x,\lambda=\mu\Bigr)\in\dom(A)$ for any $\mu\in\cM\cup\{0\}$;
\item[(ii)]$A\Res\Bigl(\frac{e^{\lambda t}}{\lambda^{2n}}R_A(\lambda)A^{n-1}y,\lambda=\mu_k\Bigr)=\Res\Bigl(e^{\lambda t}R_A(\lambda)y,\lambda=\mu_k\Bigr)$, if $\mu_k\ne0$;
\item[(iii)] $A\Res\Bigl(\frac{e^{\lambda t}}{\lambda^{2n}}R_A(\lambda)A^{n-1}y,\lambda=0\Bigr)=\Res\Bigl(e^{\lambda t}R_A(\lambda)y,\lambda=0\Bigr)-\displaystyle\sum_{j=0}^{n-1}\frac{t^{2j+1}}{(2j+1)!}A^jy$, if $\mu_k=0$;
\item[(iv)] $A\Res\Bigl(\frac{e^{\lambda t}}{\lambda^{2n}}R_A(\lambda)A^{n-1}y,\lambda=0\Bigr)=-\displaystyle\sum_{j=0}^{n-1}\frac{t^{2j+1}}{(2j+1)!}A^jy$, if $0\in\Omega_0\setminus\cM$.
\end{enumerate}
\end{lemma}
\begin{proof} Assertion \eqref{l3.2.4} yields
\begin{equation}\label{l3.6.1}
\frac{e^{\lambda t}}{\lambda^{2n}}R_A(\lambda)x\in\dom(A)\;\mbox{and}\;A\Bigl(\frac{e^{\lambda t}}{\lambda^{2n}}R_A(\lambda)x\Bigr)=\frac{e^{\lambda t}}{\lambda^{2n-2}}R_A(\lambda)x-\frac{e^{\lambda t}}{\lambda^{2n}}x	
\end{equation}
for any $\lambda\in\Omega_0\setminus\bigl(\cM\cup\{0\}\bigr)$, and so (i) follows from \eqref{l3.6.1} and \cite[Proposition 1.1.7]{ABHN}. Moreover,
\begin{equation}\label{l3.6.2}
A\Res\Bigl(\frac{e^{\lambda t}}{\lambda^{2n}}R_A(\lambda)x,\lambda=\mu\Bigr)=\Res\Bigl(\frac{e^{\lambda t}}{\lambda^{2n-2}}R_A(\lambda)x,\lambda=\mu\Bigr)-\Res\Bigl(\frac{e^{\lambda t}}{\lambda^{2n}},\lambda=\mu\Bigr)x
\end{equation}
for any $\mu\in\cM\cup\{0\}$. Using Lemma~\ref{l3.2}(iii),
\begin{equation}\label{l3.6.3}	
\frac{e^{\lambda t}}{\lambda^{2n-2}}R_A(\lambda)A^{n-1}y=e^{\lambda t} R_A(\lambda)y-\sigma_n(\lambda,t,y)\;\mbox{for any}\; \lambda\in\Omega_0\setminus\Bigl(\{\mu_k:k=1,\dots,m\}\cup\{0\}\Bigr),	
\end{equation}
where the function $\sigma_n:\CC\setminus\{0\}\times\RR\times\dom(A^{n-1})\to\bX$ is defined by
\begin{equation}\label{l3.6.4}
\sigma_n(\lambda,t,y)=\left\{\begin{array}{ll} 0,& n=1,\\
	\displaystyle\sum_{j=0}^{n-2}\frac{e^{\lambda t}}{\lambda^{2j+2}}A^jy,&n\geq 2. \end{array}\right.
\end{equation}
We note that
\begin{equation}\label{l3.6.5}
\sigma_n(\lambda,t,y)+\frac{e^{\lambda t}}{\lambda^{2n}}A^{n-1}y=\sum_{j=0}^{n-1}\frac{e^{\lambda t}}{\lambda^{2j+2}}A^jy	
\end{equation}	
for any $\lambda\in\Omega_0\setminus\bigl(\cM\cup\{0\}\bigr)$, $t\in\RR$, $y\in\dom(A^{n-1})$. From \eqref{l3.6.2}, \eqref{l3.6.3}, \eqref{l3.6.4} and \eqref{l3.6.5} we obtain that
\begin{equation}\label{l3.6.6}
A\Res\Bigl(\frac{e^{\lambda t}}{\lambda^{2n}}R_A(\lambda)A^{n-1}y,\lambda=\mu\Bigr)=\Res\Bigl(e^{\lambda t}R_A(\lambda)y,\lambda=\mu\Bigr)-\sum_{j=0}^{n-1}\Res\Bigl(\frac{e^{\lambda t}}{\lambda^{2j+2}},\lambda=\mu\Bigr)A^jy	
\end{equation}
for any $\mu\in\cM\cup\{0\}$. Assertions (ii)-(iv) follow from \eqref{l3.5.2} and \eqref{l3.6.6}.		
\end{proof}
In the sequel we need an estimate of  $\|AR_A(\lambda)_{|\bW}\|_{\mathcal{B}(\bW,\bX)}$ for $\lambda\in\Omega_0\setminus\CC_\omega^+$ analogous to \eqref{r2.6.3} . We recall the definition of the space $\bW$ in \eqref{def-bW}.
\begin{lemma}\label{l3.7}
Assume Hypothesis (H) and let $M_1$ is the constant defined in Lemma~\ref{l2.3}. Then,
\begin{equation}\label{ARA-est}
\|AR_A(\lambda)w\|\leq M_1\Bigl(\omega^{-1}+4\bigl(|\lambda|+\omega\bigr)h_1(|\lambda|)h_2(\re\lambda)\Bigr)\|w\|_\bW
\end{equation}	
for any $w\in\bW$, $\lambda\in\Omega_0\setminus\CC_\omega^+$ such that $|\im\lambda|>\max\limits_{1\leq k\leq m}|\im\mu_k|$.		
\end{lemma}
\begin{proof}
Fix $w\in\bW$ and $\lambda\in\Omega_0\setminus\CC_\omega^+$ such that $|\im\lambda|>\max\limits_{1\leq k\leq m}|\im\mu_k|$ and set $\mu=2\omega+|\re\lambda|+\rmi\im\lambda$. One can readily check that
\begin{equation}\label{l3.7.1}
\re\mu\geq 2\omega,\;|\mu|\leq 2\omega+|\lambda|\;\mbox{and}\;\frac{\re\mu-\re\lambda}{\re\mu-\omega}=\frac{2\omega+|\re\lambda|-\re\lambda}{\omega+|\re\lambda|}\leq 2.
\end{equation}
Since $\im\mu=\im\lambda$, by Lemma~\ref{l3.2}(ii)	and the resolvent identity we obtain
\begin{equation}\label{l3.7.2}
AR_A(\lambda)w=AR(\mu^2,A)w+(\re\mu-\re\lambda)(\lambda+\mu)R_A(\lambda)AR(\mu^2,A)w.
\end{equation}
From Hypothesis (H), \eqref{r2.6.3}, \eqref{l3.7.1} and \eqref{l3.7.2} we conclude that
\begin{align}\label{3.7.3}
\|AR_A(\lambda)w\|&\leq \|AR(\mu^2,A)w\|+(\re\mu-\re\lambda)|(\lambda+\mu)|\,\|R_A(\lambda)AR(\mu^2,A)w\|\nonumber\\
&\leq\frac{M_1}{\re\mu-\omega}\|w\|_\bW+M_1\bigl(|\lambda|+|\mu|\bigr)\frac{\re\mu-\re\lambda}{\re\mu-\omega}h_1(|\lambda|)h_2(\re\lambda)\|w\|_\bW\nonumber\\
&\leq M_1\Bigl(\omega^{-1}+4\bigl(|\lambda|+\omega\bigr)
h_1(|\lambda|)h_2(\re\lambda)\Bigr)\|w\|_\bW,
\end{align}
proving the lemma.
\end{proof}
We are now ready to see what are the conditions needed to switch integration in the representation formulas for the cosine family and its associated sine family given in
Lemma~\ref{l2.10} and Lemma~\ref{l2.11} to the path along the complex curve $\re\lambda=g_0(\im\lambda)+\eps$ for some $\eps>0$ small enough. The idea is to integrate along a curve parallel to  the left boundary of the set $\Omega_0$, such that all the singularities of the operator-valued function $R_A(\cdot)$ are located to the right of the curve cf. Figure 2. Since $\re\mu_k>g_0(\im\mu_k)$ for any $k=1,\dots,m$ by Hypothesis H(ii), we observe that
\begin{equation}\label{def-eps-o}
\eps_0:=\frac{1}{4}\min\bigl\{\omega-\sup_{s\in\RR}g_0(s),\re\mu_1-g_0(\im\mu_1),\dots,\re\mu_m-g_0(\im\mu_m)\bigr\}>0.	
\end{equation}	
For any $a>\max_{1\leq k\leq m}|\im\mu_k|$ and $\eps\in(0,\eps_0)$ we consider the closed simple path $\Gamma_a^\eps$ defined as the union $\Gamma_{a,\gamma}\cup\Gamma_{a,\up}^\eps\cup\Gamma_{a,\down}^\eps\cup-\Gamma_{a,\mi}^\eps$ and oriented counterclockwise, where
\begin{align}\label{def-Gamma-a-eps}
&\Gamma_{a,\gamma}:=\overline{\gamma-\rmi a,\gamma+\rmi a},\;\Gamma_{a,\up}^\eps:=\overline{\gamma+\rmi a,g_0(a)+\eps+\rmi a},\;\Gamma_{a,\down}^\eps=\overline{g_0(-a)+\eps-\rmi a,\gamma-\rmi a},\nonumber\\
&\Gamma_{a,\mi}^\eps: \lambda=g_0(s)+\eps+\rmi s,\; s\in[-a,a].
\end{align}
\begin{figure}[h]
	\begin{center}
		\includegraphics[width=0.6\textwidth]{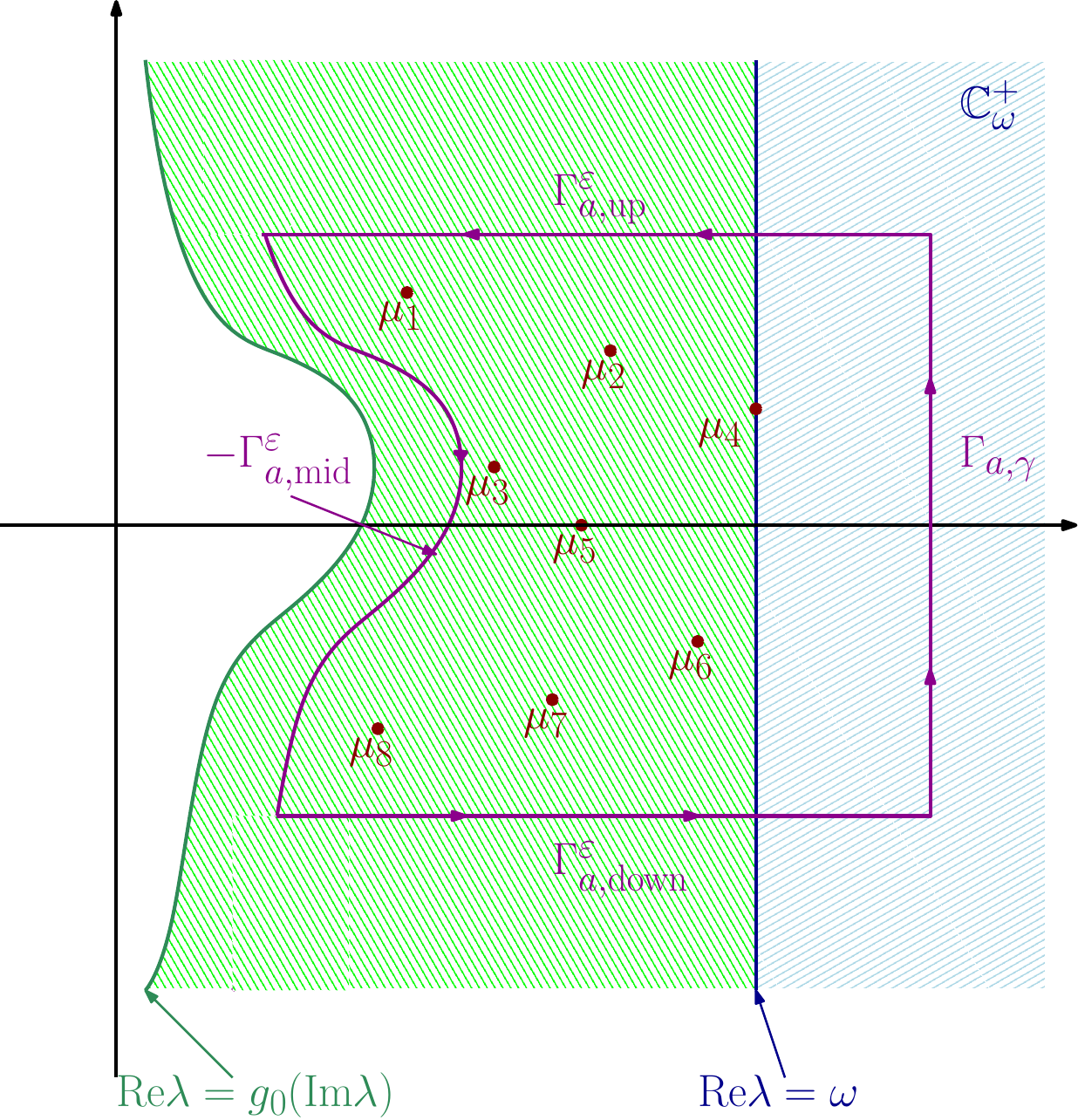}
		
		Figure 2. The path $\Gamma_a^\eps$ superimposed over the set $\Omega_0$.
	\end{center}
\end{figure}
\begin{remark}\label{r3.8}
We recall an elementary analysis result needed in the sequel. Let $G_1:[0,\infty)\to\RR$ be a measurable function such that $\lim\limits_{s\to\infty}G_1(s)=0$ and $G_2\in L^1(\RR)$. Then, by  Lebesgue's Dominated Convergence Theorem,
\begin{equation}\label{aux-limit}
\lim_{a\to\infty}\int_\RR G_1(\sqrt{\xi^2+a^2})G_2(\xi)\rmd\xi=0.	
\end{equation}	
\end{remark}
\begin{lemma}\label{l3.9}
Assume Hypothesis (H) and, in addition that $b:=\inf\limits_{s\in\RR}g_0(s)>-\infty$. Then,
\begin{equation}\label{Gamma-2-3-lim}
\lim_{a\to\infty}\int_{\Gamma_{a,\up}^\eps}\frac{e^{\lambda t}}{\lambda^k}R_A(\lambda)\rmd\lambda=\lim_{a\to\infty}\int_{\Gamma_{a,\down}^\eps}\frac{e^{\lambda t}}{\lambda^k}R_A(\lambda)\rmd\lambda=0\;\mbox{for any}\; t\geq 0,\eps\in(0,\eps_0),	
\end{equation}
in the $\mathcal{B}(\bX)$-norm for any $k\in\NN$ with $k\geq 2n_0-1$, where $n_0$ is introduced in \eqref{limit-n-zero}.	
\end{lemma}
\begin{proof} Fix $t\geq0$, $\eps\in(0,\eps_0)$ and $a>\max_{1\leq k\leq m}|\im\mu_k|$. From \eqref{def-Gamma-a-eps} we have
\begin{align}\label{l3.9.1}
\int_{\Gamma_{a,\up}^\eps}\frac{e^{\lambda t}}{\lambda^k}R_A(\lambda)\rmd\lambda&=-\int_{g_0(a)+\eps}^\gamma\frac{e^{(\xi+\rmi a)t}}{(\xi+\rmi a)^k}R_A(\xi+\rmi a)\rmd\xi,\nonumber\\
\int_{\Gamma_{a,\down}^\eps}\frac{e^{\lambda t}}{\lambda^k}R_A(\lambda)\rmd\lambda&=\int_{g_0(-a)+\eps}^\gamma\frac{e^{(\xi-\rmi a)t}}{(\xi-\rmi a)^k}R_A(\xi-\rmi a)\rmd\xi.
\end{align}		
From Hypothesis (H)(iii) we obtain
\begin{equation}\label{l3.9.2}
\Bigl\|\frac{e^{\lambda t}}{\lambda^k}R_A(\lambda)\Bigr\|=\frac{e^{(\re\lambda)t}}{|\lambda|^k}\|R_A(\lambda)\|\leq\frac{e^{\omega t}}{|\lambda|^k}h_1(|\lambda|)h_2(\re\lambda)\leq e^{\omega t}\bigl(\sup_{b\leq s\leq\gamma}h_2(s)\bigr)\frac{h_1(|\lambda|)}{|\lambda|^k},
\end{equation}
whenever $\lambda\in\mathrm{Range}(\Gamma_{a,\up}^\eps)\cup\mathrm{Range}(\Gamma_{a,\down}^\eps)$ and $\re\lambda\leq\omega$. It follows that
\begin{equation}\label{l3.9.3}
\Bigg\|\int_{g_0(\pm a)+\eps}^\omega\frac{e^{(\xi\pm\rmi a)t}}{(\xi\pm\rmi a)^k}R_A(\xi\pm\rmi a)\rmd\xi\Bigg\|	\leq e^{\omega t}\bigl(\sup_{b\leq s\leq\gamma}h_2(s)\bigr)\int_{b}^\omega\frac{h_1(\sqrt{\xi^2+a^2})}{(\xi^2+a^2)^{\frac{k}{2}}}\rmd\xi.	
\end{equation}
Since $\lim\limits_{s\to\infty}\frac{h_1(s)}{s^k}=0$ by the assumption and $\chi_{[b,\omega]}\in L^1(\RR)$, from Remark~\ref{r3.8} we infer
\begin{equation}\label{l3.9.4}
\lim_{a\to\infty}\int_{b}^\omega\frac{h_1(\sqrt{\xi^2+a^2})}{(\xi^2+a^2)^{\frac{k}{2}}}\rmd\xi=0.	
\end{equation}
From \eqref{l3.9.3} and \eqref{l3.9.4} we conclude that
\begin{equation}\label{l3.9.5}
\lim_{a\to\infty}\int_{g_0(a)+\eps}^\omega\frac{e^{(\xi+\rmi a)t}}{(\xi+\rmi a)^k}R_A(\xi+\rmi a)\rmd\xi=\lim_{a\to\infty}\int_{g_0(-a)+\eps}^\omega\frac{e^{(\xi-\rmi a)t}}{(\xi-\rmi a)^k}R_A(\xi-\rmi a)\rmd\xi=0.
\end{equation}
Since $a>\max_{1\leq k\leq m}|\im\mu_k|$ from Lemma~\ref{l3.2}(i) we have $(\omega\pm\rmi a)^2\in\rho(A)$. Moreover, from Hypothesis (H)(iii), \eqref{RA-est} and the resolvent equation we estimate:
\begin{align}\label{l3.9.6}
\|R_A(\xi\pm\rmi a)\|&\leq\|R_A(\omega\pm\rmi a)\|+\bigl\|(\xi-\omega)(\xi+\omega\pm 2\rmi a)R_A(\omega\pm\rmi a)R_A(\xi\pm\rmi a)\bigr\|\nonumber\\
&\leq\|R_A(\omega\pm\rmi a)\|+|\xi+\omega+2\rmi a|\,\|R_A(\omega\pm\rmi a)\|\,\|(\xi-\omega)R_A(\xi\pm\rmi a)\|\nonumber\\
&\leq\Bigg(1+M\sqrt{\frac{(\xi+\omega)^2+4a^2}{\xi^2+a^2}}\Bigg)h_1(\sqrt{\omega^2+a^2})h_2(\omega)\nonumber\\
& \leq\Bigg(1+M\sqrt{\frac{(\gamma+\omega)^2+4a^2}{\omega^2+a^2}}\Bigg)h_1(\sqrt{\omega^2+a^2})h_2(\omega)	
\end{align}	
for any $\xi\in[\omega,\gamma]$. From \eqref{l3.9.6} it follows that
\begin{equation}\label{l3.9.7}
\Bigg\|\int_{\omega}^\gamma\frac{e^{(\xi\pm\rmi a)t}}{(\xi\pm\rmi a)^k}R_A(\xi\pm\rmi a)\rmd\xi\Bigg\|\leq(\gamma-\omega)h_2(\omega)e^{\gamma t}\Bigg(1+M\sqrt{\frac{(\gamma+\omega)^2+4a^2}{\omega^2+a^2}}\Bigg)\frac{h_1(\sqrt{\omega^2+a^2})}{(\omega^2+a^2)^{\frac{k}{2}}}.
\end{equation}
Since $\lim\limits_{a\to\infty}\frac{h_1(\sqrt{\omega^2+a^2})}{(\omega^2+a^2)^{\frac{k}{2}}}=0$ from \eqref{l3.9.7} we obtain that
\begin{equation}\label{l3.9.8}
\lim_{a\to\infty}\int_\omega^\gamma\frac{e^{(\xi+\rmi a)t}}{(\xi+\rmi a)^k}R_A(\xi+\rmi a)\rmd\xi=\lim_{a\to\infty}\int_\omega^\gamma\frac{e^{(\xi-\rmi a)t}}{(\xi-\rmi a)^k}R_A(\xi-\rmi a)\rmd\xi=0.
\end{equation}
The lemma follows shortly from \eqref{l3.9.1}, \eqref{l3.9.5} and \eqref{l3.9.8}.	
\end{proof}
\begin{remark}\label{r3.10}
From \eqref{def-eps-o} and Hypothesis (H) one can readily check that $\cM$ is contained in the interior of the set enclosed by the path $\Gamma_a^\eps$ for any $a>\max\limits_{1\leq k\leq m}|\im\mu_k|$, $\eps\in(0,\eps_0)$. Similarly, $0\notin\mathrm{Range}(\Gamma_{a,\mi}^\eps)$. Moreover, $0$ belongs to the interior of the set enclosed by $\Gamma_a^\eps$	if and only if $g_0(0)<0$.
\end{remark}
Next, we recall the definition of convergence in the \textit{principal value sense}. Let $h:\Omega\to\bX$ be a continuous function and $\Psi:\RR\to\CC$ a piecewise smooth path. We say that $\displaystyle\int_\Psi h(\lambda)\rmd\lambda$ is \textit{convergent in the principal value sense} if $\lim\limits_{a\to\infty}\displaystyle\int_{\Psi_a}h(\lambda)\rmd\lambda$ exists in $\bX$ norm, where $\Psi_a=\Psi_{|[-a,a]}$. In this case we use the notation  $\PV\displaystyle\int_\Psi h(\lambda)\rmd\lambda=\lim\limits_{a\to\infty}\displaystyle\int_{\Psi_a}h(\lambda)\rmd\lambda$.	

We are ready to formulate the main results of this section. We introduce the path
\begin{equation}\label{def-Lambda-zero}	
\Lambda_0^\eps: \lambda=g_0(s)+\eps+\rmi s,\; s\in\RR.	
\end{equation}	
\begin{theorem}\label{t3.11}
Assume Hypothesis (H) and, in addition, that  $b:=\inf\limits_{s\in\RR}g_0(s)>-\infty$. Then,
\begin{enumerate}
\item[(i)]  The integral $\displaystyle\int_{\Lambda_0^\eps}\frac{e^{\lambda t}}{\lambda^{2n-1}}R_A(\lambda)A^nx\rmd\lambda$ is convergent in the principal value sense in $\bX$ for any $t\geq 0$, $x\in\dom(A^n)$, $\eps\in(0,\eps_0)$, $n\in\NN$, with $n\geq n_0$;
\item[(ii)] $C(t)x=\displaystyle\sum\limits_{k=1}^{m} \sum_{j=1}^{N_k}p_{1,k,j}(t)e^{\mu_k t}(A-\mu_k^2 I_\bX)^{j-1}P_kx  +\frac{1}{2\pi\rmi}\PV\displaystyle\int_{\Lambda_0^\eps}\frac{e^{\lambda t}}{\lambda^{2n-1}}R_A(\lambda)A^nx\rmd\lambda+$

$+\chi_{[0,\infty)}(g_0(0))\displaystyle\sum\limits_{j=0}^{n-1}\frac{t^{2j}}{(2j)!}A^jx$ for any $t\geq 0$, $x\in\dom(A^n)$, $\eps\in(0,\eps_0)$, $n\in\NN$, with $n\geq n_0$.
\end{enumerate}
The positive integer $n_0$ is introduced in Hypothesis (H)(iv),the projection $P_k$ and $N_k\in\NN$, $k=1,\dots,m$, were introduced in Lemma~\ref{l3.3}, and $p_{1,k,j}$ is defined in \eqref{polynomial-representation}.
\end{theorem}
\begin{proof} Fix $t\geq 0$, $x\in\dom(A^n)$, $\eps\in(0,\eps_0)$. By Remark~\ref{r3.10}, \eqref{def-eps-o} and \eqref{def-Gamma-a-eps}, the singularities of the meromorphic function $\lambda\to\frac{e^{\lambda t}}{\lambda^{2n-1}}R_A(\lambda)A^nx:\Omega_0\setminus\bigl(\cM\cup\{0\}\bigr)\to\bX$ located in the interior of the set enclosed by the closed path $\Gamma_a^\eps$ are the finitely meromorphic points $\{\mu_k:k=1,\dots,m\}$  and, possibly $0$, provided that $g_0(0)<0$. From Lemma~\ref{l3.3}, Remark~\ref{r3.4},  Lemma~\ref{l3.5} and the Residues Theorem we obtain
\begin{align}\label{t3.11.1}
\frac{1}{2\pi\rmi}\int_{\Gamma_a^\eps}\frac{e^{\lambda t}}{\lambda^{2n-1}}&R_A(\lambda)A^nx\rmd\lambda=\sum_{k=1}^{m}\Res(\lambda e^{\lambda t}R_A(\lambda)x,\lambda=\mu_k)+\chi_{(-\infty,0)}(g_0(0))\sum_{j=0}^{n-1}\frac{t^{2j}}{(2j)!}A^jx\nonumber\\
&=\sum_{k=1}^{m} \sum_{j=1}^{N_k}p_{1,k,j}(t)e^{\mu_k t}(A-\mu_k^2 I_\bX)^{j-1}P_kx+\chi_{(-\infty,0)}(g_0(0))\sum_{j=0}^{n-1}\frac{t^{2j}}{(2j)!}A^jx.	
\end{align}
for any $a>\max\limits_{1\leq k\leq m}|\im\mu_k|+1$. The theorem now follows from Lemma ~\ref{l2.10} and Lemma~\ref{l3.9} by passing to the limit as $a\to\infty$ in \eqref{t3.11.1}. 		
\end{proof}
\begin{theorem}\label{t3.12}
Assume Hypothesis (H) and, in addition, that  $b:=\inf\limits_{s\in\RR}g_0(s)>-\infty$. Then,
\begin{enumerate}
\item[(i)] The integral $\displaystyle\int_{\Lambda_0^\eps}\frac{e^{\lambda t}}{\lambda^{2n}}R_A(\lambda)A^{n-1}w\rmd\lambda$ belongs to $\dom(A)$ and it is convergent in the principal value sense in $\bX$ for any $t\geq 0$, $w\in\bW_{n-1}$, $\eps\in(0,\eps_0)$, $n\in\NN$, with $n\geq n_0$;
\item[(ii)] $S(t)w=\displaystyle\sum\limits_{k=1}^{m} \sum_{j=1}^{N_k}p_{0,k,j}(t)e^{\mu_k t}(A-\mu_k^2 I_\bX)^{j-1}P_kw  +\frac{1}{2\pi\rmi}A\PV\displaystyle\int_{\re\lambda=g_0(\im\lambda)+\eps}\frac{e^{\lambda t}}{\lambda^{2n}}R_A(\lambda)A^{n-1}w\rmd\lambda$
		
$+\chi_{[0,\infty)}(g_0(0))\sum\limits_{j=0}^{n-1}\frac{t^{2j+1}}{(2j+1)!}A^jw$ for any $t\geq 0$, $w\in\bW_{n-1}$, $\eps\in(0,\eps_0)$, $n\in\NN$, with $n\geq n_0$.
\end{enumerate}
We recall the definition of the projection $P_k$ and $N_k\in\NN$, $k=1,\dots,m$, in Lemma~\ref{l3.3}, and of the polynomial $p_{1,k,j}$  in \eqref{polynomial-representation}.
\end{theorem}
\begin{proof}
Fix $t\geq 0$, $w\in\bW_{n-1}$, $\eps\in(0,\eps_0)$. We argue as in the proof of Theorem~\ref{t3.11} with \eqref{t3.11.1} replaced by
\begin{align}\label{t3.12.1}
\frac{1}{2\pi\rmi}\int_{\Gamma_a^\eps}\frac{e^{\lambda t}}{\lambda^{2n}}R_A(\lambda)A^{n-1}w\rmd\lambda&=\sum_{\mu\in\cM\setminus\{0\}}\Res\Bigl(\frac{e^{\lambda t}}{\lambda^{2n}}R_A(\lambda)A^{n-1}w,\lambda=\mu\Bigr)\nonumber\\&+\chi_{(-\infty,0)}(g_0(0))\,\Res\Bigl(\frac{e^{\lambda t}}{\lambda^{2n}}R_A(\lambda)A^{n-1}w,\lambda=0\Bigr).
\end{align}
for any $a>\max\limits_{1\leq k\leq m}|\im\mu_k|+1$. Passing to the limit as $a\to\infty$ in \eqref{t3.12.1}, from Lemma~\ref{l2.11}(ii) and
Lemma~\ref{l3.9} we obtain
\begin{equation}\label{t3.12.2}
\lim_{a\to\infty}\int_{\Gamma_{a,\mi}^\eps}\frac{e^{\lambda t}}{\lambda^{2n}}R_A(\lambda)A^{n-1}w\rmd\lambda=\int_{\re\lambda=\gamma}\frac{e^{\lambda t}}{\lambda^{2n}}R_A(\lambda)A^{n-1}w\rmd\lambda-2\pi\rmi\cS_{n,t}w, 	
\end{equation}	
where the linear operator $\cS_{n,t}:\bW_{n-1}\to\bX$ is defined by
\begin{equation}\label{t3.12.3}
\cS_{n,t}w=\sum_{\mu\in\cM\setminus\{0\}}\Res(\frac{e^{\lambda t}}{\lambda^{2n}}R_A(\lambda)A^{n-1}w,\lambda=\mu)+\chi_{(-\infty,0)}(g_0(0))\,\Res(\frac{e^{\lambda t}}{\lambda^{2n}}R_A(\lambda)A^{n-1}w,\lambda=0).
\end{equation}
From \eqref{t3.12.2} we  conclude that $\displaystyle\int_{\Lambda_0^\eps}\frac{e^{\lambda t}}{\lambda^{2n}}R_A(\lambda)A^{n-1}w\rmd\lambda$ is convergent in the principal value sense in $\bX$.	 Moreover, since $\mathrm{Range}(\cS_{n,t})\subseteq\dom(A)$ by Lemma~\ref{l3.6}(i), from Lemma~\ref{l2.11}(iii) we infer that $\PV\displaystyle\int_{\Lambda_0^\eps}\frac{e^{\lambda t}}{\lambda^{2n}}R_A(\lambda)A^{n-1}w\rmd\lambda\in\dom(A)$, proving (i).

Next, from Lemma~\ref{l3.3}, Remark~\ref{r3.4} and Lemma~\ref{l3.6}(ii)-(iv) we obtain that $\Range(\cS_{n,t})\subseteq\dom(A)$ and
\begin{align}\label{t3.12.4}
A\cS_{n,t}w&=\sum_{k=1}^{m}\Res(\lambda e^{\lambda t}R_A(\lambda),\lambda=\mu_k)w+\chi_{(-\infty,0)}(g_0(0))\sum_{j=0}^{n-1}\frac{t^{2j}}{(2j)!}A^jx\nonumber\\
&=\sum_{k=1}^{m} \sum_{j=1}^{N_k}p_{0,k,j}(t)e^{\mu_k t}(A-\mu_k^2 I_\bX)^{j-1}P_kw+\chi_{(-\infty,0)}(g_0(0))\sum_{j=0}^{n-1}\frac{t^{2j+1}}{(2j+1)!}A^jw.		
\end{align}	
From Lemma~\ref{l2.11}(iv), \eqref{t3.12.2} and \eqref{t3.12.4} it follows that
\begin{align*}
&\frac{1}{2\pi\rmi}A\PV\displaystyle\int_{\Lambda_0^\eps}\frac{e^{\lambda t}}{\lambda^{2n}}R_A(\lambda)A^{n-1}w\rmd\lambda=\frac{1}{2\pi\rmi}A\int_{\re\lambda=\gamma}\frac{e^{\lambda t}}{\lambda^{2n}}R_A(\lambda)A^{n-1}w\rmd\lambda-A\cS_{n,t}w\nonumber\\
&\qquad\qquad=S(t)w-\sum_{j=0}^{n-1}\frac{t^{2j+1}}{(2j+1)!}A^jw-\sum_{k=1}^{m} \sum_{j=1}^{N_k}p_{0,k,j}(t)e^{\mu_k t}(A-\mu_k^2 I_\bX)^{j-1}P_kw\nonumber\\&\qquad\qquad\qquad\qquad+\chi_{(-\infty,0)}(g_0(0))\sum_{j=0}^{n-1}\frac{t^{2j+1}}{(2j+1)!}A^jw\nonumber\\
&\qquad\qquad=S(t)w-\chi_{[0,\infty)}(g_0(0))\sum\limits_{j=0}^{n-1}\frac{t^{2j+1}}{(2j+1)!}A^jw-\sum_{k=1}^{m} \sum_{j=1}^{N_k}p_{0,k,j}(t)e^{\mu_k t}(A-\mu_k^2 I_\bX)^{j-1}P_kw,
\end{align*}
proving the theorem.	
\end{proof}
We note that the conclusion of Theorem~\ref{t3.12} remain true if \eqref{limit-n-zero} is replaced by $\lim\limits_{s\to\infty}\frac{h_1(s)}{s^{2n_0}}=0$. Out next task is to estimate the integral terms in the representations of the cosine and sine families given in Theorem~\ref{t3.11} and Theorem~\ref{t3.12}, respectively. To obtain appropriate estimates we have to find a balance between the smoothness of the initial conditions of \eqref{abstract-Cauchy} and the decay properties at $+\infty$ of the function $h_1$ introduced in Hypothesis (H).
\begin{lemma}\label{l3.13}
Assume Hypothesis (H). Then, the following assertions hold true:
\begin{itemize}
	\item[(i)] There exists $M_\mi^\rmc:\NN\times(0,\infty)\times(0,\eps_0)\to(0,\infty)$  such that
\begin{equation}\label{Error-cos-mid}
\Bigg\|\int_{\Gamma_{a,\mi}^\eps}\frac{e^{\lambda t}}{\lambda^{2n-1}}R_A(\lambda)A^nx\rmd\lambda\Bigg\|\leq M_\mi^\rmc(n,a,\eps)e^{(\sup_{s\in\RR}g_0(s)+\eps)t}\|A^{n-1}x\|	
\end{equation}		
for any $t\geq 0$, $x\in\dom(A^n)$, $\eps\in(0,\eps_0)$, $a>0$;
	\item[(ii)] There exists $M_\mi^\rms:\NN\times(0,\infty)\times(0,\eps_0)\to(0,\infty)$  such that
\begin{equation}\label{Error-sine-mid}
	\Bigg\|A\int_{\Gamma_{a,\mi}^\eps}\frac{e^{\lambda t}}{\lambda^{2n}}R_A(\lambda)A^{n-1}w\rmd\lambda\Bigg\|\leq M_\mi^\rms(n,a,\eps)e^{(\sup_{s\in\RR}g_0(s)+\eps)t}\|A^{n-1}w\|
\end{equation}		
for any $t\geq 0$, $x\in\bW_{n-1}$, $\eps\in(0,\eps_0)$, $a>0$.	
\end{itemize}
\end{lemma}
\begin{proof} Fix $t\geq 0$, $x\in\dom(A^n)$, $w\in\bW_{n-1}$, $\eps\in(0,\eps_0)$. From \eqref{l3.2.2}	and \eqref{l3.2.4} we obtain that
\begin{align}\label{t3.13.1}
\frac{e^{\lambda t}}{\lambda^{2n-1}}R_A(\lambda)A^nx&=\frac{e^{\lambda t}}{\lambda^{2n-3}}R_A(\lambda)A^{n-1}x-\frac{e^{\lambda t}}{\lambda^{2n-1}}A^{n-1}x,\nonumber\\\frac{e^{\lambda t}}{\lambda^{2n}}AR_A(\lambda)A^{n-1}w&=\frac{e^{\lambda t}}{\lambda^{2n-2}}R_A(\lambda)A^{n-1}w-\frac{e^{\lambda t}}{\lambda^{2n}}A^{n-1}w.
\end{align}
for any $\lambda\in\Omega_0\setminus\bigl(\cM\cup\{0\}\bigr)$. Assertion (i) follows from \eqref{t3.13.1} for
\begin{align}\label{l3.13.2}
M_\mi^\rmc(n,a,\eps)&=\Bigl(\max\{M_{1,\mi}^\rmc(n,a,\eps),M_{2,\mi}^\rmc(n,a,\eps)\}\Bigr)\mathrm{Lenght}(\Gamma_{a,\mi}^\eps),\;\mbox{where}\nonumber\\
M_{1,\mi}^\rmc(n,a,\eps)&:=\sup_{\lambda\in\mathrm{Range}(\Gamma_{a,\mi}^\eps)}\frac{\|R_A(\lambda)\|}{|\lambda|^{2n-3}},\;M_{2,\mi}^\rmc(n,a,\eps):=\sup_{\lambda\in\mathrm{Range}(\Gamma_{a,\mi}^\eps)}\frac{1}{|\lambda|^{2n-1}}
\end{align}	
Since the linear operator $A$ is closed, \cite[Proposition 1.1.7]{ABHN} gives
\begin{equation}\label{t3.13.3}
A\int_{\Gamma_{a,\mi}^\eps}\frac{e^{\lambda t}}{\lambda^{2n}}R_A(\lambda)A^{n-1}w\rmd\lambda=\int_{\Gamma_{a,\mi}^\eps}\frac{e^{\lambda t}}{\lambda^{2n}}AR_A(\lambda)A^{n-1}w\rmd\lambda.	
\end{equation}	
Assertion (ii) follows from \eqref{t3.13.1} and \eqref{t3.13.3} for
\begin{align}\label{t3.13.4}
M_\mi^\rms(n,a,\eps)&=\Bigl(\max\{M_{1,\mi}^\rms(n,a,\eps),M_{2,\mi}^\rms(n,a,\eps)\}\Bigr)\mathrm{Lenght}(\Gamma_{a,\mi}^\eps),\;\mbox{where}\nonumber\\
M_{1,\mi}^\rms(n,a,\eps)&:=\sup_{\lambda\in\mathrm{Range}(\Gamma_{a,\mi}^\eps)}\frac{\|R_A(\lambda)\|}{|\lambda|^{2n-2}},\;M_{2,\mi}^\rms(n,a,\eps):=\sup_{\lambda\in\mathrm{Range}(\Gamma_{a,\mi}^\eps)}\frac{1}{|\lambda|^{2n}}.
\end{align}	
\end{proof}
In addition to the path $\Gamma_{a,\mi}^\eps$ introduced in \eqref{def-Gamma-a-eps}, we introduce the path
\begin{equation}\label{def-Gamma-out}
\Gamma_{a,\out}^\eps: \lambda=g_0(s)+\eps+\rmi s,\; s\in(-\infty,-a]\cup [a,\infty).
\end{equation}
To estimate integrals along the path $\Gamma_{a,\out}^\eps$ we need the elementary integral estimate \eqref{r3.14.6} below.
\begin{remark}\label{r3.14}
Assume  that the function $g_0$ introduced in Hypothesis (H) is such that $g_0,g_0'\in L^\infty(\RR)$. We define
\begin{equation}\label{r3.14.1}
\os_0=\max\big\{2(\|g_0\|_\infty+\eps_0)\|g_0'\|_\infty,\|g_0\|_\infty+\eps_0\big\}+1.	
\end{equation}
One can readily check that the function
\begin{equation}\label{r3.14.2}
s\to s^2+(g_0(s)+\eps)^2\;\mbox{is increasing on}\;[\os_0,\infty)\;\mbox{and decreasing on}\;(-\infty,-\os_0]
\end{equation}
for any $\eps\in(0,\eps_0)$. Moreover,
\begin{align}\label{r3.14.3}
\frac{\sqrt{s^2+(g_0(s)+\eps)^2}}{|s+(g_0(s)+\eps)g_0'(s)|}&\leq\frac{|s|+\|g_0\|_\infty+\eps_0}{|s|-(\|g_0\|_\infty+\eps_0)\|g_0'\|_\infty}\leq\frac{|s|+\|g_0\|_\infty+\eps_0}{|s|-\os_0/2+1/2}\leq\frac{\os_0+\|g_0\|_\infty+\eps_0}{\os_0/2+1/2}\nonumber\\
&\leq2+\frac{2(\|g_0\|_\infty+\eps_0)}{\os_0}\leq 4,\;\mbox{whenever}\;|s|\geq\os_0,\;\mbox{for any}\;\eps\in(0,\eps_0).	
\end{align}
Fix a piecewise continuous function $h:[0,\infty)\to[0,\infty)$ and $\eta>0$. Making the change of variables $\xi=\sqrt{s^2+(g_0(s)+\eps)^2}$, from \eqref{r3.14.2} and \eqref{r3.14.3} we obtain that
\begin{align}\label{r3.14.4}
\int_{\os_0}^{\infty}&\frac{h\big(\sqrt{s^2+(g_0(s)+\eps)^2}\big)\rmd s}{\big(s^2+(g_0(s)+\eps)^2\big)^\eta}\leq 4\int_{\os_0}^{\infty}\frac{h\big(\sqrt{s^2+(g_0(s)+\eps)^2}\big)}{\big(s^2+(g_0(s)+\eps)^2\big)^\eta}\,\cdot\frac{s+(g_0(s)+\eps)g_0'(s)}{\sqrt{s^2+(g_0(s)+\eps)^2}}\rmd s\nonumber\\&=4\int_{\overline{\xi}_{\eps,+}}^{\infty}\frac{h(\xi)\rmd\xi}{\xi^{2\eta}}\leq4\int_{1}^{\infty}\frac{h(\xi)\rmd\xi}{\xi^{2\eta}},\;\mbox{where}\;\overline{\xi}_{\eps,+}=\sqrt{\os_0^2+(g_0(\os_0)+\eps)^2}\geq 1.
\end{align}	
for any piecewise continuous function $h:[0,\infty)\to[0,\infty)$. Similarly,
\begin{align}\label{r3.14.5}
\int_{-\infty}^{-\os_0}&\frac{h\big(\sqrt{s^2+(g_0(s)+\eps)^2}\big)\rmd s}{\big(s^2+(g_0(s)+\eps)^2\big)^\eta}\leq 4\int_{-\infty}^{-\os_0}\frac{h\big(\sqrt{s^2+(g_0(s)+\eps)^2}\big)}{\big(s^2+(g_0(s)+\eps)^2\big)^\eta}\,\cdot\Bigg(-\frac{s+(g_0(s)+\eps)g_0'(s)}{\sqrt{s^2+(g_0(s)+\eps)^2}}\Bigg)\rmd s\nonumber\\&=4\int_{\overline{\xi}_{\eps,-}}^{\infty}\frac{h(\xi)\rmd\xi}{\xi^{2\eta}}\leq4\int_{1}^{\infty}\frac{h(\xi)\rmd\xi}{\xi^{2\eta}},\;\mbox{where}\;\overline{\xi}_{\eps,-}=\sqrt{\os_0^2+(g_0(-\os_0)+\eps)^2}\geq 1.
\end{align}	
Summarizing from \eqref{r3.14.4} and \eqref{r3.14.5} we infer that
\begin{equation}\label{r3.14.6}
\int_{|s|\geq\os_0}\frac{h\big(\sqrt{s^2+(g_0(s)+\eps)^2}\big)\rmd s}{\big(s^2+(g_0(s)+\eps)^2\big)^\eta}\leq 8\int_{1}^{\infty}\frac{h(\xi)\rmd\xi}{\xi^{2\eta}}
\end{equation}
for any piecewise continuous function $h:[0,\infty)\to[0,\infty)$ and any $\eta>0$.				
\end{remark}
\begin{lemma}\label{l3.15}
Assume Hypothesis (H) and, in addition, that $b:=\inf\limits_{s\in\RR}g_0(s)>-\infty$, $g_0'\in L^\infty(\RR)$,  $\displaystyle\int_1^\infty\frac{h_1(s)\rmd s}{s^{2n_0-1}}<\infty$. Then, for any $a>\oa_0:=\os_0+\max\limits_{1\leq k\leq m}|\im\mu_k|$, with $\os_0$ from \eqref{r3.14.1}, $t\geq 0$, $\eps\in(0,\eps_0)$, $n\in\NN$ with $n\geq n_0$, $x\in\dom(A^n)$, $w\in\bW_{n-1}$,
\begin{enumerate}
\item[(i)] The integral $\displaystyle\int_{\Gamma_{a,\out}^\eps}\frac{e^{\lambda t}}{\lambda^{2n-1}}R_A(\lambda)A^nx\rmd\lambda$ is absolutely convergent and
\begin{equation}\label{Error-cos-out}
\Bigg\|\int_{\Gamma_{a,\out}^\eps}\frac{e^{\lambda t}}{\lambda^{2n-1}}R_A(\lambda)A^nx\rmd\lambda\Bigg\|\leq8\Bigg(\int_{1}^{\infty}\frac{h_1(\xi)\rmd\xi}{\xi^{2n-1}}\Bigg)\Big(\|g_0'\|_\infty+1\Big)\Big(\sup_{\xi\leq\gamma}h_2(\xi)\Big)	 e^{(\sup_{s\in\RR}g_0(s)+\eps)t}\|A^nx\|;	
\end{equation}
\item[(ii)] The integral $\displaystyle\int_{\Gamma_{a,\out}^\eps}\frac{e^{\lambda t}}{\lambda^{2n}}AR_A(\lambda)A^{n-1}w\rmd\lambda$ is absolutely convergent and
\begin{equation}\label{Error-sine-out}
\Bigg\|\int_{\Gamma_{a,\mi}^\eps}\frac{e^{\lambda t}}{\lambda^{2n}}AR_A(\lambda)A^{n-1}w\rmd\lambda\Bigg\|\leq \oM(n)e^{(\sup_{s\in\RR}g_0(s)+\eps)t}\|A^{n-1}w\|_\bW,	
\end{equation}
where
\begin{equation}\label{def-oM}
\oM(n)=\big(\|g_0'\|_\infty+1\big)\Bigg(32M_1(1+\omega)
\Big(\sup_{\xi\leq\gamma}h_2(\xi)\Big)\int_{1}^{\infty}\frac{h_1(\xi)\rmd\xi}{\xi^{2n-1}}
+\frac{2M_1}{(2n-1)\omega}\Bigg).
\end{equation}	
\end{enumerate}		
\end{lemma}	
\begin{proof} Fix $t\geq 0$, $x\in\dom(A^n)$, $\eps\in(0,\eps_0)$. By Hypothesis (H), \eqref{def-Gamma-out} and \eqref{r3.14.6} we obtain
\begin{align}\label{l3.15.1}	
\Bigg\|\int_{\Gamma_{a,\mi}^\eps}&\frac{e^{\lambda t}}{\lambda^{2n-1}}R_A(\lambda)A^nx\rmd\lambda\Bigg\|\leq\int_{|s|\geq\os_0}\frac{|e^{(g_0(s)+\eps+\rmi s)t}|\|A^nx\|}{|g_0(s)+\eps+\rmi s|^{2n-1}}\big\|R_A\big(g_0(s)+\eps+\rmi s\big)\big\||g_0'(s)+\rmi|\rmd s\nonumber\\
&\leq\|A^nx\|\big(\|g_0'\|_\infty+1\big)\int_{|s|\geq\os_0}\frac{e^{(g_0(s)+\eps)t}}{|g_0(s)+\eps+\rmi s|^{2n-1}}h_1\big(|g_0(s)+\eps+\rmi s|\big)h_2\big(g_0(s)+\eps\big)\rmd s\nonumber\\
&\leq\|A^nx\|\big(\|g_0'\|_\infty+1\big)\Big(\sup_{\xi\leq\gamma}h_2(\xi)\Big)	e^{(\sup_{s\in\RR}g_0(s)+\eps)t}\int_{|s|\geq\os_0}\frac{h_1\big(\sqrt{s^2+(g_0(s)+\eps)^2}\big)}{\big(g_0(s)+\eps+\rmi s\big)^{\frac{2n-1}{2}}}\rmd s\nonumber\\
&\leq8\Bigg(\int_{1}^{\infty}\frac{h_1(\xi)\rmd\xi}{\xi^{2n-1}}\Bigg)\Big(\|g_0'\|_\infty+1\Big)\Big(\sup_{\xi\leq\gamma}h_2(\xi)\Big)	e^{(\sup_{s\in\RR}g_0(s)+\eps)t}\|A^nx\|,	
\end{align}	
proving (i). Next, we fix $w\in\bW_{n-1}$. Since $A^{n-1}w\in\bW$ from Lemma~\ref{l3.7} we infer that
\begin{equation}\label{l3.15.2}
\|AR_A(\lambda)A^{n-1}w\|\leq \Big(c_1+c_2|\lambda|h_1(|\lambda|)h_2(\re\lambda)\Bigr)\|A^{n-1}w\|_\bW
\end{equation}	
for any $\lambda\in\Omega_0\setminus\CC_\omega^+$ such that $|\im\lambda|>\max\limits_{1\leq k\leq m}|\im\mu_k|+1$, where $c_1=\frac{M_1}{\omega}$ and $c_2=4M_1(1+\omega)$. From  Hypothesis (H), \eqref{def-Gamma-out}, \eqref{r3.14.6} and \eqref{l3.15.2} it follows that
\begin{align}\label{l3.15.3}	
\Bigg\|\int_{\Gamma_{a,\mi}^\eps}&\frac{e^{\lambda t}}{\lambda^{2n}}AR_A(\lambda)A^{n-1}w\rmd\lambda\Bigg\|\leq\int_{|s|\geq\os_0}\frac{|e^{(g_0(s)+\eps+\rmi s)t}|}{|g_0(s)+\eps+\rmi s|^{2n}}\big\|AR_A\big(g_0(s)+\eps+\rmi s\big)A^{n-1}w\big\||g_0'(s)+\rmi|\rmd s\nonumber\\
&\leq\|A^{n-1}w\|\big(\|g_0'\|_\infty+1\big)\Bigg(c_2\int_{|s|\geq\os_0}\frac{e^{(g_0(s)+\eps)t}}{|g_0(s)+\eps+\rmi s|^{2n-1}}h_1\big(|g_0(s)+\eps+\rmi s|\big)h_2\big(g_0(s)+\eps\big)\rmd s\nonumber\\
&\qquad\qquad\qquad\qquad\qquad+c_1\int_{|s|\geq\os_0}\frac{e^{(g_0(s)+\eps)t}}{|g_0(s)+\eps+\rmi s|^{2n}}\Bigg)\nonumber\\
&\leq\|A^{n-1}w\|\big(\|g_0'\|_\infty+1\big)e^{(\sup_{s\in\RR}g_0(s)+\eps)t}\Bigg(c_2
\Big(\sup_{\xi\leq\gamma}h_2(\xi)\Big)\int_{|s|\geq\os_0}\frac{h_1\big(\sqrt{s^2+(g_0(s)+\eps)^2}\big)}{\big(g_0(s)+\eps+\rmi s\big)^{\frac{2n-1}{2}}}\rmd s\nonumber\\
&\qquad\qquad\qquad\qquad\qquad+2c_1\int_{|s|\geq1}\frac{\rmd s}{|s|^{2n}}\Bigg)\nonumber\\
&\leq\|A^{n-1}w\|\big(\|g_0'\|_\infty+1\big)e^{(\sup_{s\in\RR}g_0(s)+\eps)t}\Bigg(8c_2
\Big(\sup_{\xi\leq\gamma}h_2(\xi)\Big)\int_{1}^{\infty}\frac{h_1(\xi)\rmd\xi}{\xi^{2n-1}}
+\frac{2c_1}{2n-1}\Bigg),
\end{align}	
proving the lemma.
\end{proof}
To formulate the next theorem we introduce the notation
\begin{equation}\label{def-Natural}
\NN_{k}=\{n\in\NN:n\geq k\}\;\mbox{for}\;k\in\NN.	
\end{equation}	
\begin{theorem}\label{t3.16}
Assume Hypothesis (H) and, in addition, that $b:=\inf\limits_{s\in\RR}g_0(s)>-\infty$, $g_0'\in L^\infty(\RR)$, $\displaystyle\int_1^\infty\frac{h_1(s)\rmd s}{s^{2n_0-1}}<\infty$. Then,
\begin{enumerate}
\item[(i)] There exists $M^\rmc:\NN_{n_0}\times(0,\eps_0)\to(0,\infty)$  and $\cE^\rmc_n:\RR\to\mathcal{B}(\dom(A^n),\bX)$ such that
\begin{equation}\label{Error-cos-out-bis}
C(t)x=\sum\limits_{k=1}^{m} \sum_{j=1}^{N_k}p_{1,k,j}(t)e^{\mu_k t}(A-\mu_k^2 I_\bX)^{j-1}P_kx+\chi_{[0,\infty)}(g_0(0))\displaystyle\sum\limits_{j=0}^{n-1}\frac{t^{2j}}{(2j)!}A^jx+\cE^\rmc_n(t)x,
\end{equation}
\begin{equation}\label{error-cos2}
\|\cE^\rmc_n(t)x\|\leq 	M^\rmc(n,\eps)e^{(\sup_{s\in\RR}g_0(s)+\eps)t}\big(\|A^{n-1}x\|+\|A^nx\|\big)
\end{equation}	
for any $t\geq 0$, $x\in\dom(A^n)$, $\eps\in(0,\eps_0)$, $n\in\NN$ with $n\geq n_0$;
\item[(ii)] There exists $M^\rms:\NN\times(0,\eps_0)\to(0,\infty)$  and $\cE^\rms_n:\RR\to\mathcal{B}(W_{n-1},\bX)$ such that
\begin{equation}\label{Error-sine-out-bis}
S(t)w=\sum\limits_{k=1}^{m} \sum_{j=1}^{N_k}p_{0,k,j}(t)e^{\mu_k t}(A-\mu_k^2 I_\bX)^{j-1}P_kw+\chi_{[0,\infty)}(g_0(0))\sum\limits_{j=0}^{n-1}\frac{t^{2j+1}}{(2j+1)!}A^jw+\cE^\rms_n(t)w
\end{equation}
\begin{equation}\label{error-sine2}
	\|\cE^\rms_n(t)w\|\leq 	M^\rms(n,\eps)e^{(\sup_{s\in\RR}g_0(s)+\eps)t}\|A^{n-1}w\|_\bW
\end{equation}	
for any $t\geq 0$, $w\in\bW_{n-1}$, $\eps\in(0,\eps_0)$, $n\in\NN$ with $n\geq n_0$.
\end{enumerate}	
We recall the definition of $n_0$ in Hypothesis (H)(iv), the projection $P_k$ and numbers $N_k\in\NN$, $k=1,\dots,m$, in Lemma~\ref{l3.3}, and of the polynomial $p_{\kappa,k,j}$, $\kappa=0,1$,  in \eqref{polynomial-representation}.	
\end{theorem}	
\begin{proof}
The theorem follows from Theorem~\ref{t3.11}, Theorem~\ref{t3.12}, Lemma~\ref{l3.13} and Lemma~\ref{l3.15}.
\end{proof}

\section{Extensions across the essential spectrum}\label{sec4}
In this section we aim to refine the representations given in Section~\ref{sec3}, see, in particular, Theorem~\ref{t3.16}, for the cosine and sine operator families. To achieve this goal we need to study the \textit{resonances} of the generator of the cosine family. In the case of Schr\"odinger operators considered on $L^2(\RR^m)$, with $m$ being an odd positive integer, it is possible, see, e.g., \cite{DZ} to extend the operator-valued function $R_A(\cdot)_{|L^2_{\mathrm{comp}}(\RR^m)}$ to a meromorphic function. The extension is defined on a set that contains the essential spectrum of $A$ and takes values in the set of continuous, linear operators $\cC\cL\big(L^2_{\mathrm{comp}}(\RR^m),L^2_{\mathrm{loc}}(\RR^m)\big)$. We refer to \cite{DZ} for a very detailed discussion in this important special case.

\subsection{ General Frechet Space Framework}\label{sec4.1-new}
To proceed in our general case when $A$ is the generator of a cosine family on a Banach space $\bX$,
we need to find a subspace of $\bX$ that plays the role of the space $L^2_{\mathrm{comp}}(\RR^n)$ and a Frechet space $\bY$ that contains the original Banach space to which we can extend the generator $A$ in a natural way, and $\bY$ will be playing the role of  $L^2_{\mathrm{loc}}(\RR^m)$.

\noindent\textbf{Hypothesis (Q).} We assume that $\bX$ is a Banach space, $A$ is the generator of a cosine family on $\bX$ and, $\bY$ is a Frechet space such that $\bX\hookrightarrow\bY$, $Q_n:\bY\to\bX$, $n\in\NN$, are linear operators such that
\begin{enumerate}
\item[($Q_1$)] ${Q_n}_{|\bX}\in\mathcal{B}(\bX)$ for any $n\in\NN$;
\item[($Q_2$)] $Q_nx\xrightarrow{\enskip\bX\enskip} x$ as $n\to\infty$ for any $x\in\bX$;	
\item[($Q_3$)] $Q_nQ_k=Q_kQ_n=Q_n$ for any $k,n\in\NN$ with $k\geq n+1$;
\item[($Q_4$)]	$Q_n\big(\dom(A)\big)\subseteq\dom(A)$ and $({Q_n}_{|\bX})^*\big(\dom(A^*)\big)\subseteq\dom(A^*)$;
\item[($Q_5$)] $A\big(\dom(A)\cap\Ker(Q_n-I_\bY)\big)\subseteq\Ker(Q_n-I_\bY)$ for any $n\in\NN$;
\item[($Q_6$)] $Q_nAQ_kx=Q_nAx$ for any $x\in\dom(A)$, $k,n\in\NN$ with $k\geq n+1$;
\item[($Q_7$)] For any $n\in\NN$ the linear operator $\cG_n:\Graph(A)\to\Graph(A)$ defined by $\cG_n(x,Ax)=(Q_nx,AQ_nx)$ is continuous when $\Graph(A)$ is endowed with the $\bY\times\bY$ topology;
\item[($Q_8$)] The topology on $\bY$ is equivalent to the topology induced by the family of seminorms $\{\mfkq_n\}_{n\in\NN}$ defined by $\mfkq_n:\bY\to[0,\infty)$, $\mfkq_n(y):=\|Q_ny\|_\bX$, $n\in\NN$.
\end{enumerate}
\begin{example}\label{e4.1}
If $\bX=L^p(\RR,\CC^m)$ and $A=\displaystyle\sum_{j=0}^{k}B_j(x)\partial_x^j$, where $B_j\in\cC_{\mathrm{b}}(\RR,\CC^{m\times m})$, $j=0,1,\dots,m$. In this case $\dom(A)=W^{p,k}(\RR,\CC^m)$ and $\bY=L_{\mathrm{loc}}^p(\RR,\CC^m)$, endowed with their canonical topologies. Let $\varphi_n\in\cC^\infty(\RR)$ be such that $0\leq\varphi_n\leq1$, $\sup\limits_{n\in\NN}\|\varphi_n'\|_\infty<\infty$, $\varphi_n(t)=1$ for any $t\in[-n,n]$, $\varphi_n(t)=0$, whenever $|t|\geq n+1$. We define the linear operator $Q_n: L^p_{\mathrm{loc}}(\RR,\CC^m)\to L^p(\RR,\CC^m)$ by $Q_nf=\varphi_nf$. One can readily check that
\begin{equation}\label{e4.1.1}
(\varphi_n f)^{(j)}=\sum_{\ell=0}^{j} \left(\begin{matrix}j\\
	\ell
\end{matrix}\right)	\varphi^{(j-\ell)}f^{(\ell)}\;\mbox{for any}\; f\in W^{p,k}(\RR,\CC^m), n\in\NN, j=1,\dots,k.
\end{equation}	
\begin{equation}\label{e4.1.2}
\Ker(Q_n-I)=\{f\in L^p(\RR,\CC^m):f_{|\RR\setminus[-n,n]}=0\, \mathrm{a.e}\}\;\mbox{for any}\; n\in\NN.	
\end{equation}
Using  Lebesgue Dominated Convergence Theorem, along with \eqref{e4.1.1} and \eqref{e4.1.2}, we can see that Hypothesis (Q) holds in this case. 		
\end{example}	
Our first task is to describe immediate consequences of Hypothesis (Q). We recall that a subset $\bL\subset\bY$ is \textit{bounded} provided that
\begin{equation}\label{bounded-Frechet}
\sup_{y\in\bL}\mfkq_k(y)<\infty\;\mbox{for any}\;k\in\NN.	
\end{equation}	
The space $\bY^*$ is a locally convex space when endowed by the topology induced by the family of seminorms $\{\mfkp_{\bL,*}: \bL\;\mbox{is a bounded subset of}\;\bY\}$, where
\begin{equation}\label{def-mfkp}
	\mfkp_{\bL,*}:\bY^*\to[0,\infty)\;\mbox{is defined by}\;\mfkp_{\bL,*}(y^*)=\sup_{y\in \bL}|\langle y,y^*\rangle|.
\end{equation}
In what follows an essential role is played by the set
\begin{equation}\label{def-bX-infty}
\bX_\infty=\bigcup\limits_{n=1}^{\infty}\Ker(Q_n-I_\bY).	
\end{equation}	
\begin{lemma}\label{l4.2}
Assume Hypothesis (Q). Then, the following assertions hold true:
\begin{enumerate}
\item[(i)] $\Ker(Q_n-I_\bY)=\Ker\big((Q_n)_{|\bX}-I_\bX\big)\subset\bX$ is a closed subspace in the $\|\cdot\|_\bX$ norm for any $n\in\NN$;
\item[(ii)] $\Ker(Q_n-I_\bY)\subseteq\Range(Q_n)\subseteq\Ker(Q_{n+1}-I_\bY)\subset\bX$ for any $n\in\NN$;
\item[(iii)] $Q_kAQ_nx=AQ_nx$ for any $x\in\dom(A)$, $k,n\in\NN$ with $k\geq n+1$;
\item[(iv)] $\bX_\infty$ is a dense subspace of $\bX$ in the $\|\cdot\|_\bX$ norm, invariant under $A$;
\item[(v)]  $Q_ny\xrightarrow{\enskip\bY\enskip} y$ as $n\to\infty$ for any $y\in\bY$;
\item[(vi)] $Q_k\in\cC\cL(\bY,\bX)$ for any $k\in\NN$;
\item[(vii)] $Q_k^*\in\cC\cL(\bX^*,\bY^*)$ for any $k\in\NN$;
\item[(viii)] $\dom(A)$ is a dense subspace of $\bY$ in the $\bY$-topology;
\item[(ix)] $\dom(A)\cap\bX_\infty$ is dense in $\bX$ in the $\|\cdot\|_\bX$ norm;
\item[(x)] If, in addition, we assume that $n\in\NN$ and $Q_j(\dom(A^n))\subseteq\dom(A^n)$ for any $j\in\NN$, then $\dom(A^{n-1})\cap\Ker(Q_j-I_\bY)$ is contained in the closure of $\dom(A^n)\cap\Ker(Q_{j+1}-I_\bY)$ in the $\|\cdot\|_{\dom(A^{n-1})}$ norm for any $j\in\NN$.
\end{enumerate}	
\end{lemma}	
\begin{proof}
Assertion (i) follows shortly from ($Q_1$) and since $\Range(Q_n)\subset\bX$ for any $n\in\NN$. Obviously, $\Ker(Q_n-I_\bY)\subseteq\Range(Q_n)$. Moreover, from ($Q_3$) one can readily check that $\Range(Q_n)\subseteq\Ker(Q_{n+1}-I_\bY)$ for any $n\in\NN$, proving assertion (ii).

Next, we fix $x\in\dom(A)$. From (ii) and ($Q_4$) it follows that $Q_nx\in\dom(A)\cap\Ker(Q_{n+1}-I_\bY)$ for any $n\in\NN$. From ($Q_5$) we infer that $AQ_nx\in\Ker(Q_{n+1}-I_\bY)\subseteq\Ker(Q_k-I_\bY)$ for $k,n\in\NN$ with $k\geq n+1$, which proves assertion (iii).
From (ii) one can readily check that $\bX_\infty$ is a subspace of $\bX$. Assertion (iv) follows from ($Q_2$), ($Q_5$) and since $Q_nx\in\Range(Q_n)\subseteq\Ker(Q_{n+1}-I_\bY)\subset\bX_\infty$ for any $n\in\NN$, $x\in\bX$ by (ii).

From ($Q_3$) we know that $Q_\ell Q_ny=Q_\ell y$ for any $y\in\bY$ and any $n,\ell\in\NN$ with $n\geq\ell$. It follows that $\mfkq_\ell(Q_ny-y)\to0$ as $n\to\infty$ for any $\ell\in\NN$, $y\in\bY$, proving assertion (v). Assertion (vi) follows immediately from ($Q_8$). Next, fix $\{x_n^*\}_{n\in\NN}$ a sequence of elements of $\bX^*$ such that $x_n^*\xrightarrow{\enskip\bX^*\enskip}0$ as $n\to\infty$. Then,
\begin{equation}\label{l4.2.1}
\mfkp_{\bL,*}(Q_k^*x_n^*)=\sup_{y\in\bL}|\langle y,Q_k^*x_n^*\rangle|=\sup_{y\in\bL}|\langle Q_ky,x_n^*\rangle|\leq\big(\sup_{y\in\bL}\|Q_ky\|\big)\|x_n^*\|=\big(\sup_{y\in\bL}\mfkq_k(y)\big)\|x_n^*\|		 
\end{equation}
for any $n,k\in\NN$ and any bounded subset $\bL$ of $\bY$. Assertion (vii) follows from \eqref{bounded-Frechet} and \eqref{l4.2.1}.

Next, we fix $y\in\bY$. Since $Q_ny\in\bX$ for any $n\in\NN$ and $\dom(A)$ is a dense subspace of $\bX$ in the $\|\cdot\|_\bX$ norm, we infer that for any $n\in\NN$ there exists $x_n\in\dom(A)$ such that $\|x_n-Q_ny\|\leq\frac{1}{n}$. It follows that $x_n-Q_ny\xrightarrow{\enskip\bX\enskip}0$ as $n\to\infty$. Since $\bX\hookrightarrow\bY$,	from (v) we conclude that $x_n\xrightarrow{\enskip\bY\enskip}y$ as $n\to\infty$, proving assertion (viii).

Finally, from ($Q_1$) and the Uniform Boundedness Principle we immediately infer that
\begin{equation}\label{l4.2.2}
q_*:=\sup_{n\in\NN}\|{Q_n}_{|\bX}\|<\infty.	
\end{equation}	
Fix $x\in\bX$. Since $\dom(A)$ is a dense subspace of $\bX$ in the $\|\cdot\|_\bX$ norm, there exists $\{x_n\}_{n\geq 1}$ a sequence of elements of $\dom(A)$ such that  $x_n\xrightarrow{\enskip\bX\enskip}x$ as $n\to\infty$. From ($Q_4$) we have $Q_nx_n\in\dom(A)\cap\Range(Q_n)\subseteq\dom(A)\cap\bX_\infty$ for any $n\in\NN$. Using \eqref{l4.2.2} we have
\begin{equation}\label{l4.2.3}
\|Q_nx_n-x\|\leq\|Q_nx_n-Q_nx\|+\|Q_nx-x\|\leq q_*\|x_n-x\|+\|Q_nx-x\|\;\mbox{for any}\;n\in\NN.		
\end{equation}	
Using ($Q_1$) and \eqref{l4.2.3} we conclude that $Q_nx_n\xrightarrow{\enskip\bX\enskip}x$ as $n\to\infty$, proving assertion (ix).

Finally, we fix $n\in\NN$ and assume that $Q_j(\dom(A^n))\subseteq\dom(A^n)$ for any $j\in\NN$. Let $\nK_{i,j}:\dom(A^i)\to\bX$ be the linear operator defined by
\begin{equation}\label{l4.2.4}
\nK_{i,j}x=A^iQ_jx-Q_jA^ix\;\mbox{for}\;j\in\NN, i=0,\dots,n.	
\end{equation}
Let $\{x_k\}_{k\geq 1}$ be a sequence of elements of $\dom(A^i)$, $x\in\dom(A^i)$ and $\tx\in\bX$ such that
\begin{equation}\label{l4.2.5}
x_k\xrightarrow{\enskip\dom(A^i)\enskip}x\;\mbox{and}\; \nK_{i,j}x_k\xrightarrow{\enskip\bX\enskip}\tx\;\mbox{as}\;k\to\infty.	
\end{equation}
From ($Q_1$) and \eqref{l4.2.5} we obtain
\begin{equation}\label{l4.2.6}
Q_jx_k\xrightarrow{\enskip\bX\enskip}Q_jx\;\mbox{and}\; Q_jA^ix_k\xrightarrow{\enskip\bX\enskip}Q_jA^ix\;\mbox{as}\;k\to\infty.	
\end{equation}
From \eqref{l4.2.5} and \eqref{l4.2.6} it follows that $A^iQ_jx_k\xrightarrow{\enskip\bX\enskip}Q_jA^ix+\tx$ as $k\to\infty$. We recall that $A^i$ is a closed linear operator since $A$ is closed with non-empty resolvent set. Hence, $\nK_{i,j}$ is a closed linear operator.  Thus, by the Closed Graph Theorem, we conclude that
\begin{equation}\label{l4.2.7}
\nK_{i,j}\in\mathcal{B}(\dom(A^i),\bX)\;\mbox{for any}\; j\in\NN, i=0,\dots,n.
\end{equation}
Next, we fix $j\in\NN$ and $x\in\dom(A^{n-1})\cap\Ker(Q_j-I_\bY)$. We define
\begin{equation}\label{l4.2.8}
\tx_k=k\int_0^{1/k}T(s)x\rmd s,\;\;k\in\NN,
\end{equation}
where $\{T(z)\}_{z\in\CC_0^+}$ is the analytic semigroup generated by $A$, given by \eqref{Weierstrss}. From \cite[Proposition 1.1.7]{ABHN} we have
\begin{equation}\label{l4.2.9}
\tx_k\in\dom(A^i)\;\mbox{and}\;A^i\tx_k=k\int_0^{1/k}T(s)A^ix\rmd s\;\mbox{for any}\;k\in\NN, i=0,\dots,n-1.	
\end{equation}	
It follows that $A^i\tx_k\xrightarrow{\enskip\bX\enskip}A^ix$ as $k\to\infty$, for any $i=0,\dots,n-1$, hence
\begin{equation}\label{l4.2.10}
\tx_k\xrightarrow{\enskip\dom(A^{n-1})\enskip}A^ix\;\mbox{as}\; k\to\infty.
\end{equation}
Moreover, from \eqref{l4.2.9} we infer that $\tx_k\in\dom(A^n)$ for any $k\in\NN$. Since $Q_j(\dom(A^n))\subseteq\dom(A^n)$, from (ii) we infer that $Q_j\tx_k\in\dom(A^n)\cap\Ker(Q_{j+1}-I_\bY)$ for any $k\in\NN$. From ($Q_1$), \eqref{l4.2.7} and \eqref{l4.2.10} we obtain that
\begin{equation}\label{l4.2.11}
A^iQ_j\tx_k=\nK_{i,j}\tx_k+Q_jA^i\tx_k\xrightarrow{\enskip\bX\enskip}\nK_{i,j}x+Q_jA^ix=A^iQ_jx=A^ix\;\mbox{as}\;k\to\infty.
\end{equation}
for any $i=0,\dots,n$. We conclude that $Q_j\tx_k\xrightarrow{\enskip\dom(A^{n-1})\enskip}x$ as $k\to\infty$, proving the lemma.
\end{proof}
\subsection{The closed extension of the generator $A$ to the Frechet Space $\bY$}\label{sec4.1}

Next, we aim to show that the linear operator $A:\dom(A)\subseteq\bX\to\bX$ can be extended to a closed linear operator $\tA$ acting in the Frechet space $\bY$ and that the extension preserves at least some of the properties of $A$ associated
with the family of operators $\{Q_n\}_{n\in\NN}$, in particular the localization property ($Q_6$). We denote by $A_\bY:\dom(A)\subset\bY\to\bY$ the linear operator defined by $A_\bY y=Ay$ for $y\in\dom(A)\subseteq\bX\hookrightarrow\bY$. We will show that the linear operator $A_\bY$ is \textit{closable} in the $\bY$ topology. First, we recall the following classical result of F. E. Browder.
\begin{lemma}\label{l4.3}(\cite[Lemma 2.1(iii)]{B})
If $\bY$ is a Frechet space, $B:\dom(B)\subseteq\bY\to\bY$ is a linear operator, $\dom(B)$ is dense in $\bY$, then $B$ is \textit{closable} if and only if $\dom(B^*)$ is dense in the weak $*$ topology (denoted $\sigma(\bY^*,\bY)$) on $\bY^*$.
\end{lemma}
\begin{lemma}\label{l4.4}
Assume Hypothesis (Q). Then the following assertions hold true:
\begin{enumerate}
\item[(i)] $Q_n^*\big(\dom(A^*)\big)\subseteq\dom(A_\bY^*)$ for any $n\in\NN$;
\item[(ii)] $A_\bY$ is closable.
\end{enumerate}
\end{lemma}
\begin{proof} (i) Fix $n\in\NN$, $y\in\dom(A)$ and $x^*\in\dom(A^*)$. From ($Q_4$) and ($Q_6$) we have
\begin{equation}\label{l4.4.1}
Q_{n+1}y\in\dom(A), (Q_n)^*_{|\bX}x^*\in\dom(A^*), 	Q_nAy=(Q_n)_{|\bX}AQ_{n+1}y.
\end{equation}		
Moreover, one can readily check that
\begin{align}\label{l4.4.2}
\langle A_\bY y,Q_n^*x^*\rangle&=\langle Ay,Q_n^*x^*\rangle=\langle Q_nAy,x^*\rangle=\langle
{Q_n}_{|\bX}AQ_{n+1}y,x^*\rangle\nonumber\\&=\langle AQ_{n+1}y,{Q_n}_{|\bX}^*x^*\rangle=\langle Q_{n+1}y,A^*{Q_n}_{|\bX}^*x^*\rangle.
\end{align}
Since $A^*{Q_n}_{|\bX}^*x^*\in\bX^*$ and $Q_{n+1}\in\cC\cL(\bY,\bX)$ by Lemma~\ref{l4.2}(vi) it follows that the functional
\begin{equation}\label{l4.4.3}
y\to\big\langle Q_{n+1}y,A^*{Q_n}_{|\bX}^*x^*\big\rangle:\bY\to\CC\;\mbox{is continuous}.
\end{equation}
Hence, $Q_nx^*\in\dom(A_\bY^*)$, proving (i).

We use Lemma ~\ref{l4.3} to show that $\dom(A_\bY^*)$ is a dense subspace of $\bY^*$ in the $\sigma(\bY^*,\bY)$ topology. Fix $y^*\in\bY^*$. Since $\bX\hookrightarrow\bY$ one can readily check that $y^*_{|\bX}\in\bX^*$. Next, we define the linear functionals $f_k^*:\bX\to\CC$, $\tf_k^*:\bY\to\CC$ by
\begin{equation}\label{l4.4.4}
f_k^*(x)=\langle (Q_k)_{|\bX}x,y^*\rangle,\;\tf_k^*(y)=\langle Q_ky,y^*\rangle.
\end{equation}
Since $(Q_k)_{|\bX}\in\mathcal{B}(\bX)$ by ($Q_1$) and $Q_k\in\cC\cL(\bY,\bX)$ by Lemma~\ref{l4.2}(ii) we infer that $f_k^*\in\bX^*$ and $\tf_k^*\in\bY^*$ for any $k\in\NN$.  From ($Q_3$) and \eqref{l4.4.4} we have
\begin{equation}\label{l4.4.5}
\langle y,Q_{k+1}^*f_k^*\rangle=\langle Q_{k+1}y,f_k^*\rangle=\langle Q_kQ_{k+1}y,y^*\rangle=\langle Q_ky,y^*\rangle=\langle y,\tf_k^*\rangle\;\mbox{for any}\;y\in\bY,k\in\NN,
\end{equation}
which implies that
\begin{equation}\label{l4.4.6}
Q_{k+1}^*f_k^*=\tf_k^*\;\mbox{for any}\;k\in\NN.
\end{equation}
Moreover, from Lemma~\ref{l4.2}(v) it follows that
\begin{equation}\label{l4.4.7}
\langle y,\tf_k^*\rangle=\langle Q_ky,y^*\rangle\longrightarrow\langle y,y^*\rangle\;\mbox{as}\;k\to\infty\;\mbox{for any}\;y\in\bY.
\end{equation}
Hence,
\begin{equation}\label{l4.4.8}
\tf_k^*\xrightarrow{\enskip\sigma(\bY^*,\bY)\enskip} y^*\;\mbox{as}\;k\to\infty.
\end{equation}
Next, we prove that the sequence $\{\tf_k^*\}_{k\in\NN}$ belongs to the closure of $\dom(A_\bY^*)$ in $\sigma(\bY^*,\bY)$-topology. Fix $k\in\NN$. Since $A$ is a closed, densely defined linear operator on $\bX$, from Lemma~\ref{l4.3} we infer that $\dom(A^*)$ is dense in $\bX^*$ in the weak $*$ topology ($\sigma(\bX^*,\bX)$). Thus, there exists a sequence $\{x_n^*\}_{n\in\NN}$ of vectors of $\dom(A^*)$ such that  $x_n^*\xrightarrow{\enskip\sigma(\bX^*,\bX)\enskip}f_k^*$ as $n\to\infty$. Using \eqref{l4.4.6} we obtain that
\begin{equation}\label{l4.4.9}
\langle y, Q_{k+1}^*x_n^*\rangle=\langle Q_{k+1}y,x_n^*\rangle\longrightarrow\langle Q_{k+1}y,f_k^*\rangle=\langle y,Q_{k+1}^*f_k^*\rangle=\langle y,\tf_k^*\rangle\;\mbox{as}\;k\to\infty\;\mbox{for any}\;y\in\bY.
\end{equation}
We conclude that $Q_{k+1}^*x_n^*\xrightarrow{\enskip\sigma(\bY^*,\bY)\enskip}\tf_k^*$ as $n\to\infty$. Since $x_n^*\in\dom(A^*)$ for any $n\in\NN$, from assertion (i) we infer that $Q_{k+1}^*x_n^*\in\dom(A_\bY^*)$ for any $n\in\NN$, which shows that
\begin{equation}\label{l4.4.10}
\tf_k^*\in\overline{\dom(A_\bY^*)}\;\mbox{in the}\;\sigma(\bY^*,\bY)\;\mbox{topology for any}\;k\in\NN.
\end{equation}
From \eqref{l4.4.8} and \eqref{l4.4.10} it follows that $y^*$ belongs to the closure of $\dom(A_\bY^*)$ in the $\sigma(\bY^*,\bY)$ topology, and thus $\dom(A_\bY^*)$ is a dense subspace of $\bY^*$ in the $\sigma(\bY^*,\bY)$ topology. From Lemma~\ref{l4.3} we infer that $A_\bY$ is closable in $\bY$, proving the lemma.
\end{proof}
We denote by $\tA:\dom(\tA)\subseteq\bY\to\bY$ the closure of the closable operator $A_\bY$. Then, $\Graph(\tA)$ is the closure of $\Graph(A)$ in the $\bY\times\bY$ topology. Next, we study the properties of $\tA$.
\begin{remark}\label{r4.5}
Under Hypothesis (Q), the space $\dom(\tA)$ is a Frechet when endowed with the family of seminorms $\{\tmfkp_n\}_{n\in\NN}$ defined by
\begin{equation}\label{tilde-pn-seminorm}
\tmfkp_n:\dom(\tA)\to[0,\infty), \;\;\tmfkp(y)=\|Q_ny\|_\bX+\|Q_n\tA y\|_\bX, n\in\NN.
\end{equation}
\end{remark}
\begin{lemma}\label{l4.6}
Assume Hypothesis (Q). Then, the following assertions hold true:
\begin{enumerate}
\item[(i)] For any $n\in\NN$ the linear operator $(Q_n)_{|\dom(A)}:\dom(A)\to\dom(A)$ is continuous when we use the topology induced by the family of seminorms $\{\tmfkp_n\}_{n\in\NN}$ on the domain and the $\|\cdot\|_{\dom(A)}$ norm on the image space of the operator;
\item[(ii)] $\dom(A)$ is a dense subspace of $\dom(\tA)$ in the topology induced by the family of seminorms $\{\tmfkp_n\}_{n\in\NN}$.
\end{enumerate}
\end{lemma}
\begin{proof} Let $\{x_k\}_{k\in\NN}$ be a sequence of vectors of $\dom(A)$ such that $\tmfkp_m(x_k)\longrightarrow0$ as $k\to\infty$ for any $m\in\NN$. From ($Q_8$) we have $x_k\xrightarrow{\enskip\bY\enskip}0$ and $Ax_k\xrightarrow{\enskip\bY\enskip}0$ as $k\to\infty$. From ($Q_7$) we obtain that
\begin{equation}\label{l4.5.1}
(Q_nx_k,AQ_nx_k)=\cG_n(x_k,Ax_k)\xrightarrow{\enskip\bY\times\bY\enskip}(0,0)\;\mbox{as}\;k\to\infty\;\mbox{for any}\;n\in\NN.	
\end{equation}
From Lemma~\ref{l4.2}(iii) and \eqref{l4.5.1} it follows that
\begin{equation}\label{l4.5.2}
AQ_nx_k=Q_{n+1}AQ_nx_k\xrightarrow{\enskip\bX\enskip}0\;\mbox{as}\;k\to\infty\;\mbox{for any}\;n\in\NN.	
\end{equation}
Moreover, since $x_k\xrightarrow{\enskip\bY\enskip}0$ as $k\to\infty$ from ($Q_8$) we have $Q_nx_k\xrightarrow{\enskip\bX\enskip}0$ as $k\to\infty$. Using \eqref{l4.5.2} we conclude that
\begin{equation}\label{l4.5.3}
\|Q_nx_k\|_{\dom(A)}=\|Q_nx_k\|+\|AQ_nx_k\|\longrightarrow0\;\mbox{as}\;k\to\infty\;\mbox{for any}\;n\in\NN,
\end{equation}
proving assertion (i). Assertion (ii) follows since $\Graph(\tA)$ is the closure in the $\bY\times\bY$ topology of $\Graph(A)$.  	
\end{proof}
To discuss further properties of the linear operator $\tA$ we need the following real analysis result.
\begin{remark}\label{r4.7}
Let $(\cM_1,d_{\cM_1})$ be a metric space, $(\cM_2,d_{\cM_2})$ be a complete metric space, $\cU_1\subseteq\cM_1$, $f:\cU_1\to\cM_2$ an uniformly continuous function. Then, there exists a unique uniformly continuous function $g:\overline{\cU_1}\to\cM_2$ such that $g_{|\cU_1}=f$. Here $\overline{\cU_1}$ denotes the closure of the set $\cU_1$ in the topology induced by the metric $d_{\cM_1}$ on $\cM_1$.
\end{remark}
\begin{lemma}\label{l4.8}
Assume Hypothesis (Q). Then, the following assertions hold true:
\begin{enumerate}
\item[(i)] $Q_n\big(\dom(\tA)\big)\subseteq\dom(A)$ for any $n\in\NN$;
\item[(ii)] $Q_nAQ_ky=Q_n\tA y$ for any $y\in\dom(\tA)$, $k,n\in\NN$ with $k\geq n+1$.
\end{enumerate} 		
\end{lemma}
\begin{proof} (i) Fix $n\in\NN$. First, we apply Remark~\ref{r4.7} to $\cM_1=\dom(\tA)$ endowed with the usual Frechet space metric,
\begin{equation}\label{l4.8.1}	
 d_\tA:\dom(\tA)\times\dom(\tA)\to[0,\infty),\;\;d_\tA(y_1,y_2)=\sum_{j=1}^{\infty}\frac{1}{2^j}\frac{\tmfkp_j(y_1-y_2)}{1+\tmfkp_j(y_1-y_2)},
\end{equation}
$\cM_2=\dom(A)$ endowed with the usual Banach space metric induced by the norm $\|\cdot\|_{\dom(A)}$. We set $\cU_1=\dom(A)$ and $f=(Q_n)_{|\dom(A)}$. From Lemma~\ref{l4.6}(i) we have $f\in\cC(\cU_1,\cM_2)$ and linear, hence, $f$ is uniformly continuous. Since the closure of $\cU_1$ in the topology induced by the metric $d_\tA$ is equal to $\dom(\tA)$, from Remark~\ref{r4.7} we infer that there exists a unique uniformly continuous function $\tQ_n:\dom(\tA)\to\dom(A)$ such that $(\tQ_n)_{|\dom(A)}=(Q_n)_{|\dom(A)}$. Since the function $(Q_n)_{|\dom(A)}$ is linear, from the uniqueness of the function $\tQ_n$, we obtain that $\tQ_n\in\cC\cL(\dom(\tA),\dom(A))$ with respect to the topology induced by the family of seminorms $\{\tmfkp_n\}_{n\in\NN}$ on $\dom(\tA)$ and the graph norm $\|\cdot\|_{\dom(A)}$ on $\dom(A)$.

Next, we note that since $\dom(\tA)\hookrightarrow\bY$ and $Q_n\in\cC\cL(\bY,\bX)$ by Lemma~\ref{l4.2}(vi), we infer that $(Q_n)_{|\dom(\tA)}\in\cC\cL(\dom(\tA),\bX)$. Since $\tQ_n\in\cC\cL(\dom(\tA),\dom(A))$ and $\dom(A)\hookrightarrow\bX$ we obtain
$\tQ_n\in\cC\cL(\dom(\tA),\bX)$. Moreover, we have
\begin{equation}\label{l4.8.2}
\tQ_nx=(Q_n)_{|\dom(\tA)}x=Q_nx\;\mbox{for any}\;x\in\dom(A).	
\end{equation}	
Since $\dom(A)$ is a dense subspace of $\dom(\tA)$ in the topology induced by the family of seminorms $\{\tmfkp_n\}_{n\in\NN}$ by Lemma~\ref{l4.6}(ii), from \eqref{l4.8.2} we conclude that $\tQ_n=(Q_n)_{|\dom(\tA)}$. It follows that
\begin{equation}\label{l4.8.3}
Q_n\big(\dom(\tA)\big)=\tQ_n\big(\dom(\tA)\big)\subseteq\dom(A).
\end{equation}

Next, we fix $y\in\dom(\tA)$. From (i), ($Q_3$) and ($Q_6$) we obtain
\begin{equation}\label{l4.8.4}
Q_nAQ_my=Q_nAQ_kQ_my=Q_nAQ_ky\;\mbox{for any}\;k,m,n\in\NN\;\mbox{such that}\;m>k>n.	
\end{equation}	
The sequence $\{Q_nAQ_my\}_{m\in\NN}$ is constant for $m\geq n+1$ for any $n\in\NN$, which implies that the sequence $\{Q_nAQ_my\}_{m\in\NN}$ is convergent in $\bX$ for any $n\in\NN$. From ($Q_8$) we infer that the sequence $\{AQ_my\}_{m\in\NN}$ is convergent in $\bY$. Let $z=\lim_{m\to\infty}AQ_my$ in $\bY$. Since $\tA_{|\dom(A)}=A$ from Lemma~\ref{l4.2}(v) we have $Q_my\in\dom(A)\subset\dom(\tA)$ for any $m\in\NN$ and
\begin{equation}\label{l4.8.5}
Q_my\xrightarrow{\enskip\bY\enskip}y\;\mbox{and}\; \tA Q_my=AQ_my\xrightarrow{\enskip\bY\enskip}z\;\mbox{as}\;m\to\infty.	
\end{equation}	
Since the linear operator $\tA$ is closed we conclude that $z=\tA y$, which shows that $AQ_my\xrightarrow{\enskip\bY\enskip}\tA y$ as $m\to\infty$. Assertion (ii) follows by passing to the limit as $m\to\infty$ in \eqref{l4.8.4}.
\end{proof}
\section{Meromorphic extensions of $R_A(\cdot)_{|\bX_\infty}$ and integral representations of the Cosine and Sine functions}\label{sec5} Obviously we cannot hope to have analytic (or meromorphic) extensions of the operator valued function $R_A(\cdot)$ to a subset of the complex plane that contains the set $\Omega_0$ defined in \eqref{def-Omega-0} and some part of the essential spectrum of the generator $A$. As showed in \cite{DZ} for the case of Schr\" odinger operators, the best that we can hope for is to find meromorphic extensions of the operator-valued function $R_A(\cdot)_{|\bX_\infty}$ with values in the space $\cC\cL(\bX_\infty,\bY)$, where $\bX_\infty$ is the subspace of $\bX$ defined in \eqref{def-bX-infty}. Therefore, we need to discuss  some of the elementary properties of the subspace.
\begin{remark}\label{r4.9}
Assume Hypotheses (H) and (Q). Then, the following assertions hold true:
\begin{enumerate}
\item[(i)] $B\in\cC\cL(\bX_\infty,\bY)$ if and only if $Q_nB\in\mathcal{B}(\bX_\infty,\bX)$ for any $n\in\NN$;	
\item[(ii)] The space $\cC\cL(\bX_\infty,\bY)$ is a Frechet space when endowed with the family of seminorms $\{\tmfkq_n\}_{n\in\NN}$ defined by
\begin{equation}\label{seminroms-operator}
\tmfkq_n:\cC\cL(\bX_\infty,\bY)\to[0,\infty),\;\tmfkq_n(B)=\|Q_nB\|_{\mathcal{B}(\bX)}, n\in\NN;	
\end{equation}
\item[(iii)] If $\Omega\subseteq\CC$ is an open connected set, $F:\Omega\to\cC\cL(\bX_\infty,\bY)$ is an analytic function, then $Q_nF(\cdot)$ is a $\mathcal{B}(\bX_\infty,\bX)$-valued analytic function and $Q_nF(\cdot){Q_j}_{|\bX}$ is a $\mathcal{B}(\bX)$-valued analytic function for any $n,j\in\NN$;
\item[(iv)] The operator-valued function
\begin{equation}\label{analytic-restriction}
\lambda\to R_A(\lambda)_{|\bX_\infty}:\Omega_0\setminus\cM\to\cC\cL(\bX_\infty,\bY)\;\mbox{is analytic}.
\end{equation}  	
\end{enumerate}	
\end{remark}	
We also recall the definition of a \textit{finitely meromorphic} point given in Definition~\ref{d3.1} for the case of Banach spaces. This definition can be extended naturally to the case of $\cC\cL(\bX_\infty,\bY)$-valued functions. Let $g_0$ be the function from Hypothesis (H) above. Throughout this section we assume the following.

\noindent\textbf{Hypothesis (H-ext).}  We assume that there exist a number $\delta_0>0$, a function $g_*:\RR\to\RR$, \textit{piecewise} of class $\mathcal{C}^1$, such that  $g_*(s)+\delta_0\leq g_0(s)$ for any $s\in\RR$ and such that
\begin{enumerate}
\item[(i)] The function $R_A(\cdot)_{|\bX_\infty}$ with values in $\cC\cL(\bX_\infty,\bY)$ has a meromorphic extension from $\Omega_0\setminus\cM$ to the set
\begin{equation}\label{def-Omega-star}
\Omega_*=\{\lambda\in\CC:\re\lambda>g_*(\im\lambda)\}.
\end{equation}
The extension will be denoted $R_A^\infty(\cdot)$;
\item[(ii)] The operator-valued function $R_A^\infty(\cdot)$ has finitely many singularities in $\Omega_*\setminus\Omega_0$, denoted $\nu_\ell$, $\ell=1,\dots,p$. All singularities are finitely meromorphic points of $R_A^\infty(\cdot)$;
\item[(iii)] For any $i,j\in\NN$ there exist two piecewise continuous functions $\tih_{i,j,1},\tih_{i,j,2}:[0,\infty)\to[0,\infty)$ such that
\begin{equation}\label{RA-ext-h1-h2-est}
\|Q_iR_A^\infty(\lambda)Q_j\|\leq \tih_{i,j,1}(|\lambda|)\tih_{i,j,2}(\re\lambda)\;\mbox{for any}\;\lambda\in\Omega_*\setminus\Omega_0\;\mbox{such that}\;|\im\lambda|>\max_{1\leq \ell\leq p}|\im\nu_\ell|.	
\end{equation}	
\item[(iv)]  There exists $n_*\in\NN$ such that
\begin{equation}\label{limit-n-star}
\lim\limits_{s\to\infty}\frac{\tih_{i,j,1}(s)}{s^{2n_*-1}}=0.
\end{equation}
Moreover, $\sup\limits_{s\leq\omega}e^{\beta_{i,j}s}\tih_{i,j,2}(s)<\infty$ for some $\beta_{i,j}>0$, for any $i,j\in\NN$.
\end{enumerate}
In the sequel we denote the set of singularities by
\begin{equation}\label{def-calN}
\cN=\{\nu_\ell:\ell=1,\dots,p\}.	
\end{equation}	

\begin{figure}[h]
	\begin{center}
		\includegraphics[width=0.6\textwidth]{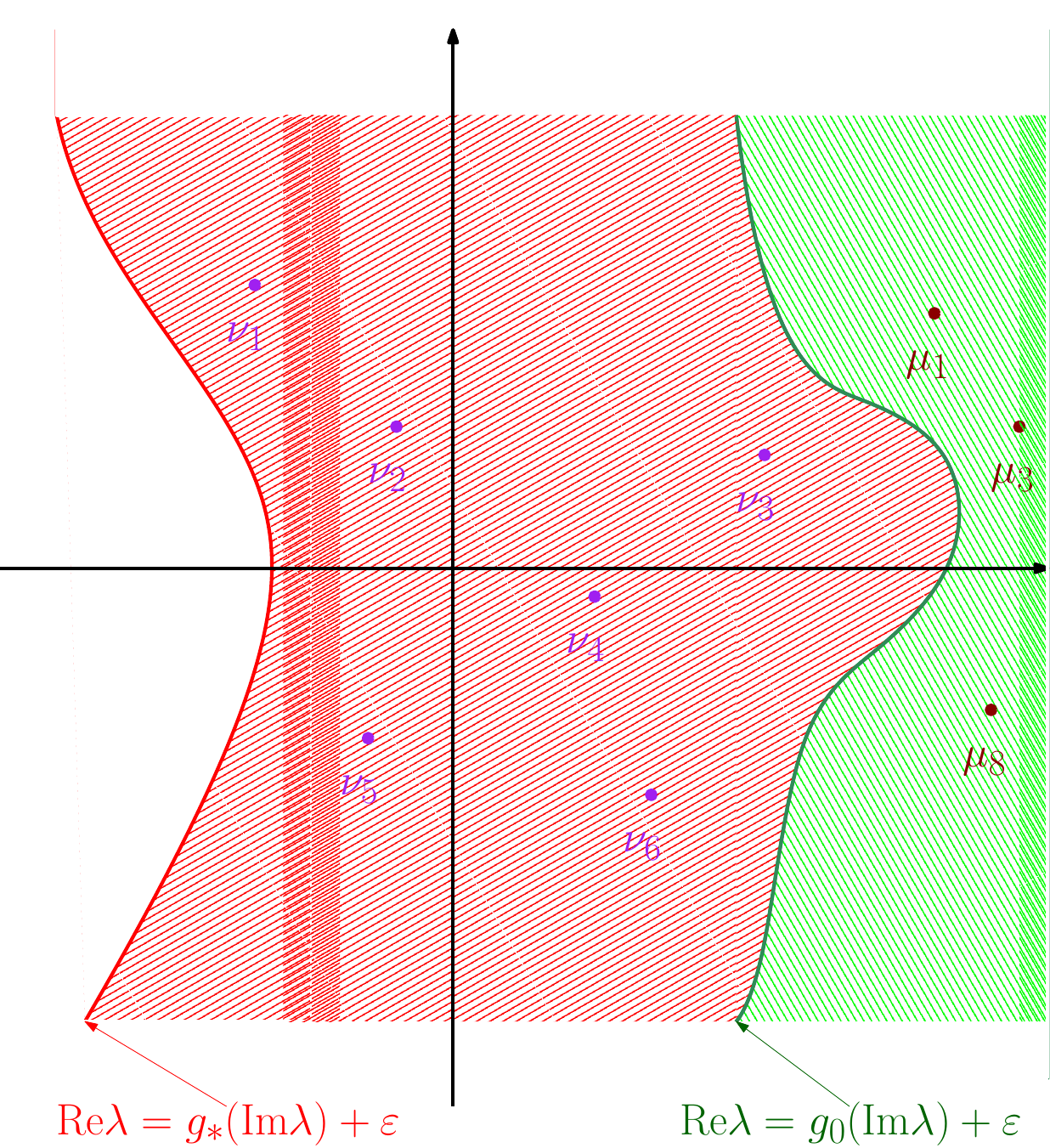}
		
		Figure 3. A generic depiction of the sets introduced in Hypothesis (H-ext). The set $\Omega_*$ consists of the union of the regions to the right of the graph of the curve $\re\lambda=g_*(\im\lambda)$.
	\end{center}
\end{figure}
Our next task is to describe the immediate consequences of Hypothesis (H-ext), similar to Lemma~\ref{l3.2} of Section~\ref{sec3}.
\begin{lemma}\label{l4.10}
Assume Hypotheses (H), (Q), (H-ext). Then, the following assertions hold true:
\begin{enumerate}
\item[(i)] $R_A^\infty(\lambda)Ax=\lambda^2R_A^\infty(\lambda)x-x$ for any $\lambda\in\Omega_*\setminus\big(\cM\cup\cN\big)$, $x\in\dom(A)\cap\bX_\infty$;
\item[(ii)] $\Range\big(R_A^\infty(\lambda)\big)\subseteq\dom(\tA)$ and $\tA R_A^\infty(\lambda)x=\lambda^2R_A^\infty(\lambda)x-x$  for any  $\lambda\in\Omega_*\setminus\big(\cM\cup\cN\big)$, $x\in\bX_\infty$;
\item[(iii)] $\Ker(Q_i-I_\bY)$ and $\bX_\infty$ are invariant subspaces of $A^n$ for any $n,i\in\NN$;
\item[(iv)] For any $n\in\NN$, $x\in\dom(A^n)\cap\bX_\infty$ and $\lambda\in\Omega_*\setminus\big(\cM\cup\cN\big)$ we have
\begin{equation}\label{RA-infty-reduction}
R_A^\infty(\lambda)A^nx=\lambda^{2n}R_A^\infty(\lambda)x-\displaystyle\sum_{j=0}^{n-1}\lambda^{2n-2-2j}A^jx.
\end{equation}		
\end{enumerate}
\end{lemma}	
\begin{proof}  Fix $x\in\dom(A)\cap\bX_\infty$. From Lemma~\ref{l3.2}(i)-(ii) we infer that
\begin{equation}\label{l4.10.1}
R_A(\lambda)Ax=\lambda^2R_A(\lambda)x-x\;\mbox{for any}\;\lambda\in\Omega_0\setminus\cM.
\end{equation}	
Assertion (i) follows from \eqref{l4.10.1} because $R_A^\infty(\cdot)$ is the analytic extension of $R_A(\cdot)_{|\bX_\infty}$ to $\Omega_*\setminus\big(\cM\cup\cN\big)$ by Hypothesis (H-ext).

Next, we fix $x\in\bX_\infty$ and define the function $\tF_x:\Omega_*\setminus\big(\cM\cup\cN\big)\to\bY\times\bY$
by $\tF_x(\lambda)=(R_A^\infty(\lambda)x,\lambda^2R_A^\infty(\lambda)x-x)$. By Hypothesis (H-ext) $\tF_x$ is analytic.  From Lemma~\ref{l3.2}(i)-(ii) we obtain
\begin{equation}\label{l4.10.2}
\tF_x(\lambda)=\bigl(R_A(\lambda)x,\lambda^2R_A(\lambda)x-x\bigr)\in\mathrm{Graph}(A)\subseteq\Graph(\tA)
\end{equation}	
for any $\lambda\in\Omega_0\setminus\cM$. Since the linear operator $\tA$ is closed, $\Graph(\tA)$ is a closed subspace of the Frechet space $\bY\times\bY$. From Lemma~\ref{lB3} and \eqref{l4.10.2} we derive that $\tF_x(\lambda)\in\Graph(\tA)$ for any $\lambda\in\Omega_*\setminus\big(\cM\cup\cN\big)$, proving assertion
(ii).

Using induction and Hypothesis ($Q_5$) we obtain
\begin{equation}\label{invariance-A-n-Qk}
A^n\big(\dom(A^n)\cap\Ker(Q_i-I_\bY)\big)\subseteq\Ker(Q_i-I_\bY)\;\mbox{for any}\; n,i\in\NN,
\end{equation}
which implies the subspace $\bX_\infty$ is invariant under $A^n$, that is
\begin{equation}\label{invariance-A-n}
	A^n\big(\dom(A^n)\cap\bX_\infty)\big)\subseteq\bX_\infty\;\mbox{for any}\; n\in\NN.
\end{equation}
Assertion (iv) follows from (i) by induction, similar to the proof of Lemma~\ref{l3.2}(iii).
\end{proof}
\begin{lemma}\label{l4.11}
Assume Hypotheses (H), (Q), (H-ext). Then, the following assertions hold true:
\begin{enumerate}
\item[(i)] $\nu_\ell^2\in\sigma_{\mathrm{point}}(\tA)$ for any $\ell=1,\dots,p$;
\item[(ii)] There exists $\ttr_*>0$, small enough, such that
\begin{equation}\label{tilde-Pk}
\tP_\ell=\frac{1}{\pi\rmi}\int_{\partial D(\nu_\ell,\ttr_*)}\lambda R_A^\infty(\lambda)\rmd\lambda\in\cC\cL(\bX_\infty,\bY),\;\oP_\ell=\frac{1}{\pi\rmi}\int_{\partial D(\nu_\ell,\ttr_*)} R_A^\infty(\lambda)\rmd\lambda\in\cC\cL(\bX_\infty,\bY)
\end{equation}
are finite rank operators for any $\ell=1,\dots,p$;	
\item[(iii)] $\tP_\ell\big(\dom(A)\cap\bX_\infty\big)=\Range(\tP_\ell)\subset\dom(\tA)$, $\oP_\ell\big(\dom(A)\cap\bX_\infty\big)=\Range(\oP_\ell)\subset\dom(\tA)$ and
$\tA\tP_\ell x=\tP_\ell Ax$, $\tA\oP_\ell x=\oP_\ell Ax$ for any $x\in\dom(A)\cap\bX_\infty$, $\ell=1,\dots,p$.	
Moreover, $\Range(\tP_\ell)$  and $\Range(\oP_\ell)$ are invariant under $\tA$ for any $\ell=1,\dots,p$;		
\item[(iv)] For any $\ell\in\{1,\dots,p\}$ there exists $\tN_\ell\in\NN$ such that $\Range(\tP_\ell)\subset\dom(\tA^{\tN_\ell})$ and $(\tA-\nu_\ell^2I_\bY)_{|\mathrm{Range}(\tP_\ell)}$  is nilpotent;
\item[(v)] If $\nu_\ell=0$ then there exists $\tN_\ell\in\NN$ such that $\Range(\oP_\ell)\subset\dom(\tA^{\tN_\ell})$ and $\tA_{|\mathrm{Range}(\oP_\ell)}$  is nilpotent;
\item[(vi)] If $\nu_\ell\ne0$ then there exists an analytic function $\tE_\ell:D(\nu_\ell,2\ttr_*)\to\cC\cL(\bX_\infty,\bY)$ such that
\begin{equation}\label{Laurent-nuk-nonzero}
R_A^\infty(\lambda)=\sum_{n=1}^{\tN_\ell}\frac{1}{(\lambda^2-\nu_\ell^2)^n}(\tA-\nu_\ell^2 I_\bY)^{n-1}\tP_\ell+\tE_\ell(\lambda)\;\mbox{for any}\;\lambda\in D(\nu_\ell,2\ttr_*)\setminus\{\nu_\ell\};
\end{equation}
\item[(vii)] If $\nu_\ell=0$ then there exists $\oN_\ell\in\NN$ with $\oN_\ell\leq\tN_\ell$ and an analytic function $\tE_\ell:D(0,2\ttr_*)\to\cC\cL(\bX_\infty,\bY)$ such that
\begin{equation}\label{Laurent-nuk-zero}
R_A^\infty(\lambda)=\sum_{j=1}^{\oN_\ell}\frac{1}{\lambda^{2j}}\tA^{j-1}\tP_\ell+\sum_{j=1}^{\oN_\ell}\frac{1}{\lambda^{2j+1}}\tA^{j-1}\oP_\ell+\tE_\ell(\lambda)
\;\mbox{for any}\;\lambda\in D(0,2\ttr_*)\setminus\{0\};
\end{equation}
\item[(viii)] If $\nu_\ell\ne0$ then
\begin{equation}\label{res-cos-sin-ext-nonzero}
\Res\Bigl(\lambda^\kappa e^{\lambda t}R_A^\infty(\lambda),\lambda=\nu_\ell\Bigr)=\sum_{j=1}^{\tN_\ell}\Bigg(\frac{1}{2\pi\rmi}\int_{\partial D(\nu_\ell,\ttr_*)}\frac{\lambda^\kappa e^{\lambda t}}{(\lambda^2-\nu_k^2)^j}\rmd\lambda\Bigg)(\tA-\nu_\ell^2 I_\bY)^{j-1}\tP_\ell,\;\kappa=0,1;
\end{equation}
\item[(ix)] If $\nu_\ell=0$ then
\begin{equation}\label{res-cos-sin-ext-zero}
\Res\Bigl(\lambda^\kappa e^{\lambda t}R_A^\infty(\lambda),\lambda=0\Bigr)=\sum_{j=1}^{\oN_\ell}\frac{t^{2j-1-\kappa}}{(2j-1-\kappa)!}\tA^{j-1}\tP_\ell+\sum_{j=1}^{\oN_\ell}\frac{t^{2j-\kappa}}{(2j-\kappa)!}\tA^{j-1}\oP_\ell,\;\kappa=0,1.
\end{equation}		
\end{enumerate}
\end{lemma}	
\begin{proof}
Let $\ttr_*>0$ be a positive number independent of $\ell\in\{1,\dots,p\}$ such that the elements of the set $\cN_\pm:=\{\nu_\ell: \ell=1,\dots,p\}\cup\{-\nu_\ell: \ell=1,\dots,p\}$ are separated by disks of radius $2\ttr_*$, that is
\begin{equation}\label{l4.11.1}	
D(\xi,2\ttr_*)\cap\cN_\pm=\{\xi\}\;\mbox{for any}\;\xi\in\cN_\pm.	
\end{equation}	
Fix $\ell\in\{1,\dots,p\}$. Since $\nu_l$ is a finitely meromorphic point by Hypothesis (H-ext) (ii), by Definition~\ref{d3.1} there exists $\tN_\ell\in\NN$, an analytic operator-valued function $\oE_\ell:D(\nu_\ell,2\ttr_*)\to\cC\cL(\bX_\infty,\bY)$ and  finite rank operators $\tR_{\ell,j}\in\cC\cL(\bX_\infty,\bY)$, $j=1,\dots,\tN_\ell$, such that
\begin{equation}\label{l4.11.2}
R_A^\infty(\lambda)=\sum_{j=1}^{\tN_\ell}\frac{1}{(\lambda-\nu_\ell)^j}\tR_{\ell,j}+\oE_\ell(\lambda)\;\mbox{for any}\;\lambda\in D(\nu_\ell,2\ttr_*)\setminus\{\nu_\ell\}.
\end{equation}
Integrating, one can readily check that
\begin{equation}\label{l4.11.3}
\tR_{\ell,j}=\frac{1}{2\pi\rmi}\int_{\partial D(\nu_\ell,\ttr_*)}(\lambda-\nu_l)^{j-1}R_A^\infty(\lambda)\rmd\lambda\;\mbox{for any}\;j=1,\dots,\tN_\ell.
\end{equation}
Next, we define
\begin{equation}\label{l4.11.4}
\tG_{\ell,n}=\frac{1}{\pi\rmi}\int_{\partial D(\nu_\ell,\ttr_*)}\lambda(\lambda^2-\nu_\ell^2)^{n-1}R_A^\infty(\lambda)\rmd\lambda\in\cC\cL(\bX_\infty,\bY)\;\mbox{for}\;n\in\NN.
\end{equation}
From \eqref{l4.11.2} and \eqref{l4.11.4} we  obtain that
\begin{equation}\label{l4.11.5}
\tG_{\ell,n}=\frac{1}{\pi\rmi}\sum_{j=1}^{\tN_\ell}\Bigg(\int_{\partial D(\nu_\ell,\ttr_*)}\frac{\lambda\bigl(\lambda^2-\nu_\ell^2\bigr)^{n-1}}{(\lambda-\nu_\ell)^j}\rmd\lambda\Bigg)\tR_{\ell,j}\;\mbox{for any}\;n\in\NN.
\end{equation}
Since $\tR_{\ell,j}$ is a finite rank operator for any $j=1,\dots,\tN_\ell$, from \eqref{l4.11.5} we conclude that
\begin{equation}\label{l4.11.6}
\tG_{\ell,n}=0\;\mbox{for any}\; n\geq 1+\tN_\ell,\;\;\tG_{\ell,n}\;\mbox{is of finite rank for any}\; n=1,\dots,\tN_\ell.	
\end{equation}
From Lemma~\ref{l4.10}(ii) and \cite[Proposition 1.1.7]{ABHN} it follows that $\Range(\tG_{\ell,n})\subseteq\dom(\tA)$ and
\begin{align}\label{l4.11.7}
(\tA-\nu_\ell^2I_\bY)\tG_{\ell,n}x&=\frac{1}{\pi\rmi}\int_{\partial D(\nu_\ell,\ttr_*)}\lambda\bigl(\lambda^2-\nu_\ell^2\bigr)^{n-1}(\tA-\nu_\ell^2I_\bY)R_A^\infty(\lambda)x\rmd\lambda\nonumber\\
&=\frac{1}{\pi\rmi}\int_{\partial D(\nu_\ell,\ttr_*)}\lambda\bigl(\lambda^2-\nu_\ell^2\bigr)^{n-1}\Bigl((\lambda^2-\nu_\ell^2)R_A^\infty(\lambda)x-x\Bigr)\rmd\lambda\nonumber\\
&=\frac{1}{\pi\rmi}\int_{\partial D(\nu_\ell,\ttr_*)}\lambda\bigl(\lambda^2-\nu_\ell^2\bigr)^nR_A^\infty(\lambda)x\rmd\lambda=\tG_{\ell,n+1}x
\end{align}	
for any $x\in\bX_\infty$, $n\in\NN$. From \eqref{l4.11.6} we conclude that $\tP_\ell=\tG_{\ell,1}$ and $\oP_\ell=\tR_{\ell,1}$ are finite rank operators for any $\ell=1,\dots,p$, proving assertion (ii). Moreover, from \eqref{l4.11.7} we have
\begin{equation}\label{l4.11.8}
\Range(\tP_\ell)\subset\dom(\tA^{\tN_\ell})\;\mbox{and}\;\tG_{\ell,n}=(\tA-\nu_\ell^2I_\bY)^{n-1}\tP_\ell\;\mbox{for any}\;\ell=1,\dots,p,\, n=1,\dots,\tN_\ell+1.
\end{equation}
From Lemma~\ref{l4.10} and \cite[Proposition 1.1.7]{ABHN} we obtain $\Range(\oP_\ell)\subseteq\dom(\tA)$ and
\begin{align}\label{l4.11.9}
\tA\tP_\ell x=\frac{1}{\pi\rmi}\int_{\partial D(\nu_\ell,\ttr_*)}\lambda\tA R_A^\infty(\lambda)x\rmd\lambda=\frac{1}{\pi\rmi}\int_{\partial D(\nu_\ell,\ttr_*)}\lambda R_A^\infty(\lambda)Ax\rmd\lambda=\tP_\ell Ax,\nonumber\\	
\tA\oP_\ell x=\frac{1}{2\pi\rmi}\int_{\partial D(\nu_\ell,\ttr_*)}\tA R_A^\infty(\lambda)x\rmd\lambda=\frac{1}{2\pi\rmi}\int_{\partial D(\nu_\ell,\ttr_*)}R_A^\infty(\lambda)Ax\rmd\lambda=\oP_\ell Ax	 
\end{align}	
for any $x\in\dom(A)\cap\bX_\infty$. From Lemma~\ref{l4.2}(ix) and since $\tP_\ell,\oP_\ell\in\cC\cL(\bX_\infty,\bY)$ we infer that $\tP_\ell(\dom(A)\cap\bX_\infty)$ and $\oP_\ell(\dom(A)\cap\bX_\infty)$ are dense subspaces of $\Range(\tP_\ell)$ and $\Range(\oP_\ell)$, respectively. Since $\Range(\tP_\ell)$ and $\Range(\oP_\ell)$ are finite dimensional vector spaces, we infer that $\tP_\ell\big(\dom(A)\cap\bX_\infty\big)=\Range(\tP_\ell)$, $\oP_\ell\big(\dom(A)\cap\bX_\infty\big)=\Range(\oP_\ell)$. This fact together with \eqref{l4.11.9} proves assertion (iii).
From \eqref{l4.11.6} and \eqref{l4.11.8} we obtain  $(\tA-\nu_\ell^2I_\bY)^{\tN_\ell}\tP_\ell=0$. In addition, since $\nu_\ell$ is a singularity of $R_A^\infty(\cdot)$ by Hypothesis (H-ext) it follows that $\tP_\ell\ne0$ for any $\ell=1,\dots,p$.
From (iii) we conclude that $(\tA-\nu_\ell^2I_\bY)_{|\Range(\tP_\ell)}$ is a nilpotent linear operator on the non-trivial, finite dimensional space $\Range(\tP_\ell)$, proving assertion (iv). Assertion (i) follows immediately from assertion (iv) and since $\tP_\ell\ne0$.

To prove the remaining assertions of the lemma we distinguish between the case when the resonance $\nu_\ell$ is at a non-zero value or at zero.

\noindent\textbf{Case 1.} $\nu_\ell\ne0$.

\noindent We fix $\lambda\in D(\nu_\ell,2\ttr_*)\setminus\{\nu_\ell\}$ and $r\in\Big(0,\min\big\{\ttr_*,\frac{|\lambda^2-\nu_\ell^2|}{4(\onu+\ttr_*)}\big\}\Big)$, where $\onu=\max\limits_{1\leq\ell\leq p}|\nu_\ell|$. From \eqref{l4.11.5} and \eqref{l4.11.6} we have
\begin{align}\label{l4.11.10}
\sum_{n=1}^{\tN_\ell}\frac{1}{(\lambda^2-\nu_\ell^2)^n}\tG_{\ell,n}&=\sum_{n=1}^{\infty}\frac{1}{(\lambda^2-\nu_\ell^2)^n}\tG_{\ell,n}=\frac{1}{\pi\rmi}\sum_{n=1}^{\infty}\sum_{j=1}^{\tN_\ell}\Bigg(\int_{\partial D(\nu_\ell,r)}\frac{\xi\bigl(\xi^2-\nu_\ell^2\bigr)^{n-1}}{(\lambda^2-\nu_\ell^2)^n(\xi-\nu_\ell)^j}\rmd\xi\Bigg)\tR_{\ell,j}\nonumber\\
&=\frac{1}{\pi\rmi}\sum_{j=1}^{\tN_\ell}\sum_{n=1}^{\infty}\Bigg(\int_{\partial D(\nu_\ell,r)}\frac{\xi\bigl(\xi^2-\nu_\ell^2\bigr)^{n-1}}{(\lambda^2-\nu_\ell^2)^n(\xi-\nu_\ell)^j}\rmd\xi\Bigg)\tR_{\ell,j}.
\end{align}
An elementary computation shows that
\begin{equation}\label{l4.11.11}
\int_{\partial D(\nu,r)}\frac{d\xi}{(\xi-\nu)^k(\xi-\mu)}=2\pi\rmi\frac{(-1)^{k-1}}{(\nu-\mu)}\;\mbox{whenever}\;k\in\NN, |\nu-\mu|>r.	
\end{equation}
Moreover, we note that
\begin{equation}\label{l4.11.12}
\Bigg|\frac{\xi^2-\nu_\ell^2}{\lambda^2-\nu_\ell^2}\Bigg|=\frac{(|\xi-\nu_\ell|)(|\xi|+|\nu_\ell|)}{|\lambda^2-\nu_\ell^2|}\leq\frac{r(2\onu+\ttr_*)}{|\lambda^2-\nu_\ell^2|}\leq\frac{1}{2}.	
\end{equation}	
From \eqref{l4.11.11} and \eqref{l4.11.12} we obtain 	 	
\begin{align}\label{l4.11.13}
\sum_{n=1}^{\infty}&\Bigg(\int_{\partial D(\nu_\ell,r)}\frac{\xi\bigl(\xi^2-\nu_\ell^2\bigr)^{n-1}\rmd\xi}{(\lambda^2-\nu_\ell^2)^n(\xi-\nu_\ell)^j}\Bigg)=\int_{\partial D(\nu_\ell,r)}\frac{\xi}{(\xi-\nu_\ell)^j(\lambda^2-\xi^2)}\rmd\xi\nonumber\\&=-\frac{1}{2}\int_{\partial D(\nu_\ell,r)}\frac{\rmd\xi}{(\xi-\nu_\ell)^j(\xi-\lambda)}-\frac{1}{2}\int_{\partial D(\nu_\ell,r)}\frac{\rmd\xi}{(\xi-\nu_\ell)^j(\xi+\lambda)}\nonumber\\
&=\frac{\pi\rmi}{(\lambda-\nu_\ell)^j}+\frac{(-1)^j\pi\rmi}{(\lambda+\nu_\ell)^j}\;\mbox{for any}\;j=1,\dots,\tN_\ell.
\end{align}
From \eqref{l4.11.10} and \eqref{l4.11.13} we conclude that
\begin{equation}\label{l4.11.14}
\sum_{j=1}^{\tN_\ell}\frac{1}{(\lambda-\nu_\ell)^j}\tR_{\ell,j}=\sum_{n=1}^{\tN_\ell}\frac{1}{(\lambda^2-\nu_\ell^2)^n}\tG_{\ell,n}+\sum_{j=1}^{\tN_\ell}\frac{(-1)^{j-1}}{(\lambda+\nu_\ell)^j}\tR_{\ell,j}\;\mbox{for any}\;\lambda\in D(\nu_\ell,2\ttr_*)\setminus\{\nu_\ell\}.
\end{equation}
From \eqref{l4.11.2}, \eqref{l4.11.8} and \eqref{l4.11.14} we infer that the Laurent expansion \eqref{Laurent-nuk-nonzero} holds, with $\tE_\ell:D(\nu_\ell,2\ttr_*)\to\cC\cL(\bX_\infty,\bY)$ defined by
\begin{equation}\label{l4.11.15}
\tE_\ell(\lambda)=\oE_\ell(\lambda)+\sum_{j=1}^{\tN_\ell}\frac{(-1)^{j-1}}{(\lambda+\nu_\ell)^j}\tR_{\ell,j}.
\end{equation}
Since $\nu_\ell\ne0$ by \eqref{l4.11.1} we have $-\nu_\ell\notin D(\nu_\ell,2\ttr_*)$, which implies that $\tE_\ell$ is analytic, proving assertion (vi). Assertion (viii) follows immediately from assertion (vi).

\noindent\textbf{Case 2.} $\nu_\ell=0$.

\noindent First, we set $\tR_{\ell,j}=0$ for each $j\geq \tN_\ell+1$. Using again Lemma~\ref{l4.10}(ii) and \cite[Proposition 1.1.7]{ABHN} we obtain that $\Range(\tR_{\ell,j})\subseteq\dom(\tA)$ and
\begin{align}\label{l4.11.16}
\tA\tR_{\ell,j}x&=\frac{1}{2\pi\rmi}\int_{\partial D(0,\ttr_*)}\lambda^{j-1}\tA R_A^\infty(\lambda)x\rmd\lambda=\frac{1}{2\pi\rmi}\int_{\partial D(0,\ttr_*)}\lambda^{j-1}\Bigl(\lambda^2R_A^\infty(\lambda)x-x\Bigr)\rmd\lambda\nonumber\\
&=\frac{1}{2\pi\rmi}\int_{\partial D(0,\ttr_*)}\lambda^{j+1}R_A^\infty(\lambda)x\rmd\lambda=\tR_{\ell,j+2}x\;\mbox{for any}\; x\in\bX_\infty,\,j\in\NN.
\end{align}	
Since $\oP_\ell=\tR_{\ell,1}$ it follows that $\Range(\oP_\ell)\subset\dom(\tA^{\tN_\ell})$ and
\begin{equation}\label{l4.11.17}
\tR_{\ell,2j+1}=\tA^{j-1}\oP_\ell\;\mbox{and}\;\tR_{\ell,2j}=\tA^{j-1}\tP_\ell\;\mbox{for any}\;j\in\NN.
\end{equation}
It follows that $\tA^{\tN_\ell}\oP_\ell=0$, which implies that $\tA_{|\Range(\oP_\ell)}$ is nilpotent proving assertion (v). Next, we define $\oN_\ell=[\tN_\ell/2]+1$, where $[\cdot]$ denotes the integer part. Assertion (vii) follows from \eqref{l4.11.2} and \eqref{l4.11.17}. Assertion (ix) follows immediately from assertion (vii).
\end{proof}
\begin{remark}\label{r4.12}
We note that the residue formula \eqref{res-cos-sin} with $\mu_k=0$ has either even or odd terms in $t$ for the case when $\kappa=0$ or $\kappa=1$, respectively. Meanwhile, \eqref{res-cos-sin-ext-zero} always has even and odd terms in both cases. This difference is due to the fact that if $0$ is a discrete eigenvalue of finite algebraic multiplicity, then in a neighborhood of $0$ we have $R_A(\lambda)=R(\lambda^2,A)$. In the case when $0$ is a resonance, it is possible that the Laurent series representation of $R_A^\infty(\lambda)$ has odd terms in $\lambda$, see, e.g., \cite[Theorem 2.2.7]{DZ}.
\end{remark}

\begin{remark}\label{r4.13}  As in Remark~\ref{r3.4}, an elementary computation shows that for any $\kappa=0,1$, $\ell=1,\dots,p$, $k\in\NN$ there exists a polynomial $\tp_{\kappa,\ell,k}$ of degree at most $k-1$ such that
\begin{equation}\label{polynomial-representation-resonances}
\frac{1}{2\pi\rmi}\int_{\partial D(\nu_k,\ttr_*)}\frac{\lambda^\kappa e^{\lambda t}}{(\lambda^2-\nu_\ell^2)^k}\rmd\lambda=\tp_{\kappa,\ell,k}(t)e^{\nu_\ell t}\;\mbox{for any}\; t\in\RR.	
\end{equation}
\end{remark}	
Next, we discuss the properties of the residues of the operator-valued functions that involve $R_A^\infty$.
\begin{lemma}\label{l4.14}
Assume Hypotheses (H), (Q), (H-ext) $n\in\NN$ and $x\in\dom(A^n)\cap\bX_\infty$. Then,
\begin{enumerate}
\item[(i)]$\Res\Bigl(\frac{e^{\lambda t}}{\lambda^{2n-1}}R_A^\infty(\lambda)A^nx,\lambda=\nu_\ell\Bigr)=\Res\Bigl(\lambda e^{\lambda t}R_A^\infty(\lambda)x,\lambda=\nu_\ell\Bigr)$, if $\nu_\ell\ne0$;
\item[(ii)] $\Res\Bigl(\frac{e^{\lambda t}}{\lambda^{2n-1}}R_A^\infty(\lambda)A^nx,\lambda=0\Bigr)=\Res\Bigl(\lambda e^{\lambda t}R_A^\infty(\lambda)x,\lambda=0\Bigr)-\displaystyle\sum_{j=0}^{n-1}\frac{t^{2j}}{(2j)!}A^jx$, if $\nu_\ell=0$;
\item[(iii)] $\Res\Bigl(\frac{e^{\lambda t}}{\lambda^{2n-1}}R_A^\infty(\lambda)A^nx,\lambda=0\Bigr)=-\displaystyle\sum_{j=0}^{n-1}\frac{t^{2j}}{(2j)!}A^jx$, if $0\in\Omega_*\setminus\big(\cM\cup\cN\big)$.
\end{enumerate}
\end{lemma}
\begin{proof} From Lemma~\ref{l4.10}(iii) we have
\begin{equation}\label{l4.14.1}	
\frac{e^{\lambda t}}{\lambda^{2n-1}}R_A^\infty(\lambda)A^nx=\lambda R_A^\infty(\lambda)x-\displaystyle\sum_{j=0}^{n-1}\frac{e^{\lambda t}}{\lambda^{2j+1}}A^jx
\end{equation}
for any $\lambda\in\Omega\setminus\bigl(\cM\cup\cN\cup\{0\}\bigr)$.	
Assertions (i)-(iii) follow from \eqref{l3.5.2} and \eqref{l4.14.1}.
\end{proof}
\begin{lemma}\label{l4.15}
Assume Hypotheses (H), (Q), (H-ext), $n\in\NN$, $x\in\dom(A^{n-1})\cap\bX_\infty$. Then,
\begin{enumerate}
\item[(i)] $\Res\Bigl(\frac{e^{\lambda t}}{\lambda^{2n}}R_A^\infty(\lambda)A^{n-1}x,\lambda=\nu\Bigr)\in\dom(\tA)$ for any $\nu\in\cN\cup\{0\}$;
\item[(ii)]$\tA\Res\Bigl(\frac{e^{\lambda t}}{\lambda^{2n}}R_A^\infty(\lambda)A^{n-1}x,\lambda=\nu\Bigr)=\Res\Bigl(e^{\lambda t}R_A^\infty(\lambda)x,\lambda=\nu\Bigr)$, if $\nu\in\cN\setminus\{0\}$;
\item[(iii)] $\tA\Res\Bigl(\frac{e^{\lambda t}}{\lambda^{2n}}R_A^\infty(\lambda)A^{n-1}x,\lambda=0\Bigr)=\Res\Bigl(e^{\lambda t}R_A^\infty(\lambda)x,\lambda=0\Bigr)-\displaystyle\sum_{j=0}^{n-1}\frac{t^{2j+1}}{(2j+1)!}A^jx$, if $0\in\cN$;
\item[(iv)] $\tA\Res\Bigl(\frac{e^{\lambda t}}{\lambda^{2n}}R_A^\infty(\lambda)A^{n-1}x,\lambda=0\Bigr)=-\displaystyle\sum_{j=0}^{n-1}\frac{t^{2j+1}}{(2j+1)!}A^jx$, if $0\in\Omega_*\setminus\big(\cM\cup\cN\big)$.
\end{enumerate}
\end{lemma}
\begin{proof} From Lemma~\ref{l4.10}(i),
\begin{equation}\label{l4.15.1}
\frac{e^{\lambda t}}{\lambda^{2n}}R_A^\infty(\lambda)A^{n-1}x\in\dom(\tA)\;\mbox{and}\;\tA\Bigl(\frac{e^{\lambda t}}{\lambda^{2n}}R_A^\infty(\lambda)A^{n-1}x\Bigr)=\frac{e^{\lambda t}}{\lambda^{2n-2}}R_A^\infty(\lambda)A^{n-1}x-\frac{e^{\lambda t}}{\lambda^{2n}}A^{n-1}x	
\end{equation}
for any $\lambda\in\Omega_*\setminus\bigl(\cM\cup\cN\cup\{0\}\bigr)$. Assertion (i) follows from \eqref{l4.15.1} and \cite[Proposition 1.1.7]{ABHN}. Moreover, we have
\begin{equation}\label{l4.15.2}
\tA\Res\Bigl(\frac{e^{\lambda t}}{\lambda^{2n}}R_A^\infty(\lambda)A^{n-1}x,\lambda=\nu\Bigr)=\Res\Bigl(\frac{e^{\lambda t}}{\lambda^{2n-2}}R_A^\infty(\lambda)A^{n-1}x,\lambda=\nu\Bigr)-\Res\Bigl(\frac{e^{\lambda t}}{\lambda^{2n}},\lambda=\nu\Bigr)A^{n-1}x
\end{equation}
for any $\nu\in\cN\cup\{0\}$. From Lemma~\ref{l4.10}(iii) it follows that
\begin{equation}\label{l4.15.3}	
\frac{e^{\lambda t}}{\lambda^{2n-2}}R_A^\infty(\lambda)A^{n-1}x=e^{\lambda t} R_A^\infty(\lambda)x-\sigma_n(\lambda,t,x)
\end{equation}
for any  $\lambda\in\Omega_*\setminus\bigl(\cM\cup\cN\cup\{0\}\bigr)$,	
where the function $\sigma_n:\CC\setminus\{0\}\times\RR\times\dom(A^{n-1})\to\bX$ was introduced in \eqref{l3.6.4}.
From \eqref{l3.6.4}, \eqref{l3.6.5}, \eqref{l4.15.2} and \eqref{l4.15.3} we obtain
\begin{equation}\label{l4.15.4}
\tA\Res\Bigl(\frac{e^{\lambda t}}{\lambda^{2n}}R_A^\infty(\lambda)A^{n-1}x,\lambda=\nu\Bigr)=\Res\Bigl(e^{\lambda t}R_A^\infty(\lambda)x,\lambda=\nu\Bigr)-\sum_{j=0}^{n-1}\Res\Bigl(\frac{e^{\lambda t}}{\lambda^{2j+2}},\lambda=\nu\Bigr)A^jx	
\end{equation}
for any $\nu\in\cN\cup\{0\}$. Assertions (ii)-(iv) follow from \eqref{l3.5.2} and \eqref{l4.15.4}.		
\end{proof}
Next, we aim to switch the integral in the representation formulas of the cosine family and its associated sine family given in Theorem~\ref{t3.11} and Theorem~\ref{t3.12}, to a path along the complex curve $\re\lambda=g_*(\im\lambda)+\eps$ for some $\eps>0$ small enough. Similarly to Section~\ref{sec3}, see \eqref{def-eps-o}, we define
\begin{equation}\label{tilde-eps}
\teps_0:=\frac{1}{4}\min\bigl\{\eps_0,\delta_0,\re\nu_1-g_*(\im\nu_1),\dots,\re\nu_p-g_*(\im\nu_p)\bigr\}>0.	
\end{equation}	
The positive constant $\delta_0$ was introduced in Hypothesis (H-ext).
Similarly to \eqref{def-Gamma-a-eps}, for any $a>\max_{1\leq \ell\leq p}|\im\nu_\ell|$ and $\eps\in(0,\teps_0)$ we define the closed simple path $\tGamma_a^\eps$ oriented counterclockwise, defined as the union $\Gamma_{a,\mi}^\eps\cup\tGamma_{a,\up}^\eps\cup\tGamma_{a,\down}^\eps\cup-\tGamma_{a,\mi}^\eps$, where
\begin{align}\label{def-tilde-Gamma-a-eps}
&\tGamma_{a,\up}^\eps:=\big[g_0(a)+\eps+\rmi a,g_*(a)+\eps+\rmi a\big],\;\tGamma_{a,\down}^\eps=\Big[g_*(-a)+\eps-\rmi a,g_0(-a)+\eps-\rmi a\Big],\nonumber\\
&\tGamma_{a,\mi}^\eps: \lambda=g_*(s)+\eps+\rmi s,\; s\in[-a,a].
\end{align}
The path $\Gamma_{a,\mi}^\eps$ was defined in \eqref{def-Gamma-a-eps}.
\begin{figure}[h]
	\begin{center}
		\includegraphics[width=0.6\textwidth]{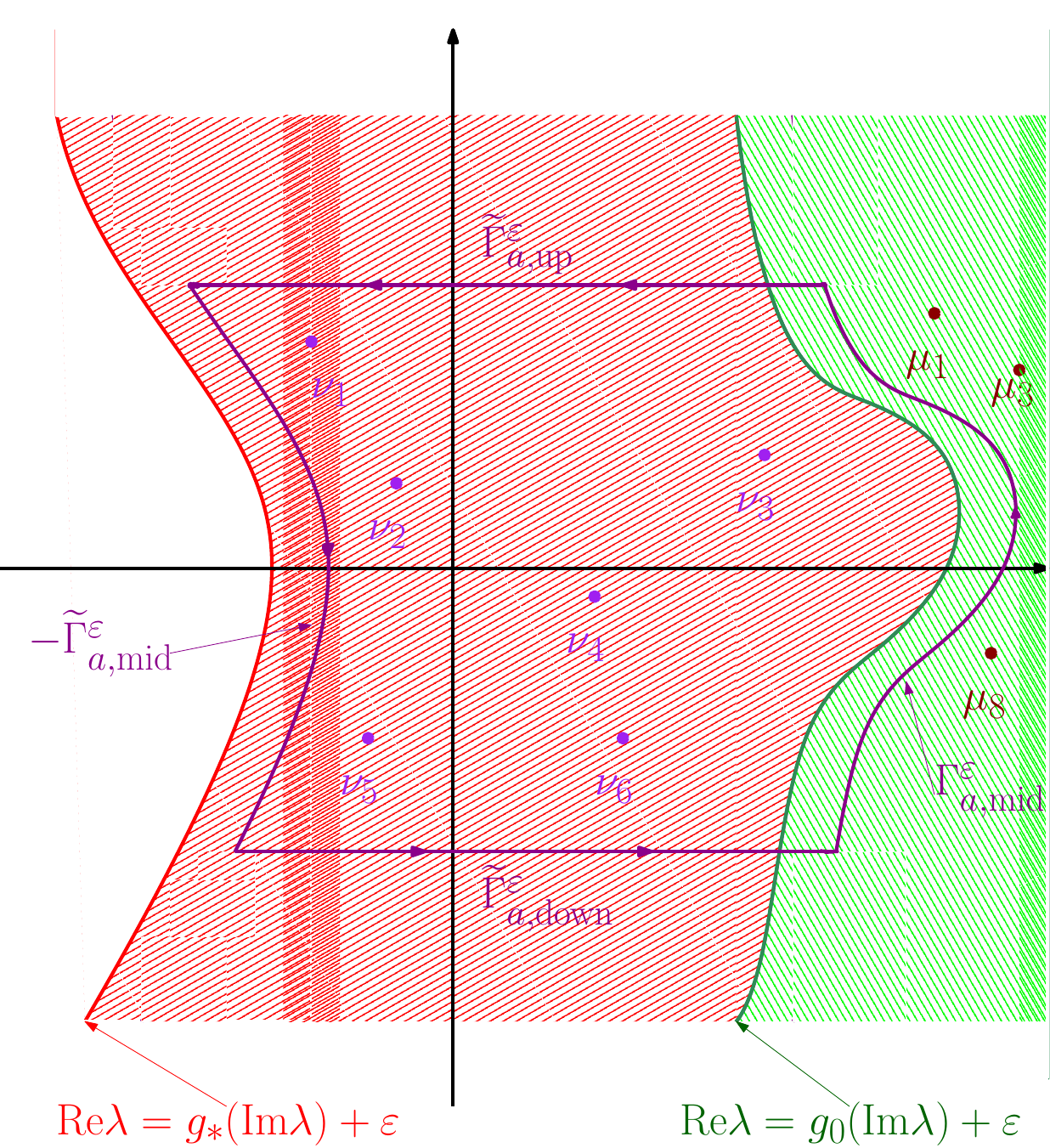}
		
		Figure 4. The path $\tGamma_a^\eps$ superimposed over the set $\Omega_*$ from \eqref{def-Omega-star}.
	\end{center}
\end{figure}
\begin{lemma}\label{l4.16}
Assume Hypotheses (H), (Q), (H-ext), $b:=\inf\limits_{s\in\RR}g_0(s)>-\infty$. Then,
 \begin{equation}\label{tilde-Gamma-2-3-lim}
\lim_{a\to\infty}\int_{\tGamma_{a,\up}^\eps}\frac{e^{\lambda t}}{\lambda^k}Q_iR_A^\infty(\lambda)Q_j\rmd\lambda=\lim_{a\to\infty}\int_{\tGamma_{a,\down}^\eps}\frac{e^{\lambda t}}{\lambda^k}Q_iR_A^\infty(\lambda)Q_j\rmd\lambda=0
\end{equation}
in the $\mathcal{B}(\bX)$-norm, for any $t>\beta_{i,j}$, $\eps\in(0,\teps_0)$, $k\in\NN$ with $k\geq 2\max\{n_0,n_*\}-1$, where $n_0$ and $n_*$ are introduced in \eqref{limit-n-zero} and \eqref{limit-n-star}, respectively.	
If, in addition, $b_*:=\inf_{s\in\RR}g_*(s)>-\infty$, the limits are zero for any $t\geq 0$.
\end{lemma}
\begin{proof} The proof of the lemma is similar to the proof of Lemma~\ref{l3.9}. First, we fix $t\geq 0$, $\eps\in(0,\teps_0)$ and $i,j\in\NN$. From Hypothesis (H-ext)(i) and since $\Range(Q_j)\subseteq\Ker(Q_{j+1}-I_\bY)\subset\bX_\infty$ by Lemma~\ref{l4.2}(ii), we have
\begin{equation}\label{l4.16.1}
R_A^\infty(\lambda)Q_j=R_A(\lambda)_{|\bX_\infty}Q_j=R_A(\lambda)Q_j\;\mbox{for any}\;\lambda\in\Omega_0\setminus\cM.	
\end{equation}
From \eqref{def-tilde-Gamma-a-eps} and \eqref{l4.16.1} we obtain
\begin{align}\label{l4.16.2}
\int_{\tGamma_{a,\up}^\eps}&\frac{e^{\lambda t}}{\lambda^k}Q_iR_A^\infty(\lambda)Q_j\rmd\lambda=-\int_{g_*(a)+\eps}^{g_0(a)+\eps}\frac{e^{(\xi+\rmi a)t}}{(\xi+\rmi a)^k}Q_iR_A^\infty(\xi+\rmi a)Q_j\rmd\xi\nonumber\\&=-\int_{g_*(a)+\eps}^{g_0(a)}\frac{e^{(\xi+\rmi a)t}}{(\xi+\rmi a)^k}Q_iR_A^\infty(\xi+\rmi a)Q_j\rmd\xi-\int_{g_0(a)}^{g_0(a)+\eps}\frac{e^{(\xi+\rmi a)t}}{(\xi+\rmi a)^k}Q_iR_A(\xi+\rmi a)Q_j\rmd\xi,\nonumber\\
\int_{\tGamma_{a,\down}^\eps}&\frac{e^{\lambda t}}{\lambda^k}Q_iR_A^\infty(\lambda)Q_j\rmd\lambda=\int_{g_*(-a)+\eps}^{g_0(-a)+\eps}\frac{e^{(\xi-\rmi a)t}}{(\xi-\rmi a)^k}Q_iR_A^\infty(\xi-\rmi a)Q_j\rmd\xi\nonumber\\&=\int_{g_*(-a)+\eps}^{g_0(-a)}\frac{e^{(\xi-\rmi a)t}}{(\xi-\rmi a)^k}Q_iR_A^\infty(\xi-\rmi a)Q_j\rmd\xi+\int_{g_0(-a)}^{g_0(-a)+\eps}\frac{e^{(\xi-\rmi a)t}}{(\xi+\rmi a)^k}Q_iR_A(\xi-\rmi a)Q_j\rmd\xi.
\end{align}
From Hypothesis (H-ext)(iii) it follows that
\begin{align}\label{l4.16.3}
\Bigl\|\frac{e^{\lambda t}}{\lambda^k}Q_iR_A^\infty(\lambda)Q_j\Bigr\|&=\frac{e^{(\re\lambda)t}}{|\lambda|^k}\|Q_iR_A^\infty(\lambda)Q_j\|\leq\frac{e^{(\re\lambda)t}}{|\lambda|^k}\tih_{i,j,1}(|\lambda|)\tih_{i,j,2}(\re\lambda)\nonumber\\
&\leq \tM_{i,j}e^{(t-\beta_{i,j})\re\lambda}\tih_{i,j,2}(s)\bigr)\frac{\tih_{i,j,1}(|\lambda|)}{|\lambda|^k},\;\mbox{whenever}\; \re\lambda\leq g_0(\im\lambda),
\end{align}
where $\tM_{i,j}:=\sup\limits_{s\leq\omega}e^{\beta_{i,j}s}\tih_{i,j,2}(s)$. We infer that
\begin{equation}\label{l4.16.4}
\Bigg\|\int_{g_*(\pm a)+\eps}^{g_0(\pm a)}\frac{e^{(\xi\pm\rmi a)t}}{(\xi\pm\rmi a)^k}Q_iR_A^\infty(\xi\pm\rmi a)Q_j\rmd\xi\Bigg\|	\leq\tM_{i,j}
\int_{b_*}^\omega e^{(t-\beta_{i,j})\xi}\frac{\tih_{i,j,1}(\sqrt{\xi^2+a^2})}{(\xi^2+a^2)^{\frac{k}{2}}}\rmd\xi,	
\end{equation}
where $b_*=\inf_{s\in\RR}g_*(s)$. We note that $\sup\limits_{s\in\RR}g_0(s)+\teps_0\leq\sup\limits_{s\in\RR}g_0(s)+\eps_0<\omega$ by \eqref{def-eps-o}. From\eqref{l3.9.2} we obtain that
\begin{equation}\label{l4.16.5}
\Bigg\|\int_{g_0(\pm a)}^{g_0(\pm a)+\eps}\frac{e^{(\xi\pm\rmi a)t}}{(\xi\pm\rmi a)^k}Q_iR_A(\xi\pm\rmi a)Q_j\rmd\xi\Bigg\|\leq e^{\omega t}\bigl(\sup_{b\leq s\leq\gamma}h_2(s)\bigr)\int_{b}^\omega\frac{h_1(\sqrt{\xi^2+a^2})}{(\xi^2+a^2)^{\frac{k}{2}}}\rmd\xi.	
\end{equation}
for any $t\geq 0$. One can readily check that
\begin{equation}\label{l4.16.6}
e^{(t-\beta_{i,j}\cdot)}\chi_{(b_*,\omega]}\in L^1(\RR)\;\mbox{whenever}\; t>\beta_{i,j}\;\mbox{or}\;t\geq 0,b_*>-\infty.
\end{equation}
Since $\lim\limits_{s\to\infty}\frac{\tih_{i,j,1}(s)}{s^k}=0$, from Remark~\ref{r3.8} and \eqref{l4.16.6} we conclude that
\begin{equation}\label{l4.16.7}
\lim_{a\to\infty}\int_{b_*}^\omega e^{(t-\beta_{i,j})\xi}\frac{\tih_{i,j,1}(\sqrt{\xi^2+a^2})}{(\xi^2+a^2)^{\frac{k}{2}}}\rmd\xi=0\;\mbox{whenever}\; t>\beta_{i,j}\;\mbox{or}\;t\geq 0,b_*>-\infty.	 
\end{equation}
The lemma follows from \eqref{l3.9.4}, \eqref{l4.16.2}, \eqref{l4.16.4}, \eqref{l4.16.5}, \eqref{l4.16.7}.
\end{proof}
\begin{remark}\label{r4.17}
From \eqref{tilde-eps} and Hypotheses (H) and (H-ext) one can readily check that the interior of the set enclosed by the path $\tGamma_a^\eps$ contains all the resonances $\nu_\ell$, $\ell=1,\dots,p$, and does not contain any of the singularities $\mu_k$, $k=1,\dots,m$, for any $a>\max_{1\leq \ell\leq p}|\im\nu_\ell|$, $\eps\in(0,\teps_0)$. Moreover, $0\notin\mathrm{Range}(\tGamma_{a,\mi}^\eps)$. In addition, $0$ belongs to the interior of the set enclosed by $\tGamma_a^\eps$ if and only if $g_*(0)<0\leq g_0(0)$.
\end{remark}
Let $P_k$ and $N_k\in\NN$, $k=1,\dots,m$ be defined in Lemma~\ref{l3.3}, $\tP_\ell$, $\oP_\ell$, $\tN_\ell$, $\oN_\ell$, $\ell=1,\dots,p$, be defined in Lemma~\ref{l4.11} and $p_{1,k,\na}$ and $\tp_{1,\ell,\na}$ be defined in \eqref{polynomial-representation} and \eqref{polynomial-representation-resonances}, respectively. Also, we introduce the path
\begin{equation}\label{def-Lambda-star}	
\Lambda_*^\eps: \lambda=g_*(s)+\eps+\rmi s,\; s\in\RR.	
\end{equation}	
\begin{theorem}\label{t4.18}
Assume Hypotheses (H), (Q) and (H-ext),  $b:=\inf\limits_{s\in\RR}g_0(s)>-\infty$. Then,
\begin{enumerate}
\item[(i)] If $x\in\dom(A^n)\cap\Ker(Q_j-I_\bY)$, then $\displaystyle\int_{\Lambda_*^\eps}\frac{e^{\lambda t}}{\lambda^{2n-1}}Q_iR_A^\infty(\lambda)A^nx\rmd\lambda$ is convergent in the principal value sense in $\bX$ for any $t\geq \beta_{i,j}$, $\eps\in(0,\teps_0)$, $n\in\NN$, with $n\geq\max\{n_0,n_*\}$;
\item[(ii)] If $x\in\dom(A^n)\cap\bX_\infty$ and $b_*:=\inf_{s\in\RR}g_*(s)>-\infty$, then $\displaystyle\int_{\Lambda_*^\eps}\frac{e^{\lambda t}}{\lambda^{2n-1}}R_A^\infty(\lambda)A^nx\rmd\lambda$ is convergent in the principal value sense in $\bY$ for any $t\geq 0$, $\eps\in(0,\teps_0)$, $n\in\NN$, with $n\geq\max\{n_0,n_*\}$;
\item[(iii)] If $x\in\dom(A^n)\cap\Ker(Q_j-I_\bY)$ then, for any $t\geq \beta_{i,j}$, $\eps\in(0,\teps_0)$, $n\in\NN$, with $n\geq\max\{n_0,n_*\}$,
\begin{align}\label{partial-resonance-cos}Q_iC(t)x&=\displaystyle\sum\limits_{k=1}^{m} \sum_{\na=1}^{N_k}p_{1,k,\na}(t)e^{\mu_k t}Q_i(A-\mu_k^2 I_\bX)^{\na-1}P_kx+\displaystyle\sum\limits_{\nu_\ell\ne0} \sum_{\na=1}^{\tN_\ell}\tp_{1,\ell,\na}(t)e^{\nu_\ell t}Q_i(\tA-\nu_\ell^2 I_\bY)^{\na-1}\tP_\ell x \nonumber\\
&+\frac{1}{2\pi\rmi}\PV\displaystyle\int_{\Lambda_*^\eps}\frac{e^{\lambda t}}{\lambda^{2n-1}}Q_iR_A^\infty(\lambda)A^nx\rmd\lambda+\chi_{[0,\infty)}(g_*(0))\displaystyle\sum\limits_{\na=0}^{n-1}\frac{t^{2\na}}{(2\na)!}Q_iA^\na x\nonumber\\
&+\chi_{\cN}(0)\Bigg(\sum_{\na=1}^{\oN_\ell}\frac{t^{2\na-2}}{(2\na-2)!}Q_i\tA^{\na-1}\tP_\ell x+\sum_{\na=1}^{\oN_\ell}\frac{t^{2\na-1}}{(2\na-1)!}Q_i\tA^{\na-1}\oP_\ell x\Bigg);
\end{align}
\item[(iv)] If $x\in\dom(A^n)\cap\bX_\infty$ and $b_*:=\inf_{s\in\RR}g_*(s)>-\infty$ then, for any $t\geq 0$, $\eps\in(0,\teps_0)$, $n\in\NN$, with $n\geq\max\{n_0,n_*\}$
\begin{align}\label{resonance-cos}C(t)x&=\displaystyle\sum\limits_{k=1}^{m} \sum_{\na=1}^{N_k}p_{1,k,\na}(t)e^{\mu_k t}(A-\mu_k^2 I_\bX)^{\na-1}P_kx+\displaystyle\sum\limits_{\nu_\ell\ne0} \sum_{\na=1}^{\tN_\ell}\tp_{1,\ell,\na}(t)e^{\nu_\ell t}(\tA-\nu_\ell^2 I_\bY)^{\na-1}\tP_\ell x \nonumber\\
&+\frac{1}{2\pi\rmi}\PV\displaystyle\int_{\Lambda_*^\eps}\frac{e^{\lambda t}}{\lambda^{2n-1}}R_A^\infty(\lambda)A^nx\rmd\lambda+\chi_{[0,\infty)}(g_*(0))\displaystyle\sum\limits_{\na=0}^{n-1}\frac{t^{2\na}}{(2\na)!}A^\na x\nonumber\\
&+\chi_{\cN}(0)\Bigg(\sum_{\na=1}^{\oN_\ell}\frac{t^{2\na-2}}{(2\na-2)!}\tA^{\na-1}\tP_\ell x+\sum_{\na=1}^{\oN_\ell}\frac{t^{2\na-1}}{(2\na-1)!}\tA^{\na-1}\oP_\ell x\Bigg).
\end{align}
\end{enumerate}
\end{theorem}
\begin{proof}
Fix $t\geq 0$, $x\in\dom(A^n)\cap\Ker(Q_j-I_\bY)$, $\eps\in(0,\teps_0)$. From Remark~\ref{r4.17}, \eqref{tilde-eps} and \eqref{def-tilde-Gamma-a-eps} it follows that the singularities of the meromorphic function $\lambda\to\frac{e^{\lambda t}}{\lambda^{2n-1}}R_A^\infty(\lambda)A^nx:\Omega_*\setminus\big(\cM\cup\cN\cup\{0\}\big)\to\bY$ that belong to the interior of the set enclosed by the closed path $\tGamma_a^\eps$ are the finitely meromorphic points $\{\nu_\ell:\ell=1,\dots,p\}$  and, possibly $0$, in the case when that $g_*(0)<0\leq g_0(0)$. From Lemma~\ref{l4.11}, Remark~\ref{r4.13},  Lemma~\ref{l4.14} and the Residues Theorem we obtain
\begin{align}\label{t4.18.1}
\frac{1}{2\pi\rmi}\int_{\tGamma_a^\eps}\frac{e^{\lambda t}}{\lambda^{2n-1}}&R_A^\infty(\lambda)A^nx\rmd\lambda=\sum_{\ell=1}^{p}\Res(\lambda e^{\lambda t}R_A^\infty(\lambda)x,\lambda=\nu_\ell)-\chi_{(g_*(0),g_0(0)]}(0)\sum_{\na=0}^{n-1}\frac{t^{2\na}}{(2\na)!}A^jx\nonumber\\
&=\sum\limits_{\nu_\ell\ne0} \sum_{\na=1}^{\tN_\ell}\tp_{1,\ell,\na}(t)e^{\nu_\ell t}(\tA-\nu_\ell^2 I_\bY)^{\na-1}\tP_\ell x -\chi_{(g_*(0),g_0(0)]}(0)\sum_{\na=0}^{n-1}\frac{t^{2\na}}{(2\na)!}A^jx\nonumber\\
&\qquad+\chi_{\cN}(0)\Bigg(\sum_{\na=1}^{\tN_\ell}\frac{t^{2\na-2}}{(2\na-2)!}\tA^{\na-1}\tP_\ell x+\sum_{\na=1}^{\tN_\ell}\frac{t^{2\na-1}}{(2\na-1)!}\tA^{\na-1}\oP_\ell x\Bigg).
\end{align}
for any $a>\max_{1\leq \ell\leq p}|\im\nu_\ell|+1$. From Lemma~\ref{l4.10}(iii) we know that $A^nx\in\Ker(Q_j-I_\bY)$, hence $A^nx=Q_jA^nx$. Since $Q_i\in\cC\cL(\bY,\bX)$, from \eqref{def-tilde-Gamma-a-eps} and \eqref{t4.18.1} we derive
\begin{align}\label{t4.18.2}
\frac{1}{2\pi\rmi}Q_i&\int_{\tGamma_{a,\mi}^\eps}\frac{e^{\lambda t}}{\lambda^{2n-1}}R_A^\infty(\lambda)A^nx\rmd\lambda=\frac{1}{2\pi\rmi}\int_{\tGamma_{a,\mi}^\eps}\frac{e^{\lambda t}}{\lambda^{2n-1}}Q_iR_A^\infty(\lambda)A^nx\rmd\lambda\nonumber\\
&=\frac{1}{2\pi\rmi}Q_i\int_{\Gamma_{a,\mi}^\eps}\frac{e^{\lambda t}}{\lambda^{2n-1}}R_A(\lambda)A^nx\rmd\lambda+\frac{1}{2\pi\rmi}\int_{\tGamma_{a,\up}^\eps}\frac{e^{\lambda t}}{\lambda^{2n-1}}Q_iR_A^\infty(\lambda)Q_jA^nx\rmd\lambda\nonumber\\
&+\frac{1}{2\pi\rmi}\int_{\tGamma_{a,\down}^\eps}\frac{e^{\lambda t}}{\lambda^{2n-1}}Q_iR_A^\infty(\lambda)Q_jA^nx\rmd\lambda-\frac{1}{2\pi\rmi}Q_i\int_{\tGamma_a^\eps}\frac{e^{\lambda t}}{\lambda^{2n-1}}R_A^\infty(\lambda)A^nx\rmd\lambda.				
\end{align}	
Assertions (i), (ii) and (iii) follow from Hypothesis ($Q_8$), Theorem~\ref{t3.11}, Lemma~\ref{l4.16}, \eqref{t4.18.1} and \eqref{t4.18.2} due to
\begin{equation}\label{t4.18.3}
\chi_{(g_*(0),g_0(0)]}(0)+\chi_{[0,\infty)}(g_0(0))=\chi_{[0,\infty)}(g_*(0)).
\end{equation}
If, in addition, $b_*=\inf_{s\in\RR}g_*(s)>-\infty$, arguing in the same way as above we conclude that assertion (iii) holds true for any $t\geq 0$. One infers assertion (iv) by passing to the limit as $i\to\infty$ in \eqref{t4.18.2}.
\end{proof}	
We recall the definition of the space $W_{n-1}$ in \eqref{def-bW-n}. Also, $P_k$ and $N_k\in\NN$, $k=1,\dots,m$ are defined in Lemma~\ref{l3.3}, $\tP_\ell$, $\oP_\ell$, $\tN_\ell$, $\oN_\ell$, $\ell=1,\dots,p$ are defined in Lemma~\ref{l4.11} and $p_{0,k,\na}$ and $\tp_{0,\ell,\na}$ are defined in \eqref{polynomial-representation} and \eqref{polynomial-representation-resonances}, respectively.
\begin{theorem}\label{t4.19}
Assume Hypotheses (H), (Q) and (H-ext),  $b:=\inf\limits_{s\in\RR}g_0(s)>-\infty$. Then,
\begin{enumerate}
\item[(i)] If $w\in\bW_{n-1}\cap\Ker(Q_j-I_\bY)$, then $\displaystyle\int_{\Lambda_*^\eps}\frac{e^{\lambda t}}{\lambda^{2n}}Q_iR_A^\infty(\lambda)A^{n-1}w\rmd\lambda\in\dom(A)$ is convergent in the principal value sense in $\bX$ for any $t\geq \beta_{i,j}$, $\eps\in(0,\eps_0)$, $n\in\NN$, with $n\geq\max\{n_0,n_*\}$;
\item[(ii)] If $w\in\bW_{n-1}\cap\bX_\infty$ and $b_*:=\inf_{s\in\RR}g_*(s)>-\infty$, then $\displaystyle\int_{\Lambda_*^\eps}\frac{e^{\lambda t}}{\lambda^{2n}}R_A^\infty(\lambda)A^{n-1}w\rmd\lambda\in\dom(\tA)$ is convergent in the principal value sense in $\bY$ for any $t\geq 0$, $\teps\in(0,\eps_0)$, $n\in\NN$, with $n\geq\max\{n_0,n_*\}$;
\item[(iii)] If $w\in\bW_{n-1}\cap\Ker(Q_j-I_\bY)$, then for any $t\geq \beta_{i,j}$, $\eps\in(0,\teps_0)$, $n\in\NN$, with $n\geq\max\{n_0,n_*\}$
\begin{align}\label{partial-resonance-sin}
Q_iS(t)w&=\displaystyle\sum\limits_{k=1}^{m} \sum_{\na=1}^{N_k}p_{0,k,\na}(t)e^{\mu_k t}Q_i(A-\mu_k^2 I_\bX)^{\na-1}P_kw+\displaystyle\sum\limits_{\nu_\ell\ne0} \sum_{\na=1}^{\tN_\ell}\tp_{0,\ell,\na}(t)e^{\nu_\ell t}Q_i(\tA-\nu_\ell^2 I_\bY)^{\na-1}\tP_\ell w \nonumber\\
&+\frac{1}{2\pi\rmi}Q_iA\Bigg(\PV\displaystyle\int_{\Lambda_*^\eps}\frac{e^{\lambda t}}{\lambda^{2n}}Q_{i+1}R_A^\infty(\lambda)A^{n-1}w\rmd\lambda\Bigg)+\chi_{[0,\infty)}(g_*(0))\displaystyle\sum\limits_{\na=0}^{n-1}\frac{t^{2\na+1}}{(2\na+1)!}Q_iA^\na w\nonumber\\
&+\chi_{\cN}(0)\Bigg(\sum_{\na=1}^{\oN_\ell}\frac{t^{2\na-1}}{(2\na-1)!}Q_i\tA^{\na-1}\tP_\ell w+\sum_{\na=1}^{\oN_\ell}\frac{t^{2\na}}{(2\na)!}Q_i\tA^{\na-1}\oP_\ell w\Bigg);
\end{align}
\item[(iv)] If $w\in\bW_{n-1}\cap\bX_\infty$ and $b_*:=\inf_{s\in\RR}g_*(s)>-\infty$ then, for any $t\geq 0$, $\eps\in(0,\teps_0)$, $n\in\NN$, with $n\geq\max\{n_0,n_*\}$
\begin{align}\label{resonance-sin}
S(t)w&=\displaystyle\sum\limits_{k=1}^{m} \sum_{\na=1}^{N_k}p_{0,k,\na}(t)e^{\mu_k t}(A-\mu_k^2 I_\bX)^{\na-1}P_kw+\displaystyle\sum\limits_{\nu_\ell\ne0} \sum_{\na=1}^{\tN_\ell}\tp_{0,\ell,\na}(t)e^{\nu_\ell t}(\tA-\nu_\ell^2 I_\bY)^{\na-1}\tP_\ell w \nonumber\\
&+\frac{1}{2\pi\rmi}\tA\Bigg(\PV\displaystyle\int_{\Lambda_*^\eps}\frac{e^{\lambda t}}{\lambda^{2n}}R_A^\infty(\lambda)A^{n-1}w\rmd\lambda\Bigg)+\chi_{[0,\infty)}(g_*(0))\displaystyle\sum\limits_{\na=0}^{n-1}\frac{t^{2\na+1}}{(2\na+1)!}A^\na w\nonumber\\
&+\chi_{\cN}(0)\Bigg(\sum_{\na=1}^{\oN_\ell}\frac{t^{2\na-1}}{(2\na-1)!}\tA^{\na-1}\tP_\ell w+\sum_{\na=1}^{\oN_\ell}\frac{t^{2\na}}{(2\na)!}\tA^{\na-1}\oP_\ell w\Bigg).
\end{align}
\end{enumerate}
\end{theorem}
\begin{proof}
Fix $t\geq 0$, $w\in\bW_{n-1}\cap\Ker(Q_j-I_\bY)$, $\eps\in(0,\eps_0)$. From Remark~\ref{r4.17}, \eqref{tilde-eps} and \eqref{def-tilde-Gamma-a-eps} it follows that the singularities of the meromorphic function $\lambda\to\frac{e^{\lambda t}}{\lambda^{2n}}R_A^\infty(\lambda)A^nx:\Omega_*\setminus\big(\cM\cup\cN\cup\{0\}\big)\to\bY$ that belong to the interior of the set enclosed by the closed path $\tGamma_a^\eps$ are the finitely meromorphic points $\{\nu_\ell:\ell=1,\dots,p\}$  and, possibly $0$, in the case when $g_*(0)<0\leq g_0(0)$.  By the Residues Theorem
\begin{equation}\label{t4.19.1}
\frac{1}{2\pi\rmi}\int_{\tGamma_a^\eps}\frac{e^{\lambda t}}{\lambda^{2n}}R_A^\infty(\lambda)A^nx\rmd\lambda=\tcS_{n,t}w
\end{equation}
for any $a>\max_{1\leq \ell\leq p}|\im\nu_\ell|+1$, where
$\tcS_{n,t}:\bW_{n-1}\to\bY$ is defined by
\begin{equation}\label{t4.19.2}
\tcS_{n,t}w=\sum_{\nu_\ell\ne0}\Res(\frac{e^{\lambda t}}{\lambda^{2n}}R_A^\infty(\lambda)A^{n-1}w,\lambda=\mu)+\chi_{(g_*(0),g_0(0)]}(0)\Res(\frac{e^{\lambda t}}{\lambda^{2n}}R_A^\infty(\lambda)A^{n-1}w,\lambda=0).
\end{equation}
From Lemma~\ref{l4.11}, Remark~\ref{r4.13} and Lemma~\ref{l4.15} we see that $\Range(\tcS_{n,t})\subseteq\dom(\tA)$ and
\begin{align}\label{l4.19.3}
\tA\tcS_{n,t}w&=\sum_{\ell=1}^{p}\Res(e^{\lambda t}R_A^\infty(\lambda)w,\lambda=\nu_\ell)-\chi_{(g_*(0),g_0(0)]}(0)\sum_{\na=0}^{n-1}\frac{t^{2\na+1}}{(2\na+1)!}A^\na w\nonumber\\
&=\sum\limits_{\nu_\ell\ne0} \sum_{\na=1}^{\tN_\ell}\tp_{0,\ell,\na}(t)e^{\nu_\ell t}(\tA-\nu_\ell^2 I_\bY)^{\na-1}\tP_\ell w -\chi_{(g_*(0),g_0(0)]}(0)\sum_{\na=0}^{n-1}\frac{t^{2\na+1}}{(2\na+1)!}A^\na w\nonumber\\
&\qquad+\chi_{\cN}(0)\Bigg(\sum_{\na=1}^{\oN_\ell}\frac{t^{2\na-1}}{(2\na-1)!}\tA^{\na-1}\tP_\ell w+\sum_{\na=1}^{\oN_\ell}\frac{t^{2\na}}{(2\na)!}\tA^{\na-1}\oP_\ell w
\Bigg).	
\end{align}	
From \eqref{def-bW-n} and Lemma~\ref{l4.10}(iii) we have $A^{n-1}w\in\bW\cap\Ker(Q_j-I_\bY)$, which implies that $A^{n-1}w=Q_jA^{n-1}$. Using again the fact that $Q_i\in\cC\cL(\bY,\bX)$, from \eqref{def-tilde-Gamma-a-eps} and \eqref{t4.19.1} we conclude that
\begin{align}\label{t4.19.4}
\frac{1}{2\pi\rmi}Q_i&\int_{\tGamma_{a,\mi}^\eps}\frac{e^{\lambda t}}{\lambda^{2n}}R_A^\infty(\lambda)A^{n-1}w\rmd\lambda=\frac{1}{2\pi\rmi}\int_{\tGamma_{a,\mi}^\eps}\frac{e^{\lambda t}}{\lambda^{2n}}Q_iR_A^\infty(\lambda)A^{n-1}w\rmd\lambda\nonumber\\
&=\frac{1}{2\pi\rmi}Q_i\int_{\Gamma_{a,\mi}^\eps}\frac{e^{\lambda t}}{\lambda^{2n}}R_A(\lambda)A^{n-1}w\rmd\lambda+\frac{1}{2\pi\rmi}\int_{\tGamma_{a,\up}^\eps}\frac{e^{\lambda t}}{\lambda^{2n}}Q_iR_A^\infty(\lambda)Q_jA^{n-1}w\rmd\lambda\nonumber\\
&+\frac{1}{2\pi\rmi}\int_{\tGamma_{a,\down}^\eps}\frac{e^{\lambda t}}{\lambda^{2n}}Q_iR_A^\infty(\lambda)Q_jA^{n-1}w\rmd\lambda-Q_i\tcS_{n,t}w.
\end{align}	
Assertions (i) and (ii) follow from  Hypothesis ($Q_8$), Theorem~\ref{t3.12}, Lemma~\ref{l4.16}, \eqref{t4.19.4} since $\dom(A)\subset\dom(\tA)$, $\Range(\tcS_{n,t})\subseteq\dom(\tA)$ and $Q_i\dom(\tA)\subseteq\dom(A)$ by Lemma~\ref{l4.8}(i).

\noindent\textit{Proof of (iii).} Using Theorem~\ref{t3.12} and passing to the limit as $a\to\infty$ in \eqref{t4.19.4}  we infer that
\begin{equation}\label{t4.19.5}
\PV\displaystyle\int_{\Lambda_*^\eps}\frac{e^{\lambda t}}{\lambda^{2n}}Q_{i+1}R_A^\infty(\lambda)A^{n-1}w\rmd\lambda=\frac{1}{2\pi\rmi}Q_{i+1}\PV\displaystyle\int_{\Lambda_*^\eps}\frac{e^{\lambda t}}{\lambda^{2n}}R_A(\lambda)A^{n-1}w\rmd\lambda-Q_{i+1}\tcS_{n,t}w,	
\end{equation}	
for any $t\geq\beta_{i,j}$. From Theorem~\ref{t3.12}(ii), Hypothesis ($Q_7$), Lemma~\ref{l4.8}(ii), \eqref{l4.19.3},  and \eqref{t4.19.5} it follows that
\begin{align}\label{t4.19.6}
\frac{1}{2\pi\rmi}&Q_iA\Bigg(\PV\displaystyle\int_{\Lambda_*^\eps}\frac{e^{\lambda t}}{\lambda^{2n}}Q_{i+1}R_A^\infty(\lambda)A^{n-1}w\rmd\lambda\Bigg)=\frac{1}{2\pi\rmi}Q_iA\PV\displaystyle\int_{\Lambda_*^\eps}\frac{e^{\lambda t}}{\lambda^{2n}}R_A^\infty(\lambda)A^{n-1}w\rmd\lambda-Q_i\tA\tcS_{n,t}w\nonumber\\
&=Q_iS(t)w-\displaystyle\sum\limits_{k=1}^{m} \sum_{\na=1}^{N_k}p_{0,k,\na}(t)e^{\mu_k t}Q_i(A-\mu_k^2 I_\bX)^{\na-1}P_kw-\chi_{[0,\infty)}(g_0(0))\displaystyle\sum\limits_{\na=0}^{n-1}\frac{t^{2\na+1}}{(2\na+1)!}Q_iA^\na w\nonumber\\
&\qquad-\displaystyle\sum\limits_{\nu_\ell\ne0} \sum_{\na=1}^{\tN_\ell}\tp_{0,\ell,\na}(t)e^{\nu_\ell t}Q_i(\tA-\nu_\ell^2 I_\bY)^{\na-1}\tP_\ell w-\chi_{(g_*(0),g_0(0)]}(0)\displaystyle\sum\limits_{\na=0}^{n-1}\frac{t^{2\na+1}}{(2\na+1)!}Q_iA^\na w\nonumber\\
&\qquad-\chi_{\cN}(0)\Bigg(\sum_{\na=1}^{\oN_\ell}\frac{t^{2\na-1}}{(2\na-1)!}\tA^{\na-1}\tP_\ell w+\sum_{\na=1}^{\oN_\ell}\frac{t^{2\na}}{(2\na)!}\tA^{\na-1}\oP_\ell w\Bigg)
\end{align}	
Assertion (iii) follows from \eqref{t4.18.3} and \eqref{t4.19.6}.

\noindent\textit{Proof of (iv).} Arguing similarly to the proof of (iii) we have
\begin{equation}\label{t4.19.7}
\PV\displaystyle\int_{\Lambda_*^\eps}\frac{e^{\lambda t}}{\lambda^{2n}}R_A^\infty(\lambda)A^{n-1}w\rmd\lambda=\frac{1}{2\pi\rmi}\PV\displaystyle\int_{\Lambda_*^\eps}\frac{e^{\lambda t}}{\lambda^{2n}}R_A(\lambda)A^{n-1}w\rmd\lambda-\tcS_{n,t}w,	
\end{equation}	
for any $t\geq 0$. From Theorem~\ref{t3.12}(ii), \eqref{l4.19.3}, \eqref{t4.19.7} and since $\tA_{|\dom(A)}=A$ we conclude that
\begin{align}\label{t4.19.8}
\frac{1}{2\pi\rmi}&\tA\Bigg(\PV\displaystyle\int_{\Lambda_*^\eps}\frac{e^{\lambda t}}{\lambda^{2n}}R_A^\infty(\lambda)A^{n-1}w\rmd\lambda\Bigg)=\frac{1}{2\pi\rmi}A\PV\displaystyle\int_{\Lambda_*^\eps}\frac{e^{\lambda t}}{\lambda^{2n}}R_A^\infty(\lambda)A^{n-1}w\rmd\lambda-\tA\tcS_{n,t}w\nonumber\\
&=S(t)w-\displaystyle\sum\limits_{k=1}^{m} \sum_{\na=1}^{N_k}p_{0,k,\na}(t)e^{\mu_k t}(A-\mu_k^2 I_\bX)^{\na-1}P_kw-\chi_{[0,\infty)}(g_0(0))\sum\limits_{\na=0}^{n-1}\frac{t^{2\na+1}}{(2\na+1)!}A^\na w\nonumber\\
&-\sum\limits_{\nu_\ell\ne0} \sum_{\na=1}^{\tN_\ell}\tp_{0,\ell,\na}(t)e^{\nu_\ell t}(\tA-\nu_\ell^2 I_\bY)^{\na-1}\tP_\ell w -\chi_{(g_*(0),g_0(0)]}(0)\sum_{\na=0}^{n-1}\frac{t^{2\na+1}}{(2\na+1)!}A^\na w\nonumber\\
&-\chi_{\cN}(0)\Bigg(\sum_{\na=1}^{\oN_\ell}\frac{t^{2\na-1}}{(2\na-1)!}\tA^{\na-1}\tP_\ell w+\sum_{\na=1}^{\oN_\ell}\frac{t^{2\na}}{(2\na)!}\tA^{\na-1}\oP_\ell w
\Bigg)
\end{align}
for any $t\geq 0$. Assertion (iv) follows from \eqref{t4.18.3} and \eqref{t4.19.8}.
\end{proof}
Similar to Theorem~\ref{t3.12}, we note that the conclusion of Theorem~\ref{t4.19} remain true if \eqref{limit-n-star} is replaced by $\lim\limits_{s\to\infty}\frac{\tih_{i,j,1}(s)}{s^{2n_*}}=0$.

\section{Frechet Space estimates of integral terms}\label{sec6}
In this subsection we aim to estimate the integral terms in the representations of cosine and sine functions from Theorem~\ref{t4.18} and Theorem~\ref{t4.19}. There are some similarities to the estimates of similar integral terms in Section~\ref{sec3}. However there are important differences. We have seen that the curve $\re\lambda=g_0(\im\lambda)$ needs to be to the right of the set $\{\lambda\in\CC:\lambda^2\in\sigma_{\mathrm{ess}}(A)\}$, therefore it is natural to assume that $\inf_{s\in\RR}g_0(s)>-\infty$. The curve $\re\lambda=g_*(\im\lambda)$ does not need to satisfy such conditions and in concrete examples (cosine families generated by a Schr\"odinger operator) it might be possible to choose a function $g_*$ such that $g_*(s)=\mathcal{O}(-\ln{|s|})$ as $s\to\pm\infty$.

Another important difference is that the extension  $R_A^\infty$ does not, in general, satisfy the Hilbert resolvent identity. More precisely, one can readily check that, in general, the range of $R(\mu^2,A)$ is not necessarily contained in $\bX_\infty$. Indeed, if $A=\partial_x^2$ one can readily check that $\bX_\infty=L^2_{\mathrm{comp}}(\RR)$, however $R(1,\partial_x^2)\chi_{[0,1]}\notin L^2_{\mathrm{comp}}(\RR)$. This shows that, in general, $R_A^\infty(\lambda)$ cannot be applied to elements of the range of $R(\mu^2,A)$, and thus $R_A^\infty(\lambda)R(\mu^2,A)$ needed for the Hilbert resolvent identity does not make sense.

First, we estimate the integral terms of the representations given in Theorem~\ref{t4.18} and Theorem~\ref{t4.19}, when the path $\re\lambda=g_*(\im\lambda)+\eps$ is replaced by the path $\tGamma_{a,\mi}^\eps$, see Figure 4.

\begin{lemma}\label{l4.20}
Assume Hypotheses (H), (Q) and (H-ext). Then, the following assertions hold true:
\begin{itemize}
\item[(i)] There exists $\tM_\mi^\rmc:\NN^3\times(0,\infty)\times(0,\eps_0)\to(0,\infty)$  such that
\begin{equation}\label{Error-cos-tilde-mid}
\Big\|\int_{\tGamma_{a,\mi}^\eps}\frac{e^{\lambda t}}{\lambda^{2n-1}}Q_iR_A^\infty(\lambda)A^nx\rmd\lambda\Big\|\leq M_\mi^\rmc(n,i,j,a,\eps)e^{(\sup_{s\in\RR}g_*(s)+\eps)t}\|A^{n-1}x\|	
\end{equation}		
for any $t\geq 0$, $x\in\dom(A^n)\cap\Ker(Q_j-I_\bY)$, $i,j\in\NN$, $\eps\in(0,\teps_0)$, $a>0$;
\item[(ii)] There exists $\tM_\mi^\rms:\NN^3\times(0,\infty)\times(0,\teps_0)\to(0,\infty)$  such that
\begin{equation}\label{Error-sine-tilde-mid}
\Big\|Q_iA\int_{\tGamma_{a,\mi}^\eps}\frac{e^{\lambda t}}{\lambda^{2n}}Q_{i+1}R_A^\infty(\lambda)A^{n-1}w\rmd\lambda\Big\|\leq \tM_\mi^\rms(n,i,j,a,\eps)e^{(\sup_{s\in\RR}g_*(s)+\eps)t}\|A^{n-1}w\|
\end{equation}		
for any $t\geq 0$, $x\in\bW_{n-1}\cap\Ker(Q_j-I_\bY)$, $i,j\in\NN$, $\eps\in(0,\teps_0)$, $a>0$.	
\end{itemize}
\end{lemma}
\begin{proof} Fix $t\geq 0$, $i,j\in\NN$, $x\in\dom(A^n)\cap\Ker(Q_j-I_\bY)$, $w\in\bW_{n-1}\cap\Ker(Q_j-I_\bY)$, $\eps\in(0,\teps_0)$. From Lemma~\ref{l4.10}(iii) we infer that $A^nx,A^{n-1}w\in\Ker(Q_j-I_\bY)$, hence $A^nx=Q_jA^nx$ and $A^{n-1}w=Q_jA^{n-1}w$. From Lemma~\ref{l4.10}(i) it follows that
\begin{equation}\label{l4.20.1}	
\frac{e^{\lambda t}}{\lambda^{2n-1}}Q_iR_A^\infty(\lambda)A^nx=\frac{e^{\lambda t}}{\lambda^{2n-3}}Q_iR_A^\infty(\lambda)Q_jA^{n-1}x-\frac{e^{\lambda t}}{\lambda^{2n-1}}Q_iA^{n-1}x	
\end{equation}	
\begin{equation}\label{l4.20.2}
\frac{e^{\lambda t}}{\lambda^{2n}}Q_i\tA R_A^\infty(\lambda)Q_jA^{n-1}w
=\frac{e^{\lambda t}}{\lambda^{2n-2}}Q_i R_A^\infty(\lambda)Q_jA^{n-1}w-\frac{e^{\lambda t}}{\lambda^{2n}}Q_iA^{n-1}w
\end{equation}
for any $\lambda\in\Omega_*\setminus\bigl(\cM\cup\cN\cup\{0\}\bigr)$. Assertion (i) follows from \eqref{l4.20.1} for
\begin{align}\label{l4.20.3}
\tM_\mi^\rmc(n,i,j,a,\eps)&=\Bigl(\max\{\tM_{1,\mi}^\rmc(n,a,\eps),M_{2,\mi}^\rmc(n,i,a,\eps)\}\Bigr)\mathrm{Lenght}(\tGamma_{a,\mi}^\eps),\;\mbox{where}\nonumber\\
\tM_{1,\mi}^\rmc(n,i,j,a,\eps)&:=\sup_{\lambda\in\mathrm{Range}(\tGamma_{a,\mi}^\eps)}\frac{\|Q_iR_A^\infty(\lambda)Q_j\|}{|\lambda|^{2n-3}},\;\tM_{2,\mi}^\rmc(n,i,a,\eps):=\sup_{\lambda\in\mathrm{Range}(\tGamma_{a,\mi}^\eps)}\frac{q_*}{|\lambda|^{2n-1}},
\end{align}	
where $q_*$ is defined in \eqref{l4.2.2}. Since the linear operator $A$ is closed and $Q_i\in\cC\cL(\bY,\bX)$, from \cite[Proposition 1.1.7]{ABHN} and ($Q_7$) we infer that
\begin{equation}\label{l4.20.4}
Q_iA\int_{\tGamma_{a,\mi}^\eps}\frac{e^{\lambda t}}{\lambda^{2n}}Q_{i+1}R_A^\infty(\lambda)A^{n-1}w\rmd\lambda=\int_{\tGamma_{a,\mi}^\eps}\frac{e^{\lambda t}}{\lambda^{2n}}Q_i\tA R_A^\infty(\lambda)A^{n-1}w\rmd\lambda.	
\end{equation}	
Assertion (ii) follows from \eqref{l4.20.2} and \eqref{l4.20.4} for
\begin{align}\label{t3.13.4-bis}
\tM_\mi^\rms(n,i,j,a,\eps)&=\Bigl(\max\{\tM_{1,\mi}^\rms(n,i,j,a,\eps),M_{2,\mi}^\rms(n,i,a,\eps)\}\Bigr)\mathrm{Lenght}(\tGamma_{a,\mi}^\eps),\;\mbox{where}\nonumber\\
\tM_{1,\mi}^\rms(n,i,j,a,\eps)&:=\sup_{\lambda\in\mathrm{Range}(\tGamma_{a,\mi}^\eps)}\frac{\|Q_iR_A^\infty(\lambda)Q_j\|}{|\lambda|^{2n-2}},\;M_{2,\mi}^\rms(n,i,a,\eps):=\sup_{\lambda\in\mathrm{Range}(\tGamma_{a,\mi}^\eps)}\frac{q_*}{|\lambda|^{2n}}.
\end{align}	
\end{proof}
Similar to Section~\ref{sec3} we define the path
\begin{equation}\label{def-tGamma-out}
\tGamma_{a,\out}^\eps: \lambda=g_*(s)+\eps+\rmi s,\; s\in(-\infty,-a]\cup [a,\infty).
\end{equation}
Our next task is to prove estimates of the integral terms in the representations given in Theorem~\ref{t4.18} and Theorem~\ref{t4.19}, when the path $\re\lambda=g_*(\im\lambda)+\eps$ is replaced by the path $\tGamma_{a,\out}^\eps$. To prove these results we introduce
\begin{equation}\label{def-overline-a-star}
\oa_*=\max\{\max_{1\leq k\leq m}|\mu_k|,\max_{1\leq\ell\leq p}|\nu_\ell|\}+1.
\end{equation}
One important assumption on the function $g_*$, introduced in Hypothesis (H-ext) is

\noindent\textbf{Hypothesis (G).}  We assume that there exist $c_*>0$, $\os_*>\oa_*$ and $\vk_*\geq 0$, such that
\begin{enumerate}
\item[(i)] $g_*(s)\leq-c_*\ln{|s|}$ whenever $|s|\geq \os_*$;
\item[(ii)] $e^{\vk_*g_*}g_*'\in L^\infty(\RR)$.
\end{enumerate}
First, we will prove two preliminary estimates needed in the sequel.
\begin{lemma}\label{l4.21}
Assume Hypotheses (H), (Q), (H-ext) and (G) and, in addition, $b:=\inf\limits_{s\in\RR}g_0(s)>-\infty$. Then,
\begin{enumerate}
\item[(i)] 	There exists $\M_{\out,1}:\NN_{2n_*-1}\times\NN^2\to(0,\infty)$ such that
\begin{equation}\label{Est-bold-M-1-out}
\Big\|\int_{\tGamma_{a,\out}^\eps}\frac{e^{\lambda t}}{\lambda^{k-2}}Q_iR_A^\infty(\lambda)Q_j\rmd\lambda\Big\|\leq	 \frac{\M_{\out,1}(k,i,j)e^{\eps(t-\beta_{i,j})}}{\big(c_*(t-\beta_{i,j}-\vk_*)-3\big)a^{c_*(t-\beta_{i,j}-\vk_*)-3}}	
\end{equation}
for any $t>\beta_{i,j}+\vk_*+3/c_*$, $i,j\in\NN$, $a\geq\os_*$,  $\eps\in(0,\teps_0)$, $k\in\NN$ with $k\geq 2n_*-1$;
\item[(ii)]	There exists $\M_{\out,2}>0$ such that
\begin{equation}\label{Est-bold-M-2-out}
\Big|\int_{\tGamma_{a,\out}^\eps}\frac{e^{\lambda t}}{\lambda^{k}}\rmd\lambda\Big|\leq	\frac{\M_{\out,2}e^{\eps t}}{\big(c_*(t-\vk_*)+k-1\big)a^{c_*(t-\vk_*)+k-1}}	
\end{equation}
for any $t>\vk_*-(k-1)/c_*$, $i,j\in\NN$, $a\geq\os_*$,  $\eps\in(0,\teps_0)$.
\end{enumerate}	
\end{lemma}
\begin{proof}
Fix $t\geq 0$, $k,i,j\in\NN$, $a\geq\os_*$, $\eps\in(0,\teps)$. First, we note that	by our assumptions on $\tih_{i,j,1}$ and $\tih_{i,j,2}$,
\begin{align}\label{l4.21.1}
\oM_*(k,i,j)&:=\Bigg(\max\Big\{\sup_{s\geq1}\frac{\tih_{i,j,1}(s)}{s^{k}},\, \sup_{s\leq\omega}e^{\beta_{i,j}s}\tih_{i,j,2}(s)\Big\}\Bigg)^{\frac{1}{2}}<\infty\nonumber\\
M_{*,1}&:=\mathrm{esssup}_{s\in\RR}|e^{\vk_*g_*(s)}g_*'(s)|<\infty.	
\end{align}
From Hypotheses (H) and (H-ext) we know that
\begin{equation}\label{l4.21.2}
g_*(s)+\eps<g_*(s)+\teps\leq g_*(s)+\delta_0\leq g_0(s)\leq\omega\;\mbox{for any}\; s\in\RR.
\end{equation}
Hence, from \eqref{l4.21.1} and \eqref{l4.21.2} we obtain the following estimate
\begin{align}\label{l4.21.3}
&\Big\|\int_{\tGamma_{a,\out}^\eps}\frac{e^{\lambda t}}{\lambda^{k-2}}Q_iR_A^\infty(\lambda)Q_j\rmd\lambda\Big\|\nonumber\\
&\leq\int_{|s|\geq a}\frac{e^{(g_*(s)+\eps)t}}{|g_*(s)+\eps+\rmi s|^{k-2}}\tih_{i,j,1}(|g_*(s)+\eps+\rmi s|)\tih_{i,j,2}(g_*(s)+\eps)|g_*'(s)+\rmi|\rmd s\nonumber\\
&\leq\oM_*(k,i,j)e^{\vk_*\eps}\int_{|s|\geq a}|g_*(s)+\eps+\rmi s|^2e^{(t-\beta_{i,j}-\vk_*)(g_*(s)+\eps)}\big|e^{\vk_*g_*(s)}g_*'(s)+\rmi e^{\vk_*g_*(s)}\big|\rmd s\nonumber\\
&\leq\oM_*(k,i,j)(M_{*,1}+e^{\vk_*\omega})e^{\vk_*\eps}\int_{|s|\geq a}|g_*(s)+\eps+\rmi s|^2e^{(t-\beta_{i,j}-\vk_*)(g_*(s)+\eps)}\rmd s.
\end{align}
Next, we estimate
\begin{align}\label{l4.21.4}
\int_{|s|\geq a}&s^2e^{(t-\beta_{i,j}-\vk_*)(g_*(s)+\eps)}\rmd s\leq e^{\eps(t-\beta_{i,j}-\vk_*)}\int_{|s|\geq a}s^2e^{-c_*(t-\beta_{i,j}-\vk_*)\ln{|s|}}\rmd s\nonumber\\
&=2e^{\eps(t-\beta_{i,j}-\vk_*)}\int_{a}^{\infty}\frac{\rmd s}{s^{c_*(t-\beta_{i,j}-\vk_*)-2}}=\frac{2e^{\eps(t-\beta_{i,j}-\vk_*)}}{\big(c_*(t-\beta_{i,j}-\vk_*)-3\big)a^{c_*(t-\beta_{i,j}-\vk_*)-3}}.	
\end{align}
for any $t>\beta_{i,j}+\vk_*+3/c_*$. Next, we define $M_{*,2}=\sup_{\xi\leq\omega}\xi^2e^{\frac{\xi}{c_*}}<\infty$. From \eqref{l4.21.2} it follows that
\begin{align}\label{l4.21.5}
&\int_{|s|\geq a}(g_*(s)+\eps)^2e^{(t-\beta_{i,j}-\vk_*)(g_*(s)+\eps)}\rmd s\leq \int_{|s|\geq a}(g_*(s)+\eps)^2e^{\frac{g_*(s)+\eps}{c_*}}e^{(t-\beta_{i,j}-\vk_*-1/c_*)(g_*(s)+\eps)}\rmd s\nonumber\\
&\leq M_{*,2}\int_{|s|\geq a}e^{(t-\beta_{i,j}-\vk_*-1/c_*)(g_*(s)+\eps)}\rmd s\leq M_{*,2}e^{\eps(t-\beta_{i,j}-\vk_*-1/c_*)}\int_{|s|\geq a}e^{-c_*(t-\beta_{i,j}-\vk_*-1/c_*)\ln{|s|}}\rmd s\nonumber\\
&\leq2M_{*,2}e^{\eps(t-\beta_{i,j}-\vk_*)}\int_{a}^{\infty}\frac{\rmd s}{s^{c_*(t-\beta_{i,j}-\vk_*)-1}}=\frac{2M_{*,2}e^{\eps(t-\beta_{i,j}-\vk_*)}}{\big(c_*(t-\beta_{i,j}-\vk_*)-2\big)a^{c_*(t-\beta_{i,j}-\vk_*)-2}}.	
\end{align}
for any $t>\beta_{i,j}+\vk_*+2/c_*$. From \eqref{l4.21.3}, \eqref{l4.21.4} and \eqref{l4.21.5} we conclude that
\begin{equation}\label{l4.21.6}
\Big\| \int_{\tGamma_{a,\out}^\eps}\frac{e^{\lambda t}}{\lambda^{k-2}}Q_iR_A^\infty(\lambda)Q_j\rmd\lambda\Big\|\leq\M_{\out,1}(k,i,j)\frac{e^{\eps(t-\beta_{i,j})}}{\big(c_*(t-\beta_{i,j}-\vk_*)-3\big)a^{c_*(t-\beta_{i,j}-\vk_*)-3}}.	
\end{equation}
for any $t>\beta_{i,j}+\vk_*+3/c_*$, where $\M_{\out,1}(k,i,j):=2\oM_*(k,i,j)(M_{*,1}+e^{\vk_*\omega})\max\{1,M_{*,2}/\os_*\}$, proving assertion (i). Similarly,
\begin{align}\label{l4.21.7}
\Big|\int_{\tGamma_{a,\out}^\eps}&\frac{e^{\lambda t}}{\lambda^k}\rmd\lambda\Big|\leq\int_{|s|\geq a}\frac{e^{(g_*(s)+\eps)t}|g_*'(s)+\rmi|}{|g_*(s)+\eps+\rmi s|^k}\rmd s\leq e^{\eps t}\int_{|s|\geq a}\frac{e^{(t-\vk_*)g_*(s)}}{|s|^k}\big|e^{\vk_*g_*(s)}g_*'(s)+\rmi e^{\vk_*g_*(s)}\big|\rmd s\nonumber\\
&\leq(M_{*,1}+e^{\vk_*\omega})e^{\eps t}\int_{|s|\geq a}\frac{e^{-c_*(t-\vk_*)\ln{|s|}}}{|s|^k}\rmd s=(M_{*,1}+e^{\vk_*\omega})e^{\eps t}\int_{|s|\geq a}\frac{\rmd s}{|s|^{c_*(t-\vk_*)+k}}.
\end{align}	
for any $t> \vk_*-(k-1)/c_*$. Assertion (ii) follows shortly from \eqref{l4.21.7}.		
\end{proof}
\begin{lemma}\label{l4.22}
Assume Hypotheses (H), (Q), (H-ext), (G) and, in addition,   $b:=\inf\limits_{s\in\RR}g_0(s)>-\infty$. Then, there exists $\tM^\rmc_\out:\NN_{\max\{n_0,n_*\}}\times\NN^2\to(0,\infty)$ such that
\begin{equation}\label{Error-cos-out-bis2}
\Big\|\int_{\tGamma_{a,\out}^\eps}\frac{e^{\lambda t}}{\lambda^{2n-1}}Q_iR_A^\infty(\lambda)A^nx\rmd\lambda\Big\|\leq	 \frac{\tM^\rmc_\out(n,i,j)e^{\eps(t-\beta_{i,j})}\|A^{n-1}x\|}{\big(c_*(t-\beta_{i,j}-\vk_*)-3\big)a^{c_*(t-\beta_{i,j}-\vk_*)-3}}	
\end{equation}
for any $t>\beta_{i,j}+\vk_*+3/c_*$, $i,j\in\NN$, $a\geq\os_*$, $x\in\dom(A^n)\cap\Ker(Q_j-I_\bY)$, $\eps\in(0,\teps_0)$, $n\in\NN$, with $n\geq\max\{n_0,n_*\}$.
\end{lemma}	
\begin{proof}
Fix $t\geq 0$, $i,j\in\NN$, $a\geq\os_*$, $x\in\dom(A^n)\cap\Ker(Q_j-I_\bY)$, $\eps\in(0,\teps_0)$.
From \eqref{l4.20.1} it follows that
\begin{equation}\label{l4.22.1}	
\int_{\tGamma_{a,\out}^\eps}\frac{e^{\lambda t}}{\lambda^{2n-1}}Q_iR_A^\infty(\lambda)A^nx\rmd\lambda=\int_{\tGamma_{a,\out}^\eps}\frac{e^{\lambda t}}{\lambda^{2n-3}}Q_iR_A^\infty(\lambda)Q_jA^{n-1}x\rmd\lambda-\int_{\tGamma_{a,\out}^\eps}\frac{e^{\lambda t}}{\lambda^{2n-1}}\rmd\lambda Q_iA^{n-1}x.
\end{equation}
We define $\tM_\out^\rmc:\NN_{\max\{n_0,n_*\}}\times\NN^2\to(0,\infty)$ by
\begin{equation}\label{l4.22.2}
\tM_\out^\rmc(n,i,j)=\M_{\out,1}(2n-1,i,j)+q_*\M_{\out,2}e^{\teps_0\beta_{i,j}},
\end{equation}
where $q_*$ is defined in \eqref{l4.2.2}. The lemma follows shortly from Lemma~\ref{l4.21} (with $k=2n-1$), \eqref{l4.22.1} and \eqref{l4.22.2}.
\end{proof}
Next, we prove an estimate of the integral term of the sine function representation given in Theorem~\ref{t4.19} by relaxing slightly the hypotheses of Lemma~\ref{l4.22}.
\begin{lemma}\label{l4.23}
Assume Hypotheses (H), (Q), (H-ext), (G) and, in addition,  $b:=\inf\limits_{s\in\RR}g_0(s)>-\infty$. Then, there exists $\tM^\rms_\out:\NN_{\max\{n_0,n_*\}}\times\NN^2\to(0,\infty)$ such that
\begin{equation}\label{Error-sin-out-bis}
\Big\|Q_iA\int_{\tGamma_{a,\out}^\eps}\frac{e^{\lambda t}}{\lambda^{2n-1}}Q_{i+1}R_A^\infty(\lambda)A^{n-1}w\rmd\lambda\Big\|\leq	 \frac{\tM^\rms_\out(n,i,j)e^{\eps(t-\beta_{i,j})}\|A^{n-1}w\|}{\big(c_*(t-\beta_{i,j}-\vk_*)-3\big)a^{c_*(t-\beta_{i,j}-\vk_*)-3}}	
\end{equation}
for any $t>\beta_{i,j}+\vk_*+3/c_*$, $i,j\in\NN$, $a\geq\os_*$, $x\in\bW_{n-1}\cap\Ker(Q_j-I_\bY)$, $\eps\in(0,\teps_0)$, $n\in\NN$, with $n\geq\max\{n_0,n_*\}$.
\end{lemma}	
\begin{proof}
Fix $t\geq 0$, $i,j\in\NN$, $a\geq\os_*$, $w\in\bW_{n-1}\cap\Ker(Q_j-I_\bY)$, $\eps\in(0,\teps_0)$.
Since the linear operator $A$ is closed and $Q_i\in\cC\cL(\bY,\bX)$, from \cite[Proposition 1.1.7]{ABHN}, Theorem~\ref{t4.19}(i), ($Q_7$) and \eqref{l4.20.2} it follows that
\begin{align}\label{l4.23.1}
Q_iA&\int_{\tGamma_{a,\out}^\eps}\frac{e^{\lambda t}}{\lambda^{2n}}Q_{i+1}R_A^\infty(\lambda)A^{n-1}w\rmd\lambda=\int_{\tGamma_{a,\out}^\eps}\frac{e^{\lambda t}}{\lambda^{2n}}Q_i\tA R_A^\infty(\lambda)A^{n-1}w\rmd\lambda\nonumber\\
&=\int_{\tGamma_{a,\out}^\eps}\frac{e^{\lambda t}}{\lambda^{2n-2}}Q_iR_A^\infty(\lambda)Q_jA^{n-1}w\rmd\lambda-\int_{\tGamma_{a,\out}^\eps}\frac{e^{\lambda t}}{\lambda^{2n}}\rmd\lambda Q_iA^{n-1}w.	 
\end{align}	
Let $\tM_\out^\rmc:\NN_{\max\{n_0,n_*\}}\times\NN^2\to(0,\infty)$ be the function defined by
\begin{equation}\label{l4.23.2}
\tM_\out^\rmc(n,i,j)=\M_{\out,1}(2n,i,j)+q_*\M_{\out,2}e^{\teps\beta_{i,j}},
\end{equation}
The lemma follows shortly from Lemma~\ref{l4.21} (with $k=2n$), \eqref{l4.23.1} and \eqref{l4.23.2}.
\end{proof}
\begin{theorem}\label{t4.24}
Assume Hypotheses (H), (Q), (H-ext), (G) and, in addition,   $b:=\inf\limits_{s\in\RR}g_0(s)>-\infty$.
Then, there exists $\tM^\rmc:\NN_{\max\{n_0,n_*\}}\times\NN^2\times(0,\eps_0)\to(0,\infty)$  and $\tcE^\rmc_{n,i}:\RR\to\cC\cL(\dom(A^{n-1})\cap\bX_\infty,\bY)$ such that
\begin{align}\label{Frechet-Error-cos}
Q_iC(t)x&=\displaystyle\sum\limits_{k=1}^{m} \sum_{\na=1}^{N_k}p_{1,k,\na}(t)e^{\mu_k t}Q_i(A-\mu_k^2 I_\bX)^{\na-1}P_kx+\displaystyle\sum\limits_{\nu_\ell\ne0} \sum_{\na=1}^{\tN_\ell}\tp_{1,\ell,\na}(t)e^{\nu_\ell t}Q_i(\tA-\nu_\ell^2 I_\bY)^{\na-1}\tP_\ell x \nonumber\\
&+\chi_{\{\nu_\ell:\ell=1,\dots,p\}}(0)\Bigg(\sum_{\na=1}^{\oN_\ell}\frac{t^{2\na-2}}{(2\na-2)!}Q_i\tA^{\na-1}\tP_\ell x+\sum_{\na=1}^{\oN_\ell}\frac{t^{2\na-1}}{(2\na-1)!}Q_i\tA^{\na-1}\oP_\ell x\Bigg)\nonumber\\
&+\chi_{[0,\infty)}(g_*(0))\displaystyle\sum\limits_{\na=0}^{n-1}\frac{t^{2\na}}{(2\na)!}Q_iA^\na x+\tcE^\rmc_{n,i}(t)x
\end{align}
\begin{equation}\label{Frechet-error-cos2}
\|Q_i\tcE^\rmc_{n,i}(t)x\|\leq 	\tM^\rmc(n,i,j,\eps)e^{(\sup_{s\in\RR}g_*(s)+\eps)t}\|A^{n-1}x\|
\end{equation}	
for any $t\geq\beta_{i,j}+\vk_*+4/c_*$, $i,j\in\NN$, $x\in\dom(A^n)\cap\Ker(Q_j-I_\bY)$, $\eps\in(0,\teps_0)$, $n\in\NN$, with $n\geq\max\{n_0,n_*\}$. If, in addition we assume that $Q_i\big(\dom(A^n)\big)\subseteq\dom(A^n)$ for any $i\in\NN$, then \eqref{Frechet-error-cos2} holds for any $x\in\dom(A^{n-1})$.
\end{theorem}	
\begin{proof}
Fix  $t\geq\beta_{i,j}+\vk_*+4/c_*$, $i,j\in\NN$ and $\eps\in(0,\teps_0)$. We define the operator-valued function $\tcE_{n,i,j}$ by solving equation \eqref{Frechet-Error-cos}. From Lemma~\ref{l4.11} we infer that $\tcE_{n,i,j}$ is well-defined. Next, we note that the function
\begin{equation}\label{t4.24.1}
t\to\frac{\big(\sup_{s\in\RR}g_*(s)\big)t}{c_*(t-\beta_{i,j}-\vk_*)}:[\beta_{i,j}+\vk_*+4/c_*,\infty)\to\RR\;\mbox{is bounded}.	
\end{equation}	
Hence, the exists $a_{i,j}>\os_*$ large enough such that
\begin{equation}\label{t4.24.2}
-c_*(t-\beta_{i,j}-\vk_*)\ln{a_{i,j}}<\big(\sup_{s\in\RR}g_*(s)\big)t\;\mbox{for any}\; t\geq\beta_{i,j}+\vk_*+4/c_*.
\end{equation}
From Lemma~\ref{l4.20}(i), Lemma~\ref{l4.22} and \eqref{t4.24.2} we obtain that
\begin{equation}\label{t4.24.3}
\|Q_i\tcE^\rmc_{n,i}(t)x\|\leq 	\tM^\rmc(n,i,j,\eps)e^{(\sup_{s\in\RR}g_*(s)+\eps)t}\|A^{n-1}x\|
\end{equation}	
for any $t\geq\beta_{i,j}+\vk_*+4/c_*$, $x\in\dom(A^n)\cap\Ker(Q_j-I_\bY)$, where the function $\tM^\rmc:\NN_{\max\{n_0,n_*\}}\times\NN^2\times(0,\eps_0)\to(0,\infty)$ is defined by
\begin{equation}\label{t4.24.4}
\tM^\rmc(n,i,j,\eps)=\tM_\mi^\rmc(n,i,j,a_{i,j},\eps)+\tM_\out^\rmc(n,i,j,\eps)	
\end{equation}
From \eqref{t4.24.3} we immediately infer \eqref{Frechet-error-cos2}. If, in addition, $Q_i\big(\dom(A^n)\big)\subseteq\dom(A^n)$ for any $i\in\NN$ from Lemma~\ref{l4.2}(x) we have $\dom(A^{n-1})\cap\Ker(Q_j-I_\bY)$ is contained in the closure of $\dom(A^n)\cap\Ker(Q_{j+1}-I_\bY)$ in the $\|\cdot\|_{\dom(A^{n-1})}$ norm. Replacing $j$ by $j+1$ in \eqref{Frechet-error-cos2}, the lemma follows since both sides of the inequality are continuous functions in the $\|\cdot\|_{\dom(A^{n-1})}$ norm.
\end{proof}
\begin{theorem}\label{t4.25}
Assume Hypotheses (H), (Q), (H-ext), (G) and, in addition,   $b:=\inf\limits_{s\in\RR}g_0(s)>-\infty$.
Then, there exists $\tM^\rmc:\NN_{\max\{n_0,n_*\}}\times\NN^2\times(0,\eps_0)\to(0,\infty)$  and $\tcE^\rmc_{n,i}:\RR\to\cC\cL(\dom(A^{n-1})\cap\bX_\infty,\bY)$ such that
\begin{align}\label{Frechet-Error-sin}
Q_iS(t)w&=\displaystyle\sum\limits_{k=1}^{m} \sum_{\na=1}^{N_k}p_{0,k,\na}(t)e^{\mu_k t}Q_i(A-\mu_k^2 I_\bX)^{\na-1}P_kw+\displaystyle\sum\limits_{\nu_\ell\ne0} \sum_{\na=1}^{\tN_\ell}\tp_{0,\ell,\na}(t)e^{\nu_\ell t}Q_i(\tA-\nu_\ell^2 I_\bY)^{\na-1}\tP_\ell w \nonumber\\
&+\chi_{\{\nu_\ell:\ell=1,\dots,p\}}(0)\Bigg(\sum_{j\na=1}^{\oN_\ell}\frac{t^{2\na-1}}{(2\na-1)!}Q_i\tA^{\na-1}\tP_\ell w+\sum_{\na=1}^{\oN_\ell}\frac{t^{2\na}}{(2\na)!}Q_i\tA^{\na-1}\oP_\ell w\Bigg)\nonumber\\
&+\chi_{[0,\infty)}(g_*(0))\displaystyle\sum\limits_{\na=0}^{n-1}\frac{t^{2\na+1}}{(2\na+1)!}Q_iA^\na w+
\tcE^\rms_{n,i}(t)w
\end{align}
\begin{equation}\label{Frechet-error-sin2}
\|Q_i\tcE^\rms_{n,i}(t)w\|\leq \tM^\rmc(n,i,j,\eps)e^{(\sup_{s\in\RR}g_*(s)+\eps)t}\|A^{n-1}w\|
\end{equation}	
for any $t\geq\beta_{i,j}+\vk_*+4/c_*$, $i,j\in\NN$, $x\in\bW_{n-1}\cap\Ker(Q_j-I_\bY)$, $\eps\in(0,\teps_0)$, $n\in\NN$, with $n\geq\max\{n_0,n_*\}$.
\end{theorem}
\begin{proof}
The theorem follows from Theorem~\ref{t4.19}, Lemma~\ref{l4.20}, Lemma~\ref{l4.23} and \eqref{t4.24.2}.
\end{proof}	
Next, we look for estimates similar to those obtained in Lemma~\ref{l4.22} and Lemma~\ref{l4.23} in the case when Hypothesis (G) is replaced by $g_*,g_*'\in L^\infty(\RR)$. First, we look for the estimate of integral term of the cosine function representation given in Theorem~\ref{t4.18}.
\begin{lemma}\label{l4.26}
Assume Hypotheses (H), (Q) and (H-ext) and, in addition,  $b:=\inf\limits_{s\in\RR}g_0(s)>-\infty$, $b_*=\inf_{s\in\RR}g_*(s)>-\infty$, $g_*'\in L^\infty(\RR)$ and $\displaystyle\int_1^\infty\frac{\tih_{i,j,1}(s)\rmd s}{s^{2n_*-1}}<\infty$. Then, there exists $L^\rmc_\out:\NN_{\max\{n_0,n_*\}}\times\NN^2\to(0,\infty)$ and $\ts_*>\oa_*$ such that	
\begin{equation}\label{Frechet-Est-cos-3}
\Bigg\|\int_{\tGamma_{a,\out}^\eps}\frac{e^{\lambda t}}{\lambda^{2n-1}}R_A^\infty(\lambda)A^nx\rmd\lambda\Bigg\|\leq L^\rmc_\out(n,i,j)	e^{(\sup_{s\in\RR}g_*(s)+\eps)t}\|A^nx\|	
\end{equation}
for any $t\geq 0$, $i,j\in\NN$, $a\geq\ts_*$, $x\in\dom(A^n)\cap\Ker(Q_j-I_\bY)$, $\eps\in(0,\teps_0)$, $n\in\NN$, with $n\geq\max\{n_0,n_*\}$.
\end{lemma}	
\begin{proof}
Fix $t\geq 0$, $i,j\in\NN$, $a\geq\oa_*$, $x\in\dom(A^n)\cap\Ker(Q_j-I_\bY)$, $\eps\in(0,\teps_0)$. From Lemma~\ref{l4.10}(iii) we have $A^nx\in\Ker(Q_j-I_\bY)$, which implies that $A^nx=Q_jA^nx$. From \eqref{l4.21.1} and \eqref{l4.21.2} it follows that
\begin{align}\label{l4.26.1}
&\Big\|\int_{\tGamma_{a,\out}^\eps}\frac{e^{\lambda t}}{\lambda^{2n-1}}Q_iR_A^\infty(\lambda)A^nx\rmd\lambda\Big\|=\Big\|\int_{\tGamma_{a,\out}^\eps}\frac{e^{\lambda t}}{\lambda^{2n-1}}Q_iR_A^\infty(\lambda)Q_jA^nx\rmd\lambda\Big\|\nonumber\\
&\leq\|A^nx\|\int_{|s|\geq a}\frac{e^{(g_*(s)+\eps)t}}{|g_*(s)+\eps+\rmi s|^{2n-1}}\tih_{i,j,1}(|g_*(s)+\eps+\rmi s|)\tih_{i,j,2}(g_*(s)+\eps)|g_*'(s)+\rmi|\rmd s\nonumber\\
&\leq\oM_*(2n-1,i,j)e^{-\beta_{i,j}b_*}(\|g_*\|_\infty+1)e^{(\sup_{s\in\RR}g_*(s)+\eps)t}\int_{|s|\geq a}
\frac{\tih_{i,j,1}(|g_*(s)+\eps+\rmi s|)}{|g_*(s)+\eps+\rmi s|^{2n-1}}\rmd s.
\end{align}
Next, we introduce
\begin{equation}\label{l4.26.2}
\ts_*=\max\big\{2(\|g_*\|_\infty+\teps_0)\|g_*'\|_\infty,\|g_*\|_\infty+\teps_0\big\}+\oa_*
\end{equation}
Replacing $g_0$ by $g_*$ in Remark~\ref{r3.14} we infer that
\begin{equation}\label{l4.26.3}
\int_{|s|\geq a}
\frac{\tih_{i,j,1}(|g_*(s)+\eps+\rmi s|)}{|g_*(s)+\eps+\rmi s|^{2n-1}}\rmd s\leq 8\int_1^\infty\frac{h_{i,j,1}(s)\rmd s}{s^{2n-1}}.	
\end{equation}	
Let $L^\rmc_\out:\NN_{\max\{n_0,n_*\}}\times\NN^2\to(0,\infty)$ be the function defined by
\begin{equation}\label{l4.26.4}
L^\rmc_\out(n,i,j)=8\oM_*(2n-1,i,j)e^{-\beta_{i,j}b_*}(\|g_*\|_\infty+1)\int_1^\infty\frac{h_{i,j,1}(s)\rmd s}{s^{2n-1}}.		
\end{equation}	
The lemma follows shortly from \eqref{l4.26.1}, \eqref{l4.26.3} and \eqref{l4.26.4}.
\end{proof}
To prove an estimate similar to Lemma~\ref{l4.26} for the integral term of the sine function given in Theorem~\ref{t4.19} we need to prove an estimate similar to the one given in Lemma~\ref{l3.7}. Since the operator-valued function $R_A^\infty(\cdot)$ does not satisfy the Hilbert resolvent identity, we need to look for a different extension of the operator-valued function $R_A(\cdot)$.
We define
\begin{equation}\label{def-X-mu}
\bX_\mu=R(\mu^2,A)\bX_\infty,\;\mbox{for}\; \mu\in\CC_\omega^+.
\end{equation}
In what follows we denote by $\CC^{\rmu\rml}_a=\{\lambda\in\CC:|\im\lambda|>a\}$ and recall the definition of $\oa_*$ in \eqref{def-overline-a-star}. In the remaining part of this subsection we make the following assumption:

\noindent\textbf{Hypothesis (H-add).}  We assume that there exists $\tbX_\infty$, a subspace of $\bX$, $\gamma_*>\omega$, such that
\begin{enumerate}
\item[(i)] $\bX_\infty\subset\tbX_\infty$ and $\bX_\mu\subset\tbX_\infty$ whenever $\re\mu=\gamma_*$;
\item[(ii)] There exists $\ta_*>\oa_*$ such that the operator-valued function $R_A(\cdot)_{|\tbX_\infty}:\Omega_0\setminus\cM\to\cC\cL(\tbX_\infty,\bY)$ has  an analytic extension to $\Omega_*\cap\CC^{\rmu\rml}_{\ta_*}$ denoted $\tR_A^\infty(\cdot)$;
\item[(iii)] For any $i\in\NN$ there exist two piecewise continuous functions $\tih_{i,3},\tih_{i,4}:[0,\infty)\to[0,\infty)$, such that
\begin{equation}\label{tRA-ext-h3-h4-est}
\|Q_i\tR_A^\infty(\lambda)\|\leq \tih_{i,3}(|\lambda|)\tih_{i,4}(\re\lambda)\;\mbox{for any}\;\lambda\in\Omega_*\cap\CC^{\rmu\rml}_{\ta_*}\setminus\Omega_0.	
\end{equation}
\item[(iv)]
$\lim\limits_{s\to\infty}\frac{h_1(s)}{s^{2n_*-1}}=0$, where $n_*$ is defined in \eqref{limit-n-star}. Moreover, the function $\tih_{i,2}$ is bounded on any compact interval.
\end{enumerate}
\begin{remark}\label{r4.27}
Since $\Omega_*\cap\{\lambda\in\CC:\pm\im\lambda>\ta_*\}$ is an open and connected subset of $\CC$, from the uniqueness property of analytic extensions we conclude that
\begin{equation}\label{r4.27.1}
\tR_A^\infty(\lambda)_{|\bX_\infty}=R_A^\infty(\lambda)\;\mbox{for any}\;\lambda\in\Omega_*\cap\CC^{\rmu\rml}_{\ta_*}.	
\end{equation}	
From Lemma~\ref{l4.2} and \eqref{r4.27.1} it follows that
\begin{equation}\label{r4.27.2}
\tR_A^\infty(\lambda)Q_j=R_A^\infty(\lambda)Q_j\;\mbox{for any}\;\lambda\in\Omega_*\cap\CC^{\rmu\rml}_{\ta_*},\,j\in\NN.	
\end{equation}	
\end{remark}
\begin{lemma}\label{l4.28}
Assume Hypotheses (H), (Q), (H-ext) and (H-add). Then, for any $x\in\bX_\infty$, $\lambda\in\Omega_*\cap\CC^{\rmu\rml}_{\ta_*}$,  $\mu\in\CC$ with $\re\mu=\gamma_*$ the following assertions hold true:
\begin{enumerate}
\item[(i)] $R_A^\infty(\lambda)x-R_A(\lambda)x=(\mu^2-\lambda^2)\tR_A^\infty(\lambda)R_A(\mu)x$;
\item[(ii)] $\Range(AR_A(\mu)_{|\bX_\infty})\subseteq\tbX_\infty$;
\item[(iii)]	$\tA R_A^\infty(\lambda)x-AR_A(\lambda)x=(\mu^2-\lambda^2)\tR_A^\infty(\lambda)AR_A(\mu)x$.
\end{enumerate}	
\end{lemma}
\begin{proof}
Fix $x\in\bX_\infty$ and $\mu\in\CC$ such that $\re\mu=\gamma_*$. From the resolvent equation we know that
\begin{equation}\label{l4.28.1}
R_A(\lambda)x-R_A(\mu)x=(\mu^2-\lambda^2)R_A(\lambda)R_A(\mu)x\;\mbox{for any}\;\lambda\in\Omega_0\setminus\cM.
\end{equation}
Using the uniqueness property of analytic extensions to open connected sets, assertion (i) follows from Hypotheses (H-ext) and (H-add) and \eqref{l4.28.1}.

From Lemma~\ref{l3.2}(ii), Hypothesis (H-add)(i) we immediately infer that
\begin{equation}\label{l4.28.2}
AR_A(\mu)x=\mu^2R_A(\mu)x-x\in\bX_\mu-\bX_\infty\subseteq\tbX_\infty,
\end{equation}
proving assertion (ii). From (i), Lemma~\ref{l3.2}(ii), Lemma~\ref{l4.10}(ii) and \eqref{r4.27.1} it follows that
\begin{align}\label{l4.28.3}
\tA R_A^\infty(\lambda)x&-AR_A(\mu)x=\lambda^2R_A^\infty(\lambda)x-\mu^2R_A(\mu)x=\mu^2\big(R_A^\infty(\lambda)x-R_A(\mu)x\big)+(\lambda^2-\mu^2)R_A^\infty(\lambda)x\nonumber\\&=\mu^2(\mu^2-\lambda^2)\tR_A^\infty(\lambda)x+(\lambda^2-\mu^2)\tR_A^\infty(\lambda)x=(\mu^2-\lambda^2)\tR_A^\infty(\lambda)\big(\mu^2R_A(\mu)x-x\big)\nonumber\\
&=(\mu^2-\lambda^2)\tR_A^\infty(\lambda)AR_A(\mu)x,
\end{align}
proving assertion (iii).
\end{proof}
In the next lemma we prove an estimate crucial in obtaining the error estimate for the sine function.
\begin{lemma}\label{l4.29}
Assume Hypotheses (H), (Q), (H-ext) and (H-add) and $b_*=\inf_{s\in\RR}g_*(s)>-\infty$. Then,
\begin{equation}\label{est-Q-tildeA-R-infty}
\|Q_i\tA R_A^\infty(\lambda)w\|\leq\frac{M_1}{\gamma_*-\omega}\Bigl(q_*+(\gamma_*-b_*)\bigl(2|\lambda|+\gamma_*-b_*\bigr)\tih_{i,3}(|\lambda|)h_{i,4}(\re\lambda)\Bigr)\|w\|_\bW
\end{equation}
for any $\lambda\in\Omega_*\cap\CC_{\ta_*}^{\rmu\rml}\setminus\Omega_0.$, $w\in\bW$, $i\in\NN$. We recall the definition of $M_1$ in Lemma~\ref{l2.3} and of $q_*$ in \eqref{l4.2.2}.		
\end{lemma}	
\begin{proof}
Fix $\lambda\in\Omega_*\cap\CC_{\ta_*}^{\rmu\rml}\setminus\Omega_0$, $w\in\bW$, $i\in\NN$ and let $\mu=\gamma_*+\rmi\im\lambda$. Since $\re\lambda\in[b_*,\omega]$ one can readily check that
\begin{equation}\label{l4.29.1}
|\lambda+\mu|=|2\lambda+\gamma_*-\re\lambda|\leq 2|\lambda|+\gamma_*-\re\lambda\leq 2|\lambda|+\gamma_*-b_*.
\end{equation}
From Lemma~\ref{l4.28}(iii), \eqref{r2.6.3}, \eqref{l4.2.2} and \eqref{l4.29.1} we conclude that
\begin{align}\label{l4.29.2}
\|Q_i\tA R_A^\infty(\lambda)w\|&\leq\|Q_iAR_A(\mu)w\|+(\re\mu-\re\lambda)|\lambda+\mu|\,\|Q_i\tA \tR_A^\infty(\lambda)AR_A(\mu)w\|\nonumber\\
&\leq\frac{M_1q_*}{\re\mu-\omega}\|w\|_{\bW}+(\gamma_*-b_*)(2|\lambda|+\gamma_*-b_*)\tih_{i,3}(|\lambda|)\tih_{i,4}(\re\lambda)\|AR_A(\mu)w\|\nonumber\\
&\leq\frac{M_1q_*}{\re\mu-\omega}\|w\|_{\bW}+(\gamma_*-b_*)(2|\lambda|+\gamma_*-b_*)\tih_{i,3}(|\lambda|)\tih_{i,4}(\re\lambda)\frac{M_1}{\re\mu-\omega}\|w\|_{\bW}\nonumber\\
&\leq\frac{M_1}{\gamma_*-\omega}\Bigl(q_*+(\gamma_*-b_*)\bigl(2|\lambda|+\gamma_*-b_*\bigr)\tih_{i,3}(|\lambda|)h_{i,4}(\re\lambda)\Bigr)\|w\|_\bW,
\end{align}
proving the lemma. 		
\end{proof}
\begin{lemma}\label{l4.30}
Assume Hypotheses (H), (Q), (H-ext) and (H-add),  $b:=\inf\limits_{s\in\RR}g_0(s)>-\infty$,  $b_*=\inf_{s\in\RR}g_*(s)>-\infty$, $g_*'\in L^\infty(\RR)$ and $\displaystyle\int_1^\infty\frac{\tih_{i,3}(s)\rmd s}{s^{2n_*-1}}<\infty$. Then, there exists $L^\rms_\out:\NN_{\max\{n_0,n_*\}}\times\NN\to(0,\infty)$ such that	
\begin{equation}\label{Frechet-Est-sin-3}
\Bigg\|Q_iA\int_{\tGamma_{a,\out}^\eps}\frac{e^{\lambda t}}{\lambda^{2n}}R_A^\infty(\lambda)A^nx\rmd\lambda\Bigg\|\leq L^\rms_\out(n,i)	e^{(\sup_{s\in\RR}g_*(s)+\eps)t}\|A^{n-1}w\|	
\end{equation}
for any $t\geq 0$, $i,j\in\NN$, $a\geq\ts_*$, $w\in\bW_{n-1}\cap\bX_\infty$, $\eps\in(0,\teps_0)$, $n\in\NN$, with $n\geq\max\{n_0,n_*\}$.
\end{lemma}	
\begin{proof} Fix $t\geq0$, $w\in\bW_{n-1}\cap\bX_\infty$, $i\in\NN$, $\eps\in(0,\teps_0)$. From \eqref{def-bW-n} we have $A^{n-1}w\in\bW$. Therefore, from Lemma~\ref{l4.29} we infer that
\begin{equation}\label{l4.30.1}
\|Q_i\tA R_A^\infty(\lambda)A^{n-1}w\|\leq \Big(\tc_1+\tc_2|\lambda|\,\tih_{i,3}(|\lambda|)\tih_{i,4}(\re\lambda)\Bigr)\|A^{n-1}w\|_\bW
\end{equation}	
for any $\lambda\in\Omega_*\cap\CC^{\rmu\rml}_{\ta_*}\setminus\Omega_0$, where $\tc_1=\frac{M_1}{\gamma_*-\omega}$ and $\tc_2=M_1(2+\gamma_*-b_*)$. From  Hypothesis (H-add), \eqref{r3.14.6}, \eqref{def-tGamma-out}, \eqref{l4.26.2} and \eqref{l4.30.1} it follows that
\begin{align}\label{l4.30.2}	
\Big\|&Q_i\tA\int_{\Gamma_{a,\mi}^\eps}\frac{e^{\lambda t}}{\lambda^{2n}} R_A^\infty(\lambda)A^{n-1}w\rmd\lambda\Big\|=\Big\|\int_{\Gamma_{a,\mi}^\eps}\frac{e^{\lambda t}}{\lambda^{2n}}Q_i\tA R_A^\infty(\lambda)A^{n-1}w\rmd\lambda\Big\|\nonumber\\
&\leq\int_{|s|\geq\os_0}\frac{|e^{(g_*(s)+\eps+\rmi s)t}|}{|g_*(s)+\eps+\rmi s|^{2n}}\big\|Q_i\tA R_A^\infty\big(g_*(s)+\eps+\rmi s\big)A^{n-1}w\big\||g_*'(s)+\rmi|\rmd s\nonumber\\
&\leq\|A^{n-1}w\|\big(\|g_*'\|_\infty+1\big)\Bigg(\tc_2\int_{|s|\geq\ts_*}\frac{e^{(g_*(s)+\eps)t}}{|g_*(s)+\eps+\rmi s|^{2n-1}}\tih_{i,3}\big(|g_*(s)+\eps+\rmi s|\big)\tih_{i,4}\big(g_*(s)+\eps\big)\rmd s\nonumber\\
&\qquad\qquad\qquad\qquad\qquad+\tc_1\int_{|s|\geq\ts_*}\frac{e^{(g_*(s)+\eps)t}}{|g_*(s)+\eps+\rmi s|^{2n}}\Bigg)\nonumber\\
&\leq\|A^{n-1}w\|\big(\|g_*'\|_\infty+1\big)e^{(\sup_{s\in\RR}g_*(s)+\eps)t}\Bigg(\tc_2
\Big(\sup_{b_*\leq s\leq\omega}\tih_{i,4}(s)\Big)\int_{|s|\geq\os_0}\frac{\tih_{i,3}\big(\sqrt{s^2+(g_0(s)+\eps)^2}\big)}{\big(g_0(s)+\eps+\rmi s\big)^{\frac{2n-1}{2}}}\rmd s\nonumber\\
&\qquad\qquad\qquad\qquad\qquad+2c_1\int_{|s|\geq1}\frac{\rmd s}{|s|^{2n}}\Bigg)\leq L^\rms_\out(n,i)e^{(\sup_{s\in\RR}g_*(s)+\eps)t}\|A^{n-1}w\|,	
\end{align}
where $L^\rms_\out:\NN_{\max\{n_0,n_*\}}\times\NN\to(0,\infty)$ is defined by
\begin{equation}\label{l4.30.3}
L_\out^\rms(n,i)=\big(\|g_*'\|_\infty+1\big)\Bigg(8\tc_2
\Big(\sup_{b_*\leq s\leq\omega}\tih_{i,4}(s)\Big)\int_{1}^{\infty}\frac{\tih_{i,3}(\xi)\rmd\xi}{\xi^{2n-1}}
+\frac{2\tc_1}{2n-1}\Bigg).
\end{equation}	
The lemma follows shortly from \eqref{l4.30.2} and \eqref{l4.30.3}.
\end{proof}
\begin{theorem}\label{t4.31}
Assume Hypotheses (H), (Q) and (H-ext),  $b:=\inf\limits_{s\in\RR}g_0(s)>-\infty$, $b_*=\inf_{s\in\RR}g_*(s)>-\infty$, $g_*'\in L^\infty(\RR)$ and $\displaystyle\int_1^\infty\frac{\tih_{i,j,1}(s)\rmd s}{s^{2n_*-1}}<\infty$. Then, there exists $L^\rmc:\NN_{\max\{n_0,n_*\}}\times\NN^2\times(0,\teps)\to(0,\infty)$ and $\tcE^\rmc_n:\RR\to\cC\cL(\dom(A^{n-1})\cap\bX_\infty,\bY)$  such that	
\begin{align}\label{Frechet-Error-cos3}
C(t)x&=\displaystyle\sum\limits_{k=1}^{m} \sum_{\na=1}^{N_k}p_{1,k,\na}(t)e^{\mu_k t}(A-\mu_k^2 I_\bX)^{\na-1}P_kx+\displaystyle\sum\limits_{\nu_\ell\ne0} \sum_{\na=1}^{\tN_\ell}\tp_{1,\ell,\na}(t)e^{\nu_\ell t}(\tA-\nu_\ell^2 I_\bY)^{\na-1}\tP_\ell x \nonumber\\
&+\chi_{\{\nu_\ell:\ell=1,\dots,p\}}(0)\Bigg(\sum_{\na=1}^{\oN_\ell}\frac{t^{2\na-2}}{(2\na-2)!}\tA^{\na-1}\tP_\ell x+\sum_{\na=1}^{\oN_\ell}\frac{t^{2\na-1}}{(2\na-1)!}\tA^{\na-1}\oP_\ell x\Bigg)\nonumber\\
&+\chi_{[0,\infty)}(g_*(0))\displaystyle\sum\limits_{\na=0}^{n-1}\frac{t^{2\na}}{(2\na)!}A^\na x+\hcE^\rmc_n(t)x
\end{align}
\begin{equation}\label{Frechet-error-cos4}
\|Q_i\tcE^\rmc_n(t)x\|\leq 	L^\rmc(n,i,j,\eps)e^{(\sup_{s\in\RR}g_*(s)+\eps)t}(\|A^{n-1}x\|+\|A^nx\|)
\end{equation}	
for any $t\geq0$, $i,j\in\NN$, $x\in\dom(A^n)\cap\Ker(Q_j-I_\bY)$, $\eps\in(0,\teps_0)$, $n\in\NN$, with $n\geq\max\{n_0,n_*\}$.
\end{theorem}
\begin{proof}
The theorem follows from Theorem~\ref{t4.18}, Lemma~\ref{l4.20} and Lemma~\ref{l4.26}
\end{proof}	
We note that in the representation result Theorem~\ref{t4.24} holds for trajectories of the cosine family with initial condition in a bigger space than the result given in Theorem ~\ref{t4.31}, while the result in the latter theorem holds for any $t\geq0$.
\begin{theorem}\label{t4.32}
Assume Hypotheses (H), (Q), (H-ext) and (H-add),  $b:=\inf\limits_{s\in\RR}g_0(s)>-\infty$,  $b_*=\inf_{s\in\RR}g_*(s)>-\infty$, $g_*'\in L^\infty(\RR)$ and $\displaystyle\int_1^\infty\frac{\tih_{i,3}(s)\rmd s}{s^{2n-1}}<\infty$. Then, there exists $L^\rms:\NN_{\max\{n_0,n_*\}}\times\NN^2\times(0,\teps)\to(0,\infty)$ and $\hcE^\rms_n:\RR\to\cC\cL(\dom(W_{n-1}\cap\bX_\infty,\bY)$  such that	
\begin{align}\label{Frechet-Error-sin3}
S(t)w&=\displaystyle\sum\limits_{k=1}^{m} \sum_{\na=1}^{N_k}p_{0,k,\na}(t)e^{\mu_k t}(A-\mu_k^2 I_\bX)^{\na-1}P_kw+\displaystyle\sum\limits_{\nu_\ell\ne0} \sum_{\na=1}^{\tN_\ell}\tp_{0,\ell,\na}(t)e^{\nu_\ell t}(\tA-\nu_\ell^2 I_\bY)^{\na-1}\tP_\ell w \nonumber\\
&+\chi_{\{\nu_\ell:\ell=1,\dots,p\}}(0)\Bigg(\sum_{\na=1}^{\oN_\ell}\frac{t^{2\na-1}}{(2\na-1)!}\tA^{\na-1}\tP_\ell w+\sum_{\na=1}^{\oN_\ell}\frac{t^{2\na}}{(2\na)!}Q_i\tA^{\na-1}\oP_\ell w\Bigg)\nonumber\\
&+\chi_{[0,\infty)}(g_*(0))\displaystyle\sum\limits_{\na=0}^{n-1}\frac{t^{2\na+1}}{(2\na+1)!}A^\na w+\hcE^\rms_n(t)w
\end{align}
\begin{equation}\label{Frechet-error-sin4}
\|Q_i\hcE^\rms_n(t)w\|\leq L^\rms(n,i,j,\eps)e^{(\sup_{s\in\RR}g_*(s)+\eps)t}\|A^{n-1}w\|
\end{equation}	
for any $t\geq0$, $i,j\in\NN$, $x\in\bW_{n-1}\cap\Ker(Q_j-I_\bY)$, $\eps\in(0,\teps_0)$, $n\in\NN$, with $n\geq\max\{n_0,n_*\}$.
\end{theorem}	
\begin{proof}
The theorem follows from Theorem~\ref{t4.19}, Lemma~\ref{l4.20} and Lemma~\ref{l4.30}
\end{proof}	
Again we point out that the difference between the representations of the sine family given in
Theorem~\ref{t4.25} and Theorem~\ref{t4.32} is that the latter result holds for any $t\geq0$, albeit under additional assumptions on the function $g_*$.

\section{Examples and Applications}\label{sec7}

In this section we discuss several applications of our results and we present certain examples that lead to the second order problems described in the previous sections. 

\noindent\textbf{Schr\"odinger operators with $\delta$-potentials.} The easiest way to illustrate our results is to look at the special case of  $\delta$-interaction of strength $\alpha$ centered at $\beta\in\RR$, following \cite{AGHKH}. This example can be seen as a generalization of the Schr\"odinger operators, where the potential is a multiple of the generalized Dirac delta function. However, we point out that the generator $A_{\alpha,\beta}$ defined below is not a bounded perturbation of the Laplacian. In this specific case we can compute the leading order terms in the expansions of the cosine and sine functions explicitly. 
We recall several immediate properties of cosine families acting on Hilbert spaces.
\begin{remark}\label{r7.2} Using the same ideas as in \cite[Example 3.14.16]{ABHN} one can immediately infer that
\begin{enumerate}	
\item[(i)]$M_\Psi$, the operator of multiplication by a measurable function $\Psi$, generates a cosine family on $L^2(\Xi,\mu)$ if and only if $|\Psi|+\re\Psi\in L^\infty(\Xi,\mu)$. Moreover, in this case the cosine family generated by $M_\Psi$, is given by $\{M_{\cosh{(t\sqrt{\Psi})}}\}_{t\in\RR}$;
\item[(ii)] If $A$ is a self-adjoint operator bounded from above on a Hilbert space $\bH$, then it generates a cosine family $\{C(t)\}_{t\geq 0}$ satisfying
\begin{equation}\label{est-cos-Hilbert}
\|C(t)\|\leq \cosh(|t|\gamma)\;\mbox{for any}\;t\in\RR, \;\mbox{where}\;\gamma^2=\sup\sigma(A).	
\end{equation}
\end{enumerate}  
\end{remark}
\begin{example}\label{delta-interaction}
Let $B_\beta:\{u\in H^2(\RR,\CC):u(\beta)=0\}\to L^2(\RR,\CC)$ be the linear operator defined by $B_\beta u=u''$. It is well known (see \cite[Theorem 3.1.1]{AGHKH}) that the self-adjoint extensions of $B_\beta$ are given by the family of linear operators
\begin{equation}\label{self-adjoint-extensions}
A_{\alpha,\beta}:\{u\in H^1(\RR)\cap H^2(\RR\setminus\{\beta\}):u(\beta^+)-u(\beta^-)=\alpha u(\beta)\}\to L^2(\RR,\CC),\;\;A_{\alpha,\beta} u=u'',	
\end{equation}	
with $-\infty<\alpha\leq\infty$. We will focus on the case $\alpha\in\RR\setminus\{0\}$. From \cite[Theorem 3.1.1]{AGHKH} we know that
\begin{equation}\label{spectrum-A-alpha-beta}
\sigma_{\mathrm{ess}}(A_{\alpha,\beta})=(-\infty,0],\,\sigma_{\mathrm{disc}}(A_{\alpha,\beta})=\left\{\begin{array}{ll} \{\alpha^2/4\},& \alpha<0,\\
\emptyset,&\alpha\geq 0. \end{array}\right.
\end{equation}	
Moreover, we have $\Ker(\alpha^2/4I_{L^2}-A_{\alpha,\beta})=\mathrm{Span}\{e^{\alpha|\cdot-\beta|/2}\}$. From \eqref{self-adjoint-extensions} and \eqref{spectrum-A-alpha-beta} it follows that $A_{\alpha,\beta}$ is self-adjoint operator, bounded from above, hence it generates a cosine family denoted $\{C_{\alpha,\beta}(t)\}_{t\in\RR}$. The associated sine family is denoted by $\{S_{\alpha,\beta}(t)\}_{t\in\RR}$. From Remark~\ref{r7.2}(ii) we infer that \eqref{exp-growth} is satisfied with $M=1$ and any $\omega>|\alpha|/2$. Next, we look at the resolvent of the operator $A_{\alpha,\beta}$. As shown in \cite[Theorem 3.1.2]{AGHKH} the following holds true:
\begin{equation}\label{resolvent-A-alpha-beta}
\CC_+\setminus\{-\alpha/2\}\subset\{\lambda\in\CC:\lambda^2\in\rho(A_{\alpha,\beta})\}\;\mbox{and}\; \Big(R(\lambda^2,A_{\alpha,\beta})f\Big)(x)=\int_{\RR}G_{\alpha,\beta}(x,y,\lambda)f(y)\rmd y,	
\end{equation}	
for any $\lambda\in\CC_+\setminus\{-\alpha/2\}$, $f\in L^2(\RR,\CC)$, where $G_{\alpha,\beta}:\RR^2\times\big(\CC\setminus\{0,-\alpha/2\}\big)\to\CC$ is the kernel defined by 
\begin{equation}\label{kernel-alpha-beta}
G_{\alpha,\beta}(x,y,\lambda)=\frac{1}{2\lambda}e^{-\lambda|x-y|}-\frac{\alpha}{2\lambda(2\lambda+\alpha)}e^{-\lambda\big(|x-\beta|+|\beta-y|\big)}.	
\end{equation}	
First, we check the hypotheses used in the current paper.
We denote by $G_0:\RR^2\times\big(\CC\setminus\{0\}\big)\to\CC$, $G_0(x,y,\lambda)=\frac{1}{2\lambda}e^{-\lambda|x-y|}$ the Green's kernel of the free Laplacian and by $R_0:\CC_+\to\mathcal{B}(L^2(\RR,\CC))$ the operator defined by
\begin{equation}\label{free-resolvent}
\big(R_0(\lambda)f\big)(x)=\int_{\RR}G_0(x,y,\lambda)f(y)\rmd y.	
\end{equation}	 
From \eqref{kernel-alpha-beta} and \eqref{free-resolvent} we obtain 
\begin{equation}\label{resolvent-decomposition}
R_{A_{\alpha,\beta}}(\lambda)f=R_0(\lambda)f-\frac{\alpha}{2\lambda+\alpha}G_0(\cdot,\beta,\lambda)(\lambda R_0(\lambda)f)(\beta)\;\mbox{for any}\;\lambda\in\CC_+\setminus\Big\{-\frac{\alpha}{2}\Big\}, f\in L^2(\RR,\CC).
\end{equation}
Thus, $R_{A_{\alpha,\beta}}(\cdot)$ is holomorphic on $\CC_+\setminus\{-\alpha/2\}$. 
We claim that if $\alpha<0$ then  $-\alpha/2$ is a finitely meromorphic point of $R_{A_{\alpha,\beta}}(\cdot)$. Indeed, by \eqref{resolvent-decomposition} $-\alpha/2$ is a pole of order $1$ of $R_{A_{\alpha,\beta}}(\cdot)$. Moreover, 
\begin{align}\label{proj-alpha-2}
P_{-\frac{\alpha}{2}}f&=\frac{1}{\pi\rmi}\int_{\partial D(-\frac{\alpha}{2},r)}\lambda R_{A_{\alpha,\beta}}(\lambda)f\rmd\lambda=-\frac{\alpha}{\pi\rmi}\int_{\partial D(-\frac{\alpha}{2},r)}\frac{\lambda G_0(\cdot,\beta,\lambda)}{2\lambda+\alpha}(\lambda R_0(\lambda)f)(\beta)\rmd\lambda\nonumber\\
&=-\alpha e^{|\cdot-\beta|/2}\int_\RR e^{|x-\beta|/2} f(x)\rmd x=(-\alpha/2)^{1/2} e^{\alpha|\cdot-\beta|/2}\big\langle f, (-\alpha/2)^{1/2} e^{\alpha|\cdot-\beta|/2}\big\rangle_{L^2} .
\end{align}	
for any $f\in L^2(\RR,\CC)$ and $r<|\alpha|/2$. This computation shows that the principal part of the expansion \eqref{finite-meromorphic} of $R_{A_{\alpha,\beta}}(\cdot)$ at $\lambda=-\alpha/2$ has only one term whose coefficient is the rank-$1$ projection $P_{-\alpha/2}$ evaluated in \eqref{proj-alpha-2}, proving the claim. 

One can readily check the following elementary inequalities:
\begin{equation}\label{elementary-estimates}
\|R_0(\lambda)\|\leq\frac{1}{|\lambda|\re\lambda},\;\|G_0(\cdot,\beta,\lambda)\|_2=\frac{1}{|\lambda|\sqrt{\re\lambda}},\;\Big|\lambda\big(R_0(\lambda)f\big)(\beta)\Big|\leq\frac{\|f\|_2}{\sqrt{\re\lambda}}	
\end{equation}	
for any $\lambda\in\CC_+$, $f\in L^2(\RR,\CC)$.  Also, 
\begin{equation}\label{note-elementary}
\mbox{if}\;|\im\lambda|\geq|\alpha|\;\mbox{then}\;|2\lambda+\alpha|\geq|\im(2\lambda+\alpha)|=2|\im\lambda|\geq2|\alpha|.
\end{equation}
Using \eqref{resolvent-decomposition}, \eqref{elementary-estimates} and \eqref{note-elementary} we obtain 
\begin{equation}\label{resolvent-estimate-1}
\|R_{A_{\alpha,\beta}}(\lambda)\|\leq \frac{3}{2|\lambda|\re\lambda}\;\mbox{for any}\;\lambda\in\CC_+,\,|\im\lambda|\geq|\alpha|.
\end{equation}	
Hence, Hypothesis (H) holds for $n_0=1$ and by choosing $g_0(s)=\eps_0$, for any small $\eps_0>0$. Example~\ref{e4.1} shows that Hypothesis (Q) holds as well, where $Q_n:L^2_{\mathrm{loc}}(\RR,\CC)\to L^2(\RR,\CC)$ is defined by $Q_nf=\varphi_n f$ and $\varphi_n$ is a smooth cutoff function. Next, we note that $G_{\alpha,\beta}(x,y,\cdot)$ is holomorphic on $\CC\setminus\{0,-\alpha/2\}$ for any $(x,y)\in\RR^2$. Moreover, $G_{\alpha,\beta}(x,y,\cdot)$ has a removable singularity at  $\lambda=0$, while $\lambda=-\alpha/2$ is a pole of order $1$. We infer that the function $R_{A_{\alpha,\beta}}(\cdot)$ can be extended meromorphically to $\CC$.  We denote this extension by $R_{A_{\alpha,\beta}}^\infty(\cdot)$. It is well known that the function $R_0(\cdot)$ can be extended holomorphically to the set $\CC\setminus\{0\}$ as a function taking values in $\cC\cL(L^2_{\mathrm{comp}}(\RR,\CC),L^2_{\mathrm{comp}}(\RR,\CC))$. We denote this extension by $R_0^\infty(\cdot)$. Using the holomorphic extension from \eqref{resolvent-decomposition} we infer that 
\begin{equation}\label{resolvent-decomposition-extension}
R_{A_{\alpha,\beta}}^\infty(\lambda)f=R_0^\infty(\lambda)f-\frac{\alpha}{2\lambda+\alpha}G_0(\cdot,\beta,\lambda)(\lambda R_0^\infty(\lambda)f)(\beta)\;\mbox{for any}\;\lambda\in\CC\setminus\Big\{0,-\frac{\alpha}{2}\Big\},
\end{equation}
$f\in L_{\mathrm{comp}}^2(\RR,\CC)$. Reasoning similar to \eqref{proj-alpha-2} shows that if $\alpha>0$ then
\begin{equation}\label{pseudo-proj-alpha-2} 
\tP_{-\frac{\alpha}{2}}f=\frac{1}{\pi\rmi}\int_{\partial D(-\frac{\alpha}{2},r)}\lambda R_{A_{\alpha,\beta}}^\infty(\lambda)f\rmd\lambda=(-\alpha/2)^{1/2} e^{\alpha|\cdot-\beta|/2}\big\langle f, (-\alpha/2)^{1/2} e^{\alpha|\cdot-\beta|/2}\big\rangle_{L^2}
\end{equation}
for any $f\in L^2_{\mathrm{comp}}(\RR,\CC)$ and $r<|\alpha|/2$. Again the principal part of the expansion \eqref{finite-meromorphic} of $R_{A_{\alpha,\beta}}^\infty(\cdot)$ at $\lambda=-\alpha/2$ has only one term whose coefficient is the rank-$1$ operator $\tP_{-\alpha/2}$ evaluated in \eqref{pseudo-proj-alpha-2}, which proves that $-\alpha/2$ is a finitely meromorphic point of $R_{A_{\alpha,\beta}}^\infty(\cdot)$ for any $\alpha>0$. 
	
As in \eqref{elementary-estimates} we estimate
\begin{align}\label{elementary-estimates-part2}
\|Q_iR_0^\infty(\lambda)Q_j\|&\leq\frac{|e^{-(i+j+2)\re\lambda}-1|}{|\lambda||\re\lambda|},\;\|Q_iG_0(\cdot,\beta,\lambda)\|_2\leq\frac{e^{-\re\lambda|\beta|\big|e^{-2\re\lambda(i+1)}-1\big|^{1/2}}}{\sqrt{2}|\lambda|\sqrt{|\re\lambda|}},\nonumber\\
\Big|\lambda\big(R_0(\lambda)Q_jf\big)(\beta)\Big|&\leq\frac{\sqrt{2}e^{-\re\lambda(|\beta|+(j+1)/2)}\big|e^{-\re\lambda(j+1)}-1\big|^{1/2}\|f\|_2}{\sqrt{|\re\lambda|}}\;\mbox{for any}\;\lambda\in\CC\setminus\{0\},\,i,j\in\NN	
\end{align}
and any $f\in L^2(\RR,\CC)$. From \eqref{note-elementary}, \eqref{resolvent-decomposition-extension} and \eqref{elementary-estimates-part2} we conclude that 
\begin{equation}\label{extension-estimates}
\|Q_iR_{A_{\alpha,\beta}}^\infty(\lambda)Q_j\|\leq\frac{|e^{-(i+j+2)\re\lambda}-1|}{|\lambda||\re\lambda|}+\frac{e^{-2\re\lambda|\beta|}}{2|\lambda||\re\lambda|}\Big((e^{-2\re\lambda(i+1)}-1)(e^{-2\re\lambda(j+1)}-e^{-\re\lambda(j+1)})\Big)^{1/2}.
\end{equation}	
for any $\lambda\in\CC\setminus\{0,-\alpha/2\}$, with $|\im\lambda|\geq|\alpha|$, $i,j\in\NN$. Thus, Hypothesis (H-ext) is satisfied with $n_*=1$ and $g_*(s)=-\eta-1-\widetilde{\eta}\ln{(1+|s|)}$, where $\eta,\widetilde{\eta}>0$. This implies that Hypothesis (G) is also satisfied. Thus,  we can apply Theorem~\ref{t4.24} and Theorem~\ref{t4.25}.  Evaluating the integrals  \eqref{polynomial-representation} and \eqref{polynomial-representation-resonances}, from \eqref{proj-alpha-2} and \eqref{pseudo-proj-alpha-2} we obtain that for any $\eta>0$, $i\in\NN$ there exist $\tcE^\rmc_{i,\alpha,\beta},\tcE^\rms_{i,\alpha,\beta}:\RR\to\cC\cL\big(L^2_{\mathrm{comp}}(\RR,\CC),L^2_{\mathrm{loc}}(\RR,\CC)\Big)$ and $t_{i,\eta}$ such that
\begin{align}\label{cos-sin-alpha-beta}
\varphi_iC_{\alpha,\beta}(t)f&=-\frac{\alpha}{4}\varphi_ie^{\alpha(|\cdot-\beta|-t)/2}\big\langle f, e^{\alpha|\cdot-\beta|/2}\big\rangle_{L^2} +\tcE^\rmc_{i,\alpha,\beta}(t)f,\nonumber\\
\varphi_iS_{\alpha,\beta}(t)f&=\frac{1}{2}\varphi_ie^{\alpha(|\cdot-\beta|-t)/2}\big\langle f, e^{\alpha|\cdot-\beta|/2}\big\rangle_{L^2} +\tcE^\rms_{i,\alpha,\beta}(t)f,\nonumber\\
&\|\varphi_i\tcE^\rmc_{i,\alpha,\beta}(t)f\|_2+\|\varphi_i\tcE^\rms_{i,\alpha,\beta}(t)f\|_2\lesssim e^{-\eta t}\|f\|_2
\end{align}	
for any $i\in\NN$, $t\geq t_{i,\eta}$, $f\in L^2_{\mathrm{comp}}(\RR,\CC)$. 
\end{example}

\noindent\textbf{Schr\"odinger operators in odd dimensions with complex-valued potentials.}
Next, we turn our attention to several interesting cases of Schr\"odinger operators. Let $\bX=L^2(\RR^r,\CC)$, where $r\in\NN$ is odd and consider $V\in L^\infty(\RR^r,\CC)$, a complex valued potential with compact support. We define the linear operator $A:H^2(\RR^r,\CC)\to L^2(\RR^r,\CC)$ by the formula $A=\Delta+V$. It is well-known that $A$ generates a cosine family. Indeed, since $A_0: H^2(\RR^r,\CC)\to L^2(\RR^r,\CC)$ defined by $A_0=\Delta$ is self-adjoint, it follows that $B_0:H^1(\RR^r,\CC)\to L^2(\RR^r,\CC)$ defined by $B_0=\rmi(-\Delta)^{\frac{1}{2}}$ is skew-symmetric, hence, $B_0$ generates a $C_0$-group on $L^2(\RR^r,\CC)$. By \cite[Example 3.14.15]{ABHN} $A_0=B_0^2$ is the generator of a cosine family on $L^2(\RR^r,\CC)$ whose phase-space is $\bW_0=\dom(B_0)=H^1(\RR^r,\CC)$. Furthermore, since the operator of multiplication by the potential $V$ is bounded from the phase space $\bW_0=H^1(\RR^r,\CC)$ to $L^2(\RR^r,\CC)$, by \cite[Corollary 3.14.13]{ABHN} $A=A_0+M_V$ generates a cosine family on $L^2(\RR^r,\CC)$ with the phase-space $H^1(\RR^r,\CC)$.

As shown in \cite[Chapters 2,3]{DZ} one can find explicit formulas for the free resolvent $(\lambda^2-\Delta)^{-1}$ as an integral operator and estimates of the resolvent kernel. Using the Birman-Schwinger principle, one can show that the resolvent $R_A(\lambda)=(\lambda^2-A)^{-1}$ satisfies Hypothesis (H). As shown in Example ~\ref{e4.1} Hypothesis (Q) is also satisfied by choosing the Frechet space $\bY=L^2_{\mathrm{loc}}(\RR^r,\CC)$. The sequence of linear operators $Q_n$ are defined as operators of multiplication by the functions $\varphi_n(|\cdot|)$, with $\varphi_n$ defined in Example~\ref{e4.1}. 

The resolvent function $R_A(\cdot)$ has a meromorphic continuation to the entire complex plane as a function taking values in the set of continuous linear operators from $L^2_{\mathrm{comp}}(\RR^r,\CC)$ to $L^2_{\mathrm{loc}}(\RR^r,\CC)$ (\cite[Theorem 2.2, 3.8]{DZ}, 
for $r=1$ and $r\geq 3$ odd, respectively). Moreover, one can 
find resonance free regions, see \cite[Theorem 2.10, 3.10]{DZ}. Hence, Hypothesis (H-ext) is satisfied for the function $g_*(s)=-\eta-\widetilde{\eta}\ln{(1+|s|)}$,  for some $\eta,\widetilde{\eta}>0$. Thus, Hypothesis (G) is also satisfied. It was already established in \cite[Theorem 2.9, Theorem 3.11]{DZ} that the solutions of the wave equation 
 \begin{equation}\label{Wave-Sch}
\left\{\begin{array}{ll}
u_{tt}=u_{xx}+V(x)u,\;t\in\RR,\\
u(0,x)=u_0,
u(0,x)=u_1.\end{array}\right.
\end{equation}
has a eigenvalue/resonance explicit expansion \cite[Formulas (2.3.3), (3.2.12)]{DZ} for the special case when the potential $V$ is real valued. This restriction is due to the fact that the proof uses  Stone's representation formula for self-adjoint linear operators. In the general case of complex-valued potentials when $A$ is not self-adjoint, we can use the representations of the cosine family generated by $A$ to achieve the same goal. By Theorem~\ref{t4.24} and Theorem~\ref{t4.25} the same expansions hold for the general case of complex-valued potentials. 

\begin{example}\label{e7.1} To illustrate the eigenvalue/resonance expansion, we now look at examples of potentials arising in Physics. An interesting example taken from \cite{LP7} is that of one dimensional Schr\"odinger operators with complex square-well potential of the form $A=\partial_x^2+\alpha\chi_{[-1,1]}$, where  $\alpha\in\CC\setminus\RR$ is such that $\sqrt{\alpha}\sinh{(\sqrt{1-\alpha})}+\sqrt{1-\alpha}=0$ and $\sqrt{\alpha}\cosh{(\sqrt{1-\alpha})}+1=0$. In this case we can explicitly compute the resolvent operator and its meromorphic extension to $\CC$ using the Jost solutions and the Jost function. The Jost solutions, denoted $U_\pm(\cdot,\lambda)$,
are defined as solutions of the equation
\begin{equation}\label{e7.1.3}
\partial_{xx}U(x)=(\lambda^2-\alpha)U(x) \text{ for $|x|\le1$, $\d_{xx}U(x)=\lambda^2U(x)$ for $|x|>1$,}	
\end{equation}	
which are such that $U_\pm(x,\lambda)=e^{\mp\lambda x}$ for any $\pm x>1$. If $\lambda\in\CC_+$ then $U_\pm(x,\lambda)$ is asymptotic to the ``plain wave" $e^{\mp x}$, as $\pm x\to\infty$, justifying the term ``Jost solution". One can readily check that $U_-(x,\lambda)=U_+(-x,\lambda)$ for any $x\in\RR$. Moreover, $U_+$ is given by
\begin{equation}\label{e7.1.4.} 
U_+(x,\lambda)= \begin{cases}
e^{-\lambda x},&  x>1, \\
e^{-\lambda}\Big(\cosh\big((x-1)\sqrt{\lambda^2-\alpha}\big)-\lambda (\lambda^2-\alpha)^{-1/2}\sinh\big((x-1)\sqrt{\lambda^2-\alpha}\big)\Big), &  |x|\le1,\\ 
h_+(\lambda)e^{-\lambda x}+h_-(\lambda)e^{\lambda x},  
&  x<-1,
\end{cases}	
\end{equation}
where 
\begin{align}\label{e7.1.5}
h_+(\lambda)&=e^{-2\lambda}
\big(
\cosh(2\sqrt{\lambda^2-\alpha})+\frac{1}{2\lambda}(\lambda^2-\alpha)^{-1/2}(2\lambda^2-\alpha)\sinh(2 \sqrt{\lambda^2-\alpha})\big),\nonumber\\ h_-(\lambda)&=\frac{\alpha}{2\lambda}(\lambda^2-\alpha)\sinh(2\sqrt{\lambda^2-\alpha}).
\end{align}	
The Jost functions is defined as the Wronskian
\begin{equation}\label{e7.1.6}
W(\lambda)=U_+(x,\lambda)\d_x (U_-(x,\lambda))-(\d_x U_+(x,\lambda)) U_-(x,\lambda),	
\end{equation}	
which is $x$-independent because $U_\pm(\cdot,\lambda)$ are solutions to \eqref{e7.1.3}. The function $W$ is given by
\begin{align}\label{e7.1.7}
W(\lambda)&=2e^{-2\lambda}\big(\lambda\cosh( \sqrt{\lambda^2-\alpha})+\sqrt{\lambda^2-\alpha}\sinh( \sqrt{\lambda^2-\alpha})\big)\nonumber\\
&\quad\times\big(\cosh(\sqrt{\lambda^2-\alpha})+\lambda (\sqrt{\lambda^2-\alpha})^{-1}\sinh( \sqrt{\lambda^2-\alpha})\big)\;\mbox{for}\;\lambda^2\ne\alpha.
\end{align}

The Jost functions $U_\pm(x,\cdot)$ and the Jost function $W$ are entire functions, see, e.g. \cite[Section 2]{LP7}. Next, we denote by $Z(W)=\{\lambda\in\CC:W(\lambda)=0\}$ and by $\tG:\RR^2\times\CC\to\CC$ the semi-separable kernel
\begin{equation}\label{e7.1.1-bis}
\tG(x,y,\lambda)=\begin{cases}U_+(x,\lambda)U_-(y,\lambda), &  x>y,\\ 
U_+(y,\lambda)U_-(x,\lambda), & x<y.\end{cases}	
\end{equation}	
We recall that 
\begin{equation}\label{e7.1.1}
\big(R_A(\lambda)f\big)(x)=\int_{-\infty}^{+\infty}G(x,y,\lambda)f(y)\, \rmd y,
\end{equation}
for any $x\in\RR$, $\lambda\in\CC_+\setminus Z(W)$ and $f\in\ L^2(\RR,\CC)$, where the Green kernel $G:\RR^2\times\big(\CC\setminus Z(W)\big)\to\CC$ is defined by 
\begin{equation}\label{e7.1.2}
G(x,y,\lambda)=\Big(W(\lambda)\Big)^{-1}\tG(x,y,\lambda).
\end{equation}	
Since $U_\pm(x,\cdot)$ and $W(\cdot)$ are entire functions we infer that $G(x,y,\cdot)$ is meromorphic on $\CC$. We note that formula \eqref{e7.1.1} holds whenever $\lambda\in\CC$ is such that $W(\lambda)\ne0$, $f\in L_{\mathrm{comp}}^2(\RR,\CC)$ and the left hand side is replaced by $R_A^\infty(\lambda)$. Moreover, we have $R_A^\infty(\lambda) \in\cC\cL\big(L^2_{\textrm{comp}}(\RR; \CC),H^2_{\loc}(\RR; \CC)\big)$ for any $\lambda\in\CC$ such that $W(\lambda)\ne 0$. 

By Lemma~\ref{l3.2}(iv) and Lemma~\ref{l4.11} (viii) the first two terms in representations \eqref{Frechet-Error-cos} and \eqref{Frechet-Error-sin} are given by the residues of the functions $\lambda\to\lambda^\kappa e^{\lambda t}R_A(\lambda)$ and 
$\lambda\to\lambda^\kappa e^{\lambda t}R_A^\infty(\lambda)$
at the zeros of $W$. Hence, it is crucial to know the order of the zeros of $W$. In the sequel we denote by  $\cR_\alpha$ the set of all zeros of $W$. From \cite[Lemma 2.5, Lemma 2.6, Lemma 2.7]{LP7} we have
\begin{equation}\label{e7.1.11} 0,\pm\sqrt{\alpha}\notin\cR_\alpha,\; \cosh{\sqrt{\lambda_0^2-\alpha}}\ne 0,\;\sinh{\sqrt{\lambda_0^2-\alpha}}\ne 0\;\mbox{for any}\;\lambda_0\in\cR_\alpha.
\end{equation}	
Moreover, from \cite[Lemma 5.10]{LP7} we know that $\oor(\lambda_0,W)=1$ if $\lambda_0\in\cR_\alpha\setminus\{-1\}$ and $\oor(-1,W)=2$.
Changing the order of integration, we infer that $\cK(t,\kappa,\lambda_0):=\Res\Bigl(\lambda^\kappa e^{\lambda t}R_A(\lambda),\lambda=\lambda_0\Bigr)$ is an integral operator on $L^2(\RR,\CC)$ having the integral kernel 
\begin{equation}\label{e7.1.9}
K(x,y,t,\kappa,\lambda_0)=\Res\Bigl(\lambda^\kappa e^{\lambda t}G(x,y,\lambda),\lambda=\lambda_0\Bigr)\;\mbox{for any}\;x,y,t\in\RR,\kappa=0,1,\lambda\in\cR_\alpha\cap\CC_+ .
\end{equation}
Similarly, $\cK^\infty(t,\kappa,\lambda_0):=\Res\Bigl(\lambda^\kappa e^{\lambda t}R_A(\lambda),\lambda=\lambda_0\Bigr)$ is an integral operator mapping $L_{\mathrm{comp}}^2(\RR,\CC)$ into
$L_{\mathrm{loc}}^2(\RR,\CC)$ having the integral kernel 
\begin{equation}\label{e7.1.10}
K^\infty(x,y,t,\kappa,\lambda_0)=\Res\Bigl(\lambda^\kappa e^{\lambda t}G(x,y,\lambda),\lambda=\lambda_0\Bigr)\;\mbox{for any}\;x,y,t\in\RR,\kappa=0,1,\lambda\in\cR_\alpha\setminus\CC_+ .
\end{equation}

We recall that 
\begin{equation}\label{e7.1.8}
\Res\Big(\frac{F}{W},\lambda=\lambda_0\Big)=\left\{\begin{array}{ll} \frac{F(\lambda_0)}{W'(\lambda_0)},& \oor(\lambda_0,W)=1,\\
\frac{2F'(\lambda_0)}{W''(\lambda_0)}-\frac{2F(\lambda_0)W'''(\lambda_0)}{3\big(W''(\lambda_0)\big)^2},&\oor(\lambda_0,W)=2, \end{array}\right.	
\end{equation}	
for any entire function $F$. Moreover, one can readily compute the derivatives of $W$ from \eqref{e7.1.7},
\begin{align}\label{e.7.1.12}
W'(\lambda_0)&=-\frac{\alpha^2(1+\lambda_0)\sinh{\Big(2\sqrt{\lambda_0^2-\alpha}\Big)}}{\lambda_0(\lambda_0^2-\alpha)^{3/2}e^{2\lambda_0}}\;\mbox{for any}\;\lambda_0\in\cR_\alpha\setminus\{-1\}\nonumber\\
W''(-1)&=\frac{2e^2\alpha}{1-\alpha},\;W'''(-1)=\frac{2e^2\alpha(6\alpha-1)}{(1-\alpha)^2}.
\end{align}
Applying \eqref{e7.1.8} we can evaluate the integral kernels from \eqref{e7.1.9} and \eqref{e7.1.10} as follows
\begin{equation}\label{e7.1.13}
\Res\Bigl(\lambda^\kappa e^{\lambda t}G(x,y,\lambda),\lambda=\lambda_0\Bigr)=-\frac{\lambda_0^{\kappa+1}(\lambda_0^2-\alpha)^{3/2}e^{(t+2)\lambda_0}}{\alpha^2(1+\lambda_0)\sinh{\Big(2\sqrt{\lambda_0^2-\alpha}\Big)}}\tG(x,y,\lambda_0).
\end{equation}	
for any $\lambda_0\in\cR_\alpha\setminus\{-1\}$, $x,y,t\in\RR$, $\kappa=0,1$. We note that formulas \eqref{e.7.1.12} and \eqref{e7.1.13} are well-defined due to \eqref{e7.1.11}. Similarly, we have
\begin{align}\label{e7.1.14}
\Res\Bigl(\lambda^\kappa &e^{\lambda t}G(x,y,\lambda),\lambda=-1\Bigr)=\frac{(-1)^\kappa(1-\alpha)}{\alpha}te^{-(t+2)}\tG(x,y,-1)\nonumber\\
&+(-1)^{\kappa+1}\Bigl(\frac{\kappa(1-\alpha)}{\alpha}+\frac{6\alpha-1}{3\alpha}\Bigr)e^{-(t+2)}\tG(x,y,-1)+	\frac{(-1)^\kappa(1-\alpha)}{\alpha}e^{-(t+2)}\partial_\lambda\tG(x,y,-1)
\end{align}	
for any $x,y,t\in\RR$, $\kappa=0,1$.

Next, we recall that Hypotheses (H), (Q), (H-ext) and (G) hold for the general case of Schr\"odinger operators with compactly supported potential, in particular for $A=\partial_x^2+\alpha\chi_{[-1,1]}$. Also, we recall the definition of the sequence of functions $\{\varphi_n\}_{n\in\NN}$ in Example~\ref{e4.1}. We denote by $\{C(t)\}_{t\in\RR}$ the cosine family generated by $A$, and by $\{S(t)\}_{t\in\RR}$ its associated sine family. 
From \cite[Theorem 2.10]{DZ} or \cite[Theorem 3.8]{LP7} we know that for any $\eta>0$ and any $\widetilde{\eta}>0$ sufficiently small there are only finitely many zeros of the function $W$ to the right of the curve $\re\lambda=-\eta-\widetilde{\eta}\ln{(1+|\im\lambda|)}$. We denote by 
\begin{equation}\label{e7.1.15}
\cR_{\alpha,\eta,\widetilde{\eta}}=\{\lambda\in\cR_\alpha:\re\lambda>-\eta-\widetilde{\eta}\ln{(1+|\im\lambda|)}\}.	
\end{equation}	
Finally, we denote by $\tcK, \tcJ:\CC\to\cC\cL\big(L^2_{\mathrm{comp}}(\RR,\CC),L^2_{\mathrm{loc}}(\RR,\CC)\Big)$ the integral operator-valued functions defined by
\begin{equation}\label{e7.1.17}
\big(\tcK(\lambda)f\big)(x)=\int_{-\infty}^{+\infty}\tG(x,y,\lambda)f(y)\, \rmd y,\;\big(\tcJ(\lambda)f\big)(x)=\int_{-\infty}^{+\infty}\partial_\lambda\tG(x,y,\lambda)f(y)\, \rmd y,
\end{equation}
for any $x\in\RR$, $\lambda\in\CC$ and $f\in\ L_{\mathrm{comp}}^2(\RR,\CC)$. From 
Theorem~\ref{t4.24} and Theorem~\ref{t4.25}, \eqref{e7.1.13} and \eqref{e7.1.14} we conclude that
for any $\eta>1$, $i\in\NN$ there exist $\tcE^\rmc_i,\tcE^\rms_i:\RR\to\cC\cL\big(L^2_{\mathrm{comp}}(\RR,\CC),L^2_{\mathrm{loc}}(\RR,\CC)\Big)$ and $t_{i,\eta}$ such that the components of the solution $u(t)=C(t)x_o+S(t)x_1$ of the initial value problem \eqref{abstract-Cauchy} satisfy
\begin{align}\label{e7.1.16}
\varphi_iC(t)f&=-\sum_{\lambda_0\in\cR_{\alpha,\eta,\widetilde{\eta}}\setminus\{-1\}}\frac{\lambda_0^2(\lambda_0^2-\alpha)^{3/2}e^{(t+2)\lambda_0}}{\alpha^2(1+\lambda_0)\sinh{\Big(2\sqrt{\lambda_0^2-\alpha}\Big)}}\tcK(\lambda_0)f+\frac{\alpha-1}{\alpha}te^{-(t+2)}\tcK(-1)f\nonumber\\
&+\frac{3\alpha+2}{3\alpha}e^{-(t+2)}\tcK(-1)f+	\frac{\alpha-1}{\alpha}e^{-(t+2)}\partial_\lambda\tcJ(-1)f+\tcE^\rmc_i(t)f,\nonumber\\
\varphi_iS(t)f&=-\sum_{\lambda_0\in\cR_{\alpha,\eta,\widetilde{\eta}}\setminus\{-1\}}\frac{\lambda_0(\lambda_0^2-\alpha)^{3/2}e^{(t+2)\lambda_0}}{\alpha^2(1+\lambda_0)\sinh{\Big(2\sqrt{\lambda_0^2-\alpha}\Big)}}\tcK(\lambda_0)f+\frac{1-\alpha}{\alpha}te^{-(t+2)}\tcK(-1)f\nonumber\\ &-\frac{6\alpha-1}{3\alpha}e^{-(t+2)}\tcK(-1)+\frac{1-\alpha}{2\alpha}e^{-(t+2)}\partial_\lambda\tcJ(-1)f+\tcE^\rms_i(t)f,\nonumber\\
&\|\varphi_i\tcE^\rmc_i(t)f\|+\|\varphi_i\tcE^\rms_i(t)f\|\lesssim e^{-\eta t}\|f\|_2
\end{align}	
for any $i\in\NN$, $t\geq t_{i,\eta}$, $f\in L^2_{\mathrm{comp}}(\RR,\CC)$. 
\end{example}

\begin{remark}\label{r7.7}
Another interesting example arises in the case when the potential $V$ in \eqref{Wave-Sch} is a matrix-valued square-well potential of the form $V=V_0\chi_{[-a,a]}$, for some $a>0$ and $V_0\in\CC^{d\times d}$ is a non symmetric matrix such that $0\notin\sigma(V_0)$ cf. \cite{LP7}. We can define the matrix $(\lambda^2 I_d-V_0)^{1/2}$ using 
Riesz-Dunford functional calculus, cf., e.g., \cite[Section VII.4]{Con}. Similar to the scalar case, we can explicitly compute the Jost solutions. The resolvent of $A=\partial_x^2 I_d+V_0\chi_{[-a,a]}$ is again an integral operator, whose integral kernel has the form \eqref{e7.1.2}. The function $W$ defined in \eqref{e7.1.7} can be defined in the matrix valued case as well and it is entire, c.f., \cite[Section 2]{LP7}. The Jost function in this case is defined as $\cW(\lambda)=\det_{d\times d}(W(\lambda))$. The zeros of the function $\cW$ determine the poles of the resolvent $R_A(\cdot)$ and its extension $R_A^\infty(\cdot)$. To obtain an expansion similar to \eqref{e7.1.16} we need to compute the residues of $\lambda\to \lambda^\kappa e^{\lambda t}G(x,y,\lambda)$ at the zeros of $\cW$. Similar to the scalar case the integral kernel has the representation $G(x,y,\lambda)=\big(\cW(\lambda)\big)^{-1}\overline{G}(x,y,\lambda)$, where the function $\overline{G}$ is defined in terms of the Jost solutions and the adjugate matrix of $W(\lambda)$, and hence it is a matrix-valued entire function. The order of the zeros of the Jost function $\cW$ can be computed using the \cite[Theorem 5.1]{LP7}. The only difference from the scalar case is that the order of a zero of $\cW$ might be bigger, thus requiring a generalization of formula \eqref{e7.1.8}, which makes the residue computation more tedious. However, the cosine and sine functions have computable expansions similar to \eqref{e7.1.16} above.
\end{remark}

\begin{remark}\label{r7.10}
The solutions of \eqref{abstract-Cauchy} admit a representation in terms of the eigenvalues and resonances of the cosine family generator $A$ also in the case when $-A$ is a Black box Hamiltonian, c.f., e.g \cite{CD1,DZ,Sjostrand}. We briefly recall that $-A$ is a black-box Hamiltonian if $A$ is a self-adjoint operator, upper semi-bounded and that there exists a ball, $B(0,R_0)=\{x\in\RR^r:|x|<R_0\}$, called the black box, such that 
\begin{enumerate}
\item[(i)] $\chi_{\RR^r\setminus B(0,R_0)}\dom(A)\subset H^2(\RR^r\setminus B(0,R_0))$;
\item[(ii)] $\chi_{\RR^r\setminus B(0,R_0)} Au=\Delta u_{|\RR^r\setminus B(0,R_0)}$ for any $u\in\dom(A)$;
\item[(iii)] If $u\in H^2(\RR^2)$ and $u_{|B(0,R_0+\eps)}=0$ for some $\eps>0$ then $u\in\dom(A)$;
\item[(iv)] $\chi_{B(0,R_0)}(\rmi I_{L^2}-A)^{-1}$ is compact.
\end{enumerate}
This general class of operators includes the obstacle scattering 
Dirichelet Laplacian, the Friedrichs extension of the Laplacian on finite volume surfaces, elliptic perturbations of the semiclassical Laplacian see, e.g. \cite[Examples, Section 4.1, Pages, 229, 246]{DZ}.  The desired expansion of solutions of \eqref{abstract-Cauchy} follows similarly to the case of Schr\"odinger operators in odd dimensions, see \cite[Theorem 4.44]{DZ}. A crucial tool in the proof of the expansion is a generalization of the Stone's formula for black box Hamiltonians. 
Under assumptions (i)-(iv) it is possible to prove that Hypotheses (H), (Q), (G) and (H-ext) hold true. This follows from the non-trivial results \cite[Theorem 4.41, Theorem 4.43]{DZ}. Thus, the expansion of solutions can also be obtained from Theorem~\ref{t4.24} and Theorem~\ref{t4.25}.

\end{remark}
In the next example, we illustrate the expansion of scattering wave solutions obtained by applying Theorem~\ref{t3.16}, which includes only the terms generated by the poles of the true resolvent $R_A(\cdot)$ defined as a $\mathcal{B}(\bX)$-valued analytic function. 
\begin{example}\label{e7.2}
Let $f:\RR\to\RR$ be a function of class at least $C^3$ and consider the equation
\begin{equation}\label{hyperbolic-pde-nonlinear}
u_{tt}=u_{xx}+f(u).	
\end{equation}	 
Time independent solutions of the nonlinear hyperbolic equation \eqref{hyperbolic-pde-nonlinear} satisfy the nonlinear pendulum equation
\begin{equation}\label{pendulum-nonlinear}
u_{xx}+f(u)=0.	
\end{equation}
We assume that $u_*$ is a homoclinic solution of \eqref{pendulum-nonlinear} that decays exponentially to $u_\infty\in\RR$ at $\pm\infty$. That is, the exist $N>0$, $c>0$ such that
\begin{equation}\label{exp-decay-infty}
|u_*(x)-u_\infty|\leq Ne^{-c|x|}\;\mbox{for any}\;x\in\RR.	
\end{equation}	
Using phase-plane analysis we infer that 
\begin{equation}\label{properties-at-infty}
f(u_\infty)=0,\;\;f'(u_\infty)<0.
\end{equation}
The linearization along the wave $u_*$ is given by 
\begin{equation}\label{hyperbolic-pde-linear}
u_{tt}=Au,	
\end{equation}
where $A=\partial_x^2+V(\cdot)$, $V:\RR\to\RR$ is defined by $V(x)=f'(u_*(x))$. The linear operator $A$ is a closed, densely defined, linear operator on $L^2(\RR,\CC)$ with domain $H^2(\RR,\CC)$. This linear operator is self-adjoint and bounded from above, hence,  by Remark~\ref{r7.2} the linear operator $A$ generates a cosine family $\{C(t)\}_{t\in\RR}$ and $\|C(t)\|=e^{\sqrt{\sup\sigma(A)}|t|}$ for any $t\in\RR$. Moreover, using the same argument as in the case of complex-valued Schr\"odinger operators in odd dimensions, the phase space in this case is $\bW=H^1(\RR,\CC)$. One can readily compute the essential spectrum of $A$ as follows:
\begin{equation}\label{ess-spectrum-linearized-hyperbplic}
\sigma_{\mathrm{ess}}(A)=\sigma_{\mathrm{ess}}(\partial_x^2+f'(u_\infty))=\sigma(\partial_x^2+f'(u_\infty))=(-\infty,f'(u_\infty)].	
\end{equation}	
Since $\lim_{x\to\pm\infty}u_*(x)=u_\infty$ it follows from \eqref{pendulum-nonlinear} and \eqref{properties-at-infty} that $\lim_{x\to\pm\infty}u_*''(x)=-f(u_\infty)=0$. Hence, $u_*''$ is bounded and thus, by Taylor's formula, we obtain $\lim_{x\to\pm\infty}u_*'(x)=0$. From \eqref{exp-decay-infty} we infer that $u_*'$ and $u_*''$ decay exponentially at $\pm\infty$. That is the positive constants $N$ and $c$ in \eqref{exp-decay-infty} can be chosen such that
\begin{equation}\label{extra-exp-decay-infty}
\max\{|u_*'(x)|, |u_*''(x)|\leq Ne^{-c|x|}\}\;\mbox{for any}\;x\in\RR.	
\end{equation}
Using Sturm-Liouville Theory, we conclude that $0\in\sigma_{\mathrm{disc}}(A)$. Moreover, $\Ker(A)=\mathrm{Span}\{u_*'\}$. Since $u_*'$ has at least one sign change, we know that $A$ has at least one positive eigenvalue. Actually, $\sigma_{\mathrm{disc}}(A)$ consists of a finite number of \textit{simple eigenvalues}, see, e.g., \cite[Theorem 2.3.3]{KP}. Therefore, there exist real numbers
$a_0>a_1>\dots>a_{p_1}>0>b_1>\dots>b_{p_2}>f'(u_\infty)$ such that 
\begin{equation}\label{spectrum-Schro-gap}
\{0,a_0\}\subseteq\sigma_{\mathrm{disc}}(A)\subseteq	\{0\}\cup\{a_j:j=0,\dots p_1\}\cup\{b_k:k=1,\dots p_2\}.
\end{equation} 
Since the linear operator $A$ is self-adjoint, 
\begin{equation}\label{first-resolvent-estimate}
 \|R(\mu,A)\|=\frac{1}{\mathrm{dist}(\mu,\sigma(A))}\;\mbox{for any}\;\mu\in\rho(A).
\end{equation}
Using \eqref{spectrum-Schro-gap} one can readily check that 
\begin{align}\label{est-distance-to-spectrum-an d-inclusion}
\mathrm{dist}(\mu,\sigma(A))=|\im\mu|\;&\mbox{whenever}\;\re\mu\leq f'(u_\infty);\nonumber\\
\{\lambda\in\CC:\lambda^2\in\sigma(A)\}\subseteq\{is:|s|\geq\sqrt{-f'(u_\infty)}\}&\cup\{0\}\cup\{\pm\sqrt{a_j}:j=0,\dots,p_1\}\nonumber\\
&\cup\{\pm\rmi\sqrt{-b_k}: k=1,\dots,p_2\}.	
\end{align}	
Next, we fix $\eps\in(0,\sqrt{a_{p_1}})$ and we define $\xi_0=\max\{\frac{2\eps\sqrt{-b_{p_2}}}{\sqrt{-f'(u_\infty)}-\sqrt{-b_{p_2}}}, \sqrt{a_0}+1\}$. We note that
\begin{equation}\label{inequalities2}
\xi_0>\frac{\eps\sqrt{-b_{p_2}}}{\sqrt{-f'(u_\infty)}-\sqrt{-b_{p_2}}}\;\;\mbox{or}\;\;\frac{\xi_0\sqrt{-f'(u_\infty)}}{\eps+\xi_0}>\sqrt{-b_{p_2}}.
\end{equation}
We define $g_0:\RR\to\RR$ by $g_0(s)=\eps$ whenever $|s|\geq\sqrt{-f'(u_\infty)}$ and $g_0(s)=\frac{|s|(\eps+\xi_0)}{\sqrt{-f'(u_\infty)}}-\xi_0$ whenever $|s|<\sqrt{-f'(u_\infty)}$. The function $g_0$ is piecewise of class $\cC^1$. From \eqref{inequalities2} and since $-\xi_0<-\sqrt{a_0}$ we have
\begin{align}\label{location-g_0-1}
\{\lambda\in\CC:\re\lambda>g_0(\im\lambda)\}&\setminus\Big(\{0\}\cup\{\pm\sqrt{a_j}:j=0,\dots,p_1\}\cup\{\pm\rmi\sqrt{-b_k}:j=1,\dots,p_2\}\Big)\nonumber\\
&\subset\{\lambda\in\CC:\lambda^2\in\rho(A)\}.	
\end{align}	
\begin{figure}[h]
	\begin{center}
		\includegraphics[width=0.6\textwidth]{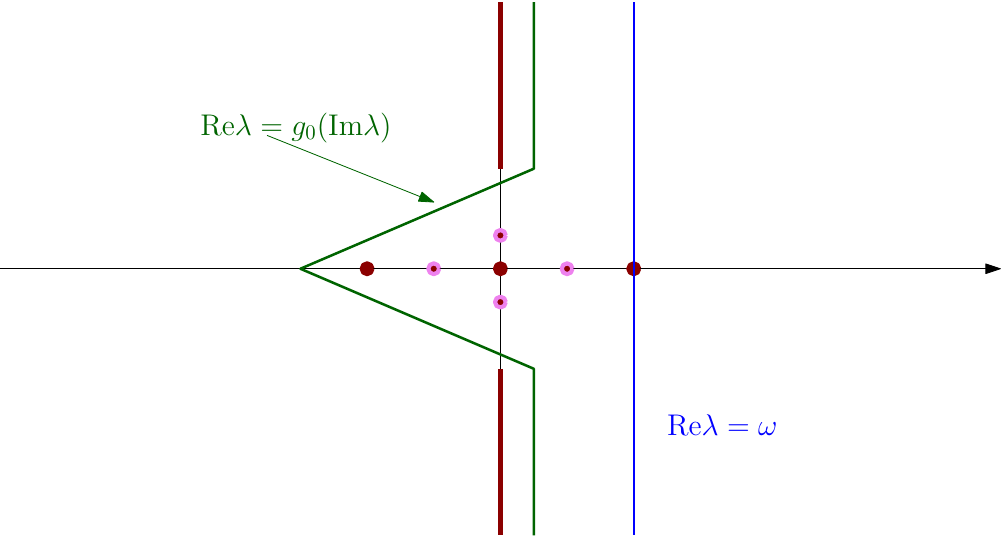}
		
		Figure 5. Depicted in dark red is the set $\{\lambda\in\CC:\lambda^2\in\sigma(A)\}$, while in dark green is the graph of the curve $\re\lambda_0=g_0(\im\lambda)$. 
	\end{center}
\end{figure}
Moreover, all the poles of the function $R_A(\lambda)=(\lambda^2I-A)^{-1}$ are located in the set $\{\lambda\in\CC:g_0(\im\lambda)<\re\lambda\leq\sqrt{a_0}\}$. Since $A$ has only semi-simple eigenvalues we infer that all the poles of $R_A(\cdot)$ are finitely meromorphic poles of order $1$. In addition,
all the projections defined by \eqref{spectral-projection} have rank $1$. If 
$\ker(A-a_jI)=\mathrm{Span}\{u_j\}$ and $\ker(A-b_kI)=\mathrm{Span}\{v_k\}$, $\|u_j\|_2=\|v_j\|_2=1$, then 
\begin{equation}\label{proj-j-k}
P_{\pm\sqrt{a_j}}=\langle\cdot,u_j\rangle u_j,\;P_{\pm\rmi\sqrt{-b_k}}=\langle\cdot,v_k\rangle u_k\;\mbox{for any}\;
j=0,\dots,p_1,\,k=1,\dots,p_2.
\end{equation}
Fix $\lambda\in\CC$ such that $g_0(\im\lambda)<\re\lambda\leq\sqrt{a_0}$ and $|\im\lambda|>\sqrt{a_0-f'(u_\infty)}$. Then, $\re\lambda\in[\delta_0,a_0]$ and
\begin{equation}\label{inequalities3}
\re(\lambda^2)=(\re\lambda)^2-(\im\lambda)^2\leq f'(u_\infty),\;|\lambda|^2\leq a_0+(\im\lambda)^2\leq a_0-f'(u_\infty)+(\im\lambda)^2\leq 2(\im\lambda)^2.
\end{equation}	
From \eqref{first-resolvent-estimate} and \eqref{inequalities3} we conclude that 
\begin{equation}\label{second-resolvent-estimate}
\|R_A(\lambda)\|=\frac{1}{|\im(\lambda^2)|}=\frac{1}{2\re\lambda|\im(\lambda)|}\leq\frac{1}{\delta_0\sqrt{2}|\lambda|}
\end{equation}
whenever $g_0(\im\lambda)<\re\lambda\leq\sqrt{a_0}$ and $|\im\lambda|>\sqrt{a_0-f'(u_\infty)}$, which proves that Hypothesis (H) holds true for $h_1(s)=\frac{1}{s}$ and $h_2(s)=\frac{1}{\delta_0\sqrt{2}}$. Evaluating integral \eqref{polynomial-representation}, applying Theorem ~\ref{t3.11} and Theorem~\ref{t3.12}, using \eqref{proj-j-k} and since $\sup_{s\in\RR}g_0(s)<\eps$, we infer that the following representation holds
\begin{align}\label{cosh-sinh-cos-sin-representation}
C(t)f&=\sum_{k=1}^{p_1}\cosh{(\sqrt{a_j} t)}\langle f,u_j\rangle u_j+\sum_{j=1}^{p_2}\cos{(\sqrt{-b_j} t)}\langle f,v_k\rangle v_k+\mathcal{O}(e^{2\eps t}),\nonumber\\
S(t)g&=\sum_{k=1}^{p_1}\frac{\sinh{(\sqrt{a_j} t)}}{\sqrt{a_j}}\langle g,u_j\rangle u_j+\sum_{j=1}^{p_2}\frac{\sin{(\sqrt{-b_j} t)}}{\sqrt{-b_j}}\langle g,v_k\rangle v_k+\mathcal{O}(e^{2\eps t})	
\end{align}	
for any $f\in H^2(\RR,\CC)$, $g\in H^1(\RR,\CC)$ and $t\in\RR$ and any $\eps>0$ small enough. 
\end{example}

\begin{example}\label{e7.4} Another class of the wave equations as in \eqref{abstract-Cauchy} that satisfy our setup are discussed in \cite{CDY1,CDY2}. The main examples include the obstacle scattering Dirichelet Laplacian in $L^2(\RR^2)$, Aharonov-Bohn Hamiltonians, Laplacians on cones and Schr\"odinger operators with real-valued, compactly supported potentials on $L^2(\RR^2)$. In all of these cases the cosine family generator $A$ is a self-adjoint operator on a Hilbert space $\bH$. Moreover, the expansion of the resolvent $R_A(\cdot)$ at $\lambda=0$ is given as a power series of $\lambda$ and $\log(\lambda)$, see \cite{CDY1}. The authors prove an expansion of solutions of \eqref{abstract-Cauchy} with $x_0=0$, in the spirit of Theorem~\ref{t4.24} and Theorem~\ref{t4.25}. Their representation has a term that counts the contribution of the discrete spectrum, an integral term that is the contribution of the zero-energy and a remainder term that decays exponentially, see \cite[Theorem 2.1]{CDY1}. Using our ``semigroup/cosine family" notation,  the assumptions of \cite{CDY1} are the following: (1) the spectrum of the cosine family generator $A$ consists of $\sigma_{\mathrm{ess}}(A)=(-\infty,0]$
and finitely many positive semi-simple discrete eigenvalues; (2) there exists a bounded linear operator $Q\in\mathcal{B}(\bH)$ and $c_0>0$	such that $QR_A(\cdot)Q$ continues holomorphically from $\{\lambda\in\CC:\re\lambda>0,\im\lambda\ne 0\}$ to $\{\lambda\in\CC:\re\lambda>-c_0,\im\lambda\ne 0\}$; (3) there exist $p,q\geq 0$ such that  $\|QR_A(\lambda)Q\|_{\dom(A^p)\to\dom(A^q)}\to 0$ as $|\im\lambda|\to\infty$, uniformly for $|\re\lambda|<c_0$ and $\int_1^\infty\|QR_A(\lambda+s)Q\|_{\dom(A^p)\to\dom(A^q)}\rmd \lambda$ for any $s>-c_0$; (4) There exist operators $A_{j,k}$ such that $QA_{j,k}Q\in\mathcal{B}(\dom(A^p),\dom(A^q))$ and numbers $\nu_j\in\RR$ and $b_{j,k}\in\CC\setminus (-\infty,0)$ such that $QR_A(\lambda)Q=\sum_{j=j_0}^\infty\sum_{k\in\ZZ}QA_{j,k}Q \lambda^{\nu_j}\log^k{(b_{j,k}\lambda)}$ uniformly on $\{\lambda\in\CC\setminus(-\infty,0]:|\re\lambda|<c_0,|\im\lambda|<c_0\}$. Assumptions (1)-(3) are in the spirit of Hypothesis (H-ext). The existence of the bounded linear operator $Q$ assumed above, is similar to Hypothesis (Q), especially since in all of the specific examples mentioned above it is a multiplication by a $C_0^\infty$ function on some $L^2$ space. 
	
Next, as a slight generalization of \cite{CDY1,CDY2}, we turn our attention to the case of Schr\"odinger operators with complex-valued, compactly supported potentials on $L^2(\RR^2)$. Arguing as in the case of Schr\"odinger operators in odd dimensions we can readily check that Hypothesis (H) holds true, and the function $g_0$ can be chosen to be some small positive constant. Using the same argument as in \cite[Subsection 3.4]{CDY1} we can check that Hypothesis (H-ext) holds true. The difference with the case of real-valued potentials is that we might have a finite number of poles of $R_A(\cdot)$ in $\CC_+$, corresponding to non-real discrete eigenvalues, and a finite number of poles of the extension $R_A^\infty(\cdot)$ in the vertical strip $\{\lambda\in\CC:-c_0<\re\lambda\leq 0\}$. In this case the function $g_*$ is as shown in Figure 6 below. The shape of the graph of $\re\lambda=g_*(\im\lambda)$ is chosen such that it is as close as possible to the boundary of the set $\{\lambda\in\CC:\re\lambda>-c_0\}\setminus(-\infty,0]$, which allows us to work on a set where the principal part of the complex logarithm function is analytic. By Example~\ref{e4.1} Hypothesis (Q) holds true if
the sequence of linear operators $\{Q_n\}_{n\in\NN}$ are defined as operators of multiplication by the functions $\varphi_n(|\cdot|)$, with $\varphi_n$ are smooth cutoff functions. Applying Theorem~\ref{t4.25} one can obtain a version of \cite[Theorem 1.4]{CDY1} for the case of complex-valued, compactly supported potentials, provided that 
an expansion similar to \cite[Equation (3.7)]{CDY1}	has been established. 
\begin{figure}[h]
	\begin{center}
		\includegraphics[width=0.6\textwidth]{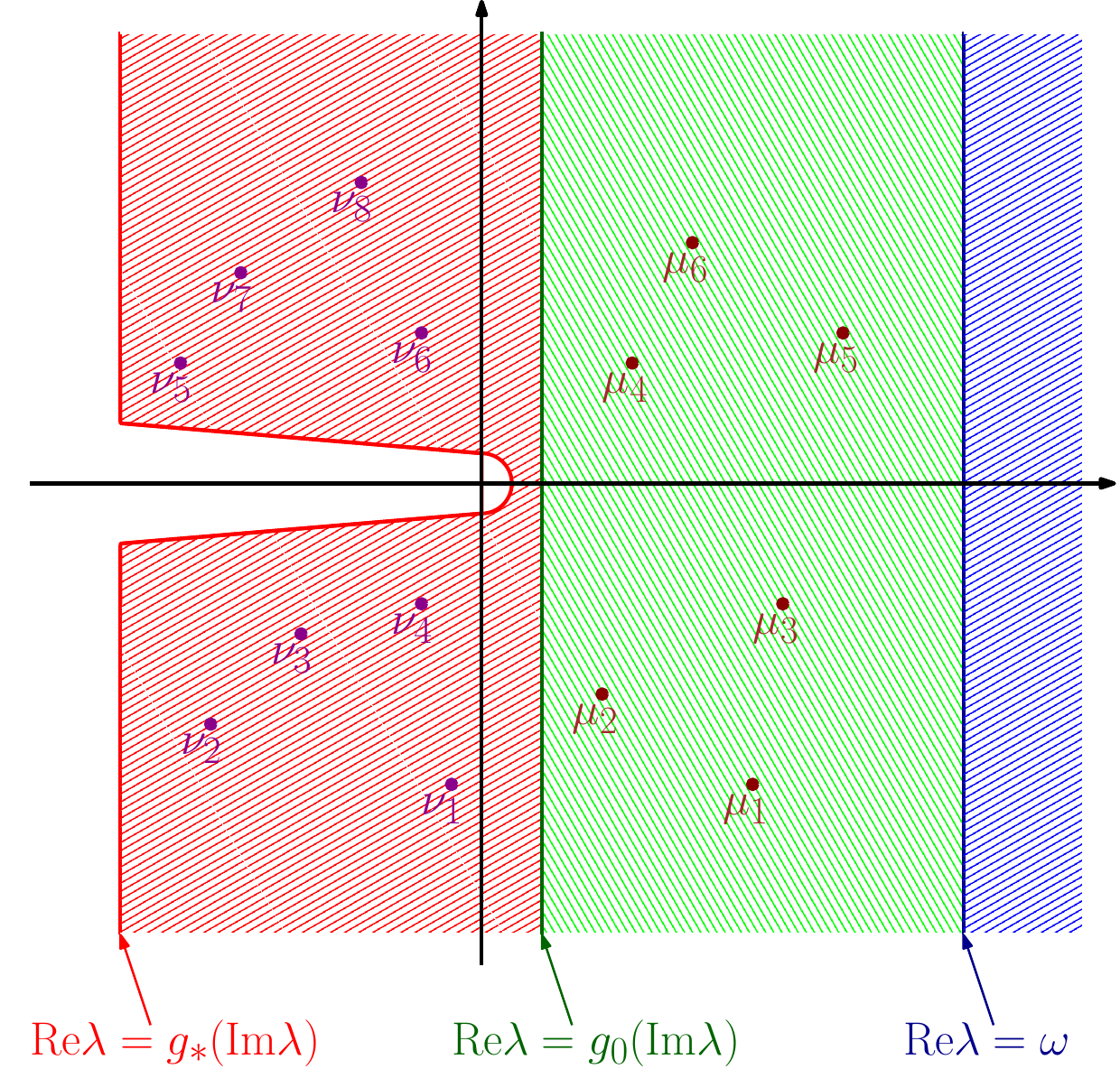}
		
		Figure 6.  The red curve represents the graph of the curve $\re\lambda=g_*(\im\lambda)$. The dark green vertical line is the graph of the constant function $g_0$. 
	\end{center}
\end{figure}
\end{example}

\appendix

\section{Laplace Transform Results}\label{Appendix A}
In this appendix we review certain results used in the Section~\ref{sec2} to invert the Laplace transform of the trajectories of cosine and sine functions using complex integration along vertical lines. Throughout this section we assume that $\bX$ is a Banach space.
\begin{lemma}\label{lAI4}
Let $f\in\mathcal{C}([0,\infty),\bX)$ be a function such that $f(0)=0$ and $\|f(t)\|\leq M_0e^{\omega t}$ for any $t\geq0$. If, in addition, $\displaystyle\int_\RR\|(\cL f)(\gamma+\rmi s)\|\rmd s<\infty$ for any $\gamma>\omega$, then
\begin{enumerate}
\item[(i)] $\displaystyle\int_{Re\lambda=\gamma}e^{\lambda t}(\cL f)(\lambda)\rmd\lambda$ is absolutely convergent for any $\gamma>\omega$, $t\geq 0$;
\item[(ii)] $f(t)=\frac{1}{2\pi\rmi}\displaystyle\int_{Re\lambda=\gamma}e^{\lambda t}(\cL f)(\lambda)\rmd\lambda$ for any $\gamma>\omega$, $t\geq 0$.
\end{enumerate} 	
\end{lemma}	
\begin{proof} (i) Fix $\gamma>\omega$. Since
\begin{equation}\label{lAI4.1}	
\int_{Re\lambda=\gamma}e^{\lambda t}(\cL f)(\lambda)\rmd\lambda=\rmi e^{\gamma t}\displaystyle\int_\RR e^{\rmi st}(\cL f)(\gamma+\rmi s)\rmd s\;\mbox{for any}\; t\geq 0,
\end{equation}
we immediately conclude from the hypothesis that assertion (i) holds true.
	
\noindent (ii) Next, we fix $x^*\in\bX^*$ and define $g_{x^*,\gamma}:\RR\to\CC$ by
\begin{equation}\label{lA4.6}
g_{x^*,\gamma}(t)=\left\{\begin{array}{ll} e^{-\gamma t}x^*(f(t)),& t\geq 0,\\
0,&t<0. \end{array}\right.	
\end{equation}	
Since $f\in\mathcal{C}([0,\infty),\bX)$ and $f(0)=0$ we obtain that $g_{x^*,\gamma}\in\mathcal{C}([0,\infty),\CC)$. Moreover, one can readily check that
\begin{equation}\label{lA4.7}
|g_{x^*,\gamma}(t)\leq M\|x^*\|e^{-(\gamma-\omega)|t|}\;\mbox{for any}\;t\in\RR,
\end{equation}
which implies that, in addition, $g_{x^*,\gamma}\in L^1(\RR)\cap L^2(\RR)$. Taking Fourier Transform, we have
\begin{equation}\label{lA4.8}
(\cF g_{x^*,\gamma})(\xi)=\int_\RR e^{-\rmi\xi t}g_{x^*,\gamma}(t)\rmd t=x^*\Bigg(\int_0^\infty e^{-(\gamma+\rmi\xi)t}f(t)\rmd t\Bigg)=x^*\Big((\cL f)(\gamma+\rmi\xi)\Big)
\end{equation}	
for any $\xi\in\RR$. Since $(\cL f)(\gamma+\rmi\cdot)\in L^1(\RR)$ and $x^*$ from \eqref{lA4.8} we infer that $\cF g_{x^*,\gamma}\in L^1(\RR)$. Since $g_{x^*,\gamma}\in\mathcal{C}([0,\infty),\CC)$ it follows that
\begin{equation}\label{lA4.9}
g_{x^*,\gamma}(t)=\frac{1}{2\pi}\int_\RR e^{\rmi\xi t}(\cF g_{x^*,\gamma})(\xi)\rmd\xi=x^*\Bigg(\frac{1}{2\pi}\int_\RR e^{\rmi\xi t}(\cL f)(\gamma+\rmi\xi)\rmd\xi\Bigg)
\end{equation}	
for any $t\in\RR$. Multiplying by $e^{\gamma t}$ in \eqref{lA4.9}, and since $x^*\in\bX^*$ was chosen random, we infer that assertion (ii) follows immediately by using \eqref{lAI4.1} and \eqref{lA4.6}.
\end{proof}
\begin{lemma}\label{lAI5}
Let $f\in L_{\mathrm{loc}}^1([0,\infty),\bX)$ be a function such that $\|f(t)\|\leq M_1e^{\omega_1 t}$ for any $t\geq0$, for some $M_1>0$ and $\omega_1\in\RR$, $\phi\in\mathcal{C}([0,\infty),\CC)$ be such that $|\phi(t)|\leq M_2e^{\omega_2 t}$ for any $t\geq0$, for some $M_2>0$ and $\omega_2\in\RR$. If, in addition,
$\displaystyle\int_\RR|(\cL\phi)(\gamma+\rmi s)|\rmd s<\infty$ for any $\gamma>\max\{\omega_1,\omega_2\}$, then
\begin{enumerate}
\item[(i)] $\displaystyle\int_{Re\lambda=\gamma}e^{\lambda t}(\cL\phi)(\lambda)(\cL f)(\lambda)\rmd\lambda$ is absolutely convergent for any $\gamma>\max\{\omega_1,\omega_2\}$, $t\geq 0$;
\item[(ii)] $(\varphi*f)(t)=\frac{1}{2\pi\rmi}\displaystyle\int_{Re\lambda=\gamma}e^{\lambda t}(\cL\phi)(\lambda)(\cL f)(\lambda)\rmd\lambda$ for any $\gamma>\max\{\omega_1,\omega_2\}$, $t\geq 0$.
\end{enumerate} 	
\end{lemma}
\begin{proof}
The main idea of the proof is to apply Lemma~\ref{lAI4} to the function $g=\phi*f$. We denote by $\omega_0=\max\{\omega_1,\omega_2\}$ and fix $\gamma>\omega_0$ and $\omega\in (0,\omega_0)$. Since $\phi\in\mathcal{C}([0,\infty),\CC)$, $f\in L_{\mathrm{loc}}^1([0,\infty),\bX)$ we infer that $g=\phi*f\in\mathcal{C}([0,\infty),\bX)$. Moreover,
\begin{align}\label{lAI5.1}
\|g(t)\|&=\Bigl\|\int_0^t\phi(t-s)f(s)\rmd s\Bigr\|\leq \int_0^t|\phi(t-s)|\|f(s)\|\rmd s\leq M_1M_2\int_0^t	e^{\omega_2(t-s)}e^{\omega_1 s}\rmd s\nonumber\\
&\leq M_1M_2te^{\omega_0 t}\leq M_1M_2 e^{\omega_0 t}\frac{1}{\omega-\omega_0}e^{(\omega-\omega_0)t}=\frac{M_1M_2}{\omega-\omega_0}e^{\omega t}\;\mbox{for any}\;t\geq 0.
\end{align}	
Moreover,  since $\|f(t)\|\leq M_1e^{\omega_1 t}$ for any $t\geq0$, we have that
\begin{equation}\label{lAI5.2}
\|(\cL f)(\lambda)\|\leq\frac{M_1}{\re\lambda-\omega_1}\;\mbox{for any}\;\lambda\in\CC_{\omega_1}^+,
\end{equation}
Since $\cL g=(\cL\phi)(\cL f)$ from \eqref{lAI5.2} we estimate
\begin{equation}\label{lAI5.3}
\int_\RR \|(\cL g)(\gamma+\rmi s)\|\rmd s=\int_\RR \|(\cL\phi)(\gamma+\rmi s)(\cL f)(\gamma+\rmi s)\|\rmd s\leq\frac{M_1}{\gamma-\omega_1}\int_\RR|(\cL\phi)(\gamma+\rmi s)|\rmd s<\infty. 	
\end{equation}
The lemma follows from \eqref{lAI5.1}, \eqref{lAI5.3}, Lemma~\ref{lAI4} and since $g\in\mathcal{C}([0,\infty),\bX)\subset L_{\mathrm{loc}}^1([0,\infty),\bX)$.
\end{proof}
\begin{remark}\label{rAI6}
If in Lemma~\ref{lAI5} we drop the assumption $\displaystyle\int_\RR\|(\cL f)(\gamma+\rmi s)\|\rmd s<\infty$ for any $\gamma>\omega$, we cannot infer, in general, the absolute convergence of the integral $\displaystyle\int_{Re\lambda=\gamma}e^{\lambda t}(\cL\phi)(\lambda)(\cL f)(\lambda)\rmd\lambda$. However, under a different assumption on the scalar function $\phi$, one can obtain a similar representation in the principal value sense, see \cite[Lemma 2.1]{Haase}.
\end{remark}
\begin{lemma}\label{lAI7}
Let $f\in\mathcal{C}([0,\infty),\bX)$ be a function such that $\|f(t)\|\leq M_0e^{\omega t}$ for any $t\geq0$, and let $k\in\NN$. Then,
\begin{enumerate}
\item[(i)] $\displaystyle\int_{Re\lambda=\gamma}\frac{e^{\lambda t}}{\lambda^{k+1}}(\cL f)(\lambda)\rmd\lambda$ is absolutely convergent for any $\gamma>\max\{\omega,0\}$, $t\geq 0$;
\item[(ii)] $\displaystyle\int_0^t(t-s)^kf(s)\rmd s=\frac{k!}{2\pi\rmi}\displaystyle\int_{Re\lambda=\gamma}\frac{e^{\lambda t}}{\lambda^{k+1}}(\cL f)(\lambda)\rmd\lambda$ for any $\gamma>\max\{\omega,0\}$, $t\geq 0$.
\end{enumerate} 	
\end{lemma}	
\begin{proof}
Let $\phi_k:[0,\infty)\to\RR$ be the function defined by $\phi_k(t)=t^k$. Clearly $\phi_k\in\mathcal{C}([0,\infty),\RR)$ and
$|\phi_k(t)|=t^k\leq\bigl(\frac{k}{\eps}\bigr)^k e^{\eps t}$ for any $t\geq 0$ and $\eps>0$. Moreover,
\begin{equation*}
\int_\RR |(\cL\phi_k)(\gamma+\rmi s)|\rmd s=\int_\RR\frac{k!}{|\gamma+\rmi s|^{k+1}}\rmd s=\int_\RR\frac{1}{(\gamma^2+s^2)^{\frac{k+1}{2}}}\rmd s<\infty\;\mbox{for any}\;\gamma>0.
\end{equation*}
From Lemma~\ref{lAI5} we have that assertion (i) and (ii) hold true for any $\gamma>\max\{\omega,\eps\}$ for any $\eps>0$. Passing to the limit as $\eps\to0$, the lemma follows shortly.  	 	
\end{proof}

\section{Frechet Space Valued Holomorphic Functions}\label{Appendix B}
In this section we present several results needed in Section~\ref{sec4}. We consider $\bY$ a Frechet space and we denote by $\bY^*$ its dual space. If $\bY_1$ is a closed subspace of $\bY$ we denote by
\begin{equation}\label{def-Y1-0}
\bY_1^{*,0}=\{y^*\in\bY^*:\langle y,y^*\rangle=0\;\mbox{for any}\;y\in\bY_1\}.
\end{equation}
\begin{remark}\label{rB1}
Since in a Frechet space the Hahn-Banach theorem holds, the following statement holds true:
\begin{equation}\label{rB1.1}
\mbox{if}\;y\in\bY\;\mbox{and}\;\langle y,y^*\rangle=0\;\mbox{for any}\;y^*\in\bY_1^{*,0}\;\mbox{then}\;y\in\bY_1.	
\end{equation}		
\end{remark}	
\begin{lemma}\label{lB2}
Let $\Omega\subseteq\CC$ be an open, connected set, $f:\Omega\to\bY$ an analytic function and $\bY_1$ a closed subspace of $\bY$. If the set $\{\lambda\in\Omega:f(\lambda)\in\bY_1\}$ has an accumulation point in $\Omega$ then $f(\lambda)\in\bY_1$ for any $\lambda\in\Omega$.	
\end{lemma}	
\begin{proof}
Fix $y^*\in\bY_1^{*,0}$. Then, the function defined by $h_{y^*}:\Omega\to\CC$ defined by $h_{y^*}(\lambda)=\langle f(\lambda), y^*\rangle$ is analytic on $\Omega$. Moreover,
\begin{equation}\label{lB2.1}
\{\lambda\in\Omega:f(\lambda)\in\bY_1\}\subseteq\{\lambda\in\Omega:g_{y^*}(\lambda)=0\}	
\end{equation}
which implies that $\{\lambda\in\Omega:g_{y^*}(\lambda)=0\}$ has an accumulation point in $\Omega$. Hence, $\langle f(\lambda), y^*\rangle=g_{y^*}(\lambda)=0$ for any $\lambda\in\Omega$. From \eqref{rB1.1} we infer that $f(\lambda)\in\bY_1$ for any $\lambda\in\Omega$, proving the lemma. 	
\end{proof}
\begin{lemma}\label{lB3}
Let $\Omega\subseteq\CC$ be an open, connected set, $\Phi$ a nonempty, open subset of $\Omega$, $f:\Omega\to\bY$ an analytic function and $\bY_1$ a closed subspace of $\bY$. If  $f(\lambda)\in\bY_1$ for any $\lambda\in\Phi$ then $f(\lambda)\in\bY_1$ for any $\lambda\in\Omega$.	
\end{lemma}	
\begin{proof}
First, we note that $\Phi$ has accumulation points in $\Omega$. Since $\Phi\subseteq\{\lambda\in\Omega:f(\lambda)\in\bY_1\}\}$, we infer that the set $\{\lambda\in\Omega:f(\lambda)\in\bY_1\}$ has an accumulation point in $\Omega$. The lemma follows shortly from Lemma~\ref{lB2}.
\end{proof}

\end{document}